\documentclass[12pt,a4paper,twoside,reqno]{amsart}
\allowdisplaybreaks

\usepackage[all,cmtip]{xy}
\usepackage[T1]{fontenc}
\usepackage{amsmath,amscd}
\usepackage{amssymb,amsfonts}
\usepackage[driver=pdftex,margin=3cm,heightrounded=true,centering]{geometry}
\usepackage{mathtools}
\usepackage{tensor}
\usepackage{url}
\usepackage[colorlinks=true,linkcolor=blue,citecolor=blue]{hyperref}
\usepackage{slashed}

\usepackage{graphicx}

\usepackage{a4wide}

\usepackage{thmtools}
\declaretheorem[style=theorem,name={Theorem}]{theoremletter}

\newtheorem{introcorollary}[theoremletter]{Corollary}

\usepackage{amsthm}

\theoremstyle{plain}

\newtheorem{theorem}{Theorem}[subsection]
\newtheorem*{thm*}{Theorem}
\newtheorem{prop}[theorem]{Proposition}
\newtheorem{lemma}[theorem]{Lemma}
\newtheorem{cor}[theorem]{Corollary}

\newtheorem{assu}[theorem]{Assumption}

\theoremstyle{definition}
\newtheorem{dfn}[theorem]{Definition}

\theoremstyle{remark} 

\newtheorem{remark}[theorem]{Remark}

\theoremstyle{plain}

\usepackage{amscd}

\numberwithin{equation}{section}

\newcommand{\alpheqn}[1][\relax]{
     \refstepcounter{equation}
     \if#1\relax \relax
       \else \label{#1}
     \fi  
     \setcounter{saveeqn}{\value{equation}}%
    \setcounter{equation}{0}%
    \renewcommand{\theequation}{\thealphequation}}
\newcommand{\reseteqn}{\setcounter{equation}{\value{saveeqn}}%
     \renewcommand{\theequation}{\thearabicequation}}

\providecommand{\mathscr}{\mathcal} 
\IfFileExists{mathrsfs.sty}{
\usepackage{mathrsfs}}         
   {
     \IfFileExists{eucal.sty}{
        \usepackage[mathscr]{eucal}} 
   {
   }
}

\newcommand{\tens}{\otimes}
\newcommand{\Lip}{\operatorname{Lip}}

\newcommand{\sa}{{\operatorname{sa}}}

\newcommand{\vertiii}[1]{{\left\vert\kern-0.25ex\left\vert\kern-0.25ex\left\vert #1 
    \right\vert\kern-0.25ex\right\vert\kern-0.25ex\right\vert}}
\newcommand{\Bvert}[1]{{\Big\vert\kern-0.25ex\Big\vert\kern-0.25ex\Big\vert #1 
    \Big\vert\kern-0.25ex\Big\vert\kern-0.25ex\Big\vert}}
\newcommand{\bvert}[1]{{\big\vert\kern-0.25ex\big\vert\kern-0.25ex\big\vert #1 
    \big\vert\kern-0.25ex\big\vert\kern-0.25ex\big\vert}}
\newcommand{\nvert}[1]{{\vert\kern-0.25ex\vert\kern-0.25ex\vert #1 
    \vert\kern-0.25ex\vert\kern-0.25ex\vert}}

\renewcommand{\leq}{\leqslant}
\renewcommand{\geq}{\geqslant}

\newcommand{\cd}{\cdot}

\newcommand{\ot}{\otimes}
\newcommand{\hot}{\widehat \otimes}

\newcommand{\op}{\oplus}

\newcommand{\plp}{+ \ldots +}

\newcommand{\ci}{\circ}

\newcommand{\ti}{\times}
\newcommand{\nn}{\mathbb{N}}
\newcommand{\zz}{\mathbb{Z}}

\newcommand{\rr}{\mathbb{R}}
\newcommand{\cc}{\mathbb{C}}

\newcommand{\al}{\alpha}
\newcommand{\be}{\beta}
\newcommand{\ga}{\gamma}
\newcommand{\Ga}{\Gamma}
\newcommand{\de}{\delta}
\newcommand{\De}{\Delta}
\newcommand{\ep}{\varepsilon}

\newcommand{\io}{\iota}

\newcommand{\la}{\lambda}
\newcommand{\La}{\Lambda}
\newcommand{\Na}{\nabla}

\newcommand{\si}{\sigma}

\newcommand{\te}{\theta}

\newcommand{\ze}{\zeta}
\newcommand{\pa}{\partial}

\newcommand{\ov}{\overline}
\newcommand{\C}[1]{\mathcal{#1}}
\newcommand{\G}[1]{\mathfrak{#1}}
\newcommand{\T}[1]{\textup{#1}}

\newcommand{\E}[1]{\emph{#1}}
\newcommand{\B}[1]{\mathbb{#1}}

\newcommand{\fork}[2]{\left\{ \begin{array}{#1} #2 \end{array} \right.} 
\newcommand{\cas}[1]{\begin{cases} #1 \end{cases} }

\newcommand{\ma}[2]{\left(\begin{array}{#1} #2 \end{array} \right)}
\newcommand{\pma}[1]{\begin{pmatrix} #1 \end{pmatrix}}
\newcommand\SmallMatrix[1]{{\tiny\arraycolsep=0.3\arraycolsep\ensuremath{\begin{pmatrix}#1\end{pmatrix}}}}

\newcommand{\su}{\subseteq}

\newcommand{\q}{\qquad}

\newcommand{\wit}{\widetilde}

\newcommand{\inn}[1]{\langle #1 \rangle}

\newcommand{\binn}[1]{\big\langle #1 \big\rangle}

\newcommand{\sem}{\setminus}

\newcommand{\comm}[3]{\tensor[_{#1}]{#2}{_{#3}}}
\newcommand{\sma}[1]{{\tiny\arraycolsep=0.3\arraycolsep\ensuremath{\begin{pmatrix}#1\end{pmatrix}}}}

\def\smallint{\begingroup\textstyle \int\endgroup}

\makeatletter
\@namedef{subjclassname@2020}{%
  \textup{2020} Mathematics Subject Classification}
\makeatother

\begin{document}

\author{Jens Kaad}

\address{Jens Kaad, Department of Mathematics and Computer Science, University of Southern Denmark, Campusvej 55, DK-5230 Odense M, Denmark}
\email{kaad@imada.sdu.dk}

\author{David Kyed}
\address{David Kyed, Department of Mathematics and Computer Science, University of Southern Denmark, Campusvej 55, DK-5230 Odense M, Denmark}
\email{dkyed@imada.sdu.dk}



\title{\sc{The quantum metric structure  of quantum $SU(2)$}} 

\begin{abstract}
We introduce a two parameter family of Dirac operators on quantum $SU(2)$ and analyse their properties from the point of view of non-commutative metric geometry. 
It is shown that these Dirac operators give rise to compact quantum metric structures, and that the corresponding two parameter family of compact quantum metric spaces varies continuously in Rieffel's quantum Gromov-Hausdorff distance.  This continuity result includes the classical case where we recover the round 3-sphere up to a global scaling factor on the metric. Our main technical tool is a quantum $SU(2)$ analogue of the Berezin transform, together with its associated fuzzy approximations, the analysis of which 
also leads to a   systematic way of approximating Lipschitz operators by means of polynomial expressions in the generators.
\end{abstract}

\maketitle

\tableofcontents

\section{Introduction}
Over the past century, the theory of operator algebras has given rise to an abundance of non-commutative analogues of classical mathematical theories, many of which have grown into successful independent research areas with exciting applications in theoretical physics.
Prominent examples of this phenomenon include the theory of quantum groups \cite{kustermans-vaes-C*-lc, wor:cpqgrps},   non-commutative geometry \`a la Connes \cite{CoMo:LIF,Con:NDG,Con:NCG,Con:GFN}, and Rieffel's theory of  quantum metric spaces  \cite{Rie:MSS, Rie:GHD},  which constitute far reaching non-commutative generalisations of classical topological groups, Riemannian (spin) manifolds and compact metric spaces, respectively.
Despite a continuous effort, it has proven very challenging to reconcile the theory of non-commutative geometry with the theory of quantum groups \cite{CoMo:TST, Maj:QNG}, and even for the most fundamental $q$-deformation, Woronowicz' $SU_q(2)$, it is not clear how one should modify Connes' axioms to obtain a non-commutative geometry which  adequately reflects the underlying $q$-geometry. Numerous candidates for Dirac operators on $SU_q(2)$ have been proposed \cite{BiKu:DQQ, CP:EST, DLSSV:DOS, KaSe:TST,  KRS:RFH, KW:TDO, NeTu:DCQ}, each having their advantages and disadvantages, but it seems unclear which (if any) of these provide $SU_q(2)$ with the right kind of non-commutative geometry. At the time of writing, it is not even known if any of these Dirac operators give rise to a  compact quantum metric structure, so also the connection between the metric geometry and the differential geometry on $SU_q(2)$ is open\footnote{Note that one may construct a Dirac operator from a length function on the dual $\widehat{SU_q(2)}$ \textcolor{black}{whose iterated commutators} 
does give rise to a compact quantum metric space \cite[Proposition 4.3 and Theorem 7.4]{BCZ:CQM}, but this construction seems to be  less related to the spin geometry of $SU_q(2)$. Indeed, for $q = 1$ the resulting Dirac operator is very different from the classical Dirac operator since the $K$-homology class of the Dirac operator coming from a length function is trivial.}. 

The first aim of the present paper is to remedy the latter problem by introducing a family of Dirac operators on $SU_q(2)$ and showing that these give rise to compact quantum metric structures.  More precisely, we investigate a new 2-parameter family of Dirac operators $D_{t,q}$,  indexed over $(0,1]\times (0,1]$, which connect some of the existing constructions in that $D_{q,q}$ agrees with the Dirac operator suggested in \cite{KaSe:TST} and $D_{1,q}$ is closely related to the one studied in \cite{KRS:RFH}.
 It is important to stress that $D_{t,q}$ does not have bounded commutators with the coordinate  algebra $\C O(SU_q(2))$ in general, and we therefore cannot obtain a genuine spectral triple.    It does, however, decompose naturally into a  ``horizontal'' part $D_{q}^H$  and ``vertical'' part $D_{t}^V$,  each of which admits bounded \emph{twisted} commutators  (though for different twists)  with elements from  $\C O(SU_q(2))$. The appearance of such twists seems inevitable when non-commutative geometry is applied in the context of $q$-deformed spaces \cite{CoMo:TST}, and even though our constructions do not fit exactly with Connes' original axioms for a spectral triple, many of the  properties of classical spectral triples admit suitable analogues in our twisted setting, as witnessed by the following result:

\begin{theoremletter}[See  Lemma \ref{l:analytic}, \ref{l:sumdirac}, \ref{l:twicommu}, \ref{l:Irela}, \ref{l:multiunidir}, Proposition \ref{p:firstorder} and Section \ref{sec:comparison-with-classical}]\label{introthm:dirac-properties}
 The Dirac operators $D_{t,q}= D_t^V + D^H_q$ are selfadjoint and the following hold:
\begin{enumerate}
\item 
There exists a  one-parameter family $(\sigma_r)_{r\in  (0,\infty)}$ of algebra automorphisms of $\C O(SU_q(2))$ such that the  twisted commutators
\[
D_t^V \sigma_t(x)-\sigma_t^{-1}(x)D_t^V \q \mbox{and} \q D_q^H\sigma_q(x) -\sigma_q^{-1}(x)D_q^H,
\]
extend to bounded operators  $\pa_t^V(x)$ and $\pa_q^H(x)$ for all $x\in \C O (SU_q(2))$.  Moreover, there exists a one-parameter family of unbounded, strictly positive operators $(\Ga_r)_{r \in (0,\infty)}$ satisfying that $\sigma_r(x) = \ov{\Ga_r^{-1} x \Ga_r}$ for all $r \in (0,\infty)$ and all $x \in \C O(SU_q(2))$. 

\item The Dirac operators $D_{t,q}$ are $SU_q(2)$-equivariant, in the sense that the  selfadjoint unbounded operators $1\hot D_{t}^V$ and $1\hot D_{q}^H$ on $(L^2(SU_q(2))\hot L^2(SU_q(2)))^{\oplus 2}$  commute with $W\oplus W$ where $W\in \mathbb{B}(L^2(SU_q(2)) \hat{\otimes} L^2(SU_q(2)) )$ denotes the multiplicative unitary for $SU_q(2)$. 

\item There exists an  antilinear unitary $I$ with $I^2=-1$ such that $D_{t,q}$ satisfies the first order condition  $[\pa_t^V(x), IyI]=0 = [\pa_q^H(x),IyI]$ for all $x,y\in \C O(SU_q(2))$.   Moreover, $I$ commutes with $D_{t,q}$ up to modular operators in the sense that the relations
\[
D_t^V \Ga_t^{-1} \cd I = I \cd D_t^V \Ga_t^{-1} \q \mbox{and} \q  D_q^H \Ga_q^{-1} \cd I = I \cd D_q^H \Ga_q^{-1} 
\]
hold on the dense subspace $\C O(SU_q(2))^{\op 2} \su L^2(SU_q(2))^{\op 2}$.
\item When $t=q=1$, the unbounded selfadjoint operator $D_{1,1}$ satisfies that $2\cdot D_{1,1} +1 =D_{S^3}$, where $D_{S^3}$ is the classical Dirac operator on $SU(2) \cong S^3$. 
\end{enumerate}
\end{theoremletter}

As already mentioned, one of the main \textcolor{black}{goals} of the paper is to investigate the metric geometry governed by the Dirac operators $D_{t,q}$, by connecting our construction to Rieffel's theory of compact quantum metric spaces \cite{Rie:MSS, Rie:GHD}.  The data defining a compact quantum metric space consists of a unital $C^*$-algebra $A$ (or, more generally, an operator system) endowed with a densely defined  seminorm $L$, and the central requirement is that the Monge-Kantorovi\v{c} extended\footnote{The adjective \emph{extended} here means that $d_L$ is apriori allowed to take the value infinity. Note, however, that by compactness this cannot be the case if $d_L$ metrises the weak$^*$ topology.} metric $d_L$ on the state space $\C{S}(A)$, defined as
\begin{align}\label{eq:connes-metric}
d_{L}(\mu, \nu):= \sup \{ |\mu(a)-\nu(a)|  \, \mid \,  L(a)\leq 1\},
\end{align}
metrises the weak$^*$ topology.  The motivating example of course comes from taking a classical compact metric space $(M,d)$ and associating to it the seminorm 
\[
L_{\T{Lip}}(f):= \sup\left\{\frac{|f(x)-f(y)|}{d(x,y)} \mid x,y\in M, x\neq y \right\} 
\]
defined on the Lipschitz functions $\T{Lip}(M)$.  Rieffel's definition is also very much inspired by constructions appearing in non-commutative geometry; see \cite{Con:CFH}. Indeed, one can associate a natural seminorm  $L_D$ to a  unital spectral triple  $(A, H, D)$ by setting
\begin{align}\label{eq:Dirac-seminorm}
L_D(a):=\big\|\ov{ [D,a] }\big\|,
\end{align}
 whenever the element $a$ belongs to a specified dense unital $*$-subalgebra $\C A \su A$ of ``differentiable'' operators. Note, however,  that this construction does not always yield a compact quantum metric space, and even in the cases where this happens, the argument is often far from trivial; see e.g.~\cite{AgKa:PSM, OzRi:Hyp, Rieffel:group-C-star-algebras-as-cqms}. \\

Having the Dirac operators $D_{t,q}$ at our disposal we obtain a family of seminorms $L_{t,q}\colon \C O(SU_q(2))\to [0,\infty)$ by setting 
 \[
L_{t,q}(x) := \big\| \pa_t^V(x) + \pa_q^H(x)\big\| 
\]
and we may ask whether they yield compact quantum metric structures on $C(SU_q(2))$. Since the seminorm $L_{t,q}$ comes from two unbounded selfadjoint operators via a twisted commutator construction we may enlarge the domain considerably and replace the coordinate algebra $\C O(SU_q(2))$ with a much larger algebra of Lipschitz elements $\T{Lip}_t(SU_q(2))$. We are considering the resulting seminorm $L_{t,q}^{\max} \colon \T{Lip}_t(SU_q(2)) \to [0,\infty)$ as the ``maximal'' seminorm associated with our spectral data whereas $L_{t,q}$ is regarded as the corresponding ``minimal'' seminorm. In analogy with the classical case, there is a wide gap between the Hopf algebra $\C O(SU_q(2))$ of polynomial expressions in the generators and the Lipschitz algebra $\T{Lip}_t(SU_q(2))$, in so far that the intersection of the domain of the closures of the twisted derivations $\pa_t^V$ and $\pa_q^H$ on the coordinate algebra does not agree with the Lipschitz algebra $\T{Lip}_t(SU_q(2))$; see for example \cite[Theorem 3.1]{GKK:QI}. One may compare the gap between the minimal and maximal seminorms to the gap between the unital $C^*$-algebra $C(SU_q(2))$ and its weak closure, the von Neumann algebra $L^\infty(SU_q(2))$. We are in this text presenting a thorough treatment of the maximal seminorms $L_{t,q}^{\max}$ and this is partly the reason for the appearance of a number of analytic challenges.
%

\begin{theoremletter}[see Theorem \ref{thm:quantum-su2-as-cqms}]\label{introthm:SUq2-is-a-cqms}
The pair $\big(C(SU_q(2)), L_{t,q}^{\max}\big)$ is a compact quantum metric space for all $t,q\in (0,1]$.
\end{theoremletter}
Enlarging the domain of a seminorm increases the difficulty of proving that it defines a compact quantum metric structure, so  Theorem \ref{introthm:SUq2-is-a-cqms} immediately implies the corresponding statement for the minimal seminorm $L_{t,q}$.
 \begin{introcorollary}[see  Corollary \ref{cor:alg-seminorm-gives-cqms}]
 The pair $\big(C(SU_q(2)), L_{t,q}\big)$ is a compact quantum metric space for all $t,q\in (0,1]$.
 \end{introcorollary}

 To prove Theorem \ref{introthm:SUq2-is-a-cqms}, we develop a set of new general tools which are likely to have applications elsewhere, and we therefore  briefly outline the main ideas involved. The central ingredient is the Podle{\'s} sphere $C(S_q^2)$ \cite{Pod:QS}, which  arises as the fixed point algebra of a certain circle action on $C(SU_q(2))$ (providing a quantised analogue of the Hopf fibration). In contrast to $SU_q(2)$, the  non-commutative metric geometry of $S_q^2$ is reasonably well understood. The work of D\c{a}browski and Sitarz \cite{DaSi:DSP} provides the Podle\'s sphere with a unital spectral triple, and it was furthermore proven in \cite{AgKa:PSM}  that $C(S_q^2)$ becomes a compact quantum metric space when equipped with the corresponding seminorm from \eqref{eq:Dirac-seminorm}. The circle action defining the Podle{\'s} sphere also gives rise to an increasing sequence of finitely generated projective modules which  suitably exhaust $C(SU_q(2))$.  These finitely generated projective modules are direct sums of spectral subspaces for the circle action and are referred to as spectral bands. The first step in proving Theorem \ref{introthm:SUq2-is-a-cqms} is  to lift the compact quantum metric structure from $C(S_q^2)$ to the  spectral bands and we develop the general theory to achieve this in Section \ref{ss:fingenproCQMS}. The second step is then to lift the compact quantum metric structure from the spectral bands all the way up to $C(SU_q(2))$. Perhaps a bit surprisingly, the main aid here comes from the theory of Schur multipliers, and we unfold this aspect in Section \ref{ss:schur}.\\

One of the main virtues of Rieffel's theory of compact quantum metric spaces, is that it allows for a natural generalisation of the classical Gromov-Hausdorff distance between compact metric spaces \cite{edwards-GH-paper, gromov-groups-of-polynomial-growth-and-expanding-maps}, naturally dubbed the quantum Gromov-Hausdorff distance \cite{Rie:GHD}. This concept has been further developed by, among others,  Kerr \cite{Ker:MQG}, Li \cite{Li:CQG, Li:GH-dist, Li:ECQ} and   Latr\'emoli\`ere \cite{Lat:DGH, Lat:QGH, Lat:MGP, Lat:GPS}, and by now exists in several different versions which take into account  more structure than Rieffel's original definition. 
The existence of such a distance function allows one to study the class of compact quantum metric spaces from a more analytical point of view, and opens the possibility to investigate a wealth of natural continuity questions. 
Over the past two decades, many positive answers have been obtained, and examples include Rieffel's fundamental result that the $2$-sphere can be approximated by the fuzzy spheres (matrix algebras) \cite{Rie:MSG}, as well as the more recent proof \cite{AKK:Podcon} that the Podle\'s spheres $S_q^2$ vary continuously in the deformation parameter  $q\in (0,1]$; for many more examples see  \cite{aguilar:thesis,  KK:DCQ,  Lat:AQQ, LatPack:Solenoids, Li:ECQ, Rie:GHD}.\\
  
In light of Theorem \ref{introthm:SUq2-is-a-cqms}, the next natural question to ask is whether one obtains quantum Gromov-Hausdorff continuity in the deformation parameters $t$  and $q$, and through a series of approximation arguments we are able to answer this in the affirmative:

\begin{theoremletter}[See Theorem \ref{thm:continuity-of-quantum-su2}]\label{introthm:qgh-continuity}
The compact quantum metric spaces $\big(C(SU_q(2)),  L_{t,q}^{\max}\big)$ vary continuously in the deformation parameter $(t,q)\in (0,1]\times (0,1]$ with respect to the quantum Gromov-Hausdorff distance.
\end{theoremletter}

We single out the following special case of Theorem \ref{introthm:qgh-continuity}, which was the original motivation for the  study undertaken in the present paper. Denoting by $d_{S^3}$ the usual round metric on $SU(2) \cong S^3\subseteq \rr^4$ and by $L_{\Lip}\colon  \Lip(SU(2)) \to [0,\infty)$ the Lipschitz constant seminorm on $C(SU(2))$ associated with the rescaled metric $2\cdot d_{S^3}$, combining Theorem \ref{introthm:qgh-continuity} and Theorem \ref{introthm:dirac-properties} yields the following:

\begin{introcorollary}[see Corollary \ref{cor:convergence-to-classical-SU2}]\label{introcor:convergence-to-classical}
The  compact quantum metric spaces  $\big(C(SU_q(2)),  L_{t,q}^{\max}\big)$ converge in quantum Gromov-Hausdorff distance to $\big(C(SU(2)), L_{\T{Lip}}\big)$ as $(t,q)$ tends to $(1,1)$.
\end{introcorollary}
The rescaling of  the metric on $S^3$ may at first sight seem strange, but it is exactly this  factor of 2 which  makes the Hopf fibration $S^3 \to S^2$ a Riemannian submersion when the 2-sphere is endowed with its round metric arising from the natural embedding into $\rr^3$. \\

The road to Theorem  \ref{introthm:qgh-continuity} is quite long, but involves a number of constructions which are of independent interest. As in the case of the Podle{\'s} sphere \cite{AKK:Podcon}, the key to such a continuity result is to construct an $SU_q(2)$ version of the Berezin transform. By means of the Berezin transform we obtain finite dimensional compact quantum metric spaces $ \T{Fuzz}_N(B_q^K) \subseteq \C O(SU_q(2))$  indexed by $N,K \in \nn_0$. We think of these compact quantum metric spaces as \emph{fuzzy spectral bands}. These fuzzy spectral bands are $\C O(SU_q(2))$-coinvariant and it is possible to describe them explicitly in terms of the usual generators for $SU_q(2)$. In Section \ref{sec:berezin} we construct our Berezin transform and prove that the fuzzy spectral bands approximate $SU_q(2)$:


\begin{theoremletter}[see Corollary \ref{cor:fuzzy-to-quantum-SU2}]\label{introthm:fuzzy-approx}
The compact quantum metric spaces  $\big(\T{Fuzz}_N(B_q^K), L_{t,q}\big)$ converge in quantum Gromov-Hausdorff distance to $\big(C(SU_q(2)), L_{t,q}^{\max}\big)$  as $N$ and $K$ tend to infinity.
\end{theoremletter} 
 
 This theorem should be viewed  as an $SU_q(2)$-analogue of Rieffel's original result \cite{Rie:MSG}, showing that the 2-sphere can be approximated in quantum Gromov-Hausdorff distance by the fuzzy 2-spheres (matrix algebras). The concrete techniques used in the construction of the Berezin transform and the fuzzy spectral bands build on the corresponding constructions for the Podle{\'s} sphere developed in \cite{AKK:Podcon}. An interesting consequence of the above fuzzy approximation is that the maximal and minimal seminorm actually give rise to the same compact quantum metric structure on $SU_q(2)$:

\begin{theoremletter}[see Corollary \ref{cor:dist-zero}]\label{introtheorem:metrics-agree}
The quantum Gromov-Hausdorff distance between  $\big(C(SU_q(2)), L_{t,q}^{\max}\big)$ and $\big(C(SU_q(2)), L_{t,q}\big)$ is zero, and the  Monge-Kantorovi\v{c} metrics $d_{t,q}^{\max}$ and $d_{t,q}$ on $\C S\big(C(SU_q(2))\big)$, induced by the two seminorms via the formula \eqref{eq:connes-metric}, agree. 
\end{theoremletter}

In particular, the continuity results in Theorem \ref{introthm:qgh-continuity} and Corollary \ref{introcor:convergence-to-classical}, which pertain to the  \emph{maximal} seminorm   $L_{t,q}^{\max}$ automatically hold true for the  \emph{minimal} seminorm $L_{t,q}$:


\begin{introcorollary}
The compact quantum metric spaces  $\big(C(SU_q(2)), L_{t,q}\big)$ vary continuously in the deformation parameters $(t,q)\in (0,1]\times (0,1]$ with respect to the quantum Gromov-Hausdorff distance. In particular, $\big(C(SU_q(2)), L_{t,q}\big)$ converges to  $\big(C(SU(2)), L_{\T{Lip}}\big)$ as $(t,q)$ tends to $(1,1)$.
\end{introcorollary}

The rest of paper is structured as follows: Section \ref{sec:qcms} contains the necessary background on compact quantum metric spaces as well as the new tools needed for the present paper. Section \ref{s:quantumsu2} contains a detailed introduction to $SU_q(2)$.  In Section \ref{s:twisspectrip} we introduce our family of Dirac operators and prove Theorem \ref{introthm:dirac-properties}. Section \ref{sec:quantum-metrics-on-quantum-su2} is devoted to proving Theorem \ref{introthm:SUq2-is-a-cqms} and in Section \ref{sec:berezin} we construct the Berezin transform and prove Theorem \ref{introthm:fuzzy-approx} and Theorem \ref{introtheorem:metrics-agree}. The final Section \ref{sec:continuity-results} pieces everything together into a proof of the main continuity result, Theorem \ref{introthm:qgh-continuity}.\\

\subsection{Notation and standing assumptions.}
Unless otherwise stated, we shall always apply the notation $\|\cdot\|$ for the unique $C^*$-norm on a $C^*$-algebra $A$ or, more generally, for its restriction to a complete operator system $X \su A$. Since the greek letter epsilon is the standard symbol both for the counit in a quantum group and an arbitrarily small positive number, we will use the symbol $\epsilon$ for the former and the symbol $\ep$ for the latter. As for tensor products, the symbols $\ot, \ot_{\min}$ and $\hot$ will denote algebraic,  minimal $C^*$-algebraic, and Hilbert space tensor products, respectively. The theory of unbounded operators plays a central role in the paper, and if $T$ is an unbounded closable operator in a Hilbert space we will denote its closure by $\overline{T}$.
Lastly, we will use the abbreviations WOT and SOT for the weak- and strong operator topology, respectively, and ucp for unital completely positive.\\

\subsection{Note added in proof.}
{Since the writing of the present paper, the research on quantum metrics on $q$-deformations has progressed further. In \cite{KM24}, it was proven that the  D'Andrea-D\k{a}browski 
spectral triples provide all quantum projective spaces $\cc P_q^\ell$ with compact quantum metric structures (the case $\ell=1$ corresponding to the Podle{\'s} sphere), and in \cite{Kaa24} the higher-dimensional Vaksman-Soibelman spheres were treated, thus providing a generalisation of Theorem \ref{introthm:SUq2-is-a-cqms}}. {The paper \cite{Kaa24} also features an updated treatment of finitely generated projective modules in the context of quantum metric spaces.}

\subsubsection*{Acknowledgments.}
The authors gratefully acknowledge the financial support from  the Independent Research Fund Denmark through grant no.~9040-00107B, 7014-00145B and 1026-00371B {and from the EU Staff Exchange project 101086394}. Moreover, thanks are due to Konrad Aguilar, since parts of the present text are based on ideas developed in connection with our joint papers \cite{AKK:Podcon} and \cite{AKK:Polyapprox}. Lastly, the authors would like to thank Ulrich Kr{\"a}hmer and  Adam Rennie for sharing their insights on Dirac operators on $q$-deformed spaces, and Marc Rieffel and Sergey Neshveyev for providing them with the important references \cite{Sain:Thesis} and \cite{IzNeTu:PBD}.  The first author would also like to thank Andreas Thom and the Technische Universit\"at Dresden for hospitality during the autumn of 2021 where a large part of the present paper was written. \\

\section{Compact quantum metric spaces}\label{sec:qcms}

In this section, we present the relevant preliminaries on compact quantum metric spaces.  For our purposes,  the theory of (concrete) operator systems provides the most convenient framework for studying compact quantum metric spaces,  and we are thus in line with the recent developments in  \cite{walter-connes:truncations} and \cite{walter:GH-convergence}, as well as the $C^*$-algebra based approaches in \cite{Li:CQG, Li:GH-dist, Rie:MSS}. The theory discussed here is also closely related to  Rieffel's original theory of \emph{order unit compact quantum metric spaces} \cite{Rie:GHD}, via the passage from an operator system to its selfadjoint part (the real subspace of selfadjoint elements). The selfadjoint part of an operator system is indeed an order unit space and the two state spaces can be identified via restriction.\\

\subsection{Definitions and basic properties}
Throughout this section, $X$ will be a \emph{complete operator system}; i.e.~$X$ will be a norm-closed subspace of a specified unital $C^*$-algebra $A_X$ such that $X$ is invariant under the adjoint operation and contains the unit from $A_X$. A \emph{state} on $X$ is a positive linear functional $\mu \colon X \to \cc$ which sends the unit $1_X$ in $X$ to the unit $1$ in $\cc$. A state on $X$ automatically has norm $1$ \cite{Pau:CBM}, and the state space $\C S (X)$  therefore becomes a compact Hausdorff space for the weak$^*$ topology. Although $X$ is not an algebra, any selfadjoint $x\in X$ may still be written as a difference of positive elements from $X$ as
\[
x=\frac12(\|x\|\cdot 1_X +x) -\frac12 (\|x\|\cdot 1_X - x),
\]
and from this it follows that any positive map $\Phi\colon X \to Y$ into another operator system $Y$ satisfies $\Phi(x^*)=\Phi(x)^*$. Lastly, we note the slight subtlety that $\Phi$ need not be a contraction, but that it is bounded with $\|\Phi\|\leq 2\|\Phi(1_X)\|$; see \cite[Proposition 2.1]{Pau:CBM}. If, however, $\Phi$ is \emph{completely} positive then  $\|\Phi\|= \|\Phi\|_{\T{cb}}=  \|\Phi(1_X)\|$. Note also, that if $\Phi$ is unital and positive and $x$ is selfadjoint then $-\|x\| \cd 1_X \leq x\leq \|x\|\cdot 1_X$ so that $\|\Phi(x)\|\leq \|x\|$. As a final observation, we note that if $\Phi\colon X \to Y$ is instead assumed  to be unital and contractive, then $\Phi$ is automatically positive; see e.g.~\cite[Proposition 2.11]{Pau:CBM}. We will apply these observations without further mentioning in the sections to follow.\\

The complete operator system $X$ gives rise to a complete order unit space $X_{\T{sa}} := \big\{ x \in X \mid x = x^*\big\}$, where the order and the unit are inherited from the surrounding unital $C^*$-algebra $A_X$. The order unit space $X_{\T{sa}}$ also has an associated state space $\C S(X_{\T{sa}})$ and we record that the restriction of states yields an affine homeomorphism $\C S(X) \to \C S(X_{\T{sa}})$.  For an arbitrary element $x \in X$ we let $\T{Re}(x)$ and $\T{Im}(x)$ in $X_{\T{sa}}$ denote the real and the imaginary part of $x$. We are interested in metrics on the state space $\C S(X)$ and in particular those metrics which metrise the weak$^*$ topology. As realised by Rieffel, these may be constructed from certain seminorms on the operator system $X$ and we now recall the key notions in this connection.

\begin{dfn}\label{d:lipschitzsem}
A seminorm $L \colon X \to [0,\infty]$ is called  a \emph{Lipschitz seminorm} when the following hold:
\begin{enumerate} 
\item $L$ is \emph{densely defined}, meaning that the domain $\T{Dom}(L) := \big\{ x \in X : L(x) < \infty \big\}$
is a norm-dense subspace of $X$;
\item the kernel of $L$ contains the scalars $\cc := \cc \cd 1_X$, thus $L(1_X) = 0$;
\item $L$ is invariant under the adjoint operation, i.e.~$L(x^*) = L(x)$ for all $x \in X$.
\end{enumerate}
\end{dfn}

It is common to require that the kernel of a Lipschitz seminorm agrees with the scalars $\cc = \cc \cd 1_X$, but we find it convenient to work with the above more flexible notion.

\begin{dfn}
Let $L \colon X \to [0,\infty]$ be a Lipschitz seminorm. The \emph{Monge-Kantorovi\v{c}} metric $d_L \colon \C S(X) \ti \C S(X) \to [0,\infty]$ is defined by
\[
d_L(\mu,\nu) := \sup\big\{ | \mu(x) - \nu(x) | \mid L(x) \leq 1 \big\}, \q \mbox{for } \mu,\nu \in \C S(X) .
\]
\end{dfn}

We remark that the Monge-Kantorovi\v{c} metric $d_L$ is not, strictly speaking, a metric since it can, a priori, take the value infinity. In fact, it can be proved that if $\ker(L)$ contains non-scalar elements, then there exist states $\mu_0$ and $\nu_0$ on $X$ such that $d_L(\mu_0,\nu_0) = \infty$; see for example \cite[Lemma 2.2]{KK:DCQ}. This possibility is excluded when $(X,L)$ is a compact quantum metric space in the following sense:

\begin{dfn}\label{d:CQMS}
Let $L \colon X \to [0,\infty]$ be a Lipschitz seminorm. We say that $(X,L)$ is a \emph{compact quantum metric space} when the Monge-Kantorovi\v{c} metric $d_L$ metrises the weak$^*$ topology on the state space $\C S (X)$. In this case, $L$ is referred to as a \emph{Lip-norm}.
\end{dfn}

\begin{dfn}\label{def:diameter}
For a compact quantum metric space $(X,L)$, the \emph{diameter} is defined as
\[
\T{diam}(X,L):=\T{diam}(\C S(X), d_L):=\sup\{d_L(\mu, \nu) \mid \mu,\nu\in \C S(X) \}.
\]
\end{dfn}

For any norm or seminorm $\vertiii{\cdot}$ on $X$, $x\in X$ and $r\geq 0$ we denote the corresponding open and closed balls as follows:
\begin{align*}
\mathbb{B}_r^{\vertiii{\cdot}}(x) :=\{y\in X \mid \vertiii{x-y}<r\} \, \, \T{ and } \, \, \,  
\overline{\mathbb{B}}_r^{\vertiii{\cdot}}(x) :=\{y\in X \mid \vertiii{x-y}\leq r\} .
\end{align*}
The following convenient characterisation of compact quantum metric spaces can be found in \cite[Theorem 1.8]{Rie:MSA}; here we let $[\  \cd \ ] \colon X \to X/\cc$ denote the quotient map and  $\| \cd \|_{X/\cc}$ denote  the quotient norm on $X/\cc$.

\begin{theorem}[Rieffel]\label{thm:rieffels-criterion}
Let $L \colon X \to [0,\infty]$ be a Lipschitz seminorm. It holds that $(X,L)$ is a compact quantum metric space if and only if the subset $\big[ \, \ov{\B B}_1^L(0) \big] \su X/\cc$ is totally bounded with respect to the quotient  norm $\| \cd \|_{X/\cc}$ on $X/\cc$.
\end{theorem}

\begin{remark}\label{rem:totallybdd}
We recall that a subset of a metric space is said to be \emph{totally bounded} if it can be covered by a finite number of $\ep$-balls for any $\ep>0$. Moreover, if the ambient metric space is complete (as it is the case for $X/\cc$), then a subset is totally bounded if and only if it has compact closure. 
We moreover notice that if $(X,L)$ is a compact quantum metric space, then the intersection $\ov{\B B}_1^{\|\cd \|}(0) \cap \ov{\B B}_1^L(0)$ is totally bounded as a subset of $X$. This follows by applying the isomorphism of Banach spaces $X \to X/\cc \op \cc$ given by $x \mapsto ([x],\mu(x))$, where $\mu \colon X \to \cc$ is a fixed state. 
\end{remark}

Let us now explain the relationship between the above operator system approach to compact quantum metric spaces and Rieffel's approach developed in the context of order unit spaces.  Consider a norm-dense real subspace $V \su X_{\T{sa}}$ satisfying that $1_X \in V$ and let $L^0 \colon V \to [0,\infty)$ be a seminorm with $L^0(1_X) = 0$. We call such a seminorm $L^0$ for an \emph{order unit Lipschitz seminorm}. This data also gives rise to a Monge-Kantorovi\v{c} metric on the state space $\C S(X)$ by putting
\[
d_{L^0}(\mu,\nu) := \sup\big\{ | \mu(x) - \nu(x) | \mid x \in V \, , \, \, L^0(x) \leq 1 \big\} .
\]

\begin{dfn}[Rieffel]\label{d:rieCQMS}
The pair  $(V,L^0)$ is an \emph{order unit compact quantum metric space} when the Monge-Kantorovi\v{c} metric $d_{L^0}$ metrises the weak$^*$ topology on the state space $\C S(X)$.
\end{dfn}

We now wish to relate the two concepts of compact quantum metric spaces given in Definition \ref{d:CQMS} and Definition \ref{d:rieCQMS}. 
To every Lipschitz seminorm $L \colon X \to [0,\infty]$ on the operator system $X$ we associate an order unit Lipschitz seminorm $L_{\T{sa}} \colon \T{Dom}(L)_{\T{sa}} \to [0,\infty)$ by restricting $L$ to the selfadjoint part of the domain $\T{Dom}(L)_{\T{sa}} := X_{\T{sa}} \cap \T{Dom}(L)$. Conversely, to every  order unit Lipschitz seminorm $L^0 \colon V \to [0,\infty)$, we associate a  Lipschitz seminorm $L^0_{\T{os}} \colon X \to  [0,\infty]$ by defining
\[
L^0_{\T{os}}(x) := \fork{cc}{
\sup_{\te \in [0,2 \pi]}L^0\big( \cos(\te)\T{Re}(x) + \sin(\te)\T{Im}(x) \big) & {\T{when }} \T{Re}(x), \T{Im}(x) \in \T{Dom}(L^0) \\
\infty & \T{otherwise}
} .
\]
We record the formula $(L^0_{\T{os}})_{\T{sa}} = L^0$. The relationship between the two notions of compact quantum metric spaces can now be made precise.
%

\begin{prop}\label{p:restrict}
If $L \colon X \to [0,\infty]$ is a Lipschitz seminorm, then we have the identity $d_L = d_{L_{\T{sa}}}$ for the associated Monge-Kantorovi\v{c} metrics on $\C S(X)$. Hence, if $(X,L)$ is a compact quantum metric space, then $(\T{Dom}(L)_{\T{sa}}, L_{\T{sa}})$ is an order unit compact quantum metric space.  Conversely, if $L^0 \colon V \to [0,\infty)$ is an order unit Lipschitz seminorm, then we have the identity $d_{L^0} = d_{L^0_{\T{os}}}$. Hence if $(V,L^0)$ is an order unit compact quantum metric space, then $(X,L^0_{\T{os}})$ is a compact quantum metric space.
\end{prop}
\begin{proof}
Let $L\colon X\to [0,\infty]$ be a Lipschitz seminorm. It clearly holds that $d_{L_{\T{sa}}} \leq d_L$. Let now $\mu,\nu \in \C S(X)$ and consider an element $\xi \in X$ with $L(\xi) \leq 1$. Choose a $\la \in S^1$ such that $\la \cd ( \mu(\xi) - \nu(\xi) ) \in \rr$. Since $L(\T{Re}(\la \cd \xi)) \leq L(\xi) $, we obtain that $\T{Re}(\la \cd \xi) \in \T{Dom}(L)_{\T{sa}}$ and $L_{\T{sa}}\big(\T{Re}(\la \cd \xi)\big) \leq 1$. We may thus estimate as follows:
\[
| \mu(\xi) - \nu(\xi) | = | \mu(\la \cd \xi) - \nu(\la \cd \xi) | = | \mu(\T{Re}(\la \cd \xi)) - \nu(\T{Re}(\la \cd \xi)) |  \leq d_{L_{\T{sa}}}(\mu,\nu) . 
\]
This shows that $d_L \leq d_{L_{\T{sa}}}$ and we may conclude that $d_L = d_{L_{\T{sa}}}$. 
Conversely, suppose that $L^0 \colon V \to [0,\infty)$ is an order unit Lipschitz seminorm. Recall that $(L^0_{\T{os}})_{\T{sa}} = L^0$ and hence $d_{L^0} = d_{L^0_{\T{os}}}$ by the first part of the proposition.
\end{proof}

The following result provides a technical condition for verifying when a pair $(X,L)$ is a compact quantum metric space. The essence of the result is that if $(X,L)$ can be suitably approximated by compact quantum metric spaces, then $(X,L)$ must also be a compact quantum metric space; see Corollary \ref{cor:approx-cor} for the precise statement. In the present text we shall apply this theorem to provide quantum $SU(2)$ with the structure of a compact quantum metric space. 

\begin{theorem}\label{t:cqmsapprox}
Let $L \colon X \to [0,\infty]$ be a Lipschitz seminorm. Suppose that for every $\ep > 0$ there exist an operator system $X_\ep$ equipped with a seminorm $L_\ep \colon X_\ep \to [0,\infty]$ and linear maps $\Phi_\ep \colon X \to X_\ep$ and $\Psi_\ep \colon X_\ep \to X$ such that
\begin{enumerate}
\item The kernel of $L_\ep$ is closed in operator norm and the subset
\[
\big[ \, \ov{\B B}_1^{L_\ep}(0)  \big] \su X_\ep / \ker(L_\ep)
\]
is totally bounded with respect to the quotient operator norm on $X_\ep / \ker(L_\ep)$;
\item We have the inclusion $\Psi_\ep( \ker(L_\ep)) \su \cc$;
\item $\Phi_\ep$ is bounded for the seminorms and $\Psi_\ep$ is bounded for the operator norms; 
\item The inequality $\| \Psi_\ep \Phi_\ep(x) - x \| \leq \ep \cd L(x)$ holds for all $x \in X$.
\end{enumerate}
Then $(X,L)$ is a compact quantum metric space.
\end{theorem}
Before embarking on the proof, it is worth emphasising  that the maps $\Phi_\ep$ and $\Psi_\ep$ are not required to be unital, and indeed this additional flexibility will be of importance when applying the criterion to prove Theorem \ref{t:fingenproCQMS} below.
\begin{proof}
By Theorem \ref{thm:rieffels-criterion}, it suffices to show that the subset $\big[\, \ov{\B B}_1^L(0) \big] \su X/\cc$ is totally bounded. Let $\ep > 0$ be given. Put $\ep' := \ep/2$ and choose a constant $C > 0$ such that $L_{\ep'}(\Phi_{\ep'}(x)) \leq C \cd L(x)$ and $\| \Psi_{\ep'}(y) \| \leq C \cd \| y \|$ for all $x \in X$ and all $y \in X_{\ep'}$. Since $\cc \su \ker(L)$ the first inequality implies that $\Phi_{\ep'}( \cc ) \su \ker(L_{\ep'})$ and we therefore have well-defined linear maps $[\Phi_{\ep'}] \colon X/\cc \to X_{\ep'}/\ker(L_{\ep'})$ and $[ \Psi_{\ep'} ] \colon X_{\ep'}/\ker(L_{\ep'}) \to X/\cc$ at the level of quotient spaces. 
We record that $[\Phi_{\ep'}]\big[\, \ov{\B B}_1^L(0) \big] \su \big[ \, \ov{\B B}_C^{L_{\ep'}}(0) \big]$. Using that the subset $\big[ \ov{\B B}_C^{L_{\ep'}}(0) \big] \su X_{\ep'} / \ker(L_{\ep'})$ is totally bounded, we may put $\de := \ep'/ C = \ep/(2C)$ and choose finitely many elements $y_1,y_2,\ldots,y_n \in X_{\ep'}$ such that
\[
[\Phi_{\ep'}]\big[ \, \ov{\B B}_1^L(0) \big] \su \bigcup_{j = 1}^n \B B_{\de}^{\|\cdot\|_{X_{\ep'}/\ker(L_{\ep'})}}\big( [y_j]\big) .
\]
We now claim that $\big[\, \ov{\B B}_1^L(0)\big] \su \bigcup_{j = 1}^n \B B_{\ep}^{\|\cdot\|_{X/\cc}}\big( [\Psi_{\ep'}(y_j)] \big)$. Indeed, for every $x \in \ov{\B B}_1^L(0)$ we may choose $j_0 \in \{1,2,\ldots,n\}$ such that $\big\| [\Phi_{\ep'}(x)] - [y_{j_0}] \big\|_{X_{\ep'}/\ker(L_{\ep'})} < \de$. Recalling that $C \cd \de = \ep' = \ep/2$ we then obtain the following inequalities:
\[
\begin{split}
\big\| [x] - [\Psi_{\ep'}(y_{j_0})] \big\|_{X/\cc} 
& \leq \big\| [x - \Psi_{\ep'} \Phi_{\ep'}(x)] \big\|_{X/\cc} 
+ \big\| [ \Psi_{\ep'} \Phi_{\ep'}(x) - \Psi_{\ep'}(y_{j_0}) ] \big\|_{X/\cc} \\
& \leq \ep' \cd L(x) + C \cd \big\| [ \Phi_{\ep'}(x) - y_{j_0}] \big\|_{X_{\ep'}/\ker(L_{\ep'})}\\
&< \ep' + C \cd \de = \ep .
\end{split}
\]
 This shows that $[x] \in \B B_{\ep}^{\|\cdot\|_{X/\cc}}\big( [\Psi_{\ep'}(y_{j_0})] \big)$ and the theorem is therefore proved.
\end{proof}

It is useful to spell out the following particular case of the above theorem.

\begin{cor}\label{cor:approx-cor}
Let $L \colon X \to [0,\infty]$ be a Lipschitz seminorm. Suppose that for every $\ep > 0$ there exist a compact quantum metric space $(X_\ep,L_\ep)$ and unital linear maps $\Phi_\ep \colon X \to X_\ep$ and $\Psi_\ep \colon X_\ep \to X$ such that
\begin{enumerate}
\item $\Phi_\ep$ is bounded for the Lipschitz seminorms and $\Psi_\ep$ is bounded for the operator norms; 
\item The inequality $\| \Psi_\ep \Phi_\ep(x) - x \| \leq \ep \cd L(x)$ holds for all $x \in X$.
\end{enumerate}
Then $(X,L)$ is a compact quantum metric space.
\end{cor}

\subsection{Quantum Gromov-Hausdorff distance}
We now review the notion of quantum Gromov-Hausdorff distance between two compact quantum metric spaces $(X,L)$ and $(Y,K)$. We are in this text applying  Rieffel's original notion of quantum Gromov-Hausdorff distance as introduced in \cite{ Rie:GHD}, although we are  paraphrasing the main definitions in order to deal with operator systems instead of order unit spaces. We would, however, like to emphasise the large body of work due to Latr\'emoli\`ere regarding quantised distance concepts in a $C^*$-algebraic context; see \cite{Lat:AQQ, Lat:BLD, Lat:DGH, Lat:QGH}. It could, in particular, be interesting to investigate whether our main continuity result for quantum $SU(2)$ (Theorem \ref{introthm:qgh-continuity}) remains valid for Latr\'emoli\`ere's notion of quantum Gromov-Hausdorff propinquity as well.

\begin{dfn}\label{d:admissible}
A Lipschitz seminorm $M \colon X \op Y \to [0,\infty]$ is said to be \emph{admissible} when the pair $(X \op Y,M)$ is a compact quantum metric space, $\T{Dom}(M) = \T{Dom}(L) \op \T{Dom}(K)$ and the quotient seminorms induced by $M_{\T{sa}}$ via the coordinate projections
$\T{Dom}(M)_{\T{sa}} \to \T{Dom}(L)_{\T{sa}}$ and $\T{Dom}(M)_{\T{sa}} \to \T{Dom}(K)_{\T{sa}}$ agree with $L_{\T{sa}}$ and $K_{\T{sa}}$, respectively. 
\end{dfn}

Whenever $M \colon X \op Y \to [0,\infty]$ is an admissible Lipschitz seminorm it follows that the coordinate projections $X \op Y \to X$ and $X \op Y \to Y$ induce isometries $\C S(X) \to \C S(X \op Y)$ and $\C S(Y) \to \C S(X \op Y)$, where the state spaces involved are equipped with the Monge-Kantorovi\v{c} metrics coming from the relevant Lip-norms. In particular, we may measure the Hausdorff distance between the state spaces $\C S(X)$ and $\C S(Y)$ with respect to the Monge-Kantorovi\v{c} metric $d_M$ on the state space $\C S(X \op Y)$. Denoting this quantity by
\[
\T{dist}_{\T H}^{d_{M}}(\C S(X),\C S(Y)) \in [0,\infty) 
\]
the \emph{quantum Gromov-Hausdorff distance} between $(X,L)$ and $(Y,K)$ is defined as the infimum over all these Hausdorff distances:
\[
\T{dist}_{\T Q}((X,L);(Y,K)) := \inf\big\{ \T{dist}_{\T H}^{d_M}(\C S(X), \C S(Y)) \mid M \colon X \op Y \to [0,\infty] \T{ admissible} \big\} .
\]
In the following lemma we apply the notation $\T{dist}_{\T Q}\big( (\T{Dom}(L)_{\T{sa}}, L_{\T{sa}}) ; (\T{Dom}(K)_{\T{sa}}, K_{\T{sa}}) \big)$ for the quantum Gromov-Hausdorff distance between the order unit compact quantum metric spaces $(\T{Dom}(L)_{\T{sa}}, L_{\T{sa}})$ and $(\T{Dom}(K)_{\T{sa}}, K_{\T{sa}})$. This notion of order unit quantum Gromov-Hausdorff distance was introduced by Rieffel in \cite{Rie:GHD}, and is defined via the obvious order unit space analogue of admissible seminorms; see  \cite[Definition 4.2]{Rie:GHD}.

\begin{lemma}\label{lem:Rie-qGH-equals-our-qGH}
We have the identity
\[
\T{dist}_{\T{Q}}((X,L);(Y,K)) 
= \T{dist}_{\T{Q}}\big((\T{Dom}(L)_{\T{sa}}, L_{\T{sa}}) ; (\T{Dom}(K)_{\T{sa}}, K_{\T{sa}})\big) .
\]
\end{lemma}
\begin{proof}
Suppose that $M \colon X \op Y \to [0,\infty]$ is admissible. By Proposition \ref{p:restrict} we then know that $\big( \T{Dom}(L)_{\T{sa}} \op \T{Dom}(K)_{\T{sa}}, M_{\T{sa}})$ is an order unit compact quantum metric space.  It moreover follows immediately from Definition \ref{d:admissible} that $M_{\T{sa}}$ is admissible in the order unit sense of Rieffel; see \cite[Section 4]{Rie:GHD}.
%
Conversely, suppose that $M^0 \colon \T{Dom}(L)_{\T{sa}} \op \T{Dom}(K)_{\T{sa}} \to [0,\infty)$ is admissible in the order unit sense. By Proposition \ref{p:restrict}, we then know that $\big( X \op Y , M^0_{\T{os}})$ is a compact quantum metric space. We record that $\T{Dom}(M^0_{\T{os}}) = \T{Dom}(L) \op \T{Dom}(K)$ and since $(M^0_{\T{os}})_{\T{sa}} = M^0$ we obtain that $M^0_{\T{os}}$ is admissible in the sense of Definition \ref{d:admissible}.
The claimed identity between quantum Gromov-Hausdorff distances now follows from Proposition \ref{p:restrict}. 
\end{proof}


Since the quantum Gromov-Hausdorff distance $\T{dist}_{\T Q}((X,L);(Y,K))$  is nothing but Rieffel's original definition from \cite{Rie:GHD} applied to the associated order unit compact quantum metric spaces,  all the main results from  \cite{Rie:GHD}  may be imported verbatim. For the readers convenience, we summarise the key features of $\T{dist}_{\T Q}$ in the theorem below. However, before doing so we need to clarify the slightly subtle notion of isometry in the setting of compact quantum metric spaces.
Fix a compact quantum metric space $(X,L)$ and consider the associated order unit compact quantum metric space $(A,L_A)$ where $A:= \T{Dom}(L)_{\T{sa}}$ and $L_A:=L\vert_A$. We let $(A^c,L_A^c)$ denote the \emph{closed} compact quantum metric space associated to $(A,L_A)$; for more details on this construction see \cite[Section 6]{Rie:GHD} and \cite[Section 4]{Rie:MSS}.
If $(Y,K)$ is another compact quantum metric space with associated order unit compact quantum metric space $(B, K_B)$, then an \emph{isometry} between $X$ and $Y$ is an order unit isomorphism $\varphi\colon A^c \to B^c$ satisfying that $L_B^c\circ \varphi= L_A^c$. The state spaces of $(A^c, L_A^c)$ and $(X,L)$ are naturally identified, and by \cite[Corollary 6.4]{Rie:GHD} one has that the isometries from $(X,L)$ to $(Y,K)$ are in bijective correspondence with the affine isometric isomorphisms from $(\C S(Y), d_K)$ to $(\C S(X), d_L)$. 

\begin{theorem}[Rieffel] \label{thm:quantumdist}
The following hold:
\begin{enumerate}
\item The quantum Gromov-Hausdorff distance is symmetric and satisfies the triangle inequality.
\item The quantum Gromov-Hausdorff distance between two compact quantum metric spaces is zero if and only if there exists an isometry between them.
\item The set of isometry classes of compact quantum metric spaces is complete for the metric induced by $\T{dist}_{\T{Q}}$.
\end{enumerate}
\end{theorem}
 
The following result provides a convenient way to estimate the distance between two compact quantum metric spaces:

\begin{prop}\label{p:almostisom}
Let $(X,L)$ and $(X',L')$ be compact quantum metric spaces and suppose that $\Phi \colon X \to X'$ and $\Psi \colon X' \to X$ are two unital positive maps satisfying that
\begin{enumerate}
\item there exist $C, C' > 0$ such that
\[
L'( \Phi(x)) \leq C \cd L(x) \q \mbox{and} \q L(\Psi(y)) \leq C' \cd L'(y)
\]
for all $x \in X$ and $y \in X'$;
\item there exist $\ep, \ep' > 0$ such that
\[
\| \Psi \Phi(x) - x \| \leq \ep \cd L(x) \q \mbox{and} \q \| \Phi \Psi(y) - y \| \leq \ep' \cd L'(y)
\]
for all $x \in X$ and $y \in X'$.
\end{enumerate}
Then the quantum Gromov-Hausdorff distance $\T{dist}_{\T{Q}}\big( (X,L) ; (X',L') \big)$ is dominated by
\[
\max\big\{ \T{diam}(X,L) \cd |1 - 1/C | + \ep/ C , \T{diam}(X',L') \cd |1 - 1/C'| + \ep'/C' \big\} .
\]
\end{prop}
\begin{proof}
To ease the notation, we put 
\[
r := \max\big\{ \T{diam}(X,L) \cd |1 - 1/C | + \ep/ C , \T{diam}(X',L') \cd |1 - 1/C'| + \ep'/C' \big\},
\]
and define a Lipschitz seminorm $K \colon X \op X' \to [0,\infty] $ by
\[
K(x,y) := 
\max\big\{ L(x), L'(y), \tfrac{1}{r} \| y - \Phi(x) \| , \tfrac{1}{r} \| x - \Psi(y) \| \big\} .
\]
Since both $(X,L)$ and $(X',L')$ are compact quantum metric spaces, we get that $K$ turns $X \op X'$ into a compact quantum metric space. Indeed, fix a state $\mu \in \C S(X)$ and put $\nu:=\mu \circ \Psi$. The fact that the image of $\ov{\B B}_1^K(0)$ becomes totally bounded in $(X \op X')/\cc$ then follows since the map 
\[
(X\oplus X')/\cc \ni [(x,y)] \longmapsto (\mu(x)-\nu(y), [x], [y] )\in \cc \oplus X/\cc \oplus X'/\cc
\] 
is an isomorphism of Banach spaces. We now show that $K$ is admissible. Clearly, $\T{Dom}(K) = \T{Dom}(L) \op \T{Dom}(L')$. Let thus $x \in \T{Dom}(L)_{\T{sa}}$ be given and let $\mu \colon X \to \cc$ be a state. Put $z = x - \mu(x)1_X$ and define the element $y := \frac{1}{C} \Phi( z) + \mu(x)1_{X'} \in  \T{Dom}(L')_{\T{sa}}$. We then obtain the estimates:\\

\begin{itemize}
\item $L'( y ) \leq \frac{1}{C} L'\big( \Phi( z ) \big) \leq L(x)$;\\
\item    $ \frac{1}{r} \| y - \Phi(x) \|  
= \frac{1}{r} \| \frac{1}{C} \Phi(z) - \Phi(z ) \| 
\leq \frac{| 1-1/C|}{r} \| z  \| 
 \leq \frac{| 1-1/C| \cd \T{diam}(X,L)}{r}  \cd L(x) \leq L(x) $;\\
 
 \item $ \frac{1}{r} \| x - \Psi(y) \| 
\leq \frac{1}{r} \| z - \frac{1}{C} \Psi \Phi(z ) \| 
 \leq \| z \| \cd \frac{ |1-1/C| }{r} 
+ \frac{1}{r} \cd \frac{1}{C} \cd \| z - \Psi \Phi(z) \|$ \\
\item[] \hspace{2.25cm} $\leq  \frac{ |1-1/C| \cd \T{diam}(X,L)}{r} \cd L(x)
+ \frac{1}{r} \cd \frac{\ep}{C} \cd L(x) \leq L(x)$.\\
\end{itemize}
This shows that $K( x, y ) \leq L(x)$. Similarly, we obtain that
\[
K\Big( \frac{1}{C'}  \Psi( x' - \nu(x') 1_{X'}) + \nu(x') 1_X, x'\Big) \leq L'(x')
\]
whenever $\nu \colon X' \to \cc$ is a state and $x' \in \T{Dom}(L')_{\T{sa}}$. We conclude that $K$ is an admissible seminorm. 
Finally, given $\mu \in \C S(X)$ it holds that $\nu:=\mu\circ \Psi\in \C S(X')$ and $d_{K}(\mu,\nu)\leq r$. By symmetry, we obtain from this that
\[
\T{dist}_{\T Q}\big( (X, L); (X', L') \big) 
\leq \T{dist}_{\T H}^{d_K}\big( \C S(X), \C S(X') \big) \leq r
\]
and this ends the proof of the present proposition.
\end{proof}

We spell out the following useful consequence of the above proposition. 


 \begin{cor}\label{cor:subspacegen}
Let $(X,L)$ be a compact quantum metric space and let $Y\su X$ be a sub-operator system such that $\T{Dom}(L)\cap Y$ is norm-dense in $Y$. Suppose there exist a constant $D \geq 0$ and an $\ep > 0$ as well as a unital positive map $\Phi \colon X \to Y$ such that $L( \Phi(x)) \leq (1 + D) \cd L(x)$ and $\| x - \Phi(x) \| \leq \ep \cd L(x)$ for all $x \in X$. Then $(Y, L)$ is a compact quantum metric space and we have the estimate 
\[
\T{dist}_{\T{Q}}\big((X,L); (Y,L) \big)\leq \T{diam}(X,L) \cd \frac{D}{1+D} + \ep .
\]
In particular, if $\Phi$ is a Lip-norm  contraction then $\T{dist}_{\T{Q}}\big((X,L); (Y,L) \big)\leq \ep$.
\end{cor}

\begin{proof}
That $(Y,L)$ is a compact quantum metric space follows from Rieffel's criterion in Theorem \ref{thm:rieffels-criterion}. 
We apply Proposition \ref{p:almostisom} to the unital positive map $\Phi \colon X \to Y$ and the inclusion $\io \colon Y \to X$. We then obtain that
\[
\begin{split}
\T{dist}_{\T{Q}}\big((X,L); (Y,L) \big)
& \leq \max\big\{ \T{diam}(X,L) \cd \big| 1- 1/(1 + D) \big| + \ep/(1 + D), \ep \big\} \\
& \leq \T{diam}(X,L) \cd \frac{D}{1+D} + \ep . \qedhere 
\end{split}
\]
\end{proof}

\begin{remark}\label{rem:intermediate}
Under the assumptions in Corollary \ref{cor:subspacegen}, if $Z$ is an intermediate operator system (i.e.~$Y\subseteq Z\subseteq X$) such that $\T{Dom}(L)\cap Z$ is dense in $Z$, then $(Z,L)$ is a compact quantum metric space as well, and the same estimate on the  quantum Gromov-Hausdorff distance holds with $(Z,L)$ instead of $(Y,L)$. Indeed, one may simply enlarge the codomain of $\Phi$ from $Y$ to $Z$  and remark that the assumptions in Corollary \ref{cor:subspacegen} are still satisfied.
\end{remark}



\begin{cor}\label{cor:subgenmet}
Under the assumptions of Corollary \ref{cor:subspacegen} we have the estimate
\[
d_L(\mu,\nu) \leq \frac{D}{1+D} \cd \T{diam}(X,L) + 2 \ep + d_L(\mu\vert_Y,\nu\vert_Y)
\]
for all $\mu,\nu \in \C S(X)$.
\end{cor}
Here the quantity $d_L(\mu\vert_Y,\nu\vert_Y)$ is to be understood as the Monge-Kantorovi\v{c} metric on $\C S(Y)$ arising from the restriction of $L$, which indeed provides $Y$ with a quantum metric structure  by Corollary \ref{cor:subspacegen} 

\begin{proof}
  Let $\mu,\nu \in \C S(X)$. By Proposition \ref{p:restrict} it suffices to show that
  \[
|\mu(x) - \nu(x) | \leq \frac{D}{1+D} \cd \T{diam}(X,L) + 2\ep + d_L(\mu\vert_Y,\nu\vert_Y)
\]
for all $x \in X_{\sa}$ with $L(x) \leq 1$. Let $x \in X_{\sa}$ with $L(x) \leq 1$ be given. By \cite[Proposition 2.2]{Rie:MSS} it holds that $\inf_{\la \in \rr} \|x - \la \cd 1_X\| \leq \T{diam}(X,L)/2$. Since $\Phi$ is unital and positive (and $x$ is selfadjoint), we then have that $\inf_{\la \in \rr} \|\Phi(x - \la \cd 1_X)\| \leq \T{diam}(X,L)/2$. We moreover notice that $\frac{1}{1+D} \Phi(x - \la \cd 1_X) \in Y$ and that the estimate $\frac{1}{1+D} L\big( \Phi(x - \la \cd 1_X) \big) \leq 1$ is satisfied for all $\la \in \rr$. For every $\la \in \rr$ we put $x_\la := x - \la \cd 1_X$ and compute as follows:
\[
\begin{split}
| \mu(x) - \nu(x)| & = \inf_{\la \in \rr}\big| \mu(x_\la) - \nu(x_\la) \big| \\
& \leq \inf_{\la \in \rr}\Big( \big| \mu(x_\la) - \frac{1}{1+D}\mu( \Phi(x_\la)) \big| + \frac{1}{1+D} \big| \mu( \Phi(x_\la)) - \nu( \Phi(x_\la)) \big| \\
& \ \ \  + \big| \frac{1}{1+D}\nu( \Phi(x_\la)) - \nu(x_\la) \big| \Big) \\
& \leq 2 \cd \| x - \Phi(x) \| + 2 \cd \inf_{\la \in \rr} \| \frac{D}{1+D} \Phi(x_\la) \|  + d_L( \mu\vert_Y,\nu\vert_Y) \\
& \leq 2 \ep + \frac{D}{1+D} \T{diam}(X,L) + d_L( \mu\vert_Y,\nu\vert_Y) . \qedhere
\end{split}
\]
\end{proof}

The first step in proving that $C(SU_q(2))$ is a compact quantum metric space, is to utilise that this is known to be the case for the $C^*$-subalgebra $C(S_q^2)$ (see \cite{AgKa:PSM}), and then bootstrap to certain finitely generated projective modules over $C(S_q^2)$. We therefore need to develop a bit of general theory to ensure that our finitely generated projective modules do indeed become compact quantum metric spaces, and we carry out this part of the program in the following section.

\subsection{Finitely generated projective modules}\label{ss:fingenproCQMS}
Let $A$ be a unital $C^*$-algebra, let $B \su A$ be a unital $C^*$-subalgebra and suppose that $E \colon A \to B$ is a conditional expectation.  Remark that $E$ is automatically unital and completely positive and the operator norm of $E$ is therefore equal to one. We moreover consider a complete operator system $X \su A$ such that $B \su X$ and  suppose in addition that the multiplication in $A$ induces a right $B$-module structure on $X$.  On top of this data we fix a Lipschitz seminorm $L \colon A \to [0,\infty]$, and suppose that the domain of $L$ is a unital $*$-subalgebra of $A$.
Our aim is now to impose conditions which ensure that $(X,L)$ is a compact quantum metric space. On the algebraic side we make the following:

\begin{assu}\label{a:fingenerated}
Let $n\in \nn_0$ and assume that there exist elements $v_j \in A$ {and $w_j \in X$} for $j = 0,1,\ldots,n$ with $v_0 = w_0 = 1_A$ such that
\[
\sum_{j = 0}^n w_j \cd E(v_j \cd x) = x \q \mbox{for all } x \in X .
\]
Assume, moreover,  that $E(v_j) = 0$ for all $j \in \{1,2,\ldots,n\}$.
\end{assu}

We define the $B$-linear maps 
\[
\begin{split}
\Phi & \colon X \to \bigoplus_{j = 0}^n B \q \T{ by } \q  \Phi(x) = \sum_{j = 0}^n e_j \cd E(v_j \cd x) \q \T{and} \\
\Psi & \colon \bigoplus_{j = 0}^n B \to X \q \T{ by } \q   \Psi \big( \sum_{j = 0}^n e_j \cd b_j \big) = \sum_{j = 0}^n w_j \cd b_j ,
\end{split}
\]
where $e_0,\dots, e_n$ denotes the standard basis in the free module $\op_{j=0}^n B $. It then follows from Assumption \ref{a:fingenerated} that $(\Psi \ci \Phi)(x) = x$ for all $x \in X$. {In particular, we obtain that $X$ is finitely generated projective as a right $B$-module.} Since $E(v_j) = 0$ for all $j \in \{1,2,\ldots,n\}$ we moreover get that
\begin{align}\label{eq:Phi-of-b-equation}
\Phi(b) = \sum_{j = 0}^n e_j \cd E(v_j \cd b) = \sum_{j = 0}^n e_j \cd E(v_j) \cd b = e_0 \cd b
\end{align}
for all $b \in B \su X$. 


\begin{assu}\label{a:fingensemi}
We impose the following extra conditions on our data:
\begin{enumerate}
\item The conditional expectation $E \colon A \to B$ is bounded for the seminorm $L \colon A \to [0,\infty]$;
\item The restriction $L \colon B \to [0,\infty]$ gives $B$ the structure of a compact quantum metric space;
\item There exists a constant $C_0 > 0$ such that $\| x - E(x) \| \leq C_0 \cd L(x)$ for all $x \in X$;
\item The elements $v_j$ and $w_j$ belong to $\T{Dom}(L)$ for all $j = 0,1,\ldots,n$;
\item For each $v\in \T{Dom(L)}$ the left-multiplication operator $m(v) \colon X \cap \ker(E) \to A$ is bounded with respect to the seminorm $L$.
\end{enumerate}
\end{assu}

A few remarks are in place. First of all, since $L \colon A \to [0,\infty]$ is a Lipschitz seminorm it follows from Assumption \ref{a:fingensemi} $(1)$ that $\T{Dom}(L) \cap B \su B$ is norm-dense and hence that
the restriction $L \colon B \to [0,\infty]$ is a Lipschitz seminorm.  Next, since $\Psi : \op_{j = 0}^n B \to X$ is surjective and $\T{Dom}(L)$ is an algebra, we obtain from Assumption \ref{a:fingensemi} $(4)$ that $X \cap \T{Dom}(L) \su X$ is norm-dense and hence that the restriction $L \colon X \to [0,\infty]$ is also a Lipschitz seminorm. As the following theorem shows, this restriction is actually a Lip-norm.


\begin{theorem}\label{t:fingenproCQMS}
Under Assumption \ref{a:fingenerated} and  \ref{a:fingensemi}, the restriction $L \colon X \to [0,\infty]$ provides $X$ with the structure of a compact quantum metric space.
\end{theorem}
\begin{proof}
We first record that the direct sum 
\[
Y := \op_{j = 0}^n B \cong C\big( \{0,1,2,\ldots,n\},B\big)
\]
becomes a unital $C^*$-algebra when equipped with the supremum norm. We are going to apply Theorem \ref{t:cqmsapprox} with $\Phi_\ep = \Phi \colon X \to Y$ and $\Psi_\ep = \Psi \colon Y \to X$ for all $\ep > 0$. Indeed, condition $(4)$ in Theorem \ref{t:cqmsapprox} is satisfied since $(\Psi \circ \Phi)(x) = x$ for all $x \in X$. 
Let us define the seminorm $K \colon Y \to [0,\infty]$ by
\[
 K\Big( \sum_{j = 0}^n e_j \cd b_j\Big) := 
 \max\big\{ L(b_0) , L(b_1),\ldots,L(b_n), \| b_1 \|, \ldots, \| b_n\| \big\} .
\]
Assumption \ref{a:fingensemi} $(2)$ then implies that the kernel of $K$ is given by the closed subspace 
$\ker(K) = \cc \cd e_0 \su Y$. Moreover, Theorem \ref{thm:rieffels-criterion} and Remark \ref{rem:totallybdd} together with Assumption \ref{a:fingensemi} $(2)$ shows that the subset $\big[ \ov{\B B}^K_1(0) \big] \su Y / \ker(K)$ is contained in
\[
[\mathbb{B}_1^L(0)] \times \big(\mathbb{B}_1^L(0)\cap \mathbb{B}_1^{\|\cdot\|}(0) \big) \times \cdots \times \big(\mathbb{B}_1^L(0)\cap \mathbb{B}_1^{\|\cdot\|}(0)\big) \su  B/\cc \op B^{\op n}
\]
and therefore totally bounded with respect to the quotient operator norm. We have thus verified condition $(1)$ in Theorem \ref{t:cqmsapprox}. Condition $(2)$ in Theorem \ref{t:cqmsapprox} follows immediately since $\Psi( e_0) = w_0$ and $w_0 = 1_A = 1_X$. It is moreover clear that $\Psi \colon Y \to X$ is bounded for the operator norms. In order to establish the remaining condition $(3)$ in Theorem \ref{t:cqmsapprox} we therefore only need to show that $\Phi \colon X \to Y$ is bounded for the seminorms involved.
By Assumption \ref{a:fingensemi} $(1)$ we may choose a constant $C_1 > 0$ such that $L(E(x)) \leq C_1 \cd L(x)$ for all $x \in A$. Moreover, by Assumption \ref{a:fingensemi} $(5)$ we may choose constants $D_j > 0$ for $j = 1,2,\ldots,n$ such that
\[
L(v_j \cd x) \leq D_j \cd L(x) \q \T{for all } x \in \ker(E) \cap X .
\]
Using that $\Phi(E(x))=  e_0\cdot E(x)$  (see \ref{eq:Phi-of-b-equation}) and that $v_0=1$, we then obtain that
\[
\begin{split}
K( \Phi(x) ) 
& \leq K\big( \Phi(x - E(x)) \big) + K( \Phi(E(x))) \\
& = K\Big( \sum_{j = 1}^n e_j \cd E\big( v_j \cd (x - E(x))\big) \Big) + L( E(x)) \\
& \leq K\Big( \sum_{j = 1}^n e_j \cd E\big( v_j \cd (x - E(x))\big) \Big) + C_1 \cd L(x)
\end{split}
\]
for all $x \in X$. Moreover, for each $j \in \{1,2,\ldots,n\}$ and $x\in X$ we obtain the inequality
\[
\| E\big( v_j \cd (x - E(x))\big) \| 
\leq \| v_j \| \cd \| x - E(x) \| \leq \| v_j \| \cd C_0 \cd L(x)
\]
together with the inequality
\[
\begin{split}
L\big( E\big( v_j \cd (x - E(x))\big) \big) 
& \leq C_1 \cd L\big( v_j \cd (x - E(x))\big)  \leq C_1 \cd D_j \cd L(x - E(x)) \\
& \leq C_1 \cd D_j \cd (1 + C_1) \cd L(x) .
\end{split}
\]
This shows that $\Phi \colon X \to Y$ is indeed bounded for the  seminorms involved and we have proved the theorem.
\end{proof}

\section{Preliminaries on quantum $SU(2)$} \label{s:quantumsu2}

The main object of study in the present text is the unital $C^*$-algebra $C(SU_q(2))$, known as \emph{quantum $SU(2)$},  introduced by Woronowicz in \cite{Wor:UAC}. There are numerous good sources describing this object, and in addition to the original texts by Woronowicz we refer the reader to the monographs  \cite{KlSc:QGR} and \cite{Timmermann-book} for general background information. Let $q \in (0,1]$. Aligning our notation with the papers  \cite{AgKa:PSM, AKK:Podcon, AKK:Polyapprox, DaSi:DSP, GKK:QI},  we define the $C^*$-algebraic version of quantum $SU(2)$ as the universal unital $C^*$-algebra $C(SU_q(2))$ with two generators $a$ and $b$ subject to the relations
\begin{align*}
 ba &= q ab \, \,\, \qquad b^* a = q a b^*  \, \qquad \, \, bb^* =b^* b \\
 1&=a^* a + q^2 bb^* \qquad \hspace{0.61cm} a a^* + bb^*= 1 .
\end{align*}

These relations are best justified by noting that they are equivalent to the requirement that
\[
u := \pma{a^* & - qb \\ b^* & a} \in \B M_2\big( C(SU_q(2))\big)
\]
is a unitary matrix, in the following  referred to as the \emph{fundamental unitary}.
Inside the unital $C^*$-algebra $C(SU_q(2))$ we have the \emph{coordinate algebra} $\C O(SU_q(2))$ defined as the  unital $*$-subalgebra generated by $a$ and $b$.  The set $\{\xi^{klm}\mid k\in \zz, l,m\in \nn_0\}$  with elements given by
\begin{align}\label{eq:standard-basis}
\xi^{klm}:= \begin{cases}
  a^kb^l(b^*)^m  &  k,l,m\geq 0  \\
  b^l(b^*)^m (a^*)^{-k}  &  k<0 ,l,m\geq 0  
\end{cases}
\end{align}
constitutes a linear basis for $\C O(SU_q(2))$; see \cite[Proposition 6.2.5]{Timmermann-book}. The coordinate algebra $\C O(SU_q(2))$ is in fact a Hopf $*$-algebra and the coproduct $\De$, the antipode $S$ and the counit $\epsilon$ are best described in terms of the fundamental unitary
by means of the formulae $\De(u) = u \ot u$, $S(u) = u^*$ and $\epsilon(u) =  \SmallMatrix{1 & 0 \\ 0 & 1}$. The coproduct $\De$ extends to a unital $*$-homomorphism $\De \colon C(SU_q(2)) \to C(SU_q(2)) \ot_{\T{min}} C(SU_q(2))$, which turns $C(SU_q(2))$ into a $C^*$-algebraic compact quantum group in the sense of Woronowicz; see \cite{wor:cpqgrps}.  For general $C^*$-algebraic compact quantum groups, it is not true that one can find a bounded counit, but since $C(SU_q(2))$ is known to be \emph{coamenable}, the counit $ \epsilon \colon \C O (SU_q(2))\to \cc$ actually does extend to a unital $*$-homomorphism $\epsilon \colon C(SU_q(2)) \to \cc$; see \cite{BMT:comamenability}. 

\subsection{The quantum enveloping algebra}\label{subsec:enveloping}
We are also interested in the \emph{quantum enveloping algebra} $\C U_q(\G{su}(2))$. For $q \in (0,1)$, this is defined (see \cite[Chapter 4]{KlSc:QGR}) as the universal unital  $\cc$-algebra with generators $e,f,k,k^{-1}$ subject to the relations
\begin{align}\label{eq:deformed-lie-alg-relations}
kk^{-1} = 1 = k^{-1} k \, , \quad \, \, ek = q ke \, , \quad \, \, kf = q fk \,  \ \ \ \T{and} \ \ \ \, \,   fe - ef = \frac{k^2 - k^{-2}}{q - q^{-1}} .
\end{align}
 The quantum enveloping algebra becomes a unital $*$-algebra for the adjoint operation determined by the formulae $k^* = k$ and $e^* = f$. For $q = 1$, the (quantum) enveloping algebra is defined as the universal unital algebra with generators $e,f,h$ satisfying the relations
\begin{align*}
[h,e] = -2e \, , \quad \, \, [h,f] = 2f \, \, \, \ \T{ and } \ \, \, \, \, [f,e] = h,
\end{align*}
with involution given by $h^* = h$ and $e^* = f$; i.e.~it agrees with the enveloping algebra of the Lie algebra $\mathfrak{su}(2)$ as one would expect. Note that we have chosen to follow the notation from \cite{DaSi:DSP}, and that the quantum enveloping algebra just defined is the one denoted $\E{\u{U}}_q(\T{sl}_2)$ in \cite{KlSc:QGR}.
The quantum enveloping algebra $\C U_q(\G{su}(2))$ is also a Hopf $*$-algebra. For $q \neq 1$, the comultiplication, antipode and counit are determined by the formulae
\begin{align*} 
 \De(e) &= e \ot k + k^{-1} \ot e &  S(e) &= -q^{-1}e &  \epsilon(e) &= 0 \notag  \\
 \De(f) &= f \ot k + k^{-1} \ot f &   S(f) &= - qf &  \epsilon(f) &= 0   \\
\De(k) &= k \ot k&   S(k) &= k^{-1} &  \epsilon(k) &= 1   \notag
\end{align*}
and for $q = 1$ by
\begin{align*}
 \De(e) &= e \ot 1 + 1 \ot e  &  S(e) &= - e  &  \epsilon(e) &= 0 \notag  \\
 \De(f) &= f \ot 1 + 1 \ot f&   S(f) &= - f  &  \epsilon(f) &= 0   \\
\De(h) &= h \ot 1 + 1 \ot h &   S(h) &= - h &  \epsilon(h) &= 0   \notag
\end{align*}
In order to unify our notation, it is convenient to put $k = 1$ in the case where $q = 1$.\\

The coordinate algebra $\C O(SU_q(2))$ and the quantum enveloping algebra $\C U_q(\G{su}(2))$ are related to one another by means of a non-degenerate dual pairing of Hopf $*$-algebras \cite[Chapter 4, Theorem 21]{KlSc:QGR}. For $q \neq 1$, this pairing can be described as follows:
\begin{align}\label{eq:generators-on-fundamental-unitary}
\inn{k,u} = \ma{cc}{q^{-1/2} & 0 \\ 0 & q^{1/2}} \, , \quad \, \, \inn{e,u} = \ma{cc}{0 & 1 \\ 0 & 0} \,  \ \ \text{ and } \ \  \, \, \inn{f,u} = \ma{cc}{0 & 0 \\ 1 & 0} 
\end{align}
and for $q = 1$  the same formulae apply together with the additional identity
\[
\inn{h,u} = \ma{cc}{-1 & 0 \\ 0 & 1} .
\]
The dual pairing yields a left action and a right action of $\C U_q(\G{su}(2))$ on $\C O(SU_q(2))$. These actions play a central role in the present text and for $\eta \in \C U_q(\G{su}(2))$ they are defined by the linear endomorphisms
\[
\pa_\eta := (1 \ot \inn{\eta,\cd}) \De  \q \T{and} \q \de_\eta := ( \inn{\eta, \cd} \ot 1) \De
\]
of $\C O(SU_q(2))$.  Thus, $\pa_\eta$ denotes the left action associated to $\eta$ whereas $\de_\eta$ denotes the corresponding right action.
Pairing the generators of $\C O(SU_q(2))$ and $\C U_q(\G{su}(2))$ one obtains the following explicit formulae for the endomorphisms coming from $e$ and $f$ (we are here only listing the non-zero values):
\begin{align}\label{eq:derexpI}
\pa_e(a) &= b^* & \pa_f(a^*) &= -q b  & \de_e(a^*) &= b^* & \de_f(a) &= -q  b \notag \\
\pa_e(b) &= -q^{-1} a^* & \pa_f(b^*) &= a  & \de_e(b) &= -q^{-1} a & \de_f(b^*) &= a^* .
\end{align}
The endomorphisms coming from $e$ and $f$ in $\C U_q(\G{su}(2))$ are related to one another via the adjoint operation, meaning that 
\begin{align}\label{eq:pae-and-paf-and-star}
\pa_e(x^*) = - q^{-1} \pa_f(x)^* \quad \T{and} \quad \de_e(x^*) = -q^{-1}\de_f(x)^*
\end{align}
for all $x \in \C O(SU_q(2))$. We furthermore record that $\pa_k$ and $\de_k$ are algebra automorphisms of $\C O(SU_q(2))$. The relationship between these automorphisms and the adjoint operation is given by $\pa_k(x^*) = \pa_k^{-1}(x)^*$ and $\de_k(x^*) = \de_k^{-1}(x)^*$ for all $x \in \C O(SU_q(2))$. The relevant formulae on generators are listed here:
\begin{align}\label{eq:derexpII}
\pa_k(a) &= q^{1/2} a &  \pa_k(b) &= q^{1/2} b &  \de_k(a) &= q^{1/2} a &  \de_k(b^*) &= q^{1/2}  b^* .
\end{align}
All these formulae may be derived directly from the defining relations for $\C O(SU_q(2))$ and $\C U_q(\G{su}(2))$ and the definition of a dual pairing of Hopf $*$-algebras \cite[Chapter 1, Definition 5 \& Equation (41)]{KlSc:QGR}. In the same way one sees that both $\pa_e$ and $\pa_f$ are twisted derivations, in the sense that 
\begin{align}\label{eq:twisted-leibnitz-for-pae-and-paf}
\pa_e(xy)&=\pa_e(x)\pa_k(\textcolor{black}{y})+\pa_{k^{-1}}(x)\pa_e(y)\notag\\
\pa_f(xy)&=\pa_f(x)\pa_k(\textcolor{black}{y})+\pa_{k^{-1}}(x)\pa_f(y)
\end{align}
for all $x,y\in \C O(SU_q(2))$. 


We shall encounter such twisted derivations numerous times in the sections to follow and we therefore formalise this notion in the following short section.

\subsection{Twisted derivations}

\begin{dfn}\label{def:twisted-derivation}
Let $A$ and $B$ be $C^*$-algebras and let $\si$ and $\te \colon \C A \to B$ be algebra homomorphisms defined on a dense $*$-subalgebra $\C A \su A$. We say that a linear map $d \colon \C A \to B$ is a \emph{twisted derivation} when $d(x \cd y) = d(x) \cd \te(y) + \si(x) \cd d(y)$ for all $x,y \in \C A$. A twisted derivation is called a \emph{twisted $*$-derivation} when $d(x^*) = - d(x)^*$ and $\si(x^*)^* = \te(x)$ for all $x \in \C A$.  
\end{dfn}
We remark that a twisted derivation $d \colon \C A \to B$ is the same thing as a derivation $d \colon \C A \to B$ when $B$ is given the bimodule structure determined by the algebra homomorphisms $\si$ and $\te \colon \C A \to B$.

\subsubsection{$q$-numbers}
 We are going to need two versions of $q$-numbers. For $q \in (0,1]$ and $n \in \nn$ we define the quantity
\begin{align}\label{eq:our-q-integers}
\inn{n}_q := 1 + q^2 +\cdots + q^{2(n-1)} . 
\end{align}
Furthermore, the classical \emph{$q$-number} makes sense for every $a\in \rr$ and is defined by
\[
[a]_q:=  \begin{cases} \frac{q^{a}-q^{-a}}{q-q^{-1}} & {q\in (0,1)} \\
a & q=1
\end{cases}.
\]
 Whenever no confusion can arise, we omit the subscript $q$ from the notation. 

\subsection{Corepresentation theory}
The (co-)representation theory of $SU_q(2)$ is well understood, and turns out to be equivalent with that of $SU(2)$; see \cite[Section 5]{Wor:UAC}. We may therefore choose a complete set of irreducible corepresentation unitaries $u^n \in \B M_{n+1}(\C O(SU_q(2)))$, $n \in \nn_0$, where the matrix entries $u^n_{ij}$ are labelled by indices $i,j \in \{0,1,\ldots,n\}$. For $q \neq 1$, we fix this choice of irreducible corepresentation unitaries such that
\begin{equation}\label{eq:exppai}
\begin{split}
\inn{k,u^n_{ij}} & = \de_{ij} \cd q^{j -n/2} \\
 \inn{e,u^n_{ij}} &= \de_{i,j - 1} \cd q^{\frac{1 - n}{2}} \sqrt{ \inn{n - j +1}_{q} \inn{j}_{q}  }\\
\inn{f,u^n_{ij}} & = \de_{i,j + 1} \cd q^{\frac{1-n}{2}} \sqrt{ \inn{n-j}_{q} \inn{j+1}_{q} } ,
\end{split}
\end{equation}
and for $q = 1$ we fix the same formulae together with the additional identity
\[
\inn{h,u^n_{ij}} = \de_{ij} \cd (2j - n) .
\]
We record that the fundamental unitary $u$ agrees with the irreducible corepresentation unitary $u^1$ and that $u^0 = 1$. We shall often refer to the entries $u^n_{ij} \in \C O(SU_q(2))$ as the \emph{matrix coefficients} and we apply the convention that $u_{ij}^n:=0$ whenever one of the parameters $n,i,j$ is outside of its natural range; i.e.~when $n<0$ or $(i,j)\notin \{0,\dots, n\}^2$. The adjoint operation can be described at the level of the matrix coefficients via the formula
\begin{align}\label{eq:adjoint-of-matrix-coefficients}
(u^n_{ij})^* = (-q)^{j-i} u^n_{n-i,n-j} ;
\end{align}
see for instance \cite[Section 2]{DLSSV:DOS}.  For more details on the corepresentation theory for quantum $SU(2)$, we refer the reader to \cite[Chapter 3, Theorem 13 \& Chapter 4, Proposition 16 and 19]{KlSc:QGR}.  
Using the $q$-Clebsch-Gordan coefficients (see \cite[Section 3]{DLSSV:DOS} and \cite[Chapter 3.4]{KlSc:QGR}) one may explicitly describe the products between the generators and the matrix coefficients:
\begin{equation}\label{eq:leftmult}
\begin{split}
a^* \cd u^n_{ij} & = q^{i + j} \tfrac{ \sqrt{\inn{n - i + 1} \inn{n-j+1}}}{ \inn{n+1}} \cd u^{n+1}_{ij}
+ \tfrac{\sqrt{\inn{i} \inn{j}}}{\inn{n+1}} \cd u^{n-1}_{i-1,j-1} \\
b^* \cd u^n_{ij} & = q^j \tfrac{ \sqrt{\inn{i+1}\inn{n-j + 1}}}{\inn{n+1}} \cd u^{n+1}_{i+1,j}
- q^{i+1} \tfrac{ \sqrt{\inn{n-i} \inn{j}}}{\inn{n+1}} \cd u^{n-1}_{i,j-1} \\
a \cd u^n_{ij} & = \tfrac{ \sqrt{\inn{i + 1} \inn{j+1}}}{ \inn{n+1}} \cd u^{n+1}_{i+1,j+1}
+ q^{i+j+2}\tfrac{\sqrt{\inn{n-i} \inn{n-j}}}{\inn{n+1}} \cd u^{n-1}_{ij} \\
b \cd u^n_{ij} & = 
- q^{i-1}\tfrac{\sqrt{\inn{j+1} \inn{n-i +1}}}{\inn{n+1}} \cd u^{n+1}_{i,j+1}
+ q^j \tfrac{ \sqrt{\inn{n -j} \inn{i}}}{ \inn{n+1}} \cd u^{n-1}_{i-1,j} .
\end{split}
\end{equation}
In particular, it holds that $u_{00}^n=(a^*)^n$ for all $n\in \nn_0$, a fact that will be used several times throughout the paper. 

\subsection{The Haar state}\label{subsec:Haar-state}
Quantum $SU(2)$ comes equipped with its \emph{Haar state} $h \colon C(SU_q(2)) \to \cc$ which can be expressed on the matrix coefficients by the simple relations
\[
h(1) = 1 \, \, \T{ and } \, \, \, h(u^n_{ij}) = 0
\]
for all $n \in \nn$ and $i,j \in \{0,1,\ldots,n\}$; see e.g.~\cite[Chapter 4, Equation (50)]{KlSc:QGR}. On the elements $\xi^{klm}$ of the linear basis \eqref{eq:standard-basis}, the Haar state vanishes if $k\neq 0$, and for $k=0$ it furthermore vanishes when $l\neq m$. Finally, when $k=0$ and $l=m$ it holds that
\begin{align}\label{eq:Haar-on-standard-basis}
h(b^{m}b^{*m})=\frac{1}{\inn{m+1}_q};
\end{align}
see e.g.~\cite[Theorem 6.2.17]{Timmermann-book}. As the name suggests, the Haar state is bi-invariant with respect to the comultiplication in the sense that
\[
(h\ot 1)\Delta(x)=(1\ot h)\Delta(x)=h(x)\cdot 1 \quad \T{for all } x\in C(SU_q(2)) .
\]
For $q\neq 1$, the Haar state is not a trace, but it is a \emph{twisted trace} with respect to the algebra automorphism $\nu:=\de_{k^{-2}}\circ \pa_{k^{-2}}$, in the sense that
\begin{align}\label{eq:modular}
h(x y) = h(\nu(y) x) \q \T{for all } x,y \in \C O(SU_q(2));
\end{align}
see \cite[Chapter 4, Proposition 15]{KlSc:QGR}. Using the formulae in \eqref{eq:exppai}, one sees that the modular automorphism $\nu$ is given by the following formula on the matrix coefficients:
\begin{align}\label{eq:modular-function}
\nu(u^n_{ij}) =  q^{2(n - i - j)}   \cd u^n_{ij}.
\end{align}
The  algebra automorphisms $\de_{k^{-1}} \circ \pa_{k^{-1}}$ and $\de_{k}\circ \pa_{k}$ will be denoted $\nu^{1/2}$ and $\nu^{-1/2}$, respectively.\\

The Haar state is faithful and we denote the corresponding GNS Hilbert space by $L^2(SU_q(2))$ and the natural embedding $C(SU_q(2))\subseteq L^2(SU_q(2))$ by $\Lambda$. Furthermore, we denote the associated injective $*$-homomorphism by $\rho \colon C(SU_q(2)) \to \B B\big(L^2(SU_q(2))\big)$ and the notation $L^\infty(SU_q(2))$ refers to the enveloping von Neumann algebra so that $L^\infty(SU_q(2))$ agrees with the double commutant $\rho(C(SU_q(2)))''\subseteq \mathbb{B}\big(L^2(SU_q(2))\big)$. Lastly, the diagonal representation of $C(SU_q(2))$ on two copies of $L^2(SU_q(2))$ plays a prominent role in the sections to follow and will be denoted by $\pi\colon C(SU_q(2))\to \mathbb{B}\big(L^2(SU_q(2))^{\oplus 2}\big)$. Whenever convenient, we apply the notation $H_q:=L^2(SU_q(2))$. 
The matrix  units $u_{ij}^n$ constitute an orthogonal basis in $L^2(SU_q(2))$, and  the $2$-norms of $u_{ij}^n$ and $(u_{ij}^n)^*$ are given by
\begin{align}\label{eq:haarmatrix}
\inn{u^n_{ij},u^n_{ij}} & = h( (u^n_{ij})^* u^n_{ij}) = \frac{q^{2(n-i)}}{\inn{n + 1}_q } \notag \\ 
\binn{ (u^n_{ij})^* , (u^n_{ij})^*} & = h( u^n_{ij} (u^n_{ij})^*) = \frac{q^{2j}}{\inn{n + 1}_q };  
\end{align}
whenever $n\in \nn_0$ and $i,j\in \{0,\dots, n\}$; see \cite[Chapter 4, Theorem 17]{KlSc:QGR}.

\subsection{Circle actions}\label{ss:circle}
The unital $C^*$-algebra $C(SU_q(2))$ carries two distinguished circle actions 
\[
\si_L \T{ and } \si_R \colon S^1 \ti C(SU_q(2)) \to C(SU_q(2))
\]
referred to as the \emph{left circle action} and the \emph{right circle action}, respectively. These two circle actions are given on the matrix coefficients by the formulae
\begin{align}\label{eq:sigma-L-and-R-on-matrix-units}
\si_L(z,u^n_{ij}) = z^{2j - n} u^n_{ij} \, \, \ \T{ and } \ \, \, \, \si_R(z,u^n_{ij}) = z^{2i -n} u^n_{ij}
\end{align}
for all $z \in S^1$, $n \in \nn_0$ and $i,j \in \{0,1,\ldots,n\}$; see for example \cite[Section 2.2]{KRS:RFH}. The spectral subspaces for the left circle action play a special role in the present text and they are denoted by
\[
A^m_q := \big\{ x \in C(SU_q(2)) \mid \si_L(z,x) = z^m \cd x \, \, \T{ for all }  z \in S^1 \big\}, \q m \in \zz .
\]
For each $m \in \zz$ we also define the algebraic spectral subspace $\C A^m_q := A^m_q \cap \C O(SU_q(2))$. Note that the \emph{Podle\'s sphere} (see \cite{Pod:QS}) agrees with the fixed point algebra so that $C(S_q^2) = A^0_q$, and the coordinate algebra $\C O(S_q^2)$ agrees with the algebraic fixed point algebra $\C A^0_q$. The algebraic spectral subspaces are left comodules over $\C O(SU_q(2))$ in the sense that the coproduct restricts to a coaction $\De \colon \C A^m_q \to \C O(SU_q(2)) \ot \C A^m_q$ for each $m \in \zz$.
The spectral subspace $A_q^m$ comes with an associated \emph{spectral projection} $\Pi^L_m\colon C(SU_q(2))\to A_q^m$
defined by the norm-convergent Riemann integral
\begin{align}\label{eq:spec-proj-norm-def-on-suq2}
\Pi^L_m(x) = \frac{1}{2\pi} \int_0^{2\pi} \sigma_L(e^{ir},x) \cd e^{-irm} dr .
\end{align}
Note that $\Pi^L_m$ is a contraction and that $\Pi^L_m\big(\C O (SU_q(2))\big) \subseteq \C A_q^m$. We apply the notation $H_q^m \su H_q$ for the Hilbert space closure of $\La( \C A^m_q) \su H_q$.
For each $M \in \nn_0$, we introduce the \emph{spectral band}  
\begin{equation}\label{eq:specbanddef}
B_q^M :=\sum_{m=-M}^M A_q^m .
\end{equation}
The spectral band also exists in an algebraic version, namely $\C B_q^M := \sum_{m = -M}^M \C A_q^m$. We note that $B_q^M$ agrees with the norm-closure of the algebraic spectral band, where the non-trivial inclusion follows by using the spectral projections.

\subsection{Analytic elements}\label{ss:analytic}
For each $s \in (0,1]$, we define the closed strip
\begin{equation}\label{eq:closstrip}
I_s := \left\{ z \in \cc \mid \T{Im}(z) \in \left[ \tfrac{\log(s)}{2}, -\tfrac{\log(s)}{2} \right] \right\} \su \cc .
\end{equation}

\begin{dfn}\label{d:analyt}
Let $s \in (0,1]$. We say that an element $x \in C(SU_q(2))$ is \emph{analytic of order $-\log(s)/2$} when the continuous map $\rr \to C(SU_q(2))$ given by $r \mapsto \si_L(e^{ir},x)$ extends to a continuous map $I_s \to C(SU_q(2))$ which is analytic on the interior $I_s^{\ci} \su I_s$. If so, we denote this (unique) continuous extension by $z \mapsto \si_L(e^{iz},x)$. 
\end{dfn}

Let $x,y \in C(SU_q(2))$ be analytic of order $-\log(s)/2$. Applying the basic properties of operator valued analytic maps we obtain that $x \cd y$ and $x^*$ are analytic of order $-\log(s)/2$ and that we have the relations
\begin{align}\label{eq:sigmaL-relations}
\si_L(e^{iz}, x \cd y) = \si_L(e^{iz},x) \cd \si_L(e^{iz},y) \, \, \T{ and } \, \, \, 
\si_L(e^{iz},x^*) = \si_L(e^{i \cd \ov{z}},x)^*
\end{align}
for all $z \in I_s$. The set of elements that are analytic of order $-\log(s)/2$ thus constitutes a unital $*$-subalgebra.

\begin{lemma}\label{l:boundedanalytic}
Let $s \in (0,1]$ and let $x$ be analytic of order $-\log(s)/2$. If $T \colon C(SU_q(2)) \to C(SU_q(2))$ is a bounded operator which is equivariant with respect to the circle action $\si_L$, then $T(x)$ is analytic of order $-\log(s)/2$ and it holds that $T (\si_L(e^{iz},x)) = \si_L(e^{iz},T(x))$ for all $z \in I_s$.
\end{lemma}
\begin{proof}
Since $T$ is bounded, the map $I_s \ni z\mapsto T(\sigma_L(e^{iz},x)) \in C(SU_q(2))$ is continuous and analytic on the interior $I_s^{\circ}$. Moreover, for $r\in \rr$ we have $T(\sigma_L(e^{ir}, x))=\sigma_L(e^{ir}, Tx)$, so it follows that $T(x)$ is analytic of order $-\log(s)/2$ and, by the identity theorem for analytic functions, that $T( \si_L(e^{iz},x)) = \si_L(e^{iz},T(x))$ for all $z \in I_s$.
\end{proof}


\begin{lemma}\label{l:sigmaonspec}
Let $m \in \zz$ and $x \in A_q^m$. It holds that $x$ is analytic of order $-\log(s)/2$ for all $s \in (0,1]$ and that the associated extension is given by
\[
\si_L(e^{iz},x) = e^{iz \cd m} \cd x \q \mbox{for all } z \in \cc .
\]
\end{lemma}
\begin{proof}
This follows since $\si_L(e^{it},x) = e^{it \cd m} \cd x$ and since $z \mapsto e^{iz \cd m}$ is analytic. 
\end{proof}

It follows from Lemma \ref{l:sigmaonspec} that every $x \in \C O(SU_q(2))$ is analytic of order $-\log(s)/2$ for all $s \in (0,1]$ and that we have an algebra automorphism
\[
\si_L(e^{iz}, \cd) \colon \C O(SU_q(2)) \to \C O(SU_q(2))
\]
for all $z \in \cc$. Moreover, it holds that $\si_L(e^{iz}, \si_L(e^{iw}, x) ) = \si_L(e^{i(z + w)},x)$ for all $z,w \in \cc$ and $x \in \C O(SU_q(2))$. As a consequence Lemma \ref{l:sigmaonspec} we also obtain that
\[
\sigma_L(q^{1/2}, x)= \pa_k(x) \q \T{for all } x\in \C O(SU_q(2)) .
\]
For each $s \in (0,\infty)$, we also introduce the unbounded operator $\Ga_{s,0} \colon \C O(SU_q(2))^{\op 2} \to L^2(SU_q(2))^{\op 2}$ given by the formula 
\begin{align}\label{eq:modudef}
\Ga_{s,0}\pma{\xi \\ \eta} := \pma{s^{\frac{1-n}{2}} & 0 \\ 0 & s^{\frac{-1-m}{2}} } \pma{\xi \\ \eta}
\end{align}
for all $\xi \in \C A^n_q$ and $\eta \in \C A^m_q$. Since $\Ga_{s,0}$ admits an orthonormal basis of eigenvectors with strictly positive eigenvalues, we obtain that $\Ga_{s,0}$ is closable and that the closure is a positive unbounded operator with dense image. We denote this closure by 
\[
\Ga_s \colon \T{Dom}(\Ga_s) \to L^2(SU_q(2))^{\op 2} .
\]
The inverse of $\Ga_s$ is again a positive unbounded operator with dense image and we have the following identities regarding images and domains:
\[
\T{Dom}(\Ga_s^{-1}) = \T{Im}(\Ga_s) \q \T{and} \q \T{Im}(\Ga_s^{-1}) = \T{Dom}(\Ga_s) .
\]
The inverse $\Ga_s^{-1}$ agrees with the closure of the unbounded operator $\Ga_{s^{-1},0} \colon \C O(SU_q(2))^{\op 2} \to L^2(SU_q(2))^{\op 2}$ and we therefore have the identity $\Ga_s^{-1} = \Ga_{s^{-1}}$. 


\begin{lemma}\label{l:analytic}
Let $s \in (0,1]$. If $x \in C(SU_q(2))$ is analytic of order $-\log(s)/2$, then it holds that $x\big( \T{Dom}(\Ga_s) \big) \su \T{Dom}(\Ga_s)$ and  $x\big( \T{Im}(\Ga_s) \big) \su \T{Im}(\Ga_s)$ and we have the relations
\[
\Ga_s x \Ga_s^{-1}(\xi) =  \si_L(s^{-1/2},x)(\xi) \q \mbox{and} \q   \Ga_s^{-1} x \Ga_s(\eta) = \si_L(s^{1/2},x)(\eta)
\]
for all $\xi \in \T{Im}(\Ga_s)$ and $\eta \in \T{Dom}(\Ga_s)$.
\end{lemma}
\begin{proof}
Suppose that $x \in C(SU_q(2))$ is analytic of order $-\log(s)/2$. We focus on showing that $x\big( \T{Dom}(\Ga_s) \big) \su \T{Dom}(\Ga_s)$ and that $\Ga_s x \Ga_s^{-1}(\xi) = \si_L(s^{-1/2},x)(\xi)$ for all $\xi \in \T{Dom}(\Ga_s)$, since the remaining identities follow by similar arguments. 
We apply the notation $\C E := \C O(SU_q(2))^{\op 2}$ for the defining core for $\Ga_s$. It then suffices to show that
\begin{equation}\label{eq:sigmaconj}
\inn{\Ga_s \eta, x \Ga_s^{-1} \xi} = \inn{\eta, \si_L(s^{-1/2},x) \xi}
\end{equation}
for all $\xi, \eta \in \C E$. Let thus $\xi,\eta \in \C E$ be given. For each $r \in \rr$ we consider the unitary operator $\Ga_s^{ir} \colon L^2(SU_q(2))^{\op 2} \to L^2(SU_q(2))^{\op 2}$. It can then be verified that these unitary operators implement the left circle action in the sense that the identity 
\begin{equation}\label{eq:gammasigma}
\Ga_s^{ir} x \Ga_s^{-ir} = \si_L( e^{-ir\log(s)/2},x)
\end{equation}
holds for all $r \in \rr$. Indeed, when $x$ belongs to a spectral subspace the above identity follows from Lemma \ref{l:sigmaonspec} and therefore holds in general by density and continuity. \\
Let us define the closed strip $I := \{ z \in \cc \mid \T{Im}(z) \in [-1,1] \}$ together with the continuous functions $f,g \colon I \to \cc$ given by the formulae
\[
f(z) := \inn{\Ga_s^{i \cd \ov z} \eta, x \Ga_s^{i \cd z} \xi} \q \T{and} \q g(z) := \inn{\eta, \si_L( e^{ i \cd z \log(s)/2},x) \xi} 
\]
for all $z \in \cc$. Notice that the first of these functions makes sense since $\eta,\xi \in \C E$ and the second makes sense because $x$ is analytic of order $-\log(s)/2$. Both of these functions are then holomorphic on the interior of the strip $I^{\ci}$ and they agree on the real line $\rr \su I$ by an application of \eqref{eq:gammasigma}. This implies that $f(z) = g(z)$ for all $z \in I$ and we obtain the identity in \eqref{eq:sigmaconj} by evaluating at $z = i$.
\end{proof}

For each $t \in (0,1]$, we apply the notation $\T{Ana}_t(SU_q(2))$ for the unital $*$-subalgebra of $C(SU_q(2))$ consisting of elements $x \in C(SU_q(2))$ which are analytic of order $\T{max}\big\{ -\log(q)/2 , - \log(t)/2 \big\}$. We equip $\T{Ana}_t(SU_q(2))$ with the norm $\| \cd \|_{t,q}$ defined by  
\[
\|x\|_{t,q} := \max\big\{\| \si_L(t^{1/2},x) \| + \| \si_L(q^{1/2},x) \| , \| \si_L(t^{-1/2},x) \| + \| \si_L(q^{-1/2},x) \| \big\} 
\]
and record that $\T{Ana}_t(SU_q(2))$ is then a unital Banach $*$-algebra. For more details, see for instance \cite[Example 1.5]{BlKaMe:OCA}. 
We end this section by a small lemma providing an estimate on the norm $\| \cdot \|_{t,q}$ on a fixed spectral band.

\begin{lemma}\label{l:anaband}
Let $M \in \nn_0$ and $x \in B_q^M$. It holds that $x \in \T{Ana}_t(SU_q(2))$ and we have the estimate
\[
\| x \|_{t,q} \leq \sum_{m = - M}^M ( t^{m/2} + q^{m/2}) \cd \| x \| \q \mbox{for all } t \in (0,1] .
\]
\end{lemma}
\begin{proof}
This follows from Lemma \ref{l:sigmaonspec}. Indeed, for every $s \in (0,1]$ we have the estimate:
\[
\big\| \si_L(s^{\pm 1/2},x) \big\| = \Big\| \sum_{m = -M}^M s^{\pm m/2} \Pi_m^L(x) \Big\|
\leq \sum_{m = - M}^M s^{\pm m/2} \| x \| .  \qedhere
\]
\end{proof}


\subsection{The continuous field}\label{subsec:cont-field}
It is possible to consider the unital $C^*$-algebras $C(SU_q(2))$ for different values of $q \in (0,1]$ as fibres in a continuous field of $C^*$-algebras, as was shown by Blanchard in \cite{Bla:DCH}. For the sake of clarity, we will, for a moment,  adorn the elements in $SU_q(2)$ with an additional $q$, thus writing $a_q$ and $b_q$ for the generators. Let us fix $\de\in (0,1)$. We obtain from \cite[Th\'eor\`eme 3.3 \& Proposition 7.1]{Bla:DCH} that there exists a unital continuous field of $C^*$-algebras $C(SU_\bullet(2))$ over $[\delta,1]$ whose fibre at $q$ agrees with $C(SU_q(2))$. Concretely, the continuous field $C(SU_\bullet(2))$ is defined as the universal $C^*$-algebra generated by three elements $a_\bullet,b_\bullet$ and $f$ subject to the relations
\begin{itemize}
\item[-] $f$ commutes with $a_\bullet$ and $b_\bullet$; 
\item[-] $f$ is selfadjoint and the spectrum of $f$ agrees with the interval $[\de,1]$;
\item[-] $u_\bullet = \begin{pmatrix} a_\bullet^* & - f b_\bullet \\ b_\bullet^* & a_\bullet \end{pmatrix}$ is a unitary element in $\mathbb{M}_2(C(SU_\bullet(2)))$.
\end{itemize}
For each $q \in [\de,1]$, the evaluation homomorphism $\T{ev}_q\colon C(SU_\bullet(2)) \to C(SU_q(2))$ is defined by sending the generators $a_\bullet$ and $b_\bullet$ to the corresponding generators $a_q$ and $b_q$ in $C(SU_q(2))$ and by sending $f$ to the scalar $q$.  In what follows, we will tacitly identify $C^*(f)$ with $C([\de, 1])$. 
We denote by $\C O(SU_\bullet(2))$ the unital $*$-subalgebra generated by $C^*(f), a_\bullet$ and $b_\bullet$.  Note that it follows from the discussion in the beginning of Section \ref{s:quantumsu2} that the elements
\begin{align}\label{eq:standard-basis-field}
\xi^{klm}_\bullet:= \begin{cases}
  a_\bullet^kb_\bullet^l(b_\bullet^*)^m   &   k,l,m\geq 0  \\
  b_\bullet^l(b_\bullet^*)^m (a_\bullet^*)^{-k} &   k<0 ,l,m\geq 0  
\end{cases}
\end{align}
constitute a basis for $\C O(SU_\bullet(2))$ when considered as a $C([\de,1])$-module.
 Let $k \in \zz$ and $l,m \in \nn_0$. As a consequence of the twisted Leibniz rule from \eqref{eq:twisted-leibnitz-for-pae-and-paf} there exists a unique element $\pa_{e_\bullet}(\xi^{klm}_\bullet) \in \C O(SU_\bullet(2))$ such that
\[
\T{ev}_q\big( \pa_{e_\bullet}(\xi^{klm}_\bullet) \big) =\pa_e\big(\T{ev}_q(\xi^{klm}_\bullet)\big) \q \T{for all } q \in [\de,1] .
\]
We may thus define $\pa_{e_\bullet}  \colon \C O(SU_\bullet(2)) \to \C O(SU_\bullet(2))$, by mapping each basis element $\xi^{klm}_\bullet$ to  $\pa_{e_\bullet}(\xi^{klm}_\bullet)$ and extending by $C([\de,1])$-linearity. By construction, it holds that 
\[
\T{ev}_q\big( \pa_{e_\bullet}(x_\bullet) \big) =\pa_e\big(\T{ev}_q(x_\bullet)\big) \q \T{for all } x_\bullet \in \C O(SU_\bullet(2)) \, \T{ and } \, \, q \in [\de,1] .
\]
In a similar fashion, we define a $C([\de,1])$-linear map $\pa_{f_\bullet} \colon \C O(SU_\bullet(2)) \to \C O(SU_\bullet(2))$ satisfying that
\[
\T{ev}_q\big( \pa_{f_\bullet}(x_\bullet) \big) =\pa_f\big(\T{ev}_q(x_\bullet)\big) \q \T{for all } x_\bullet \in \C O(SU_\bullet(2)) \, \T{ and } \, \, q \in [\de,1] .
\]

\section{Spectral geometry on quantum $SU(2)$}\label{s:twisspectrip}
In this section we provide a detailed treatment of the non-commutative geometry of quantum $SU(2)$. As alluded to in the introduction, it has turned out remarkably difficult to properly unify the theory of quantum groups with Connes' non-commutative geometry, and the general consensus seems to be that one needs to relax Connes' axioms by allowing for  certain twists; see \cite{CoMo:TST}. There are by now a number of candidates for Dirac operators on $SU_q(2)$ with various advantages and disadvantages \cite{KRS:RFH, KaSe:TST, DLSSV:DOS, BiKu:DQQ, CP:EST, KW:TDO, NeTu:DCQ, BCZ:CQM}, and here we wish to give a detailed analysis of the Dirac operators proposed in \cite{KaSe:TST} and \cite{KRS:RFH} from the quantum metric point of view. In order to treat both Dirac operators simultaneously, it will be an advantage to allow for an additional parameter $t$ which, for fixed $q$, interpolates between the Dirac operator from \cite{KaSe:TST} and that from \cite{KRS:RFH} on $SU_q(2)$.  We emphasise that for $t \neq q$ we do not work with a single Dirac operator but rather with a pair of Dirac operators, aligning with the terminology from classical fiber bundles, we refer to them as the vertical and horizontal Dirac operator, respectively. The vertical and horizontal Dirac operators are in fact incompatible in the sense that their interactions with the coordinate algebra require the use of two different twists. \\

\subsection{The horizontal and vertical Dirac operators}
Let us fix two parameters $t,q \in (0,1]$. We define two unbounded operators $\C D^H_q$ and $\C D^V_t \colon \C O(SU_q(2))^{\op 2} \to L^2(SU_q(2))^{\op 2}$. The first of these unbounded operators is referred to as the \emph{horizontal Dirac operator} and is given by the matrix
\begin{equation}\label{eq:dirhori}
\C D^H_q := \pma{ 0 & -q^{-1/2} \pa_{fk^{-1}} \\ -q^{1/2} \pa_{ek^{-1}} & 0} .
\end{equation}
We remark that $\C D^H_q$ is independent of the parameter $t \in (0,1]$. The second  unbounded operators is referred to as the \emph{vertical Dirac operator} and given by the assignment
\begin{equation}\label{eq:dirvert}
\C D^V_t \pma{\xi \\ \eta} := \pma{ t^{\frac{-n + 1}{2}} \big[ \frac{n-1}{2}\big]_t & 0 \\ 0 & - t^{\frac{-m - 1}{2}} \big[ \frac{m+1}{2}\big]_t} \cd \pma{\xi \\ \eta}  
\end{equation}
for all $\xi \in \C A^n_q$ and $\eta \in \C A^m_q$. 
A direct computation verifies that both $\C D^V_t$ and $\C D^H_q$ are symmetric and for both operators there exists a family of  orthogonal finite dimensional invariant subspaces which span a dense subspace in $L^2(SU_q(2))^{\oplus 2}$; it may even be deduced from \eqref{eq:exppai} that we can obtain a joint invariant family of  finite dimensional subspaces, by setting
\[
V^n_{ij} := \left\{ \begin{pmatrix} \la \cd u^n_{ij} \\ \mu \cd u^n_{i,j-1} \end{pmatrix} \mid \la,\mu \in \cc \right\} 
\su \C O(SU_q(2))^{\op 2}, \qquad n\in \nn_0, i,j\in \{0,\dots, n\}.
\]
It therefore follows that $\C D^H_q, \C D_t^V$ and $\C D_t^V + \C D^H_q$ are essentially selfadjoint, and we denote the selfadjoint closures of the horizontal and vertical Dirac operators by $D^H_q$ and $D^V_t$, respectively. Moreover, we have the following convenient description of the closure of $\C D_t^V + \C D^H_q \colon \C O(SU_q(2))^{\op 2} \to L^2(SU_q(2))^{\op 2}$:
\begin{lemma}\label{l:sumdirac}
The unbounded operator $\C D^V_t + \C D^H_q$ is essentially selfadjoint. Moreover, it holds that $\T{Dom}(\overline{\C D^V_t + \C D^H_q})= \T{Dom}(D_t^V)\cap \T{Dom}(D^H_q)$ and $\overline{\C D^V_t + \C D^H_q}= D_t^V + D^H_q$.
\end{lemma}
\begin{proof}
We already argued that $\C D^V_t + \C D^H_q$ is essentially selfadjoint. Let $n,m \in \zz$. Using that $\pa_e(\C A_q^n)\subseteq \C A_q^{n-2}$ and $\pa_f(\C A_q^m)\subseteq \C A_q^{m+2}$, we obtain for $\xi\in \C A_q^n$ and $\eta\in \C A_q^m$ that
\begin{align*}
\C D^H_q \C D_t^V \pma{\xi \\ \eta} &= \C D^H_q \pma{t^{-\frac{n-1}{2}}\left[\frac{n-1}{2}\right]_t\xi \\ - t^{-\frac{m+1}{2}}\left[\frac{m+1}{2}\right]_t \eta}= \pma{q^{-\frac{1}{2}} t^{ -\frac{m+1}{2}}\left[\frac{m+1}{2}\right]_t\pa_{fk^{-1}}\eta \\     -q^{\frac12 }t^{-\frac{(n-1)}{2}}\left[\frac{n-1}{2}\right]_t \pa_{ek^{-1}}\xi }\\
&= -\pma{ t^{-\frac{(m+2)-1}{2}}\big[\frac{(m+2)-1}{2}\big]_t  & 0 \\  0 &  - t^{-\frac{(n-2)+1}{2}}\big[\frac{(n-2)+1}{2}\big]_t } \pma{-q^{-\frac12} \pa_{f k^{-1}}\eta \\ -q^{\frac12 } \pa_{ek^{-1}}\xi } \\
&=- \C D_t^V \C D^H_q  \pma{\xi \\ \eta}.
\end{align*}
Hence, $\C D_t^V$ and $\C D^H_q$ anti-commute on the core $\C O (SU_q(2))^{\oplus 2}$ and from \cite[Proposition 2.3]{MeLe} it therefore follows that $D_t^V$ and $D^H_q$ weakly anti-commute in the sense of  \cite[Definition 2.1]{MeLe}. An application of \cite[Theorem 2.6]{MeLe} therefore gives that $D_t^V + D^H_q $ is selfadjoint on  $\T{Dom}(D_t^V)\cap \T{Dom}(D^H_q)$; see also \cite{Mes:UCN, KaLe:LGR}. Hence $\overline{\C D^V_t + \C D^H_q} \subseteq D_t^V + D^H_q$ and since both operators are selfadjoint the opposite inclusion follows trivially.
\end{proof}

\subsection{The origin of the Dirac operators}\label{subsec:comparison-with-KS-and-KRS}
We now describe the precise relationship between the Dirac operators constructed above and those introduced in  \cite{KaSe:TST} and \cite{KRS:RFH}.
Setting $t=q$, a direct computation verifies that 
\[
\C D_q :=\C D^V_q+\C  D^H_q =
\begin{cases} \SmallMatrix{ \frac{ 1 - q \pa_{k^{-2}}}{q - q^{-1}}  & -q^{-1/2} \pa_{fk^{-1}}  \\ -q^{1/2} \pa_{ek^{-1}} & \frac{q^{-1}\pa_{k^{-2}} - 1}{q - q^{-1}}  } &\mbox{ for } q\in (0,1) \\
\SmallMatrix{ \frac{1}{2} (\pa_h - 1) & -\pa_f \\ -\pa_e & -\frac{1}{2} (\pa_h + 1) } & \mbox{ for } q=1 \end{cases} .
\]
Comparing with the Dirac operator $\C D_q^{\T{KS}}$ introduced in \cite{KaSe:TST} we then have the identity
\[
\C D_q = \pma{0 & 1 \\ - 1 & 0} \C D_q^{\T{KS}} \pma{0 & - 1 \\ 1 & 0} .
\]
In \cite{KRS:RFH}, Kr\"ahmer, Rennie and Senior proposed another candidate for a Dirac operator,  $\C D_q^{\T{KRS}}$, which they apply to construct a non-trivial twisted Hochschild 3-cocycle; see \cite[Theorem 3.5]{KRS:RFH}. This provides one way of formalising the intuition that $SU_q(2)$ ought to have dimension 3 as a non-commutative manifold, avoiding the typical dimension drop phenomenon. In our notation, their Dirac operator is given by
\[
\C D_q^{\T{KRS}}:= \C D^V_1 + \underbrace{\pma{ 0 & q^{-1/2}\pa_{kf} \\ q^{1/2}\pa_{ke} &0 }}_{=: \C D^{H}_{\T{KRS}}}.
\]
The relationship between our horizontal Dirac operator and the horizontal Dirac operator introduced by Kr\"ahmer, Rennie and Senior is governed by the unbounded strictly positive operator $\Gamma_{q,0}$ via the relation
\begin{equation}\label{eq:dampening}
\Gamma_{q,0} \C D^{H}_{\T{KRS}} \Gamma_{q,0}=- \C D^H_q .
\end{equation}

 The vertical and horizontal Dirac operators $D_1^V$ and $D_q^H$ are also compatible with the unbounded Kasparov product in a way which we will now explain; see \cite{KaLe:SFU,Mes:UCN,MeRe:NMU}. It is however important to realise that the triple $(C(SU_q(2)), L^2(SU_q(2))^{\op 2}, D_1^V + D_q^H)$ is \emph{not} a spectral triple unless $q = 1$, so that we are formally beyond the scope of the current state of the art in unbounded $KK$-theory. We let $D_q^0$ denote the Dirac operator associated with the D\c{a}browski-Sitarz spectral triple $(C(S_q^2),H_q^1 \op H_q^{-1},D_q^0)$; see \cite{DaSi:DSP}. We are going to discuss this even spectral triple in more details in Section \ref{ss:podles-revisited}, but record for the moment that $D_q^0$ agrees with the closure of the unbounded symmetric operator
\[
\C D_q^0 := \pma{0 & -\pa_f \\ - \pa_e & 0} \colon \C A_q^1 \op \C A_q^{-1} \to H_q^1 \op H_q^{-1} .
\]
The grading operator on $H_q^1 \op H_q^{-1}$ is denoted by $\ga := \SmallMatrix{1 & 0 \\ 0 & -1}$, and the derivation on $\C O(S_q^2)$ coming from $D_q^0$ by taking commutators is denoted by $\pa^0 \colon \C O(S_q^2) \to \B B(H_q^1 \op H_q^{-1})$.
Let $E$ denote the Hilbert $C^*$-module obtained by completing $\C O(SU_q(2))$ with respect to the $C(S_q^2)$-valued inner product given by $\inn{x,y} := \Pi^L_0(x^* y)$. We may turn $E$ into a $C^*$-correspondence from $C(SU_q(2))$ to $C(S_q^2)$ where the left action of $C(SU_q(2))$ is induced by the product structure in $\C O(SU_q(2))$. The $C^*$-correspondence $E$ can moreover be equipped with the unbounded selfadjoint and regular operator $N \colon \T{Dom}(N) \to E$ defined on the core $\C O(SU_q(2)) \su E$ by putting $N(x) = n \cd x$ whenever $x \in \C A^n_q$. The pair $(C(SU_q(2)),E,N)$ is then an odd unbounded Kasparov module from $C(SU_q(2))$ to $C(S_q^2)$;  see \cite{CNNR:TEK} for more details.
Following the scheme of unbounded $KK$-theory, we should in principle  be able to form the unbounded Kasparov product of the odd unbounded Kasparov module $(C(SU_q(2)),E,N)$ and the even spectral triple $(C(S_q^2),H_q^1 \op H_q^{-1},D_q^0)$. The result of this operation is in general not a spectral triple on $C(SU_q(2))$, but we may still investigate the involved unbounded operators on the Hilbert space $E \hot_{C(S_q^2)} (H_q^1 \op H_q^{-1})$, which arises as the interior tensor product between the $C^*$-correspondence $E$ and the $C^*$-correspondence $H_q^1 \op H_q^{-1}$. 
The interior tensor product $E \hot_{C(S_q^2)} (H_q^1 \op H_q^{-1})$ is isomorphic to $L^2(SU_q(2))^{\op 2}$ and the isomorphism is induced by the product structure in $\C O(SU_q(2))$. The unbounded selfadjoint and regular operator $N \colon \T{Dom}(N) \to E$ gives rise to the unbounded selfadjoint and regular operator $N \hot \ga \colon \T{Dom}(N \hot 1) \to E \hot_{C(S_q^2)} (H_q^1 \op H_q^{-1})$ which is given by $N \ot \ga$ on the core $\T{Dom}(N) \ot_{C(S_q^2)} (H_q^1 \op H_q^{-1}) \su E \hot_{C(S_q^2)} (H_q^1 \op H_q^{-1})$. Under the isomorphism between $E \hot_{C(S_q^2)} (H_q^1 \op H_q^{-1})$ and $L^2(SU_q(2))^{\op 2}$ it can be verified that $N \hot \ga$ agrees with $D^V_1$. This explains the relationship between the vertical Dirac operator $D^V_1$ and the expected formula from unbounded $KK$-theory.

In order to explain the relationship between the horizontal Dirac operator and constructions appearing in unbounded $KK$-theory, we define the \emph{Gra\ss mann connection}
\[
\Na \colon \C O(SU_q(2)) \to E \underset{C(S_q^2)}{\hot} \B B(H_q^1 \op H_q^{-1})
\]
by putting $\Na(x) := \sum_{i = 0}^n (u^n_{i0})^* \ot \pa^0(u^n_{i0} \cd x)$ whenever $x$ belongs to the algebraic spectral subspace $\C A^n_q \su \C O(SU_q(2))$. Combining this Gra\ss mann connection with the Dirac operator from the D\c{a}browski-Sitarz spectral triple we obtain the linear map

\begin{align*}
 1 \ot_\Na \C D_q^0 \colon \C O(SU_q(2)) \underset{\C O(S_q^2)}{\ot} (\C A_q^1 \op \C A_q^{-1}) &\longrightarrow E\underset{C(S_q^2)}{\hot} (H_q^1 \op H_q^{-1}) 
\end{align*}
given by $ (1 \ot_\Na \C D_q^0)(x \ot y) := \Na(x)(y) + x \ot \C D_q^0(y)$, where the domain agrees with the balanced tensor product $\C O(SU_q(2)) \ot_{\C O(S_q^2)} (\C A_q^1 \op \C A_q^{-1})$. It can then be verified that $1 \ot_\Na \C D_q^0$ induces an unbounded symmetric operator on the Hilbert space $E \hot_{C(S_q^2)} (H_q^1 \op H_q^{-1})$ and this unbounded symmetric operator is unitarily equivalent to $\C D_H^{\T{KRS}} \colon \C O(SU_q(2))^{\op 2} \to L^2(SU_q(2))^{\op 2}$. The dampening procedure applied in \eqref{eq:dampening} in order to pass from the horizontal Dirac operator $\C D_H^{\T{KRS}}$ to the horizontal Dirac operator $\C D^H_q$ appears in many places and is systematically investigated in \cite{Kaa:DAH,Kaa:UKM} from the point of view of unbounded $KK$-theory. We record however that the modular operators applied in \cite{Kaa:DAH,Kaa:UKM} are all assumed to be bounded (even though inverses are allowed to be unbounded).

\subsection{Bounded twisted commutators}
Recall that $\pi \colon C(SU_q(2)) \to \B B\big( L^2(SU_q(2))^{\op 2} \big)$ denotes the injective $*$-homomorphism obtained by letting the GNS representation $\rho$ act diagonally. We now wish to describe the interaction between the coordinate algebra $\C O(SU_q(2))$ and the horizontal and vertical Dirac operators. To this end, it is convenient to introduce the linear maps $\pa^1$ and $\pa^2 \colon \C O(SU_q(2)) \to \C O(SU_q(2))$ given by the formulae 
\begin{equation}\label{eq:twideriI}
\pa^1 := q^{1/2} \pa_e, \q \pa^2 := q^{-1/2} \pa_f,
\end{equation}
as well as the linear map $\pa^3_t \colon \C O(SU_q(2)) \to \C O(SU_q(2))$ given by
\begin{equation}\label{eq:twideriII}
\pa^3_t(x) := [n/2]_t \cd x \q \T{for all } x \in \C A^n_q .
\end{equation}


The following lemma shows that suitably twisted commutators with the horizontal and vertical Dirac do indeed give rise to bounded operators, which may be explicitly described via the maps just introduced. Note that for $t=q$ the twist is the same and in this case the lemma below becomes the statement from  \cite[Lemma 3.2]{KaSe:TST}; cf.~Section \ref{subsec:comparison-with-KS-and-KRS}.

\begin{lemma}\label{l:twicommu}
For each $x \in \C O(SU_q(2))$, it holds that the twisted commutators
\[
\begin{split}
 \C D^H_q \cd \si_L(q^{1/2},x) - \si_L(q^{-1/2},x) \cd \C D^H_q \colon  \C O(SU_q(2))^{\op 2} &\longrightarrow L^2(SU_q(2))^{\op 2} \q \mbox{and} \\
 \C D^V_t \cd \si_L(t^{1/2},x) - \si_L(t^{-1/2},x) \cd \C D^V_t : \C O(SU_q(2))^{\op 2}& \longrightarrow L^2(SU_q(2))^{\op 2}
\end{split}
\]
extend to bounded operators on $L^2(SU_q(2))^{\op 2}$ given, respectively, by
\[
\pa_q^H(x) := \pma{ 0 & - \pa^2(x) \\ -\pa^1(x) & 0 } \q \mbox{and} \q \pa^V_t(x) := \pma{\pa^3_t(x) & 0 \\ 0 & -\pa^3_t(x) } .
\]
\end{lemma}
\begin{proof}
Note first that $\sigma_L(s^{\pm \frac12}, -)$ preserves $\C O(SU_q(2))$ for all $s\in (0,1]$ by Lemma \ref{l:sigmaonspec}, so that the compositions in the lemma are indeed well-defined. By linearity, it suffices to fix an  $n\in \zz$ and prove the statements for  $x\in \C A_q^n$. It then holds that  $\sigma_L(q^{\pm \frac12},x)=q^{\pm \frac{n}{2}}x=\pa_{k^{\pm 1} }(x)$. Using the twisted Leibniz rule from \eqref{eq:twisted-leibnitz-for-pae-and-paf}, the first formula may now be verified by a direct computation.  For the second equality, one computes the twisted commutator on an arbitrary vector in $\C A_q^k \oplus \C A_q^m$, and again a direct computation yields the desired formula.
\end{proof}

In classical Riemannian spin geometry, it is well known (see e.g.~\cite[Chapter 6, Lemma 1]{Con:NCG}) that a continuous function has bounded commutator with the Dirac operator exactly if the function in question is Lipschitz with respect to the Riemannian metric. Our next aim is to provide a suitable counterpart for the algebra of Lipschitz functions in the $q$-deformed setting. We recall that both of the parameters $t$ and $q$ in $(0,1]$ are currently fixed.

\begin{dfn}\label{def:Lipschitz-elements}
Let $x \in C(SU_q(2))$. We say that $x$ is \emph{horizontally Lipschitz} when
\begin{enumerate}
\item $x$ is analytic of order $-\log(q)/2$;
\item the bounded operator $\si_L(q^{1/2},x)$ preserves the domain of $D^H_q$; 
\item the twisted commutator
\[
D^H_q \cd \si_L(q^{1/2},x) - \si_L(q^{-1/2},x) \cd D^H_q \colon \T{Dom}(D^H_q) \to L^2(SU_q(2))^{\op 2}
\]
extends to a bounded operator $\pa_q^H(x)$ on $L^2(SU_q(2))^{\op 2}$.  The set of horizontally Lipschitz elements is denoted $\T{Lip}^{H}(SU_q(2))$.
\end{enumerate}
We say that $x$ is \emph{vertically Lipschitz} when
\begin{enumerate}
\item $x$ is analytic of order $-\log(t)/2$;
\item the bounded operator $\si_L(t^{1/2},x)$ preserves the domain of $D^V_t$; 
\item the twisted commutator
\[
D^V_t \cd \si_L(t^{1/2},x) - \si_L(t^{-1/2},x) \cd D^V_t \colon \T{Dom}(D^V_t) \to L^2(SU_q(2))^{\op 2}
\]
extends to a bounded operator $\pa^V_t(x)$ on $L^2(SU_q(2))^{\op 2}$.  The set of vertically Lipschitz elements is denoted $\T{Lip}_t^{V}(SU_q(2))$
\end{enumerate} 
We apply the notation $\T{Lip}_t(SU_q(2))$ for the subset of $C(SU_q(2))$ consisting of elements which are both horizontally and vertically Lipschitz. 
\end{dfn}

A few remarks are in place. The subset $\T{Lip}_t(SU_q(2)) \su C(SU_q(2))$ is in fact a unital $*$-subalgebra which we refer to as the \emph{Lipschitz algebra}. Moreover, we obtain from Lemma \ref{l:twicommu} that $\C O(SU_q(2)) \su \T{Lip}_t(SU_q(2))$ and hence that $\T{Lip}_t(SU_q(2))$ is norm-dense in $C(SU_q(2))$. The basic algebraic properties of the linear maps
\[
\pa_q^H \T{ and } \pa^V_t  \colon \T{Lip}_t(SU_q(2)) \longrightarrow \B B\big( L^2(SU_q(2))^{\op 2} \big)
\]
can be summarised as follows:

\begin{lemma}\label{l:twilip}
The linear maps $\pa_q^H, \pa^V_t  \colon \T{Lip}_t(SU_q(2)) \longrightarrow \B B\big( L^2(SU_q(2))^{\op 2} \big)$ are twisted $*$-derivations, in the sense that the formulae 
\[
\begin{split}
\pa_q^H(x^*) & = - \pa_q^H(x)^* \, \, , \, \, \, \, \,  \pa_q^H(x \cd y) = \pa_q^H(x) \si_L(q^{1/2},y) + \si_L(q^{-1/2},x) \pa_q^H(y) \q \mbox{and} \\
\pa^V_t(x^*) & = - \pa^V_t(x)^* \, \, , \, \, \, \, \, \pa^V_t(x \cd y) = \pa^V_t(x) \si_L(t^{1/2},y) + \si_L(t^{-1/2},x) \pa^V_t(y)
\end{split}
\]
hold for all $x,y \in \T{Lip}_t(SU_q(2))$.
\end{lemma}
\begin{proof}
The twisted Leibniz rules are verified by a direct computation, and the $*$-compatibility follows from the selfadjointness of the involved unbounded operators and the formula $\sigma_L(s^{\frac12}, x)^* =\sigma_L(s^{-\frac12}, x^*)$, which can be derived from \eqref{eq:sigmaL-relations}.
\end{proof}

We are interested in the linear map
\[
\pa_{t,q} := \pa^V_t + \pa^H_q \colon \T{Lip}_t(SU_q(2)) \longrightarrow \B B\big( L^2(SU_q(2))^{\op 2} \big) .
\]
It is important to clarify that $\pa_{t,q}$ is \emph{not} a twisted derivation unless $t = q$. It does however hold that $\pa_{t,q}(x^*) = - \pa_{t,q}(x)^*$ for all $x \in \T{Lip}_t(SU_q(2))$. Later on, in Proposition \ref{p:twistdericlos}, we shall moreover see that $\pa_{t,q}$ is closable for the norm topology. \\

Let us denote the standard matrix units in $\B M_2(\cc)$ by $e_{ij}$, $i,j \in \{0,1\}$, and introduce the twisted derivations $\pa^1, \pa^2, \pa^3_t \colon \T{Lip}_t(SU_q(2)) \to \B B\big( L^2(SU_q(2)) \big)$ by putting
\[
\pa^1(x) := - e_{11} \cd \pa_{t,q}(x) \cd e_{00} \, \, ,  \, \, \, 
\pa^2(x) := - e_{00} \cd \pa_{t,q}(x) \cd e_{11} \, \, \T{ and } \ \, \, \, 
\pa^3_t(x) := e_{00} \cd \pa_{t,q}(x) \cd e_{00}
\]
for all $x \in \T{Lip}_t(SU_q(2))$. By Lemma \ref{l:twicommu}, this notation is compatible with the notation introduced in \eqref{eq:twideriI} and \eqref{eq:twideriII}.
The adjective \emph{twisted} above is here to be understood in the sense of Definition \ref{def:twisted-derivation} where the twists are given by $\sigma(q^{1/2}, \cdot )$ and $\sigma(q^{-1/2}, \cdot )$ for $\pa^1$ and $\pa^2$, and  by $\sigma(t^{1/2},\cdot)$ and $\sigma(t^{-1/2}, \cdot)$ for $\pa_t^3$.


\begin{remark}\label{rem:pat-on-Lipschitz}
Let $x\in \T{Lip}_t(SU_q(2))$ be given. A direct computation shows that $\binn{\zeta',e_{00} \pa_q^H(x) e_{00} \cd \zeta}= \binn{\zeta',e_{11} \pa_q^H(x)e_{11} \cd \zeta}=0$ for all $\zeta, \zeta'\in \C O(SU_q(2))^{\oplus 2}$. We thereby obtain that $\pa_q^H(x)=\SmallMatrix{ 0 & -\pa^2(x) \\ - \pa^1(x) & 0  }$. Similarly, one sees that $ \pa^V_t(x)=\SmallMatrix{ \pa_t^3(x) & 0\\ 0 & \pa_t^4(x) }$ for some twisted derivation $\pa^4_t \colon \T{Lip}_t(SU_q(2)) \to \mathbb{B}\big(L^2(SU_q(2))\big)$. As a consequence, the following inequality holds:
\begin{align}\label{eq:praktisk-ulighed}
\max\left\{ \|\pa_t^V(x)\|, \|\pa_q^H(x)\|  \right\} \leq  \| \pa_{t,q}(x) \|  
\end{align}
In analogy with the algebraic case described in Lemma \ref{l:twicommu}, we shall later show (see Remark \ref{rem:form-of-patv}) that $\pa^4_t(x)=-\pa_t^3(x)$, implying that
\begin{align*}
\pa_{t,q}(x)= \pma{ \pa_t^3(x) &  -\pa^2(x)\\  -\pa^1(x) & -\pa_t^3(x)  } \q \T{for all } x \in \T{Lip}_t(SU_q(2)).
\end{align*}
\end{remark}

\begin{dfn}\label{def:max-and-min-seminorm}
We define two seminorms, $L_{t,q}$ and  $L^{\max}_{t,q}$, on $C(SU_q(2))$ by setting
\begin{align*}
L_{t,q}(x)&:= \begin{cases} \|\pa_{t,q}(x)\| & {\T{for }} x\in \C O(SU_q(2))\\
\infty & {\T{for }}  x\in C(SU_q(2))\setminus \C O(SU_q(2))
\end{cases} \\
L_{t,q}^{\max}(x)&:= \begin{cases} \|\pa_{t,q}(x)\| & {\T{for }} x\in \T{Lip}_t(SU_q(2))\\
\infty & {\T{for }} x\in C(SU_q(2))\setminus \T{Lip}_t(SU_q(2))
\end{cases}
\end{align*}
The (extended) metrics on $\C S(SU_q(2))$ induced by the seminorms $L_{t,q}$ and $L_{t,q}^{\max}$ through the formula \eqref{eq:connes-metric} will be denoted $d_{t,q}$ and $d_{t,q}^{\max}$, respectively.
\end{dfn}
\begin{remark}
It follows from Lemma \ref{l:twicommu} and Lemma \ref{l:twilip} that $L_{t,q}^{\T{max}}$ and $L_{t,q}$ are both Lipschitz seminorms in the sense of Definition \ref{d:lipschitzsem}. 
\end{remark}

In Latr\'emoli\`ere's approach to the quantised Gromov-Hausdorff distance \cite{ Lat:DGH, Lat:QGH}, a central role is played by an axiom demanding that the seminorm in question satisfies a certain Leibniz inequality \cite[Equation (1.1)]{Lat:QGH}. Since $\pa_{t,q}$ is not a derivation, we only get a twisted version of the Leibniz inequality, where the operator norm appearing in  \cite[Equation (1.1)]{Lat:QGH} is replaced by the norm $\| \cdot \|_{t,q}$ introduced in Section \ref{ss:analytic}.

\begin{lemma}\label{l:leftrightbound}
Let $x,y \in \T{Lip}_t(SU_q(2))$. Then we have the estimate
\[
\begin{split}
L_{t,q}^{\T{max}}(x \cd y) & \leq 
\|x\|_{t,q} \cd L_{t,q}^{\T{max}}(y) 
+ L_{t,q}^{\T{max}}(x) \cd \| y \|_{t,q} .
\end{split}
\]
\end{lemma}
\begin{proof}
Let $x,y \in \T{Lip}_t(SU_q(2))$. We first notice that the following inequalities hold: 
\[
\begin{split}
\| \pa_q^H(x \cd y) \| & \leq \| \pa_q^H(x)\| \cd \| \si_L(q^{1/2},y) \| + \| \si_L(q^{-1/2},x) \| \cd \| \pa_q^H(y) \| \\
& \leq L_{t,q}^{\T{max}}(x) \cd \| \si_L(q^{1/2},y) \| + \| \si_L(q^{-1/2},x) \| \cd L_{t,q}^{\T{max}}(y) .
\end{split}
\]
Since a similar computation shows that
\[
\| \pa^V_t(x \cd y) \| \leq L_{t,q}^{\T{max}}(x) \cd \| \si_L(t^{1/2},y) \| + \| \si_L(t^{-1/2},x) \| \cd L_{t,q}^{\T{max}}(y) ,
\]
we obtain the result of the present lemma.
\end{proof}

One of the main results of the present paper is Theorem \ref{introthm:SUq2-is-a-cqms}, which shows that $L_{t,q}^{\T{max}}$ turns $C(SU_q(2))$ into a compact quantum metric space. Knowing this,  it then follows (cf.~Theorem \ref{thm:rieffels-criterion}) that $L_{t,q}$ also has this property.  The proof of Theorem \ref{introthm:SUq2-is-a-cqms}  is contained in Section \ref{sec:quantum-metrics-on-quantum-su2} below, but before proceeding to this, we will need to carry out a rather detailed analysis of the spectral geometry on $SU_q(2)$ arising from the horizontal and vertical Dirac operators introduced above. We first show how one recovers the classical spin geometry on $SU(2)$ when $t = q = 1$.

\subsection{Comparison with the classical Dirac operator}\label{sec:comparison-with-classical}
In this section, we analyse the classical case where both of the parameters $t$ and $q$ are equal to one. 
Consider therefore the compact Lie group $SU(2)$ of special unitary $2 \ti 2$-matrices. The unital $C^*$-algebra of continuous functions on $SU(2)$ agrees with $C(SU_1(2))$ and the fundamental representation $U \colon SU(2) \to U(\cc^2)$ identifies with the fundamental unitary $u \in \B M_2\big(C(SU_1(2))\big)$. We equip $SU(2)$ with the Haar measure $\mu$ and record that the corresponding state on $C(SU(2))$ agrees with the Haar state $h\colon C(SU_1(2)) \to \cc$. In particular, the Hilbert space of (equivalence classes) of square integrable functions $L^2(SU(2))$ coincides with $L^2(SU_1(2))$. 
We are now going to explain how the classical Dirac operator on $SU(2)$ identifies with the sum of the vertical and horizontal Dirac operators, $\C D^V_1$ and $\C D^H_1$, from \eqref{eq:dirvert} and \eqref{eq:dirhori} up to rescaling and addition of a constant. \\

The Lie algebra of $SU(2)$ is denoted by $\G{su}(2)$ and is explicitly given by the space of skew-hermitian $(2 \ti 2)$-matrices of trace zero. We equip the Lie algebra $\G{su}(2)$ with the inner product defined by
\[
\inn{X,Y} := \T{TR}(X^* Y) \q \T{for all } X,Y \in \G{su}(2) ,
\]
where $\T{TR} \colon \B M_2(\cc) \to \cc$ denotes the normalised trace satisfying that $\T{TR}(1) = 1$. We single out the orthonormal basis for $\G{su}(2)$ consisting of the matrices
\[
X_1 := \pma{0 & - 1 \\ 1 & 0 } \q X_2 := \pma{0 & i \\ i & 0} \q X_3 := \pma{i & 0 \\ 0 & - i} .
\]
The elements in $\G{su}(2)$ can be identified with left-invariant vector fields on $SU(2)$. Indeed, for each element $X \in \G{su}(2)$ one obtains a derivation $X \colon C^\infty(SU(2)) \to C^\infty(SU(2))$ by the formula
\begin{align}\label{eq:inv-vector-field}
X(f)(g) := \frac{d}{dt}\big( f( g \cd e^{tX} ) \big)\big|_{t = 0} 
\q \T{for all } f \in C^\infty(SU(2)) \, , \,\, g \in SU(2) .
\end{align}
In this way, the inner product on the Lie algebra $\G{su}(2)$ yields a Riemannian metric on $SU(2)$ and therefore in particular a metric on $SU(2)$. Upon identifying $SU(2)$ with the $3$-sphere $S^3$ via the map
\[
\pma{z_1 & -\ov{z_2} \\ z_2 & \ov{z_1}} \mapsto (z_1,z_2) 
\]
it can be verified that the corresponding metric on $S^3$ agrees with the classical round metric. This means that $S^3$ sits inside $\rr^4$ as a sphere of radius one, or more precisely that the standard inclusion $S^3 \to \rr^4$ becomes a Riemannian immersion.\\ 

The spinor bundle for $SU(2)$ is the trivial complex hermitian vector bundle of rank $2$. The fundamental representation of the Lie algebra $\G{su}(2)$ on $\cc^2$ induces a representation of the Clifford algebra associated to $\G{su}(2)$ on $\cc^2$. The classical Dirac operator $\C D_{S^3} \colon C^\infty(SU(2))^{\op 2} \to L^2(SU(2))^{\op 2}$ on $SU(2)$ is then given by the expression
\[
\C D_{S^3}(\xi) := \sum_{i = 1}^3 X_i \cd X_i(\xi) = \pma{ i X_3(\xi) & -X_1(\xi) + i X_2(\xi) \\ X_1(\xi) + i X_2(\xi) & - i X_3(\xi)} ;
\]
see for example \cite[Section 3.5]{Friedrich:Dirac}. Notice that we are here considering $\C D_{S^3}$ as an unbounded operator on the Hilbert space of $L^2$-sections of the spinor bundle. We denote the closure of $\C D_{S^3}$ by $D_{S^3}$ and record that $D_{S^3}$ is a selfadjoint unbounded operator.\\

At the level of the coordinate algebra $\C O(SU_1(2))$, which we tacitly identify with a unital $*$-subalgebra of $C^\infty(SU(2))$, we now single out the correspondence between the derivations associated to $X_1,X_2,X_3 \in \G{su}(2)$ and the derivations $\pa_e,\pa_f,\pa_h$ defined in Section \ref{s:quantumsu2}. 
Using the formula \eqref{eq:inv-vector-field} on may verify the relations
\[
\pa_e = -\frac{1}{2} (X_1 + i X_2) \q \pa_f = \frac{1}{2} (X_1 - i X_2) \q \pa_h = i X_3 ,
\]
directly on the generators $a,b,a^*,b^*$,  and since all maps are derivations the same relations hold on all of   $\C O(SU_1(2))$. We may thus rewrite the unbounded operator $\C D^V_1 + \C D^H_1 \colon \C O(SU_1(2))^{\op 2} \to L^2(SU_1(2))^{\op 2}$ as follows:
\[
\C D^V_1 + \C D^H_1 = \frac{1}{2} \cd \pma{ i X_3 & -X_1 + i X_2 \\ X_1 + i X_2 & - i X_3} - \frac{1}{2} .
\]
At the level of unbounded operators on $L^2(SU(2))^{\op 2}$ we therefore obtain that $2 \cd (\C D^V_1 + \C D^H_1) + 1 \su \C D_{S^3}$. Since both of the unbounded operators $2 \cd (\C D^V_1 + \C D^H_1) + 1$ and $\C D_{S^3}$ are essentially selfadjoint we conclude that their closures agree,  resulting in the identity 
\[
2 \cd \ov{\C D^V_1 + \C D^H_1} + 1 = D_{S^3} .
\]
We moreover recall from Lemma \ref{l:sumdirac} that $\ov{\C D^V_1 + \C D^H_1} = D_1^V + D_1^H$.
Lastly, we spell out some consequences of the above identity of Dirac operators from the point of view of quantum metric spaces. Let us denote the classical round metric by $d_{S^3} \colon S^3 \ti S^3 \to [0,\infty)$ and the corresponding Lipschitz algebra by $\T{Lip}(S^3)$. The Lipschitz constant associated to a Lipschitz function $f \colon S^3 \to \cc$ is denoted by $L_{\T{Lip}}(f)$.
For each point $p \in S^3$ we apply the notation $\T{ev}_p \colon C(SU_1(2)) \to \cc$ for the pure state given by evaluation in the point $p$. We are here suppressing the $*$-isomorphisms $C(SU_1(2)) \cong C(SU(2)) \cong C(S^3)$. 

\begin{theorem}\label{t:classical}
The pair $\big( C(SU_1(2)), L_{1,1}^{\T{max}} \big)$ is a compact quantum metric space. The Lipschitz algebra $\T{Lip}_1(SU_1(2))$ identifies with the Lipschitz algebra $\T{Lip}(S^3)$ and for every $f \in \T{Lip}(S^3)$ it holds that
\[
L_{1,1}^{\T{max}}(f) = \frac{1}{2} L_{\T{Lip}}(f) .
\]
In particular, for every pair of points $p_0,p_1 \in S^3$ we obtain the formula
\[
2 \cd d_{S^3}(p_0,p_1) = d_{1,1}^{\max}(\T{ev}_{p_0},\T{ev}_{p_1}) ,
\]
where the metric on the right hand side denotes the Monge-Kantorovi\v{c} metric associated with the Lip-norm $L_{1,1}^{\max}$. 
\end{theorem}
\begin{proof}
A continuous function $f \colon S^3 \to \cc$ has bounded commutator with $D_{S^3}=2\cd (D^V_1 + D^H_1 )+1$ if and only if $f$ is Lipschitz with respect to $d_{S^3}$ \cite[Chapter 6, Lemma 1]{Con:NCG}, and by the paragraph following \cite[Chapter 6, Lemma 1]{Con:NCG} one has that $\big\| \ov{ [D_{S^3}, f] } \big\|$ equals the Lipschitz constant $L_{\T{Lip}}(f)$.
 Since $t=q=1$, all twists appearing in the definition of the Lipschitz algebra $\T{Lip}_1(SU_1(2))$ are trivial. Using that $D_{S^3} = 2 \cd (D^V_1 + D^H_1) + 1$ and, in particular, that the domain of $D_{S^3}$ is the intersection of the domains of $D_1^V$ and $D_1^H$, it can then be verified that a continuous function $f \colon S^3 \to \cc$ has bounded commutator with $D_{S^3}$ if and only if $f$ is both vertically and horizontally Lipschitz (meaning that $f$ has bounded commutators with $D^V_1$ and with $D^H_1$). The Lipschitz algebra $\T{Lip}_1(SU_1(2))$ therefore agrees with the Lipschitz algebra $\T{Lip}(S^3)$ and the formula $L_{1,1}^{\T{max}}(f) = \frac{1}{2} L_{\T{Lip}}(f)$ now follows. The comparison formula for the two metrics $d_{S^3}$ and $d_{1,1}^{\max}$ is now a consequence of  \cite[Chapter 6, Formula 1]{Con:NCG}; see also  \cite[Proposition 1]{Con:CFH}.
\end{proof}

\subsection{The real structure} \label{subsec:first-order-condition}
In Connes' non-commutative geometry, one encounters the notion of a \emph{real structure} for a spectral triple $(A, H, D)$; see \cite{Con:GFN}. A real structure captures the dimension (modulo 8) of the non-commutative spin manifold in question and is encoded by an antilinear  unitary $J\colon H\to H$ (subject to a couple of conditions).  Even though we are working on the borderline of non-commutative geometry we shall nevertheless show that one may define an analogue of a real structure in our setting. As one would expect, this real structure gives $SU_q(2)$ real dimension 3; see Remark \ref{rem:twisted-real-structure} below for more details. \\

Let us fix the parameters $t,q \in (0,1]$. Define the antilinear map $\mathcal{J}\colon \C O(SU_q(2))\to \C O(SU_q(2))$ by setting $\mathcal{J}(x) = (\pa_k \de_k)(x^*)$. Using that the modular automorphism $\nu$ is given by $\de_{k^{-2}}  \pa_{k^{-2}} \colon \C O(SU_q(2)) \to \C O(SU_q(2))$ a direct computation shows that $\mathcal{J}$ extends to an antilinear unitary $J$ on $L^2(SU_q(2))$. In fact, $J$ is the modular conjugation arising when applying Tomita-Takesaki theory (see e.g.~\cite[Chapter VI]{Takesaki-vol-II}) to the left Hilbert algebra $\C O(SU_q(2))$ equipped with the inner product $\inn{x,y} := h(x^* y)$. In particular, it therefore holds that $JL^\infty(SU_q(2))J = L^\infty(SU_q(2))'$; see \cite[Chapter VI , Theorem 1.19]{Takesaki-vol-II}. \\

We now define the antilinear map $\C I  := \pma{0 & \C J \\ - \C J & 0} \colon  \C O(SU_q(2))^{\op 2} \to \C O(SU_q(2))^{\op 2}$ together with the associated antilinear unitary operator $I := \pma{0 & J \\ - J & 0} \colon  L^2(SU_q(2))^{\op 2} \to L^2(SU_q(2))^{\op 2}$.  We record that $I^2 = -1$. This is the map that will be our substitute for a real structure, and our next aim is therefore to prove a version of the first order condition, which in our setting amounts to a relation of the form $[\pa_{t,q}(x), IyI]=0$; see Proposition \ref{p:firstorder}. To achieve this, the unbounded operator $\Ga_{s,0}$ defined in \eqref{eq:modudef} turns out to be essential, and we analyse its interaction with $\C I, \C D_t^V$ and $\C D_q^H$ in the following series of lemmas.

\begin{lemma}\label{l:modudir}
The horizontal Dirac operator $\C D^H_q$ commutes with $\Ga_{q,0}$ and the vertical Dirac operator $\C D^V_t$ commutes with $\Ga_{s,0}$ for all $s \in (0,1]$. 
\end{lemma}
\begin{proof}
Let $n,m \in \zz$. By linearity, it suffices prove the two commutation relations on vectors of the form $\pma{\xi \\ \eta}\in \C A_q^n \oplus \C A_q^m$. Since $\Ga_{q,0}$ preserves the algebraic spectral subspaces and $\pa_e(\C A_q^n)\subseteq \C A_q^{n-2}$ and $\pa_f(\C A_q^m)\subseteq \C A_q^{m+2}$ we obtain that:
\begin{align*}
\C D^H_q \Gamma_{q,0}\pma{\xi \\ \eta} &= \pma{-q^{-\frac12}q^{\frac{-1-m}{2}}  \pa_{fk^{-1}}(\eta) \\ -q^{\frac12}q^{\frac{1-n}{2}}\pa_{ek^{-1}}(\xi) }= \pma{ q^{\frac{1-(m+ 2)}{2}} & 0 \\ 0 & q^{\frac{-1-(n-2)}{2}}  } \C D^H_q \pma{\xi \\ \eta}= \Gamma_{q,0}\C D^H_q \pma{\xi \\ \eta},
\end{align*}
thus proving the first commutation relation.  Since both $\C D_t^V$ and $\Gamma_{s,0}$ are diagonal on $\C A_q^n\oplus \C A_q^m$ they clearly commute here.
\end{proof}


\begin{lemma}\label{l:Irela}
It holds that $\C I \cd \Ga_{s,0}^{-1} =\Ga_{s,0} \cd \C I$ for all $s \in (0,1]$. Moreover, we have the commutation relations
\[
(\C D^H_q \Ga_{q,0}^{-1}) \cd \C I = \C I \cd (\C D^H_q \Ga_{q,0}^{-1}) \quad \mbox{and} \quad 
(\C D^V_t  \Ga_{t,0}^{-1}) \cd \C I = \C I \cd (\C D^V_t  \Ga_{t,0}^{-1}) .
\]
\end{lemma}
\begin{proof}
By linearity, it suffices to check the three commutation relations on subspaces of the form $ \C A_q^n \oplus \C A_q^m$ for arbitrary $n,m\in \zz$. 
The first commutation relation $\C I \cd \Ga_{s,0}^{-1} =\Ga_{s,0} \cd \C I$ follows on $\C A_q^n \oplus \C A_q^m$ by noting that $\C J(\C A_q^k) = \C A_q^{-k}$ for all $k \in \zz$.
For the second commutation relation, we first remark that $\pa_f \C J(\xi) = - \C J \pa_e(\xi)$ for all vectors $\xi \in \C O(SU_q(2))$. Indeed, using the defining relations for $\C U_q(\mathfrak{su}(2))$ from \eqref{eq:deformed-lie-alg-relations} and the $*$-relations from \eqref{eq:pae-and-paf-and-star} we may compute as follows:
\[
\pa_f \C J(\xi) = \pa_f \pa_k \de_k(\xi^*) = \pa_k \de_k \pa_f(\xi^*) \cd q^{-1} = - \pa_k \de_k \pa_e(\xi)^* = -\C J \pa_e(\xi) .
\]
Similarly, one sees that $\pa_e \C J= -\C J \pa_f$, and the second commutation relation then follows by noting that
\[
\C D^H_q \Ga_{q,0}^{-1} = \pma{ 0 & -q^{-1/2} \pa_{fk^{-1}} \\ - q^{1/2} \pa_{ek^{-1}} & 0 } \pma{q^{-1/2} \pa_k & 0 \\ 0 & q^{1/2} \pa_k}
= \pma{0 & - \pa_f \\ - \pa_e & 0} .
\]
To prove the last commutation relation, observe that the restriction of the unbounded operator $\C D_t^V \Gamma_{t,0}^{-1}$ to the subspace $\C A_q^n \oplus \C A_q^m$ is represented by the matrix $\SmallMatrix{[\frac{n-1}{2}]_t & 0 \\ 0 & -[\frac{m+1}{2}]_t  }$. Using one more time that $\C J(\C A_q^k) = \C A_q^{-k}$ for all $k \in \zz$, we now obtain the identity $(\C D^V_t  \Ga_{t,0}^{-1}) \cd \C I = \C I \cd (\C D^V_t  \Ga_{t,0}^{-1})$ on $\C A_q^n \oplus \C A_q^m$ from a direct computation.
\end{proof}

\begin{lemma}\label{l:commureal}
For each $y \in \C O(SU_q(2))$ we have the identities
\begin{alignat}{2}
[\C D^H_q, \C I y \C I]  &= \Ga_{q,0} \cd \C I \pa_q^H(y) \C I \cd \Ga_{q,0}  &&= \C I \pa_q^H(\pa_k(y)) \C I \cd \Ga_{q,0}^2\notag\\
[\C D^V_t, \C I y \C I]  &= \Ga_{t,0} \cd \C I \pa^V_t(y) \C I \cd \Ga_{t,0} &&=\C I \pa^V_t(\si_L(t^{1/2},y)) \C I \cd \Ga_{t,0}^2\notag
\end{alignat}
on the subspace $\C O(SU_q(2))^{\op 2} \su L^2(SU_q(2))^{\op 2}$. 
\end{lemma}
\begin{proof}
Using Lemma \ref{l:analytic}, Lemma \ref{l:modudir} and Lemma \ref{l:Irela}, we may compute as follows:
\[
\begin{split}
\C D^H_q \cd \C I y \C I & = \C D^H_q \Ga_{q,0}^{-1} \cd \Ga_{q,0} \C I y \C I \\
&= \C D^H_q \Ga_{q,0}^{-1} \cd \C I \si_L(q^{1/2},y) \C I \cd \Ga_{q,0} \\
& = \Ga_{q,0} \cd \C I \C D^H_q \si_L(q^{1/2},y) \C I \cd \Ga_{q,0} \\
& = \Ga_{q,0} \cd \C I \pa_q^H(y) \C I \cd \Ga_{q,0}
+ \Ga_{q,0} \cd \C I \si_L(q^{-1/2},y) \C D^H_q \C I \cd \Ga_{q,0} \\
& = \Ga_{q,0} \cd \C I \pa_q^H(y) \C I \cd \Ga_{q,0}
+ \C I y \C I \cd \C D^H_q .
\end{split}
\]
This proves the first identity regarding the commutator with the horizontal Dirac operator. The second one follows by a similar computation, using the same series of lemmas as above:
\begin{align*}
\Ga_{q,0} \C I \pa_q^H(y) \C I \Ga_{q,0} &= \C I \Ga_{q,0}^{-1} \left( \C D^H_q \Ga_{q,0}^{-1}y\Ga_{q,0}- \Ga_{q,0}y \Ga_{q,0}^{-1}\C D^H_q   \right) \C I \Ga_{q,0}\\
&=\C I \left( \C D^H_q \Ga_{q,0}^{-2}y\Ga_{q,0}^2- y \C D^H_q   \right) \Ga_{q,0}^{-1} \C I \Ga_{q,0}\\
&= \C I \pa_q^H(\pa_k(y))\C I \Ga_{q,0}^{2}.
\end{align*}
The remaining identities regarding the commutator with the vertical Dirac operator are proven by completely analogous computations.
\end{proof}

With the above lemmas at our disposal, we may now state and prove the analogue of the first order condition.

\begin{prop}\label{p:firstorder}
For each $y\in L^\infty(SU_q(2))$ and $x\in \T{Lip}_t(SU_q(2))$ we have the identities
\[
[IyI, \pa_q^H(x)] = 0 = [I y I, \pa^V_t(x)] .
\]
\end{prop}
\begin{proof}
Since the von Neumann algebra $L^\infty(SU_q(2))$ agrees with the closure of the coordinate algebra $\C O(SU_q(2)) \su \B B( L^2(SU_q(2)))$ with respect to the strong operator topology, it suffices to treat the case where $y\in \C O(SU_q(2))$. Let thus $y \in \C O(SU_q(2))$ be given. We will just focus on proving that $IyI$ commutes with $\pa^V_t(x)$ since the proof of the analogous result for $\pa_q^H(x)$ follows the same pattern. From Lemma \ref{l:analytic} and Lemma \ref{l:commureal} we obtain the identities
\begin{equation}\label{eq:firstorderI}
\begin{split}
I y I \cd \si_L(t^{-1/2},x) \C D^V_t 
& = \si_L(t^{-1/2},x) \C I y \C I \cd \C D^V_t \\
& = \si_L(t^{-1/2},x) \C D^V_t \cd \C I y \C I - \si_L(t^{-1/2},x) \Ga_{t,0} \cd \C I \pa^V_t(y) \C I \cd \Ga_{t,0} \\
& = \si_L(t^{-1/2},x) \C D^V_t \cd \C I y \C I - \Ga_{t} \cd I \pa^V_t(y) I x \cd \Ga_{t,0} 
\end{split}
\end{equation}
of unbounded operators defined on the dense subspace $\C O(SU_q(2))^{\op 2} \su L^2(SU_q(2))^{\op 2}$. Similarly, using Lemma \ref{l:analytic} and Lemma \ref{l:commureal} one more time, we obtain that
\begin{equation}\label{eq:firstorderII}
\begin{split}
& \binn{ I y I \cd D^V_t \si_L(t^{1/2},x) \xi, \eta } 
= \binn{ \si_L(t^{1/2},x) \xi, \C D^V_t \cd \C I y^* \C I \eta} \\
& \q = \binn{ \si_L(t^{1/2},x) \xi, \C I y^* \C I \cd \C D^V_t \eta}
+ \binn{ \si_L(t^{1/2},x) \xi, \Ga_{t,0} \cd \C I \pa^V_t(y^*) \C I \cd \Ga_{t,0} \eta} \\
& \q = \binn{ D^V_t \si_L(t^{1/2},x) \cd I y I \xi, \eta}
- \binn{ \Ga_t \cd I \pa^V_t(y) I x \cd \Ga_{t,0} \xi, \eta} 
\end{split}
\end{equation}
for all  $\xi,\eta \in \C O(SU_q(2))^{\op 2}$. Combining the identities in \eqref{eq:firstorderI} and \eqref{eq:firstorderII} we see that
\[
\begin{split}
I y I \cd \pa^V_t(x)(\xi)
& = I y I \cd D^V_t \si_L(t^{1/2},x)(\xi) - I y I \cd \si_L(t^{-1/2},x) \C D^V_t(\xi) \\
& = D^V_t \si_L(t^{1/2},x) \cd I y I (\xi) - \Ga_t \cd I \pa^V_t(y) I x \cd \Ga_{t,0}(\xi) \\
& \q - \si_L(t^{-1/2},x) \C D^V_t \cd \C I y \C I(\xi) + \Ga_t \cd I \pa^V_t(y) I x \cd \Ga_{t,0}(\xi) \\
& = \pa^V_t(x) \cd I y I(\xi)
\end{split}
\]
for all  $\xi \in \C O(SU_q(2))^{\op 2}$. This proves the proposition.
\end{proof}

\begin{cor}\label{cor:values-in-L-infty}
The twisted $*$-derivations $\pa_t^V$ and $\pa^H_q \colon \T{Lip}_t(SU_q(2))\to \mathbb{B}\big(L^2(SU_q(2))^{\op 2}\big)$ both take values in $\mathbb{M}_2\big(L^\infty(SU_q(2))\big)$.
\end{cor}
\begin{proof}
Since $J$ is the modular conjugation for the left Hilbert algebra $\C O(SU_q(2))$ with inner product coming from the Haar state, it holds that $L^\infty(SU_q(2))'=JL^\infty(SU_q(2))J$ as an identity between operator algebras in $\mathbb{B}\big(L^2(SU_q(2))\big)$; see \cite[Chapter VI , Theorem 1.19]{Takesaki-vol-II}. For $x \in \T{Lip}_t(SU_q(2))$, it therefore suffices to show that each entry in $\pa_q^H(x), \pa_t^V(x)\in\mathbb{B}\big(L^2(SU_q(2))^{\op 2} \big) =\mathbb{M}_2\big(\mathbb{B}(L^2(SU_q(2)))\big)$ belongs to the commutant $(JL^\infty(SU_q(2))J)'$. For $y\in L^\infty(SU_q(2))$ it holds that 
\[
IyI=-\begin{pmatrix} JyJ & 0 \\ 0 & JyJ \end{pmatrix},
\]
 and hence it suffices to show that $[IyI, \pa_t^V(x)]=[IyI, \pa_q^H(x)]=0$, but this was already proven in Proposition \ref{p:firstorder}.
\end{proof}

\begin{remark}\label{rem:twisted-real-structure}
In the classical setting of non-commutative geometry, a real 3-dimensional structure for an odd spectral triple $(A, H, D)$  with coordinate algebra $\C A \su A$ is given by an antilinear  unitary $J\colon H\to H$. This data is then supposed to satisfy the conditions $J^2=-1$, $DJ=JD$ and for all $a,b\in \C A$ one has $[a, JbJ]=0$ and  $\big[\ov{[D,a]}, JbJ\big]=0$; see \cite{Con:GFN}. In our setting,  the antilinear unitary $I\in \mathbb{B}(L^2(SU_q(2))^{\oplus 2})$ provides the substitute for a real structure. Lemma \ref{l:Irela} may thus be viewed as a twisted analogue of the relation $DJ=JD$,  while  Proposition \ref{p:firstorder}  is the analogue of the first order condition $\big[\ov{[D,a]}, JbJ\big]=0$. The relation $[a, IbI]=0$ also holds by Tomita-Takesaki theory as already remarked in the beginning of the present section. 
\end{remark}

\subsection{The equivariance condition}
We are now going to investigate the equivariance properties of the spectral geometric data governed by our pair of Dirac operators. In some of the literature on Dirac operators on $q$-deformed spaces (see e.g.~\cite{DaSi:DSP, DLSSV:DOS}) the equivariance is to be  understood in the sense that the Dirac operator in question commutes with the right action of $\C U_q(\mathfrak{su}(2))$; i.e. with the diagonal action of operators of the form $\de_{\eta}$ with $\eta\in \C U_q(\mathfrak{su}(2))$ on the core $\C O (SU_q(2))^{\oplus 2}\subseteq L^2(SU_q(2))^{\oplus 2}$. Since $\C D_q^H$ is constructed explicitly using the \emph{left} action  it clearly commutes with $\de_\eta$, and since $\de_\eta$ preserves the spectral subspaces 
it also follows easily that $\C D_t^V$ commutes with $\de_\eta$. Thus, this type of equivariance is basically built into the construction of $D_{t,q}$. 
In this section we shall show another kind of equivariance, in that we will show that our spectral data is compatible with the coproduct on the $C^*$-algebraic quantum group $C(SU_q(2))$. More precisely, we will show in Lemma \ref{l:multiunidir} below that the vertical and horizontal Dirac operators both commute with the multiplicative unitary for $SU_q(2)$, which seems to be an equivariance condition which is more closely related  with the $SU(2)$-equivariance of the classical Dirac operator on $S^3$; see Remark \ref{rem:classical-equivariance} for more details. Throughout the section,  we are still keeping the two parameters $t$ and $q$ in  $(0,1]$ fixed unless explicitly stated otherwise. \\



 Let us consider the Hilbert space tensor product $L^2(SU_q(2)) \hot L^2(SU_q(2))$ and introduce the unitary operator 
\[
W \colon L^2(SU_q(2)) \hot L^2(SU_q(2)) \to L^2(SU_q(2)) \hot L^2(SU_q(2))
\]
given by the formula $W( x \ot y) := \De(y)  \cd ( x \ot 1 )$ for all elements $x,y \in \C O(SU_q(2))$.  We record that $W\big( \C O(SU_q(2)) \ot \C O(SU_q(2))\big) = \C O(SU_q(2)) \ot \C O(SU_q(2))$ and hence that $W^*\big(\C O(SU_q(2)) \ot \C O(SU_q(2))\big) = \C O(SU_q(2)) \ot \C O(SU_q(2))$ as well. The unitary operator $W$ implements the coproduct $\De \colon C(SU_q(2)) \to C(SU_q(2)) \ot_{\T{min}} C(SU_q(2))$ in the sense that
\[
\De(z) = W(1 \ot z) W^* \quad \T{for all } z \in C(SU_q(2)) .
\]
The operator $W$ is referred to as the \emph{multiplicative unitary} for quantum $SU(2)$; see \cite{BaSk:UMD} for more details on these matters. For each $x \in \T{Lip}_t(SU_q(2))$ we may use the multiplicative unitary to make sense of the expressions $\De(\pa_q^H(x))$ and $\De(\pa^V_t(x))$. Indeed, since $\pa_q^H(x)$ and $\pa^V_t(x)$ are bounded operators on $L^2(SU_q(2))^{\op 2}$ we may apply the following definitions:
\[
\begin{split}
\De(\pa_q^H(x)) & := (W \op W)( 1 \ot \pa_q^H(x) )(W \op W)^* \q \T{and} \\
\De(\pa^V_t(x)) & := (W \op W)( 1 \ot \pa^V_t(x) )(W \op W)^* ,
\end{split}
\]
where both of the right hand sides are bounded operators on the Hilbert space tensor product $L^2(SU_q(2)) \hot L^2(SU_q(2))^{\op 2}$.
We would like to commute the coproduct past the twisted $*$-derivations $\pa_q^H$ and $\pa^V_t$ obtaining formulae of the form
\[
(1 \ot \pa_q^H) \De(x) = \De( \pa_q^H(x) ) \, \, \T{ and } \, \, \, (1 \ot \pa^V_t) \De(x) = \De( \pa^V_t(x) ) .
\]
In order to make sense of the left hand sides of these expressions we first investigate the unbounded selfadjoint operators $1 \hot D^H_q$ and $1 \hot D^V_t$, defined, respectively, as the closures of the unbounded symmetric operators
\[
1 \ot \C D^H_q \, \, \T{ and } \, \, \, 1 \ot \C D^V_t \colon \C O(SU_q(2)) \ot \C O(SU_q(2))^{\op 2}
\longrightarrow L^2(SU_q(2)) \hot L^2(SU_q(2))^{\op 2} .
\]

\begin{lemma}\label{l:multiunidir}
The unitary operator $W \op W$ preserves the subspaces $\T{Dom}(1 \hot D^H_q)$ and $\T{Dom}(1 \hot D^V) \su L^2(SU_q(2)) \hot L^2(SU_q(2))^{\op 2}$. Moreover, it holds that
\begin{equation}\label{eq:equivar}
\begin{split}
& [ 1 \hot D^H_q , W \op W ](\xi) = 0 \q \mbox{for all } \xi \in \T{Dom}(1 \hot D^H_q) \q \mbox{and} \\ 
& [ 1 \hot D^V_t , W \op W ]( \ze) = 0 \q \mbox{for all }  \ze \in \T{Dom}(1 \hot D^V_t) .
\end{split}
\end{equation}
\end{lemma}
\begin{proof}
We first remark that the direct sum $W \op W$ preserves the common core $\C O(SU_q(2)) \ot \C O(SU_q(2))^{\op 2}$ for the two selfadjoint unbounded operators $1 \hot D^H_q$ and $1 \hot D^V_t$. Using standard results on commutators with selfadjoint unbounded operators, it therefore suffices to verify  the identities in \eqref{eq:equivar} for elements of the form $\xi = \ze = x \ot y$ with $x \in \C O(SU_q(2))$ and $y = \pma{y_1 \\ y_2} \in \C O(SU_q(2))^{\op 2}$. 
Using the coassociativity of $\Delta$, one sees that $\De \pa_\eta(w) = (1 \ot \pa_\eta)\De(w)$ for all $\eta \in \C U_q(\G{su}(2))$ and $w \in \C O(SU_q(2))$. It therefore follows that
\[
\begin{split}
(1 \ot \C D^H_q)(W \oplus W)(\xi) 
& = - \pma{ 0 & 1 \ot q^{-1/2} \pa_{fk^{-1}} \\ 1 \ot q^{1/2} \pa_{ek^{-1}} & 0}
\pma{ \De(y_1) \cd (x \ot 1) \\ \De(y_2) \cd (x \ot 1) } \\
& = - \pma{ q^{-1/2}\De(\pa_{fk^{-1}}(y_2)) \cd (x \ot 1)  \\ q^{1/2}\De(\pa_{ek^{-1}}(y_1)) \cd (x \ot 1) } \\
& = (W \op W)(1 \ot \C D^H_q)(\xi) .
\end{split}
\]
This proves the relevant identity in the case of the horizontal Dirac operator. To treat the commutator with the vertical Dirac operator, we simply record that $W$ preserves the subspace $\C O(SU_q(2)) \ot \C A^n_q$ for all values of $n \in \zz$. The commutation relation now follows since $\C D_t^V$ acts as a diagonal  scalar matrix on $\C A^n_q \op \C A^m_q$ for all $n,m \in \zz$.
\end{proof}

\begin{remark}\label{rem:classical-equivariance}
In the situation where $t=q=1$,  Lemma \ref{l:multiunidir} together with the formulae $\lambda_g \xi = (\T{ev}_{g^{-1}} \ot 1 )\Delta(\xi) = (\T{ev}_{g^{-1}} \ot 1 )(W \op W)(1 \ot \xi)$ for the left translation operator $\lambda_g\colon \C O (SU(2))\to \C O (SU(2))$ implies that $\lambda_g \circ \C D_{1,1}=\C D_{1,1}\circ \lambda_g$ as operators on $\C O(SU(2))^{\oplus 2}$. Lemma \ref{l:multiunidir} thus recovers the $SU(2)$-equivariance of the classical Dirac operator in this case (cf.~ Section \ref{sec:comparison-with-classical}).
\end{remark}

 Next, we would like to introduce the analogue of the Lipschitz algebra $\T{Lip}_t(SU_q(2))$ for the minimal tensor product $C(SU_q(2)) \ot_{\T{min}} C(SU_q(2))$, but in this case associated with the unbounded selfadjoint operators $1 \hot D^H_q$ and $1 \hot D^V_t$ instead of the unbounded selfadjoint operators $D^H_q$ and $D^V_t$. For this to make sense, we let the minimal tensor product of $C^*$-algebras $C(SU_q(2)) \ot_{\T{min}} C(SU_q(2))$ act on $L^2(SU_q(2)) \hot L^2(SU_q(2))^{\op 2}$ via the representation $\rho \ot \pi$, which we will from now on often suppress. Note that in this representation, the coproduct is implemented by $W\oplus W$ in the sense that
 \begin{align}\label{eq:implement-in-repr}
(\rho \ot \pi)(\Delta(x))=(W\oplus W)(1 \ot \pi(x))(W\oplus W)^* \q \T{for all } x \in C(SU_q(2)) .
 \end{align}
 
We start out by expanding our notion of analytic elements. To this end, we record that the left circle action $\si_L \colon S^1 \ti C(SU_q(2)) \to C(SU_q(2))$ induces a left circle action $1 \ot \si_L$ on the minimal tensor product $C(SU_q(2)) \ot_{\T{min}} C(SU_q(2))$ given on simple tensors by
\[
(1 \ot \si_L)(z, x \ot y) := x \ot \si_L(z,y) \q \T{for all } z \in S^1 \, , \, \, x,y \in C(SU_q(2)) .
\]
We recall that the closed strip $I_s \su \cc$ was introduced in \eqref{eq:closstrip} for all values of $s \in (0,1]$.

\begin{dfn}\label{d:minanalyt}
Let $s \in (0,1]$. We say that an element $x \in C(SU_q(2)) \ot_{\T{min}} C(SU_q(2))$ is \emph{analytic of order $-\log(s)/2$} when the continuous map $\rr \to C(SU_q(2)) \ot_{\T{min}}  C(SU_q(2))$ given by $r \mapsto (1 \ot \si_L)(e^{ir},x)$ extends to a continuous function $I_s \to C(SU_q(2)) \ot_{\T{min}}  C(SU_q(2))$ which is analytic on the interior $I_s^{\ci} \su I_s$. We denote this (unique) continuous extension by $z \mapsto (1 \ot \si_L)(e^{iz},x)$.
\end{dfn}

We record that the $*$-algebra structure on the minimal tensor product $C(SU_q(2)) \ot_{\T{min}} C(SU_q(2))$ induces a $*$-algebra structure on the subset of elements which are analytic of order $-\log(s)/2$. 
%

\begin{lemma}\label{l:copanalyt}
Let $s \in (0,1]$. If an element $x \in C(SU_q(2))$ is analytic of order $-\log(s)/2$ in the sense of Definition \ref{d:analyt}, then $\De(x)$ is analytic of order $-\log(s)/2$ in the sense of Definition \ref{d:minanalyt}. Moreover, we have the formula
\[
(1 \ot \si_L)(e^{iz}, \De(x)) = \De( \si_L(e^{iz},x)) \q \mbox{for all } \, \, \, z \in I_s .
\]
\end{lemma}
\begin{proof}
We first notice that $(1 \ot \si_L)(e^{ir}, \De(x)) = \De( \si_L(e^{ir},x))$ for all $x \in C(SU_q(2))$ and $r \in \rr$. To verify this identity, it suffices to use the formula in \eqref{eq:sigma-L-and-R-on-matrix-units} for the matrix coefficients and then extend  to all of $C(SU_q(2))$ by continuity and linearity. 
 Suppose next that $x \in C(SU_q(2))$ is analytic of order $-\log(s)/2$. The map $r\mapsto \De( \si_L(e^{ir},x))$ then extends continuously to $I_s$ and the extension is analytic on $I_s^{\circ}$. It follows that $\De(x)$ is analytic of order $-\log(s)/2$ as well. The desired formula for all $z\in I_s$ is then a consequence of the Identity Theorem in complex analysis.
\end{proof}

We now have the data needed in order to formally introduce the Lipschitz algebras associated to the unbounded selfadjoint operators $1 \hot D^H_q$ and $1 \hot D^V_t$. Let $x \in  C(SU_q(2)) \ot_{\T{min}} C(SU_q(2))$. We say that $x$ is \emph{horizontally Lipschitz} when
\begin{enumerate}
\item $x$ is analytic of order $-\log(q)/2$ (with respect to $1 \ot \si_L$);
\item the bounded operator $(1 \ot \si_L)(q^{1/2},x)$ preserves the domain of $1 \hot D^H_q$;
\item the twisted commutator
\[
(1 \hot D^H_q) \cd (1 \ot \si_L)(q^{1/2},x) - (1 \ot \si_L)(q^{-1/2},x) \cd (1 \hot D^H_q)
\]
extends to a bounded operator, denoted $(1 \ot \pa_q^H)(x)$, on the Hilbert space $L^2(SU_q(2)) \hot L^2(SU_q(2))^{\op 2}$.
\end{enumerate}
We say that $x$ is \emph{vertically Lipschitz} when
\begin{enumerate}
\item $x$ is analytic of order $-\log(t)/2$ (with respect to $1 \ot \si_L$);
\item the bounded operator $(1 \ot \si_L)(t^{1/2},x)$ preserves the domain of $1 \hot D^V_t$;
\item the twisted commutator
\[
(1 \hot D^V_t) \cd (1 \ot \si_L)(t^{1/2},x) - (1 \ot \si_L)(t^{-1/2},x) \cd (1 \hot D^V_t)
\]
extends to a bounded operator, denoted $(1 \ot \pa^V_t)(x)$, on the Hilbert space $L^2(SU_q(2)) \hot L^2(SU_q(2))^{\op 2}$.
\end{enumerate}
The \emph{Lipschitz algebra} $\Lip_t\big(SU_q(2) \ti SU_q(2) \big)$ then consists of the elements in $C(SU_q(2)) \ot_{\T{min}} C(SU_q(2))$ which are both horizontally and vertically Lipschitz. We record that the Lipschitz algebra $\Lip_t\big(SU_q(2) \ti SU_q(2) \big)$ is a norm-dense $*$-subalgebra of the minimal tensor product $C(SU_q(2)) \ot_{\T{min}} C(SU_q(2))$.  

\begin{lemma}\label{l:partialcommutes}
For $x \in \Lip_t(SU_q(2))$ it holds that $\De(x) \in \Lip_t\big(SU_q(2) \ti SU_q(2)\big)$ and we have the formulae
$(1 \ot \pa_q^H)\De(x) = \De(\pa_q^H(x))$ and $(1 \ot \pa^V_t)\De(x) = \De(\pa^V_t(x))$.  
\end{lemma}
\begin{proof}
Let $x \in \T{Lip}_t(SU_q(2))$ be given. We focus on showing that $\De(x)$ is horizontally Lipschitz and that $(1 \ot \pa_q^H)\De(x) = \De(\pa_q^H(x))$, since the same argument applies to the vertical case as well. 
First note that $\De(x)$ is analytic of order $-\log(q)/2$ by Lemma \ref{l:copanalyt}.
Let $\xi \in \C O(SU_q(2)) \ot \C O(SU_q(2))^{\op 2}$ be an element in the core for $1 \hot D^H_q$, and recall that $(W^* \op W^*)(\xi) \in \C O(SU_q(2)) \ot \C O(SU_q(2))^{\op 2}$. Using Lemma \ref{l:multiunidir}, Lemma \ref{l:copanalyt} and \eqref{eq:implement-in-repr} we then see that
\[
\begin{split}
(1 \ot \si_L)(q^{1/2},\De(x))(\xi) 
& = \De(\si_L(q^{1/2},x))(\xi) \\ 
& = (W \op W) (1 \ot \si_L(q^{1/2},x)) (W^* \op W^*)(\xi) \in \T{Dom}(1 \hot D^H_q) .
\end{split}
\]
Using this, another application of Lemma \ref{l:multiunidir}, Lemma \ref{l:copanalyt} and \eqref{eq:implement-in-repr} shows that the twisted commutator may be computed on $\xi$ as follows:
\[
\begin{split}
& (1 \hot D^H_q) (1 \ot \si_L)(q^{1/2},\De(x))(\xi) - (1 \ot \si_L)(q^{-1/2}, \De(x)) (1 \hot D^H_q)(\xi) \\
& \q = (W \op W) \big( 1 \ot D^H_q \si_L(q^{1/2},x) - 1 \ot \si_L(q^{-1/2},x) D^H_q  \big) (W^* \op W^*)(\xi) \\
& \q = (W \op W)(1 \ot \pa_q^H(x) ) (W^* \op W^*)(\xi) = \De(\pa_q^H(x))(\xi) .
\end{split}
\]
The result of the lemma now follows since $\De(\pa_q^H(x))$ is a bounded operator and since $\C O(SU_q(2)) \ot \C O(SU_q(2))^{\op 2}$ is a core for the selfadjoint unbounded operator $1 \hot D_q^H$.
\end{proof}

For $\xi,\ze \in L^2(SU_q(2))$ we let $\phi_{\xi,\ze} \colon C(SU_q(2)) \to \cc$ denote the bounded linear functional $\phi_{\xi,\ze}(x) := \inn{\xi, \rho(x) \ze}$. Let us moreover introduce the two bounded operators $T_\xi$ and $T_\ze \colon L^2(SU_q(2))^{\op 2} \to L^2(SU_q(2)) \hot L^2(SU_q(2))^{\op 2}$ given by the formulae $T_\xi(\eta) := \xi \ot \eta$ and $T_\ze(\eta) := \ze \ot \eta$. We define the bounded operator
\[
\begin{split}
 \phi_{\xi,\ze} \ot 1 \colon \B B\big( L^2(SU_q(2)) \hot L^2(SU_q(2))^{\op 2}\big) &\longrightarrow \B B\big( L^2(SU_q(2))^{\op 2} \big) 
\end{split}
\]
given by $(\phi_{\xi,\ze} \ot 1)(z) := T_\xi^* z T_\ze$, 
and record that we have the estimate $\| \phi_{\xi,\ze} \ot 1  \| \leq \| \xi \| \cd \| \ze \|$ on the operator norm. 

The last result of the present section shows how the Lipschitz seminorm and the coproduct interact with the slice maps just introduced. This result will be essential in our analysis of the Berezin transform; see Proposition \ref{prop:berezin-approximates-identity}. 

\begin{prop}\label{prop:slice-and-Lip}
For each $\xi,\ze \in L^2(SU_q(2))$ and $z \in \Lip_t\big(SU_q(2) \ti SU_q(2)\big)$ it holds that
$(\phi_{\xi,\ze} \ot 1)(z) \in  \Lip_t(SU_q(2))$ and we have the identities 
\[
\begin{split}
& \pa_q^H\big( (\phi_{\xi,\ze} \ot 1)(z) \big) = (\phi_{\xi,\ze} \ot 1)(1 \ot \pa_q^H)(z) \q \mbox{and} \\ 
& \pa^V_t\big( (\phi_{\xi,\ze} \ot 1)(z) \big) = (\phi_{\xi,\ze} \ot 1)(1 \ot \pa^V_t)(z) .
\end{split}
\]
In particular, we have the estimate
\[
L_{t,q}^{\max}\big( (\phi_{\xi,\ze} \ot 1)(\De(x)) \big) \leq \| \xi \| \| \ze\| \cd L_{t,q}^{\max}(x)
\]
for all $x \in \Lip_t(SU_q(2))$.
\end{prop}
\begin{proof}
Let $\xi,\ze \in L^2(SU_q(2))$ and $z \in \Lip_t\big(SU_q(2) \ti SU_q(2)\big)$ be given. We focus on showing that $(\phi_{\xi,\ze} \ot 1)(z)$ is vertically Lipschitz and that $\pa^V_t\big( (\phi_{\xi,\ze} \ot 1)(z) \big) = (\phi_{\xi,\ze} \ot 1)(1 \ot \pa^V_t)(z)$. The analogous claim regarding the horizontal Dirac operator follows by a similar argument. 
Notice first that we have the inclusion $T_\ze D^V_t \su (1 \hot D^V_t) T_\ze$ of unbounded operators on $L^2(SU_q(2))^{\oplus 2}$. Since the same inclusion holds with $T_\xi$ instead of $T_\ze$ we also obtain the inclusion 
$T_\xi^* (1 \hot D^V_t) \su D^V_t T_\xi^*$ by applying the adjoint operation. Secondly, since $(\phi_{\xi,\ze} \ot 1)(y_1\ot y_2)= \phi_{\xi,\ze}(y_1)y_2$ for $y_1,y_2 \in C(SU_q(2))$ it follows that 
\[
\sigma_L(e^{ir}, (\phi_{\xi,\ze} \ot 1)(y_1\ot y_2))=(\phi_{\xi,\ze}\ot 1)(1 \ot \sigma_L)(e^{ir}, y_1 \ot y_2) \q \T{for all }  r\in \rr ,
\]
and hence the same formula holds globally on $C(SU_q(2))\ot_{\min} C(SU_q(2))$ by linearity and density.  We thereby obtain that $x := (\phi_{\xi,\ze} \ot 1)(z)\in C(SU_q(2))$ is analytic of order $-\log(t)/2$ (in the sense of Definition \ref{d:analyt}) and that we have the identity
\[
\si_L(e^{iw},x) = (\phi_{\xi,\ze} \ot 1)(1 \ot \si_L)(e^{iw}, z) \q \T{for all } w \in I_t .
\] 
It follows from the above observations that the bounded operator
\[
\si_L(t^{1/2},x) = (\phi_{\xi,\ze} \ot 1)(1 \ot \si_L)(t^{1/2},z) = T_\xi^* (1 \ot \si_L)(t^{1/2},z) T_\ze
\]
preserves the domain of $D^V_t$. Moreover, we may compute as follows for any vector $\eta \in \T{Dom}(D^V_t)$:
\[
\begin{split}
D^V_t \cd \si_L(t^{1/2},x) (\eta) 
& = D^V_t \cd T_\xi^* (1 \ot \si_L)(t^{1/2},z) T_\ze(\eta) 
= T_\xi^* (1 \hot D^V_t ) \cd (1 \ot \si_L)(t^{1/2},z) T_\ze(\eta) \\
& = T_\xi^* (1 \ot \pa^V_t)(z) T_\ze(\eta) + T_\xi^* (1 \ot \si_L)(t^{-1/2},z) \cd (1 \hot D^V_t) T_\ze(\eta) \\
& = (\phi_{\xi,\ze} \ot 1)(1 \ot \pa^V_t)(z) (\eta) + \si_L(t^{-1/2},x) \cd D^V_t (\eta) .
\end{split}
\]
This ends the proof of the first part of the lemma.  The second part of the lemma (regarding the estimate relating to the seminorm $L_{t,q}^{\max}$) now follows immediately by an application of Lemma \ref{l:partialcommutes}.
\end{proof}

\subsection{Conjugating the Dirac element with the fundamental unitary}
The main technical tool for proving quantum Gromov-Hausdorff continuity for the Podle{\'s} spheres $S_q^2$ \cite[Theorem A]{AKK:Podcon} is a trivialisation of the ``spinor bundle'' $\C A^1_q \op \C A^{-1}_q$ implemented by the fundamental corepresentation unitary $u:=u^1\in \mathbb{M}_2\big(\C O (SU_q(2))\big)$. Note that this trivialisation is not compatible with the $\zz/2\zz$-grading on the spinor bundle. As in Section \ref{subsec:comparison-with-KS-and-KRS}, we let $\pa^0\colon \C O (S_q^2) \to \mathbb{B}(H_q^1 \oplus H_q^{-1})$ denote the derivation arising by taking the commutator with the D\c{a}browski-Sitarz Dirac operator; see \cite{DaSi:DSP}. In the paper \cite{AKK:Podcon} we analysed the linear map $\de^0:=u\pa^0 u^*$, a key feature of which is that it gives rise to the same seminorm as $\pa^0$ after composition with the operator norm. Moreover, we saw in \cite[Proposition 3.12]{AKK:Podcon} that $\de^0$ can be described by means of the right action of the quantum enveloping algebra $\C U_q(\G{su}(2))$ on the coordinate algebra $\C O(SU_q(2))$.

To obtain quantum Gromov-Hausdorff continuity also at the level of quantum $SU(2)$, it is therefore relevant to analyse  the analogue of $\de^0$ in this context as well, and we carry out the relevant details in this section. For the analysis below to work of out we need the vertical and horizontal derivations to obey the same twisted Leibniz rule. We thus focus exclusively on the special case where the two parameters $t$ and $q \in (0,1]$ agree, and consider the twisted $*$-derivation (see Definition \ref{def:twisted-derivation})
\[
\pa := \pa_{q,q} = \pma{ \pa^3_q & - \pa^2 \\ - \pa^1 & - \pa^3_q} \colon
\C O(SU_q(2)) \to \B M_2\big( \C O(SU_q(2))\big),
\]
where the twists are given by $\pa_k$ and $\pa_{k^{-1}}$ so that $\pa(xy) = \pa(x)\pa_k(y) + \pa_{k^{-1}}(x)\pa(y)$ for all $x,y\in \C O(SU_q(2))$.

Note that the twisted $*$-derivation $\pa^3_q$ can be described by the formula
\[
\pa^3:= \pa^3_q = \fork{ccc}{ \frac{\pa_k - \pa_{k^{-1}}}{q - q^{-1}} & \T{for} & q \neq 1 \\ 
\frac{1}{2} \pa_h & \T{for} & q = 1} .
\]
so that $\pa$ is defined entirely in terms of the \emph{left} action of $\C U_q(\G{su}(2))$ on the coordinate algebra $\C O(SU_q(2))$. The main point is to show that when $\pa$ is conjugated with the fundamental corepresentation unitary we obtain a twisted $*$-derivation which can be expressed in terms of the \emph{right} action of the quantum enveloping algebra.
Recall that for $\eta\in \C U_q (\mathfrak{su}(2))$, the right action of $\eta$ is defined by the linear endomorphism $\de_\eta\colon \C O(SU_q(2))\to \C O (SU_q(2))$ given by the formula $\de_\eta := ( \inn{\eta, \cd} \ot 1) \De$, and in this way we obtain three twisted derivations $\de^1, \de^2, \de^3 \colon \C O(SU_q(2)) \to \C O(SU_q(2))$ by setting
\[
\de^1 := q^{1/2} \de_e  \quad \de^2 := q^{-1/2} \de_f \quad \T{and} \q
\de^3 := \fork{ccc}{ \frac{\de_k - \de_{k^{-1}}}{q - q^{-1}} & \T{for} & q \neq 1 \\ 
\frac{1}{2} \de_h & \T{for} & q = 1} .
\]
which are all twisted by the automorphisms $\de_k$ and $\de_{k^{-1}}$ so that  $\de^i(xy)=\de^i(x)\de_k(y) + \de_{k^{-1}}(x)\de^i(y)$ for all $x,y\in \C O(SU_q(2))$ and $i\in \{1,2,3\}$. We now assemble this data into a single twisted $*$-derivation 
\[
\de := \pma{ \de^3 & -\de^2 \\ -\de^1 & -\de^3} \colon \C O(SU_q(2)) \longrightarrow \mathbb{M}_2\big( \C O(SU_q(2)) \big), 
\]
where the twists are again given by $\de_k$ and $\de_{k^{-1}}$. Recalling that $u = u^1$ denotes the fundamental corepresentation unitary, the main result of this section is the identity
\begin{align}\label{eq:fundamental-identity}
u \pa(x) u^* = \de(x) \q \T{for all } x \in \C O(SU_q(2)) ,
\end{align}
which will play crucial role in our further analysis. The strategy for proving \eqref{eq:fundamental-identity} will be to first show that  $u\pa(-)u^*$ satisfies the same twisted Leibniz rule as $\de$,
thus reducing the proof to verifying \eqref{eq:fundamental-identity} on the generators of $\C O(SU_q(2))$. To this end, we will need the algebra automorphism $\nu^{-1/2} := \de_k \circ  \pa_k   \colon \C O(SU_q(2)) \to \C O(SU_q(2))$. Notice that it follows from the defining commutation relations in $\C O(SU_q(2))$ that
\begin{equation}\label{eq:taucommu}
b x = \nu^{-1/2}(x) b \quad \T{and} \quad b^* x = \nu^{-1/2}(x) b^* \quad \T{for all } x \in \C O(SU_q(2)) .
\end{equation}
 Before we proceed we introduce some relevant notation: if $\si, \te \colon \C O(SU_q(2)) \to \C O(SU_q(2))$ are algebra automorphisms, we shall write
\[
\comm{\si}{[y,x]}{\te} := y \te(x) - \si(x) y 
\]
for the twisted commutator between two elements $x, y \in \C O(SU_q(2))$. 

\begin{lemma}\label{l:deriII}
We have the identities
\begin{align*}
 \comm{\de_k}{[a^*,x]}{\pa_k} &= (1 - q^2) b \pa^1(x)  &   \mbox{and} \qquad \quad
\comm{\de_k}{[a, x]}{\pa_k} &= (1 - q^2) q^{-1} b^* \pa^2(x) \\
 \comm{\de_k}{[b^*,x]}{\pa_k} &= b^* (\pa_k - \pa_{k^{-1}})(x)  &  \mbox{and} \qquad \quad
\comm{\de_k}{[b,x]}{\pa_k} &= b (\pa_k - \pa_{k^{-1}})(x)
\end{align*}
for all $x \in \C O(SU_q(2))$.
\end{lemma}
\begin{proof}
A direct computation  reveals that the operation $x \mapsto \comm{\de_k}{[a^*,x]}{\pa_k}$ satisfies the following twisted Leibniz rule:
\[
\comm{\de_k}{[a^*, x y]}{\pa_k} = \comm{\de_k}{[a^*,x]}{\pa_k} \cd \pa_k(y) + \de_k(x)\cd \comm{\de_k}{[a^*, y]}{\pa_k} \q \T{for all } x, y \in \C O(SU_q(2)) .
\]
It moreover follows from \eqref{eq:taucommu} that the operation $x \mapsto b \pa^1(x)$ satisfies the same twisted Leibniz rule so that
\[
b \pa^1(xy) = b \pa^1(x) \pa_k(y) + \de_k(x) b \pa^1(y) \q \T{for all } x,y \in \C O(SU_q(2)) .
\]
In order to prove the first identity of the lemma, it thus suffices to check that 
\[
a^* \pa_k(x) - \de_k(x) a^* = (1 - q^2) b \pa^1(x)
\]
for $x \in \{a,a^*,b,b^*\}$.  In these four cases, one may verify the relevant identity by a straightforward computation.
The second identity of the lemma can be proved by a similar argument. 
The two last identities (those involving twisted commutators with $b$ and $b^*$) follow immediately from \eqref{eq:taucommu}.
\end{proof}

\begin{lemma}\label{l:deriIII}
We have the identities
\[
\comm{\de_k}{[u, x]}{\pa_k} = { (q^2 - 1)}  \ma{cc}{0 & b \\ q^{-1} b^* & 0} \pa(x) \quad \mbox{and} \quad
\comm{\pa_{k^{-1}}}{[u^*, x]}{\de_{k^{-1}}} = { (q^2 - 1)}  \pa(x) \ma{cc}{0 & q^{-1} b \\ b^* & 0} 
\]
for all $x \in \C O(SU_q(2))$.
\end{lemma}
\begin{proof}
The relevant identities are trivially satisfied for $q = 1$ so we focus on the case where $q \neq 1$ and let $x \in \C O(SU_q(2))$ be given. Applying the definition of the fundamental corepresentation unitary $u$ together with Lemma \ref{l:deriII} we obtain that
\[
\begin{split}
\comm{\de_k}{[u, x]}{\pa_k} & = \pma{ \comm{\de_k}{[a^*,x]}{\pa_k} & -q \cd \comm{\de_k}{[b,x]}{\pa_k} \\ \,\comm{\de_k}{[b^*,x]}{\pa_k} & \comm{\de_k}{[a,x]}{\pa_k}} \\
& =  \pma{ (1 - q^2)b \pa^1(x) & -qb (\pa_k - \pa_{k^{-1}})(x) \\ b^* (\pa_k - \pa_{k^{-1}})(x) & (1 - q^2) q^{-1} b^* \pa^2(x)} \\
& = (q^2 - 1) \pma{0 & b \\ q^{-1}b^* & 0} \pa(x) .
\end{split}
\]
This proves the first identity of the lemma. The remaining identity then follows from the first via the following computation:
\[
\begin{split}
(q^2-1) \pa(x) \pma{0 & q^{-1} b \\ b^* & 0} & = 
(1 - q^2) \left( \pma{0 & b \\ q^{-1} b^* & 0} \pa(x^*) \right)^* 
= ( \de_k(x^*) u - u \pa_k(x^*) )^* \\ 
& = u^* \de_{k^{-1}}(x) - \pa_{k^{-1}}(x) u^* = \comm{\pa_{k^{-1}}}{[u^* ,x]}{\de_{k^{-1}}} . \qedhere
\end{split}
\]
\end{proof}

We are now ready to show that the operation $x \mapsto u \pa(x) u^*$ is a  twisted derivation.

\begin{prop}\label{p:deri}
It holds that
\[
u \pa(xy) u^* = u \pa(x) u^* \de_k(y) + \de_{k^{-1}}(x) u \pa(y) u^*
\]
for all $x,y \in \C O(SU_q(2))$.
\end{prop}
\begin{proof}
Let $x,y \in \C O(SU_q(2))$ be given. We compute that
\[
\begin{split}
u \pa(xy) u^* & = u \pa(x) \pa_k(y) u^* + u \pa_{k^{-1}}(x) \pa(y) u^* \\  
& = u \pa(x) u^* u \pa_k(y) u^* + u \pa_{k^{-1}}(x) u^* u \pa(y) u^* \\ 
& = u \pa(x) u^* \de_k(y) + \de_{k^{-1}}(x) u \pa(y) u^*  \\ 
& \q + u \pa(x) u^* \cd \comm{\de_k}{[u,y]}{\pa_k} u^* - u \cd \comm{\pa_{k^{-1}}}{[u^*,x]}{\de_{k^{-1}}} u \pa(y) u^* .
\end{split}
\]
Notice now that
\[
\pma{0 & q^{-1}b \\ b^* & 0} u = \pma{q^{-1} bb^* & ab \\ q^{-1} a^* b^* & - qb^* b} 
= u^* \pma{0 & b \\ q^{-1} b^* & 0} .
\]
Thus, applying Lemma \ref{l:deriIII} we obtain that
\[
\begin{split}
u \pa(x) u^* \cd \comm{\de_k}{[u,y]}{\pa_k} u^* 
& = { (q^2-1)} u \pa(x) u^* \pma{0 & b \\ q^{-1} b^* & 0} \pa(y) u^* \\
& = {(q^2-1)} u \pa(x) \pma{0 & q^{-1}b \\ b^* & 0} u \pa(y) u^* \\
& = u \cd \comm{\pa_{k^{-1}}}{[u^*,x]}{\de_{k^{-1}}} \cd u \pa(y) u^* .
\end{split}
\]
This proves the proposition.
\end{proof}

We are now ready to verify that $u$ conjugates $\pa$ into $\de$.

\begin{prop}\label{p:derV}
It holds that $u \pa(x) u^* = \de(x)$ for all $x \in \C O(SU_q(2))$.
\end{prop}
\begin{proof}
Using Proposition \ref{p:deri} we see that the operations $x \mapsto u \pa(x) u^*$ and $x \mapsto \de(x)$ satisfy the same twisted Leibniz rule. Since they also behave in the same way with respect to the adjoint operation, it therefore suffices to verify the required identity on the generators $a,b \in \C O(SU_q(2))$. To treat the case $q=1$ and $q<1$ on the same footing,  we define
\[
\mu := [1/2]_q = \frac{1}{q^{1/2} + q^{-1/2}} .
\]
so that  $\pa^3(a)=\mu a$ and  $\pa^3(b)=\mu b$. The two relations may now be proven by a straightforward computation, indeed:
\begin{align*}
u \pa(a) u^* & = \pma{a^* & - qb \\ b^* & a}  \pma{\mu \cd a & 0 \\ -q^{1/2} b^* & -\mu \cd a} \pma{a & b \\ -qb^* & a^*} \\
& =  \pma{ \mu \cd a^* a^2 + q^{3/2} bb^* a - q^2 \mu \cd ba b^* 
& \mu \cd a^* a b + q^{3/2} bb^*b + q \mu \cd baa^* \\
\mu \cd b^* a^2 - q^{1/2} ab^* a + q \mu \cd a^2 b^*
& \mu \cd b^* ab - q^{1/2} ab^* b - \mu \cd a^2 a^* } \\
&=  \pma{\mu \cd a &  q^{1/2} b \\ 0 & -\mu \cd a}  = \de(a), \quad \T{and} \\
u \pa(b) u^* & = \pma{a^* & - qb \\ b^* & a} \pma{ \mu \cd b & 0 \\ q^{-1/2}a^*  & -\mu \cd b } \pma{a & b \\ -qb^* & a^*} \\
& = \pma{\mu \cd a^* b a - q^{1/2} b a^* a - q^2 \mu \cd b^2 b^* & 
\mu \cd a^* b^2 - q^{1/2} ba^* b + q \mu \cd b^2 a^* \\ 
\mu \cd b^*b  a + q^{-1/2} aa^* a + q \mu \cd a bb^* 
& \mu \cd b^* b^2 + q^{-1/2} aa^* b - \mu \cd aba^* } \\
& =  \pma{-\mu \cd b  & 0 \\ q^{-1/2} a &  \mu \cd b }  = \de(b) . \qedhere
\end{align*}
\end{proof}

We now have the tools needed to properly investigate the quantum metric space structure on $SU_q(2)$, and we proceed to do so in the following section.

\section{Quantum metrics on quantum $SU(2)$}\label{sec:quantum-metrics-on-quantum-su2}
We now return to the general setting, and consider again two parameters $t,q \in (0,1]$ which will be fixed throughout this section. The aim of this section is to show that $\big(C(SU_q(2)), L_{t,q}^{\max}\big)$ is a compact quantum metric space. The proof consists of several steps and we therefore first explain the general strategy.
For each $M \in \nn_0$, we recall from Section \ref{ss:circle} that the algebraic spectral $M$-band is defined as the subspace $\C B_q^M := \sum_{m = - M}^M \C A_q^m \su \C O(SU_q(2))$ and that the spectral $M$-band $B_q^M$ agrees with the norm closure of $\C B_q^M$ with respect to the $C^*$-norm on $C(SU_q(2))$. The Lipschitz seminorm
\[
L_{t,q}^{\T{max}} \colon C(SU_q(2)) \to [0,\infty]
\]
restricts to a Lipschitz seminorm $L_{t,q}^{\T{max}} \colon B_q^M \to [0,\infty]$ with domain $B_q^M\cap \T{Lip}_t(SU_q(2))$. 
 We start by proving that the pair $(B_q^M,L_{t,q}^{\T{max}})$ is a compact quantum metric space for all $M \in \nn_0$. 
Knowing this, the next step is to construct a  Lip-norm contraction $ C(SU_q(2))\to B_q^M$ for each $M\in \nn_0$. This sets the stage for an application of  Corollary \ref{cor:approx-cor}, from which we will finally deduce that  $\big(C(SU_q(2)), L_{t,q}^{\max}\big)$ is a compact quantum metric space; see Theorem \ref{thm:quantum-su2-as-cqms} below for details.
In the following section we first treat the case where $M = 0$, which plays a special role, since this provides the connection with the Podle\'s sphere, whose quantum metric structure was investigated in \cite{AgKa:PSM}; see also \cite{AKK:Podcon, AKK:Polyapprox}. 

\subsection{The Podle\'s sphere revisited}\label{ss:podles-revisited}
Notice first of all that the spectral $0$-band $B_q^0$ agrees with the Podle\'s sphere $C(S_q^2)$. For each $m \in \zz$, we recall from Section \ref{ss:circle} that $H^m_q \su L^2(SU_q(2))$ denotes the Hilbert space completion of the algebraic spectral subspace $\C A^m_q$ with respect to the inner product coming from the Haar state $h \colon C(SU_q(2)) \to \cc$. The GNS Hilbert space $L^2(SU_q(2))$ is then isomorphic to the Hilbert space direct sum
\[
L^2(SU_q(2)) \cong \bigoplus_{m = -\infty}^\infty H^m_q.
\]
The horizontal Dirac operator $\C D^H_q \colon \C O(SU_q(2))^{\op 2} \to L^2(SU_q(2))^{\op 2}$ restricts to the unbounded operator
\[
\C D^0_q = \pma{0 & -\pa_f \\ -\pa_e & 0} \colon \C A^1_q \op \C A^{-1}_q \longrightarrow H^1_q \op H^{-1}_q 
\]
and we denote the closure by $D^0_q \colon \T{Dom}(D_q^0) \to H^{1}_q \op H^{-1}_q$. The vertical Dirac operator $\C D^V_t \colon \C O(SU_q(2))^{\op 2} \to L^2(SU_q(2))^{\op 2}$ restricts to the trivial operator zero on the direct sum $\C A^1_q \op \C A^{-1}_q$, and the diagonal representation $\pi \colon C(SU_q(2)) \to \B B(L^2(SU_q(2))^{\op 2})$ restricts to a representation
\[
\pi^0 \colon C(S_q^2) \longrightarrow \B B( H^1_q \op H^{-1}_q) .
\]
We equip the Hilbert space $H^{1}_q \op H^{-1}_q$ with the $\zz/2\zz$-grading operator $\gamma=\sma{1& 0 \\0&-1}$ and record that the triple $\big( C(S_q^2), H^{1}_q \op H^{-1}_q, D^0_q\big)$ agrees with the D\c{a}browski-Sitarz spectral triple (up to conjugation with the grading operator $\ga$); see \cite{DaSi:DSP,NeTu:LFQ}. We remark that we are now within the standard realm of non-commutative geometry, in so far that $\big( C(S_q^2), H^{1}_q \op H^{-1}_q, D^0_q\big)$ is a genuine (even) spectral triple on $C(S_q^2)$. This is in contrast to the situation for $C(SU_q(2))$, where we are just relying on the spectral data given by the horizontal and vertical Dirac operators. The D\c{a}browski-Sitarz spectral triple therefore has its own Lipschitz algebra
\[
\T{Lip}(S_q^2):=\{x\in C(S_q^2)\mid \pi^0(x)(\T{Dom}(D^0_q))\subseteq \T{Dom}(D^0_q) \T{ and }   \ov{[D^0_q, \pi^0(x)]} \text{ is bounded} \}.
\]
For each $x \in \T{Lip}(S_q^2)$ we apply the notation $\pa^0(x) := \ov{ [D^0_q,\pi^0(x)]}$. The main result in \cite{AgKa:PSM} is that the Lipschitz seminorm $L_q^{0,\max} \colon C(S_q^2)\to [0,\infty]$ defined by
\[
L_q^{0,\max}(x):= \cas{ \big\| \pa^0(x) \big\| & x \in \T{Lip}(S_q^2) \\ 
\infty & x \in C(S_q^2) \sem \T{Lip}(S_q^2) }
\]
turns $C(S_q^2)$ into a compact quantum metric space.  We shall now prove that the two settings are compatible, in the sense that the restriction of $L_{t,q}^{\max}$ to $C(S_q^2)$ agrees with $L_{q}^{0,\max}$. \\

Recall that $\nu^{1/2}$ denotes the algebra automorphism $\pa_{k^{-1}}\circ \de_{k^{-1}}\colon \C O(SU_q(2)) \to \C O(SU_q(2))$ while $\C J \colon \C O(SU_q(2)) \to \C O(SU_q(2))$ denotes the antilinear antihomomorphism  $x\mapsto (\de_k \pa_k)(x^*)$, which extends to the  antilinear unitary operator $J$ on $L^2(SU_q(2))$; see Section \ref{subsec:first-order-condition} for more details. 

\begin{lemma}\label{lem:right-mult-lem-v2}
For each $x,y \in \C O (SU_q(2))$ it holds that $J\nu^{1/2}(y)^* J\Lambda (x)= \Lambda(xy)$
\end{lemma}
\begin{proof}
This follows from a straightforward computation, using that $\de_k(x)^*=\de_{k^{-1}}(x^*)$ and $\pa_k(x^*)=\pa_{k^{-1}}(x)^*$ for all $x\in \C O(SU_q(2))$.
\end{proof}


\begin{prop}\label{p:podisomet}
The inclusion $C(S_q^2)  \subseteq C(SU_q(2))$ in an isometry with respect to the seminorms $L_q^{0,\max} \colon C(S_q^2) \to [0,\infty]$ and $L_{t,q}^{\max} \colon C(SU_q(2)) \to [0,\infty]$. In particular, it holds that $\T{Lip}(S_q^2) \su \T{Lip}_t(SU_q(2))$.
\end{prop}
\begin{proof}
By restricting the Haar state, we obtain a GNS representation $\rho^0\colon C(S_q^2)\to \mathbb{B}(H_q^0)$, and we denote by $L^\infty(S_q^2)\subseteq \mathbb{B}(H_q^0)$ the enveloping von Neumann algebra $\rho^0(C(S_q^2))''$. By standard von Neumann algebraic techniques, the inclusion $C(S_q^2) \su C(SU_q(2))$ extends to a normal inclusion $\io \colon L^\infty(S_q^2) \to L^\infty(SU_q(2))$ with the property that $\iota (x)\cdot \xi =x\cdot \xi$ for $\xi \in L^2(S_q^2)\subseteq L^2(SU_q(2))$. \\
%

Let $x \in \T{Lip}(S_q^2)$ be given. We recall from \cite[Lemma 3.7]{AKK:Polyapprox} that the operator $\de^0(x) := u \pa^0(x) u^*$ belongs to $\B M_2( L^\infty(S_q^2))$ and we may thus define the element
\[
\io(\pa^0(x)) := u^* \io(\de^0(x)) u \in \B M_2( L^\infty(SU_q(2))) .
\]
It clearly holds that $\| \io(\pa^0(x)) \| = \| \pa^0(x) \|$. Moreover, we remark that whenever $ \ze \in H_q^1 \op H_q^{-1} \su L^2(SU_q(2)) \op L^2(SU_q(2))$ it holds that $ u \cd \ze \in H_q^0\oplus H_q^0$ and hence that
\begin{equation}\label{eq:iopartial}
\io(\pa^0(x)) I y I  \ze = I y I \io(\pa^0(x))  \ze = I y I \pa^0(x)  \ze \q \T{for all } y \in L^\infty(SU_q(2)) .
\end{equation}
 Let now $\xi \in \C O(SU_q(2))^{\op 2}$ be given. We aim to show that $x \cd \xi \in \T{Dom}(D^H_q) \cap \T{Dom}(D^V_t)$ and that we have the identities
\begin{equation}\label{eq:hvcomm}
[D^H_q,x] \xi = \io(\pa^0(x)) \xi \q \T{and} \q [D^V_t,x] \xi = 0 .
\end{equation}
This suffices to prove the present theorem: indeed, since $x\in \T{Lip}(S_q^2)$ both twists involved in the  definitions of $\pa_t^V(x)$ and $\pa_q^H(x)$ are trivial. Moreover, if one proves the relations in \eqref{eq:hvcomm} for $\xi$ in the core $\C O(SU_q(2))^{\op 2}$, an approximation argument shows that $x(\T{Dom}(D^H_q))\subseteq \T{Dom}(D^H_q)$, $x(\T{Dom}(D_t^V))\subseteq \T{Dom}(D_t^V)$ and that the relations in \eqref{eq:hvcomm} hold on the two domains. \\

Let us start out by proving the claims relating to the vertical Dirac operator. Without loss of generality, we may assume that $\xi \in \C A^n_q \op \C A^m_q$ for some $n,m \in \zz$. Since $x \in C(S_q^2)$, it follows that $x \xi \in H^n_q \op H^m_q$. But $H^n_q \op H^m_q \su \T{Dom}(D^V_t)$ and the relevant commutator $[D^V_t,x] \xi$ is trivial since the restriction of $D^V_t$ to $H^n_q \op H^m_q$ is given by multiplication with the diagonal matrix
\[
\sma{ t^{\frac{-n+1}{2}} \big[\frac{n-1}{2}\big]_t & 0 \\ 0 & - t^{\frac{-m - 1}{2}} \big[\frac{m+1}{2}\big]_t}.
\]
Next we focus on the claims relating to the horizontal Dirac operator. Let first $\eta \in \T{Dom}(D^0_q)$ and $z \in \C O(SU_q(2))$ be given. We begin by showing that
\begin{equation}\label{eq:rightdom}
I z I \eta \in \T{Dom}(D^H_q) \q \T{and} \q [D^H_q, I z I] \eta = I \pa_q^H(\pa_k(z)) I \eta .
\end{equation}
Since $\C A^1_q \op \C A^{-1}_q$ is a core for $\T{Dom}(D^0_q)$ we may, without loss of generality, assume that $\eta \in \C A^1_q \op \C A^{-1}_q$. We then remark that $I z I \eta = \C I z \C I \eta \in \C O(SU_q(2)) \op \C O(SU_q(2))$ and that Lemma \ref{l:commureal} therefore implies that
\[
[D^H_q, I z I] \eta = \C I \pa_q^H(\pa_k(z)) \C I \Ga_{q,0}^2 \eta .
\]
The desired formula for the commutator $[D^H_q,IzI] \eta$ then follows by noting that $\Ga_{q,0}$ restricts to the identity operator on $\C A^1_q \op \C A^{-1}_q$. 
To proceed, we denote the two columns in $u^*$ by $v_1 \in \C A^1_q \op \C A^{-1}_q$ and $v_2 \in \C A^1_q \op \C A^{-1}_q$, thus
\[
v_1 = \begin{pmatrix} a \\ -q b^* \end{pmatrix} \q \T{and} \q v_2 = \begin{pmatrix} b \\ a^* \end{pmatrix} .
\]
For our fixed element $\xi \in \C O(SU_q(2))^{\op 2}$, we may therefore choose $y_1,y_2 \in \C O(SU_q(2))\subseteq L^2(SU_q(2))^{\oplus 2}$ such that
\[
\xi = u^* u \xi = v_1 \cd y_1 + v_2 \cd y_2 = - I \nu^{1/2}(y_1)^* I \cd v_1 - I \nu^{1/2}(y_2)^* I \cd v_2,
\]
where the last equality follows from Lemma \ref{lem:right-mult-lem-v2} (suppressing the embedding $\Lambda$ for notational convenience).
To ease the notation, put $z_1 := - \nu^{1/2}(y_1)^*$ and $z_2 := - \nu^{1/2}(y_2)^*$. We then have that
\[
 \xi = Iz_1 I \cd v_1 + I z_2 I \cd v_2 \, \, \ \T{ and } \ \, \, \, x \xi = I z_1 I \cd x v_1 + I z_2 I \cd x v_2 .
\]
 Since $v_1$ and $v_2$ belong to $\C A^1_q \op \C A^{-1}_q \su \T{Dom}(D^0_q)$ and $x \in \T{Lip}(S_q^2)$ we know that $xv_1$ and $xv_2 \in \T{Dom}(D^0_q)$. We thus obtain from  \eqref{eq:rightdom} that $x \xi \in \T{Dom}(D^H_q)$ and  moreover that
\[
\begin{split}
D^H_q x \xi 
& = I z_1 I \cd D^H_q x v_1 + I z_2 I \cd D^H_q x v_2
+ I \pa_q^H(\pa_k(z_1)) I x v_1 + I \pa_q^H(\pa_k(z_2)) I x v_2 \\
& = I z_1 I \cd \pa^0(x) v_1 + I z_2 I \cd \pa^0(x) v_2
+ I z_1 I x D^H_q v_1 + I z_2 I x D^H_q v_2 \\
& \q + x I \pa_q^H(\pa_k(z_1)) I v_1 + x I \pa_q^H(\pa_k(z_2)) I v_2 \\
& = I z_1 I \cd \pa^0(x) v_1 + I z_2 I \cd \pa^0(x) v_2
+ x D^H_q Iz_1 I v_1 + x D^H_q I z_2 I v_2 \\
& = \io(\pa^0(x)) \xi + x D^H_q \xi ,
\end{split}
\]
where the last equality follows from \eqref{eq:iopartial}. This ends the proof of the present proposition.
\end{proof}
 
\begin{cor}\label{cor:zeroth-spec-band-is-cqms}
We have the identity $(B_q^0, L_{t,q}^{\max}) = (C(S_q^2),L_q^{0,\max})$. In particular, it holds that $(B_q^0, L_{t,q}^{\max})$ is a compact quantum metric space.
\end{cor}
\begin{proof} It suffices to establish the identity $(B_q^0, L_{t,q}^{\max}) = (C(S_q^2),L_q^{0,\max})$ since we already know from \cite[Theorem 8.3]{AgKa:PSM} that $(C(S_q^2),L_q^{0,\max})$ is a compact quantum metric space.
We have $B_q^0=C(S_q^2)$ and, by Proposition \ref{p:podisomet}, $\T{Lip}(S_q^2) \subseteq B_q^0 \cap \T{Lip}_t(SU_q(2))$ with $L_q^{0,\max}(x)=L_{t,q}^{\max}(x)$ for all $x\in \T{Lip}(S_q^2)$. We therefore only need to show that $B_q^0 \cap \T{Lip}_t(SU_q(2)) \subseteq  \T{Lip}(S_q^2)$. Denote by $\io \colon H_q^1 \op H_q^{-1} \to L^2(SU_q(2)) \op L^2(SU_q(2))$ the inclusion of Hilbert spaces. It can then be verified that $\io^* \C D_q^H \su \C D_q^0 \io^*$ and from this inclusion it follows that $\io^* D_q^H \su D_q^0 \io^*$. Similarly, we have the inclusion $\io D_q^0 \su D_q^H \io$ of unbounded operators.

The above inclusions can now be applied as follows: for each $x \in B_q^0 \cap \T{Lip}_t(SU_q(2))$ and each $\xi \in \T{Dom}(D_q^0)$ we get that $\pi^0(x) \xi = \io^* \pi(x) \io \xi \in \T{Dom}(D_q^0)$. Moreover, it holds that
\[
D_q^0 \pi^0(x) \xi = D_q^0 \io^* \pi(x) \io \xi 
= \io^* \pa^H_q(x) \io \xi + \io^* \pi(x) \io D_q^0 \xi = \io^* \pa^H_q(x) \io \xi + \pi^0(x) D_q^0 \xi .
\]
This shows that $x \in \T{Lip}(S_q^2)$ and the corollary is proved.
\end{proof}

\begin{remark}\label{rem:algebraic-podisomet}
The algebraic counterpart to Proposition \ref{p:podisomet} is basically a triviality. Indeed, for  $x\in \C O(S_q^2)$ we have that $\pa_{t,q}(x)=  \sma{0 & - q^{-1/2} \pa_f(x) \\ -q^{1/2}\pa_e(x) & 0} = \pa^0(x)$, and it moreover holds that
\[
\sma{0 & - q^{-1/2} \pa_f(x) \\ -q^{1/2}\pa_e(x) & 0}^* \sma{0 & - q^{-1/2} \pa_f(x) \\ -q^{1/2}\pa_e(x) & 0} = \pa^0(x)^*\pa^0(x) \in \mathbb{M}_2(\C O(S_q^2)) .
\]
We thereby obtain that
\[
 L_{t,q}(x)^2 = \| \pa_{t,q}(x)^*\pa_{t,q}(x) \|_{C(SU_q(2))}= \| \pa^0(x)^*\pa^0(x) \|_{C(S_q^2)}=  L_q^0(x)^2,
\]
where $L_q^0$ denotes the variation of $L_q^{0,\max}$ whose domain is $\C O(S_q^2)$.
\end{remark}


%


\subsection{Spectral projections and twisted derivations}\label{subsec:spectral-projections}
Let $\te \colon S^1 \ti C(SU_q(2)) \to C(SU_q(2))$ be a strongly continuous action of the circle on quantum $SU(2)$. In this section we are investigating the relationship between the spectral projections coming from $\te$ and the  twisted $*$-derivations
\[
\pa_q^H \T{ and } \pa^V_t \colon \T{Lip}_t(SU_q(2)) \to \B B\big( L^2(SU_q(2))^{\op 2} \big) 
\]
introduced in Section \ref{s:twisspectrip}. As a first consequence of these efforts, we shall  establish, in Proposition \ref{p:twistdericlos} below, that the sum of  twisted $*$-derivations $\pa_{t,q} = \pa_q^H + \pa^V_t$ is closable. \\

We suppose that there exists a $2\pi$-periodic, strongly continuous one-parameter unitary group $(U_r)_{r \in \rr}$ acting on the Hilbert space $L^2(SU_q(2))^{\op 2}$ such that
\begin{enumerate}
\item $U_r \pi(x) U_{-r} = \pi\big( \te(e^{ir},x) \big)$ for all $r \in \rr$ and $x \in C(SU_q(2))$,  where $\pi$ is the diagonal unital $*$-homomorphism introduced in Section \ref{subsec:Haar-state}.
\item $[D^H_q, U_r] = 0 = [D^V_t, U_r]$ for all $r \in \rr$;
\item $\te(z, \si_L(w,x) ) = \si_L(w, \te(z,x))$ for all $z,w \in S^1$, $x \in C(SU_q(2))$.
\end{enumerate}
Since the map $r \mapsto U_r$ is strongly continuous, we obtain that the map $r \mapsto U_r T U_{-r}$ is weakly continuous for every $T \in \mathbb{B}(L^2(SU_q(2))^{\oplus 2})$. For each $n \in \zz$ we may therefore define the \emph{n$^{\T{\emph{th}}}$ spectral projection}   $\Pi_n^\te \colon \B B( L^2(SU_q(2))^{\op 2}) \to \B B( L^2(SU_q(2))^{\op 2})$, implicitly, by the formula
\begin{equation}\label{eq:specprojWOT}
\begin{split}
& \inn{\xi,\Pi_n^\te(T)\eta} = \frac{1}{2\pi} \int_0^{2\pi} \inn{\xi, U_r T U_{-r} \eta} e^{-irn} dr,   \q \xi, \eta \in L^2(SU_q(2))^{\oplus 2} .
\end{split}
\end{equation}
We remark  that the  spectral projections separate points; i.e.~that for $T\in \mathbb{B}(L^2(SU_q(2))^{\oplus 2})$, it holds that $T=0$  if and only if $\Pi_n^\te(T)=0$ for all $n\in \zz$. 
It follows from our conditions that the spectral projection $\Pi^\te_n \colon \mathbb{B}(L^2(SU_q(2))^{\oplus 2}) \to \mathbb{B}(L^2(SU_q(2))^{\oplus 2})$ induces a spectral projection $\Pi^\te_n \colon C(SU_q(2)) \to C(SU_q(2))$ satisfying that $\Pi^\te_n(\pi(x)) = \pi( \Pi^\te_n(x))$ for all $x \in C(SU_q(2))$ (we hope that this slight abuse of notation does not cause unnecessary confusion). The spectral projection on $C(SU_q(2))$ is given by the norm-convergent Riemann integral
\begin{align}\label{eq:spec-proj-norm-def-theta}
\Pi^\te_n(x) = \frac{1}{2\pi} \int_0^{2\pi} \te(e^{ir},x) \cd e^{-irn} dr .
\end{align}


\begin{lemma}\label{l:spectralprojderiv}
For each $n \in \zz$ and $x\in \T{Lip}_t(SU_q(2))$, it holds that $\Pi^\te_n(x) \in \T{Lip}_t(SU_q(2))$ and we have the identities
\begin{equation}\label{eq:projder}
\Pi^\te_n( \pa_q^H(x) ) = \pa_q^H( \Pi^\te_n(x) ) \q \mbox{and} \q \Pi^\te_n( \pa_t^V(x) ) = \pa_t^V( \Pi^\te_n(x) ) .
\end{equation}
In particular, it holds that the spectral projection $\Pi^\te_n \colon C(SU_q(2)) \to C(SU_q(2))$ is a contraction for our Lipschitz seminorm, meaning that 
\[
L_{t,q}^{\T{max}}\big( \Pi^\te_n(x)\big) \leq L_{t,q}^{\T{max}}(x) \q \mbox{for all }  x \in C(SU_q(2)) .
\]
\end{lemma}
\begin{proof}
 Let $n \in \zz$. The fact that $\Pi^\te_n$ becomes a contraction for the seminorm $L_{t,q}^{\T{max}}$ is going to follow from the identities in \eqref{eq:projder} together with the fact that the spectral projection $\Pi^\te_n$ is a norm-contraction. 
 Let $x \in \T{Lip}_t(SU_q(2))$ be given.  We focus on showing that $\Pi^\te_n(x)$ is horizontally Lipschitz and that the identity $\Pi^\te_n(\pa_q^H(x)) = \pa_q^H( \Pi^\te_n(x))$ is satisfied. The vertical case follows by a similar argument.

 We first record that the assumption (3) on $\theta$ implies that $\sigma_L(e^{ir},\Pi^\te_n(x))= \Pi^\te_n( \si_L(e^{ir},x))$ for all $r \in \rr$. Hence, Lemma \ref{l:boundedanalytic} shows that $\Pi^\te_n(x)$ is analytic of order $-\log(q)/2$ and that the identities  
\[
\Pi^\te_n(\sigma_L(q^{1/2},x))=\sigma_L(q^{1/2},\Pi^\te_n(x)) \q \T{and} \q \Pi^\te_n(\si_L(q^{-1/2},x)) = \si_L(q^{-1/2},\Pi^\te_n(x))
\]
are satisfied. Let now $\xi,\eta \in \T{Dom}(D^H_q)$ be given. We may then compute as follows:
\[
\begin{split}
& \inn{ D^H_q \xi, \si_L(q^{1/2}, \Pi^\te_n(x)) \eta} 
= \frac{1}{2\pi} \int_0^{2\pi} \inn{D^H_q \xi, U_r \si_L(q^{1/2},x) U_{-r} \eta} \cd e^{-irn} \, dr  \\
& \q = \frac{1}{2\pi} \int_0^{2\pi} \inn{\xi, U_r D^H_q \si_L(q^{1/2},x) U_{-r} \eta} \cd e^{-irn} \, dr \\
& \q = \frac{1}{2\pi} \int_0^{2\pi} \inn{\xi, U_r \pa_q^H(x) U_{-r} \eta} \cd e^{-irn} \, dr
+ \frac{1}{2\pi} \int_0^{2\pi} \inn{\xi, U_r \si_L(q^{-1/2},x) U_{-r} D^H_q \eta} \cd e^{-irn} \, dr \\
& \q = \inn{ \xi, \Pi^\te_n( \pa_q^H(x)) \eta }
+ \inn{\xi, \si_L(q^{-1/2},\Pi^\te_n(x)) D^H_q \eta} .
\end{split}
\]
This shows that $\si_L(q^{1/2}, \Pi^\te_n(x)) \eta \in  \T{Dom}((D^H_q)^*)= \T{Dom}(D^H_q)$, and moreover that
$\Pi^\te_n(x)$ is horizontally Lipschitz with $\pa_q^H( \Pi^\te_n(x)) = \Pi^\te_n( \pa_q^H(x))$.
\end{proof}




Our prime example, where the above lemma applies, is given by the $2\pi$-periodic strongly continuous one-parameter unitary group $(U_r^L)_{r \in \rr}$ defined by 
\begin{align}\label{eq:U_t-unitary-grp}
U_r^L\begin{pmatrix}\xi \\ \eta\end{pmatrix}
= \begin{pmatrix} e^{ir(k - 1)} \cd \xi \\ e^{ir(m + 1)} \cd \eta \end{pmatrix} \q \xi \in H^k_q \, , \, \, \eta \in H^m_q, 
\end{align}
This unitary group induces the left circle action on $C(SU_q(2))$ in the sense that
\begin{equation}\label{eq:unitaryleft}
U_r^L \pi(x) U_{-r}^L = \pi\big(\si_L(e^{ir},x)\big) \q \T{for all } r \in \rr \, \T{ and }  \, x \in C(SU_q(2)) .
\end{equation}
For each $n \in \zz$ we denote the corresponding spectral projection by $\Pi_n^L$. The following lemma now verifies that the last assumption (2) is indeed satisfied.

\begin{lemma}\label{l:diracinvariant}
It holds that $[D^H_q, U_r^L] = 0 = [D^V_t, U_r^L]$ for all $r \in \rr$. 
\end{lemma}
\begin{proof}
Recall that  $\C O(SU_q(2))^{\op 2}$ is a core for both of the unbounded selfadjoint operators $D^H_q$ and $D^V_t$. Moreover, we know that $\C O(SU_q(2))$ agrees with the algebraic linear span of the algebraic spectral subspaces $\C A^k_q$, $k \in \zz$.  It therefore suffices to prove the relevant commutator identities on vectors of the form $\pma{\xi \\ \eta}$ with $\xi \in  \C A^n_q$ and $\eta \in \C A^m_q$ for some  $n,m \in \zz$. The vanishing result for the commutator with the vertical Dirac operator $\C D^V_t$ is then clearly satisfied, so we focus on the horizontal Dirac operator $\C D^H_q$. In this case, the vanishing result follows since
\[
\C D^H_q \pma{\xi \\ \eta} = -\pma{q^{-1/2}\pa_{fk^{-1}}(\eta) \\ q^{1/2}\pa_{ek^{-1}}(\xi)} 
\in  \C A^{m+2}_q \op \C A^{n-2}_q . \qedhere
\]  
\end{proof} 

The assumptions in Lemma \ref{l:spectralprojderiv} are therefore met, and it yields the following:

\begin{cor}\label{c:diracinvariant}
Let $x \in \T{Lip}_t(SU_q(2))$ and $n \in \zz$. It holds that $\Pi_n^L(x) \in \T{Lip}_t(SU_q(2))$ and we have the identities
\[
\pa^V_t( \Pi_n^L(x) ) = \Pi_n^L(\pa^V_t(x)) \, \, \mbox{ and } \, \, \, \pa^H_q( \Pi_n^L(x) ) = \Pi_n^L(\pa^H_q(x)) .
\]
In particular, it holds that $L_{t,q}^{\T{max}}(\Pi_n^L(x)) \leq L_{t,q}^{\T{max}}(x)$.
\end{cor}

\begin{prop}\label{p:twistdericlos}
The sum of  twisted $*$-derivations $\pa_{t,q} = \pa_q^H + \pa^V_t \colon \T{Lip}_t(SU_q(2)) \to \B B\big( L^2(SU_q(2))^{\op 2} \big)$ is closable.
\end{prop}
\begin{proof}
 We first record that \eqref{eq:praktisk-ulighed} yields that
\[
\max\{ \| \pa_q^H(x) \|, \| \pa^V_t(x) \| \} \leq \| \pa_{t,q}(x) \| \leq \| \pa_q^H(x) \| + \| \pa^V_t(x) \| 
\q \T{for all } x \in  \T{Lip}_t(SU_q(2)).
\]
To see that $\pa_{t,q}$ is closable is thus suffices to show that $\pa_q^H$ and $\pa^V_t$ are both closable.
We focus on showing that $\pa_q^H$ is closable since the proof is almost the same for $\pa^V_t$.
For $m \in \zz$, we first remark that the restriction $\pa_q^H \colon \T{Lip}_t(SU_q(2)) \cap A^m_q \to \B B\big( L^2(SU_q(2))^{\op 2} \big)$ is closable:  indeed, for each $x \in \T{Lip}_t(SU_q(2)) \cap A^m_q$ and each $\xi \in \T{Dom}(D^H_q)$ we obtain from Lemma \ref{l:sigmaonspec} that
\[
\pa_q^H(x) \xi = D^H_q \si_L(q^{1/2},x) \xi - \si_L(q^{-1/2},x) D^H_q \xi 
= q^{m/2} D^H_q x \xi - q^{-m/2} x D^H_q \xi .
\]
The fact that the relevant restriction is closable then follows from the selfadjointness of $D^H_q$.  An application of Corollary \ref{c:diracinvariant} now shows that $\pa_q^H \colon \T{Lip}_t(SU_q(2)) \to \B B\big( L^2(SU_q(2))^{\op 2} \big)$ is closable. Indeed, as already remarked, for a bounded operator $y \in \B B( L^2(SU_q(2))^{\op 2})$ it holds that $y = 0$ if and only if $\Pi_n^L(y) = 0$ for all $n \in \zz$.
%
\end{proof}

\subsection{Spectral bands as compact quantum metric spaces}
We fix again our two parameters $t,q \in (0,1]$  together with an $M \in \nn_0$. We are now going to establish that the spectral band  $B_q^M =\sum_{m=-M}^M A_q^m$ becomes a compact quantum metric spaces when equipped with the Lipschitz seminorm $L_{t,q}^{\T{max}} \colon B_q^M \to [0,\infty]$ introduced in Definition \ref{def:max-and-min-seminorm}.  To this end, we utilise the general theory about finitely generated projective modules developed in Section \ref{ss:fingenproCQMS}. We emphasise that the  domain of the restriction $L_{t,q}^{\T{max}} \colon B_q^M \to [0,\infty]$ is substantially larger than the algebraic spectral band $\C B_q^M$. We start out by stating (and reproving) a well-known result regarding the spectral subspaces $(A_q^m)_{m\in \zz}$  (see e.g.~\cite[Proposition 3.5]{HaMa:PMM}). Recall that $\Pi^L_m\colon C(SU_q(2)) \to C(SU_q(2))$ denotes the spectral projection defined in \eqref{eq:spec-proj-norm-def-theta} associated with the circle action $\sigma_L$. 

\begin{lemma}\label{l:spec-subspaces-projective}
For any $x\in C(SU_q(2))$  and any $m\in \zz$ it holds that
\[
\Pi^L_m(x) = \cas{
\sum_{i = 0}^m (u_{i0}^m)^* \cd \Pi^L_0( u_{i0}^m \cd x ) &  m \geq 0 \\ 
\sum_{i = 0}^{|m|} (u_{i |m|}^{|m|})^* \cd \Pi^L_0( u_{i |m|}^{|m|} \cd x) & m < 0
}
\]
In particular, we obtain that $A_q^m$ is finitely generated and projective as a right module over $A_q^0=C(S_q^2)$.
\end{lemma}
\begin{proof}
By continuity and density, it suffices to check the identity for $x\in \C O (SU_q(2))$ and by linearity we may furthermore assume that $x\in \C A_q^k$ for some $k\in \zz$. If $m\geq 0$, then $u_{i0}^m \in \C A_q^{-m}$ (cf.~\eqref{eq:sigma-L-and-R-on-matrix-units}) and hence $u_{i0}^m x\in \C A_{q}^{k-m}$. Hence both sides are zero if $m\neq k$ and for $m=k$ the identity follows from the fact that $u^m$ is a unitary matrix.   The final statement about projectivity now follows, since the identity just proven shows that the map $A_q^m \to (A_q^0)^{\oplus(m+1)}$ given by $x\mapsto u_{\bullet 0}\cdot x$ provides an embedding of $A_q^{m}$ as a direct summand in a finitely generated free module. The case $m<0$ follows analogously.
\end{proof}

To show that  $(B_q^M, L_{t,q}^{\T{max}})$ is a compact quantum  metric space, we wish to apply  Theorem \ref{t:fingenproCQMS}, and we therefore need to compare the Lipschitz seminorm $L_{t,q}^{\T{max}}$ with the operator norm on quantum $SU(2)$ (see Assumption \ref{a:fingensemi}). This comparison takes place in the next two lemmas.

\begin{lemma}\label{lem:paV-on-spectral-subspaces}
For every $m\in \zz$, it holds that $A_q^m \subseteq \T{Lip}_t^{V}(SU_q(2))$ and 
\[
\pa_t^V (x)= \pma{[m/2]_tx & 0 \\ 0 & -[m/2]_tx} \q \mbox{for all } x\in A_q^m .
\]
\end{lemma}
\begin{proof}
Let $m\in \zz$ and $x \in A^m_q$ be given. We then know from Lemma \ref{l:sigmaonspec} that $x$ is analytic of order $-\log(t)/2$. Let now $n,k \in \zz$ and $y \in \C A_q^n \op \C A_q^k$ be given. We then have that
\[
\si_L(t^{1/2},x) \cd y = t^{m/2} x \cd y \in A^{m + n}_q \op A^{m + k}_q \su \T{Dom}(D^V_t). 
\]
Using the relation $[r+s]_t- t^{-r}[s]_t= t^s[r]_t$, which is valid for all $r,s\in \rr$, a direct computation shows that
\[
\big( D^V_t \si_L(t^{1/2},x) - \si_L(t^{-1/2},x) D^V_t \big) y
= \pma{ [m/2]_t \cd x & 0 \\ 0 & -[m/2]_t \cd x} \cd y .
\]
This proves that the twisted commutator
\[
D^V_t \si_L(t^{1/2},x) - \si_L(t^{-1/2},x) D^V_t
\]
is well-defined on the core $\C O(SU_q(2))^{\op 2}$ for the vertical Dirac operator and that it extends to the bounded operator $[m/2]_t \pma{x & 0 \\ 0 & -x}$. 
From this it follows that $ \si_L(t^{1/2},x) $ preserves $\T{Dom}(D_t^V)$ and that $\pa_t^V(x)=  \pma{[m/2]_tx & 0 \\ 0 & -[m/2]_t x}$ as desired.
\end{proof}


\begin{remark}\label{rem:form-of-patv}
As an aside, we remark that it is now easy to verify that the algebraic formula for $\pa_t^V$ obtained in  Lemma \ref{l:twicommu} actually extends to the whole Lipschitz algebra, in the sense that
\begin{align}\label{eq:form-of-patv}
 \pa_t^V(x)=\pma{\pa_t^3(x) & 0 \\ 0 & -\pa_t^3(x)} \q \T{for all } x\in \T{Lip}_t(SU_q(2)) .
 \end{align}
By Remark \ref{rem:pat-on-Lipschitz}, we already know that the off-diagonal elements in $\pa_t^V(x)$ are zero and that the upper left hand entry is $\pa_t^3(x)$. 
Conjugating $\pa_t^V(x)$ with the unitary $S:=\SmallMatrix{0& 1 \\1& 0} $ interchanges the diagonal entries, so it suffices to show that $S\pa_t^V(x)S=-\pa_t^V(x)$.  
A direct computation shows that the unitaries $(U_r^L)_{r\in \rr}$ defined in \eqref{eq:U_t-unitary-grp} satisfies
\[
SU_{r}^L =\pma{e^{2ir} & 0 \\ 0& e^{-2ir}}U_r^L S \q \T{for all } r \in \rr 
\]
and since $ \pa_t^V(x)$ is diagonal it commutes with the unitary $\SmallMatrix{e^{2ir} & 0 \\ 0 & e^{-2ir}}$. Using this, it is not difficult to see that
\[
\Pi_n^L(S\pa_t^V(x) S)=S\Pi_n^{L}(\pa_t^V(x))S \q \T{for all } n \in \zz  .
\]
By Lemma \ref{lem:paV-on-spectral-subspaces} we know that \eqref{eq:form-of-patv} is valid whenever $x$ belongs to a spectral subspace, and since 
$\Pi_n^L$ commutes with $\pa_t^V$ (see Corollary \ref{c:diracinvariant}) we therefore obtain that
\[
\Pi_n^L(S\pa_t^V(x) S)=S\Pi_n^{L}(\pa_t^V(x))S= S\pa_t^V(\Pi_n^{L}x)S=-\pa_t^V(\Pi_n^{L}x)=-\Pi_n^L(\pa_t^V(x)).
\]
Since the spectral projections separate points, it follows that $S\pa_t^V(x)S=-\pa_t^V(x)$ and hence that \eqref{eq:form-of-patv} holds.
\end{remark}

\begin{lemma}\label{l:lipdomnorm}
For each $m \in \zz$,  it holds that $| [m/2]_t | \cd \| \Pi_m^L(x) \| \leq L_{t,q}^{\T{max}}(x)$ for all  $x \in \T{Lip}_t(SU_q(2))$.
\end{lemma}
\begin{proof}
The result follows from Corollary \ref{c:diracinvariant} and Lemma \ref{lem:paV-on-spectral-subspaces} via the estimate
\[
| [m/2]_t | \cd \| \Pi_m^L(x) \| = \| \pa^V_t( \Pi_m^L(x) ) \| \leq L_{t,q}^{\T{max}}( \Pi_m^L(x)) \leq L_{t,q}^{\T{max}}(x) . \qedhere
\]
\end{proof}
Lastly, in order to apply Theorem \ref{t:fingenproCQMS}, we need to verify that Assumption \ref{a:fingensemi} (5) is satisfied, which is the contents of the following lemma: 

\begin{lemma}\label{l:leftbound}
Let $v \in \T{Lip}_t(SU_q(2))$ and let $M \in \nn_0$. Then the left-multiplication operator
\[
m(v) \colon B_q^M \cap \ker(\Pi_0^L) \to C(SU_q(2))
\]
is bounded with respect to the seminorm $L_{t,q}^{\T{max}}$.
\end{lemma}
\begin{proof}
We first remark that Lemma \ref{l:lipdomnorm} shows that there exists a constant $D_M > 0$ such that $\| x \| \leq D_M \cd L_{t,q}^{\T{max}}(x)$ for all $x \in B_q^M \cap \ker(\Pi_0^L)$. Next, it follows from Lemma \ref{l:anaband} and Lemma \ref{l:leftrightbound} that
\[
\begin{split}
L_{t,q}^{\max}(v \cd x) & \leq \| v\|_{t,q} \cd L_{t,q}^{\max}(x) + L_{t,q}^{\max}(v) \cd \| x \|_{t,q} \\
& \leq \Big( \| v\|_{t,q} + \sum_{m = -M}^M (t^{m/2} + q^{m/2}) D_M \Big) \cd L_{t,q}^{\max}(x) 
\end{split}
\]  
for all $x \in B_q^M \cap \ker(\Pi_0^L)$. This proves the present lemma.
\end{proof}

We are now in position to state and prove the main result of this section, which shows that the spectral bands are compact quantum metric spaces. Notice that it follows from Lemma \ref{l:spec-subspaces-projective} that the spectral bands are finitely generated projective modules. In fact, with a little extra effort it can be proved that they are free (but this does not help to ease the argumentation).

\begin{theorem}\label{t:bandcqms}
Let $M \in \nn_0$. The spectral band $B_q^M \su C(SU_q(2))$, the conditional expectation $\Pi_0^L \colon C(SU_q(2)) \to C(S_q^2)$ and the Lipschitz seminorm $L_{t,q}^{\T{max}} \colon C(SU_q(2)) \to [0,\infty]$ satisfy Assumption \ref{a:fingenerated} and Assumption \ref{a:fingensemi}. In particular, it holds that the restriction $L_{t,q}^{\T{max}} \colon B_q^M \to [0,\infty]$ provides $B_q^M$ with the structure of a compact quantum metric space. 
\end{theorem}
\begin{proof}
It follows from Lemma \ref{l:spec-subspaces-projective} that $B_q^M$ satisfies Assumption \ref{a:fingenerated}: indeed we may apply the elements in $C(SU_q(2))$, defined, for each $m \in \{-M,-M+1,\ldots,M\}$ and each $i \in \{0,1,\ldots,|m|\}$, by
\[
v_{im} := \cas{u^m_{i,0} & m \geq 0 \\ u^{-m}_{i,-m} & m < 0} \, \T{ and } \, \, w_{im} := v_{im}^*.
\]
Notice in this respect that $1=v_{00} = w_{00}$ and that $\Pi_0^L(v_{im}) = 0 = \Pi_0^L(w_{im})$ as soon as $(i,m) \neq (0,0)$ (cf.~\eqref{eq:sigma-L-and-R-on-matrix-units}). 
To see that conditions (1)-(5) in Assumption  \ref{a:fingensemi}  are satisfied, notice that (1) follows from Corollary \ref{c:diracinvariant}, while (2) follows from Corollary \ref{cor:zeroth-spec-band-is-cqms}. Condition (3) is a consequence of Lemma \ref{l:lipdomnorm} and condition (4) is trivially satisfied since $v_{im},  w_{im}\in \C O(SU_q(2))$ for all $m \in \{-M,-M+1,\ldots,M\}$ and $i \in \{0,1,\ldots,|m|\}$. Condition (5) is exactly the contents of Lemma \ref{l:leftbound} and Theorem \ref{t:fingenproCQMS} therefore shows that $(B_q^M, L_{t,q}^{\T{max}})$ is a compact quantum metric space.
%
\end{proof}
Knowing  that the spectral bands are compact quantum metric spaces, our next main goal will be to show that the same is true for quantum $SU(2)$. 
We wish to do so by an application of Corollary  \ref{cor:approx-cor}, but verifying that the assumptions there are indeed fulfilled turns out to be a slightly delicate matter.  One of our objectives will be to construct an ``anti-derivative'' of the twisted $*$-derivation $\pa^V_t \colon \T{Lip}_t(SU_q(2)) \to \B B(L^2(SU_q(2))^{\op 2})$. To this end we need the theory of Schur multipliers, and we  gather all the results needed within this context in the following section.

\subsection{Schur multipliers}\label{ss:schur}
Let $\| \cd \|_2 \colon \ell^2(\zz) \to [0,\infty)$ denote the usual Hilbert space norm on the Hilbert space of $\ell^2$-sequences indexed by $\zz$. The standard basis vectors in $\ell^2(\zz)$ are denoted by $e_i$, $i \in \zz$. We recall the following essential result due to Grothendieck:


\begin{prop}[Grothendieck]\label{p:groth}
Let $H$ and $K$ be Hilbert spaces and assume that they are $\zz$-graded as $H = \oplus_{i = -\infty}^\infty H_i$ and $K = \oplus_{i = -\infty}^\infty K_i$ such that each bounded operator $T\in \B B(H,K)$ is represented by a matrix $(T_{ij})_{i,j\in \zz}$. Let $\varphi\colon \zz\times \zz \to \cc$ be given and assume that there exist $a(i), b(i)\in \ell^2(\zz)$ for every $i \in \zz$ such that
\begin{enumerate}
\item $c(a):=\sup_{i\in \zz} \| a(i) \|_2 < \infty$ and $c(b):=\sup_{i\in \zz} \| b(i) \|_2 < \infty$;
\item $\varphi(i,j)= \inn{a(i), b(j)}$ for all $i,j\in \zz$.
\end{enumerate}
Then for every $T\in \B B(H,K)$ the matrix $(\varphi(i,j)T_{ij})_{i,j \in \zz}$ also defines a bounded operator from $H$ to $K$ and the map $\B M(\varphi) \colon \B B(H,K) \to \B B(H,K)$, which associates to $T$ the bounded operator with matrix $(\varphi(i,j)T_{ij})_{i,j}$, is completely bounded with cb-norm at most {$c(a)c(b)$}.
\end{prop}

Under the hypotheses of the theorem above, the map $\varphi$ is called a \emph{Schur multiplier}. For a more elaborate treatment of the theory of Schur multipliers the reader is referred to \cite{Pisier:Similarity}, but for the readers convenience we sketch the proof of Proposition \ref{p:groth} here.
\begin{proof}
Defining $a \colon H\to \ell^2(\zz) \hot H = \bigoplus_{i\in \zz} \ell^2(\zz)\hot H_i$ by  $a((\xi_i)_i):= (a(i) \ot \xi_i)_i$ and $b\colon K\to \ell^2(\zz)\hot K$ by $b( (\eta_i)_i) := (b(i) \ot \eta_i)_i$, one sees that $a$ and $b$ are bounded with $\|a\|\leq c(a)$ and $\|b\|\leq c(b)$. Moreover, one verifies that $\mathbb{M}(\varphi)(T)=b^*(1\ot T) a$ and hence we get $\|\mathbb{M}(\varphi)\|\leq c(a)c(b)$. The same argument works over matrices, so we indeed obtain that $\|\mathbb{M}(\varphi)\|_{\T{cb}}\leq c(a)c(b)$.
\end{proof}
%
For each $t,q \in (0,1]$, we wish to construct an anti-derivative of $\pa^V_t \colon \T{Lip}_t(SU_q(2)) \to \B B\big( L^2(SU_q(2))^{\op 2} \big)$, which will be given in terms of a Schur multiplier $\varphi_t \colon \zz \ti \zz \to \cc$ defined by the formula
 \begin{equation}\label{eq:varphi}
\varphi_t(i,j) := \begin{cases} \frac{1}{[ (i-j)/2]_t} & i \neq j\\
0 & i=j
\end{cases} .
\end{equation}
In order to show that $\varphi_t \colon \zz \ti \zz \to \cc$ is indeed a Schur multiplier we start out by recording a well known lemma on  $q$-numbers (including the proof for lack of a good reference):

\begin{lemma}\label{l:qnumest}
It holds that $[n/2]_q \geq \frac{n}{q^{1/2} + q^{-1/2}}$ for all $q \in (0,1]$ and $n \in \nn$.
\end{lemma}
\begin{proof}
Let $n \in \nn$ be given.  The inequality clearly holds for $q = 1$, so assume that $q \in (0,1)$. We first notice that $[n]_q \geq n$. Indeed, for $n = 2k$ even this inequality follows since
\[
[2k]_q - 2k = (q^{-2k + 1} + q^{2k -1} - 2) + (q^{-2k + 3} + q^{2k-3} - 2) \plp (q^{-1} + q - 2) \geq 0 
\]
and for $n = 2k+1$ odd we obtain the inequality since
\[
[2k+1]_q - (2k+1) = (q^{-2k} + q^{2k} - 2) + (q^{-2k + 2} + q^{2k - 2} - 2) \plp (q^{-2} + q^2 - 2) \geq 0 .
\]
We then obtain that
\[
[n/2]_q = \frac{q^{n/2} - q^{-n/2}}{(q^{1/2} - q^{-1/2})(q^{1/2} + q^{-1/2})} 
= \frac{ [n]_{q^{1/2}} }{q^{1/2} + q^{-1/2}} \geq \frac{n}{q^{1/2} + q^{-1/2}} . \qedhere
\] 
\end{proof}

\begin{lemma}\label{l:quaintest}
Let $t \in (0,1]$. The function $\varphi_t \colon \zz \ti \zz \to \cc$ is a Schur multiplier and we have the estimate
\[
\| \B M(\varphi_t) \|_{\T{cb}} \leq \frac{ \pi \cd (t^{1/2} + t^{-1/2} ) }{\sqrt{3}} 
\]
on the cb-norm of the associated completely bounded operator.
\end{lemma}
\begin{proof}
In order to apply Proposition \ref{p:groth}, we define the sequences $a(i) := \sum_{k = -\infty}^\infty \varphi_t(i,k) \cd e_k$ and $b(i) := e_i$ for all $i \in \zz$ and note that $\varphi_t(i,j)=\inn{a(i), b(j)}$. For each $i \in \zz$, we then apply Lemma \ref{l:qnumest} to obtain the estimate
\[
\| a(i) \|_2^2 = \| a(0) \|_2^2 = 2 \sum_{k = 1}^\infty \frac{1}{[ k/2]_t^2} 
\leq 2 \sum_{k = 1}^\infty \frac{(t^{1/2} + t^{-1/2})^2}{k^2} = \frac{\pi^2 (t^{1/2} + t^{-1/2})^2}{3} 
\]
on the Hilbert space norm. Since we moreover have that $\| b(i) \|_2^2 =1$ for all $i\in \zz$, the relevant estimate on the cb-norm now follows from Proposition \ref{p:groth}:
\[
\| \B M(\varphi_t) \|_{\T{cb}} \leq \sup_{i \in \zz}\| a(i) \|_2 \leq \frac{\pi \cd (t^{1/2} + t^{-1/2})}{\sqrt{3}} . \qedhere
\]
\end{proof}

We shall also need a systematic method for approximating elements in $C(SU_q(2))$ by elements in the spectral bands $B_q^M$, $M \in \nn_0$. This approximation will also take place by means of Schur multipliers. For each $M \in \nn_0$, we define the function $\ga_M \colon \zz \ti \zz \to \cc$ by the formula
\begin{equation}\label{eq:gamma}
\gamma_M(i,j) := \begin{cases} \frac{M+1-|i-j|}{M+1}  & |i - j|\leq M  \\
0 & |i-j| > M 
\end{cases} .
\end{equation}

\begin{lemma}\label{l:compcont}
For each $M \in \nn_0$, the function $\ga_M \colon \zz \ti \zz \to \cc$ is a Schur multiplier and we have the estimate 
\[
\| \B M(\gamma_M)  \|_{\T{cb}} \leq 1
\]
on the cb-norm of the associated completely bounded operator.
\end{lemma}
\begin{proof}
We are going to apply Proposition \ref{p:groth}.  Let $M \in \nn_0$ be given, and define, for each $i \in \zz$, the sequences
\[
a(i) = b(i) := \frac{1}{\sqrt{M+1}} \cd \sum_{k = i}^{i + M} e_k .
\]
We record that $\| a(i) \|_2^2 = 1 = \| b(i) \|_2^2$. Let now $i,j \in \zz$ be given, and assume first that $i \leq j$. We then compute that 
\[
\begin{split}
\inn{a(i),b(j)} & = \frac{1}{M+1} \cd \Big \langle \sum_{k = i}^{i + M} e_k, \sum_{l = j}^{j+M} e_l \Big\rangle
= \begin{cases}
\frac{1}{M+1} \sum_{k = j}^{i + M} 1 & j \leq i + M \\
0 & j > i + M
\end{cases} 
\\
& = 
\begin{cases}
\frac{M + i - j + 1}{M+1} & j - i \leq M \\
0 & j - i > M
\end{cases}
 = \ga_M(i,j) .
\end{split}
\]
For $j \leq i$ we get from the above identities that $\inn{a(i),b(j)} = \inn{a(j),b(i)} = \ga_M(j,i) = \ga_M(i,j)$. The proof is therefore complete.
\end{proof}

For each $\de \in (0,1)$ we define the null-sequence of positive real numbers $(\ep(\de,M))_{M = 0}^\infty$ by putting
\begin{equation}\label{eq:epnull}
\ep(\de,M) := 2^{1/2} \cd (\de^{1/2} + \de^{-1/2}) \cd \Big(  \frac{M}{(M+1)^2} + \sum_{k = M+1}^\infty \frac{1}{k^2} \Big)^{1/2} \q \T{for all } M \in \nn_0 .
\end{equation}
In particular, we record that $\ep(\de,0) = \frac{\pi \cd (\de^{1/2} + \de^{-1/2})}{\sqrt{3}}$.

\begin{lemma}\label{l:emtail}
Let $\de \in (0,1)$. It holds that
\[
\big \| \B M(\varphi_t) (1- \B M(\ga_M)) \big\|_{\T{cb}} \leq \ep(\de,M)
\]
for all $M \in \nn_0$ and all $t \in [\de,1]$.
\end{lemma}
\begin{proof}
Let $M \in \nn_0$ and $t \in [\de,1]$ be given. We are going to apply Proposition \ref{p:groth} to the function $\rho_{t,M} \colon \zz \ti \zz \to \cc$ given by the formula
\[
\rho_{t,M}(i,j) := \varphi_t(i,j) \cd \big(1 - \ga_M(i,j) \big)
 = \begin{cases}
\frac{|i - j|}{(M + 1) \cd [(i-j)/2]_t} & 0 < |i - j| \leq M \\
0 & |i - j | = 0 \\
\frac{1}{[(i-j)/2]_t} & |i - j | > M
\end{cases} .
\]
For each $i,j \in \zz$ we define the sequences
\[
a(i) := \sum_{k = -\infty}^\infty \rho_{t,M}(i,k) \cd e_k \q \T{and} \q b(j) = e_j .
\]
Applying Lemma \ref{l:qnumest}, we may estimate the Hilbert space norm of $a(i)$ as follows:
\[
\begin{split}
\| a(i) \|_2^2 
& = \| a(0) \|_2^2 = \sum_{k = -\infty}^\infty | \rho_{t,M}(0,k) |^2 
= 2 \cd \sum_{k =  M + 1}^\infty  \frac{1}{[k/2]_t^2} 
+ { \frac{2}{(M+1)^2}} \cd \sum_{k = 1}^{ M} \frac{k^2}{[k/2]_t^2} \\
& \leq 2 \left( t^{1/2} + t^{-1/2}\right)^2 \cd \Big( \sum_{k = { M+1}}^\infty \frac{1}{k^2}\Big)
+ { \frac{2M}{(M+1)^2}} \cd \left( t^{1/2} + t^{-1/2}\right)^2 \\
& \leq 2 \left(\de^{1/2} + \de^{-1/2}\right)^2 \cd \Big( \frac{M}{(M+1)^2} + \sum_{k = { M+1}}^\infty \frac{1}{k^2} \Big) = \ep(\de,M)^2 .
\end{split}
\]
This shows that $\rho_{t,M}$ is a Schur multiplier satisfying the estimate $\| \B M( \rho_{t,M}) \|_{\T{cb}} \leq \ep(\de,M)$ on the cb-norm of the associated completely bounded operator. The result of the present lemma now follows by noting that $\B M(\rho_{t,M}) = \B M(\varphi_t) \cd (1 - \B M(\ga_M) )$ by construction. 
\end{proof}

We end this subsection by re-introducing spectral projections in the context of Schur multipliers. For each $n \in \zz$ we define the Schur multiplier 
\begin{equation}\label{eq:delta}
\de_n \colon \zz\times \zz \to \cc \q \de_n(i,j):=\de_{n,i-j}.
\end{equation}
The associated operator $\B M(\de_n)$ is then completely contractive and can be interpreted in terms of spectral projections.  To explain this, suppose that $H = \op_{m = -\infty}^\infty H_m$ is a $\zz$-graded Hilbert space and define the unitary operator $V_r \colon H \to H$ by
\[
 \q V_r\Big( \sum_{m = -\infty}^\infty e_m \cd \xi_m \Big) := \sum_{m = -\infty}^\infty e_m \cd e^{irm} \xi_m
\]
for every $r \in \rr$. This yields a $2\pi$-periodic strongly continuous one-parameter unitary group $(V_r)_{r \in \rr}$ and it holds that
\[
\inn{\xi,\B M(\de_n)(T) \eta} = \frac{1}{2\pi}\int_0^{2\pi} \inn{\xi, V_r T V_{-r} \eta} e^{-irn} dr
\]
for all $T \in \B B(H)$ and $\xi,\eta \in H$. Thus, $\B M(\de_n)$ is the $n^{\T{th}}$ spectral projection associated with our $2\pi$-periodic strongly continuous unitary group $(V_r)_{r \in \rr}$; cf. \eqref{eq:specprojWOT}. 


\subsection{Projecting onto the spectral bands}


Throughout this section we again fix the parameters $t,q \in (0,1]$. We are going to apply the theory of Schur multipliers to the $\zz$-grading 
\begin{equation}\label{eq:decompspin}
L^2(SU_q(2))^{\op 2} = \bigoplus_{m = -\infty}^\infty (H_q^{m+1}\oplus H_q^{m-1}) .
\end{equation}
 This $\zz$-grading is simply the spectral subspace decomposition associated with the circle action on $L^2(SU_q(2))^{\op 2}$ induced by the $2\pi$-periodic strongly continuous one-parameter unitary group $(U_r^L)_{r \in \rr}$ introduced in \eqref{eq:U_t-unitary-grp}.  
For $M \in \nn_0$ and $n \in \zz$, we consider the Schur multipliers $\gamma_M, \de_n \colon \zz \times \zz \to \cc$ introduced  in \eqref{eq:gamma} and \eqref{eq:delta} and  apply the following notation for the completely bounded  operators they induce:
\[
 E_M^L := \B M(\ga_M)  \, \, , \, \, \, \Pi^L_n :=  \B M(\de_n)  \colon \B B\left( L^2(SU_q(2))^{\op 2}\right) \to 
\B B\left( L^2(SU_q(2))^{\op 2}\right) .
\]
 This notation is compatible with our already existing notation for spectral projections by the remarks at the end of  Section \ref{ss:schur}.
We emphasise that both $\Pi^L_n$ and $E_M^L$ induce operators on $C(SU_q(2))$ via the relations
\[
\Pi^L_n( \pi(x)) = \pi( \Pi^L_n(x)) \q \T{and} \q E_M^L(\pi(x)) = \pi( E_M^L(x))
\]
for all $x \in C(SU_q(2))$. Notice in this respect that 
\[
E_M^L = \sum_{m = - M}^{M} \frac{M + 1 - |m|}{M+1} \Pi_m^L .
\]

The aim of this subsection is to prove that $E_M^L$ is an $L_{t,q}^{\max}$-contraction onto the spectral $M$-band $B^M_q$ and that $E_M^L$ approximates the identity map on the $L_{t,q}^{\max}$-unit ball better and better as $M$ grows, \textcolor{black}{thus setting  the stage for an application of Corollary \ref{cor:approx-cor}. }

We are also interested in the completely bounded operator $\int_t^V \colon \B B( L^2(SU_q(2))^{\op 2}) \to \B B( L^2(SU_q(2))^{\op 2})$ defined by the formula
\begin{equation}\label{eq:quaintdef}
\int_t^V T := \B M(\varphi_t)( \ga \cd T ) ,
\end{equation}
where we recall that $\gamma:=\SmallMatrix{1& 0 \\0&-1}  \in \B B( L^2(SU_q(2))^{\op 2})$ and that $\varphi_t \colon \zz \ti \zz \to \cc$ was introduced in \eqref{eq:varphi}. Remark that the Schur multiplier $\B M(\varphi_t)$ is also defined relative to the spectral subspace decomposition given in \eqref{eq:decompspin}. We record that $\B M(\varphi_t)$ induces a bounded operator on $C(SU_q(2))$:  indeed, for each $m \in \zz$ and $x \in A^m_q$ we have the formula
\[
\B M(\varphi_t)(x) = \cas{ \frac{1}{[m/2]_t} \cd x & m \neq 0 \\ 0 & m = 0}, 
\] 
from which it follows that $\B M(\varphi_t)$ preserves $\C O (SU_q(2))$ and hence also $C(SU_q(2))$ by boundedness. We start out by proving that $\int_t^V$ serves as an anti-derivative with respect to $\pa_t^V$ providing a noncommutative analogue of the fundamental theorem of calculus.

\begin{prop}\label{prop:antiderivative}
For each $x\in \T{Lip}_t(SU_q(2))$, it holds that $\displaystyle\int_t^V \pa_t^V(x)=(1 - \Pi_0^L)(x)$.
\end{prop}
\begin{proof}
Let $x \in \T{Lip}_t(SU_q(2))$ be given. First note that if $x\in \T{Lip}_t(SU_q(2))\cap A_q^m$ for some $m\in \zz$, then the statement follows from Lemma \ref{lem:paV-on-spectral-subspaces}: indeed, in this case we have that
\[
\int_t^V \pa_t^V(x) = \int_t^V \pma{ [m/2]_t x & 0 \\ 0 & -[m/2]_t x} = [m/2]_t \cd \B M(\varphi_t)(x)
= (1 - \Pi_0^L)(x) .
\]
To prove the general statement, it suffices to show that
\[
\Pi_n^L\Big(\int_t^V \pa_t^V(x)\Big)  = \Pi_n^L(1 - \Pi_0^L)(x) \q \T{for all } n\in \zz .
\]
Let thus $n \in \zz$ be given. Since both $\Pi_n^L$ and $\B M(\varphi_t)$ are Schur multipliers with respect to the same $\zz$-grading on $L^2(SU_q(2))^{\op 2}$, they commute. Moreover, we notice that the grading operator $\ga$ preserves the spectral subspace $H^{m+1}_q \op H^{m-1}_q \su L^2(SU_q(2))^{\op 2}$ for all $m \in \zz$ and hence  it holds that left multiplication with $\ga$ commutes with $\Pi_n^L$. The relevant identity therefore becomes a consequence of Corollary \ref{c:diracinvariant} through the following computation: 
\[
\Pi_n^L\Big(\int_t^V \pa_t^V(x) \Big) = \int_t^V \pa_t^V(\Pi_n^L x)
= (1-\Pi_0^L)\Pi_n^L( x ) =\Pi_n^L(1-\Pi_0^L)(x). \qedhere
\]
\end{proof}


 The next step is to prove that $E_M^L$ is a contraction for $L_{t,q}^{\max}$, thus verifying part of the hypotheses in Corollary \ref{cor:approx-cor}.

\begin{lemma}\label{l:emlip}
Let $M \in \nn_0$ and $x\in \T{Lip}_t(SU_q(2))$. It holds that $E_M^L(x)\in  B_q^M\cap \T{Lip}_t(SU_q(2))$ and
$L_{t,q}^{\max}(E_M^L(x)) \leq L_{t,q}^{\max}(x)$. 
\end{lemma}
\begin{proof}
We start out by recalling that $E_M^L =  \sum_{m=-M}^{M} \tfrac{M+1-|m|}{M+1} \Pi_m^L$. It therefore follows from Lemma \ref{l:compcont}  and Corollary \ref{c:diracinvariant} that $E_M^L(x)\in \T{Lip}_t(SU_q(2))\cap B_q^M$ and that 
\[
L_{t,q}^{\max}(E_M^L(x)) = \| \pa_{t,q}(E_M^L(x))\|= \| E_M^L( \pa_{t,q}(x)) \| \leq \|\pa_{t,q}(x) \|=L_{t,q}^{\max}(x) . \qedhere
\]
\end{proof}

We now show that the sequence of $L_{t,q}^{\T{max}}$-contractions $(E_M^L)_{M = 0}^\infty$ approximates the identity map on the $L_{t,q}^{\T{max}}$-unit ball, thus verifying the last hypothesis in Corollary \ref{cor:approx-cor}.  In fact, this approximation can be obtained uniformly in the deformation parameters $t,q \in (0,1]$.  For each $\de \in (0,1)$ we recall the definition of the null-sequence of positive real numbers $(\ep(\de,M))_{M = 0}^\infty$ from \eqref{eq:epnull}.

\begin{prop}\label{p:bandapprox-without-saying-it}
Let $\de \in (0,1)$. It holds that
\[
\|x-E_M^L(x)\|\leq \ep(\de,M) \cd L_{t,q}^{\T{max}}(x)
\]
for all $M \in \nn_0$, $(t,q) \in [\de,1] \ti (0,1]$ and $x\in \T{Lip}_t(SU_q(2))$.
\end{prop}
\begin{proof}
We apply Proposition \ref{prop:antiderivative}  in combination with Lemma \ref{l:emtail} to obtain that
\[
\begin{split}
\|(1 - E_M^L)(x) \| & = \| (1 - E_M^L)(1 - \Pi_0^L)(x) \| = \big\| (1 - E_M^L) \int_t^V \pa_t^V(x) \big\| \\ 
& \leq \big\|  (1-\B M(\ga_M)) \B M(\varphi_t) \big\| \cd \| \ga \cd \pa_t^V(x) \|
\leq \ep(\de,M) \cd L_{t,q}^{\max}(x) 
\end{split}
\]
for all $M \in \nn_0$, $t \in [\de,1]$ and $x \in \T{Lip}_t(SU_q(2))$.
\end{proof}

\subsection{Quantum $SU(2)$ as a compact quantum metric space}
We are now ready to show that quantum $SU(2)$ becomes a compact quantum metric space when equipped with the Lipschitz seminorm $L_{t,q}^{\T{max}} \colon C(SU_q(2)) \to [0,\infty]$. 

\begin{theorem}\label{thm:quantum-su2-as-cqms}
The pair $\big(C(SU_q(2)), L_{t,q}^{\max}\big)$ is a compact quantum metric space for all $t,q \in (0,1]$.
\end{theorem}
\begin{proof}
For each $M \in \nn_0$ we know from Theorem \ref{t:bandcqms} that the spectral band $B_q^M$ becomes a compact quantum metric space when equipped with the restricted Lipschitz seminorm $L_{t,q}^{\T{max}} \colon B_q^M \to [0,\infty]$. We may then apply Corollary \ref{cor:approx-cor} using the compact quantum metric spaces $(B_q^M, L_{t,q}^{\max})$, together with the unital linear maps $E_M^L \colon C(SU_q(2)) \to B_q^M$ and the inclusions $\io_M \colon B_q^M \to C(SU_q(2))$. That the assumptions in Corollary \ref{cor:approx-cor}  are indeed met by this data follows from Lemma \ref{l:emlip} and Proposition \ref{p:bandapprox-without-saying-it}. 
\end{proof}

\begin{cor}\label{cor:alg-seminorm-gives-cqms}
The pair $\big(C(SU_q(2)), L_{t,q}\big)$ is a compact quantum metric space for all $t,q \in (0,1]$
\end{cor}
\begin{proof}
Since the $L_{t,q}$-unit ball is contained in the $L_{t,q}^{\max}$-unit ball this follows from Theorem \ref{thm:quantum-su2-as-cqms} and Theorem \ref{thm:rieffels-criterion}.
\end{proof}

We can also show that the spectral bands converge towards quantum $SU(2)$ in the quantum Gromov-Hausdorff distance. In fact, as the following theorem shows, the convergence can even be obtained in a uniform manner with respect to the deformation parameters $t, q \in (0,1]$. For each $\de \in (0,1)$ we recall the definition of the null-sequence of positive real numbers $(\ep(\de,M))_{M = 0}^\infty$ from \eqref{eq:epnull}.

\begin{theorem}\label{thm:band-approx}
 Let $\de\in (0,1)$. It holds that 
\begin{align*}
\T{dist}_{\T{Q}}\big((C(SU_q(2)), L_{t,q}^{\max}); (B_q^M, L_{t,q}^{\max}) \big)  \leq \ep(\de,M)
\end{align*}
for all $M\in \nn_0$ and $(t,q)\in [\de, 1] \ti (0,1]$. Moreover, for all  $\mu,\nu \in \C S\big(C(SU_q(2))\big)$ it holds that
\[
d_{t,q}^{\max}( \mu,\nu) \leq 2 \cd \ep(\de,M) + d_{t,q}^{\max}\big(\mu\vert_{B_q^M}, \nu\vert_{B_q^M} \big)
\]
for all $M \in \nn_0$, all $(t,q) \in [\de,1] \ti (0,1]$ 
\end{theorem}
\begin{proof}
   By Lemma \ref{l:emlip} and Proposition \ref{p:bandapprox-without-saying-it}, the unital positive operator $E^L_M\colon C(SU_q(2))\to B_q^M$ satisfies the assumptions in
    Corollary \ref{cor:subspacegen}  with $D=0$ and $\ep=\ep(\de, M)$. The first statement therefore follows from Corollary  \ref{cor:subspacegen} and the second from Corollary \ref{cor:subgenmet}.
\end{proof}

We may also provide an estimate on the diameter (see Definition \ref{def:diameter}) of quantum $SU(2)$ in terms of the  diameter of the Podle{\'s} sphere. 

\begin{prop}\label{prop:diameter-estimate}
For all $t,q\in (0,1]$ it holds that
\[
\T{diam}\big(C(SU_q(2)), L_{t,q}^{\max}\big) \leq  \frac{ 2 \pi \cd (t^{1/2} + t^{-1/2})}{\sqrt{3}} + \T{diam}\big(C(S_q^2),  L_q^{0,\max}\big) .
\]
\end{prop}
\begin{proof}
 By \cite[Proposition 5.5]{Rie:GHD} we have that 
\[
\T{diam}\big(C(SU_q(2)), L_{t,q}^{\max}\big) = 2 \cd \T{dist}_{\T{Q}}\big( (C(SU_q(2)), L_{t,q}^{\max}); (\cc,0) \big).
\]
Using the triangle inequality for the quantum Gromov-Hausdorff distance we then obtain that
\[
\begin{split}
\T{diam}\big(C(SU_q(2)), L_{t,q}^{\max}\big) & \leq 2 \cd \T{dist}_{\T Q}\big( (C(SU_q(2)), L_{t,q}^{\max}); (C(S_q^2),L_q^{0,\max}) \big) \\
& \q + \T{diam}\big(C(S_q^2), L_q^{0,\max}\big) .
\end{split}
\]
The result of the proposition now follows from Corollary \ref{cor:zeroth-spec-band-is-cqms} and Theorem \ref{thm:band-approx} in the case where $M = 0$.
%
\end{proof}

\begin{remark}\label{rem:bounded-diamter}
In \cite[Theorem 4.18]{AKK:Podcon} we proved that the family $\big( (C(S_q^2),L_{q}^{0,\max})\big)_{q\in (0,1]}$ varies continuously in the quantum Gromov-Hausdorff distance, and thus, in particular, that the function $ (0,1] \ni q \mapsto \T{diam}\big(C(S_q^2),L_{q}^{0,\max}\big) =2\cdot \T{dist}_{\T{Q}}((C(S_q^2), L_{q}^{0,\max}); (\cc,0))$ is continuous.  An application of Proposition \ref{prop:diameter-estimate} therefore shows that the function $(t,q)\mapsto \T{diam}\left(C(SU_q(2)), L_{t,q}^{\max}\right)$ is bounded on compact subsets of $(0,1]\times (0,1]$.
\end{remark}

\section{The quantum Berezin transform}\label{sec:berezin}
We now introduce the second key ingredient in the analysis of the quantum metric structure of $SU_q(2)$, namely an analogue of the classical Berezin transform (see e.g.~\cite{Sch:BTQ} and references therein) in this context. The Berezin transform was already essential in Rieffel's seminal results in \cite{Rie:MSG},  where he proves that the 2-sphere can be approximated by matrices. The Berezin transform also played a pivotal role in the analysis of the quantum metric structure on the Podle{\'s} spheres $S_q^2$, $q\in (0,1]$, in  \cite{AKK:Podcon} and \cite{AKK:Polyapprox}. In the present context it will serve to  firstly establish the fact that the maximal and minimal  Lip-norm, $L_{t,q}^{\max}$ and $L_{t,q}$ introduced in Definition \ref{def:max-and-min-seminorm}, actually give rise to the same quantum metric structure (see  Corollary \ref{cor:dist-zero} below).  Secondly, the Berezin transform provides us with finite dimensional quantum metric spaces which we will show approximate quantum $SU(2)$ in a suitably uniform manner. This, in turn, will be the key to our main continuity result, Theorem \ref{introthm:qgh-continuity}.

\subsection{Definition of the Berezin transform}
  Throughout this section, we fix the deformation parameter $q\in (0,1]$. The other parameter $t \in (0,1]$ is irrelevant in this section, since  we are currently only concerned with the $C^*$-algebras and not the Lip-norms.
 For each $N,M \in \nn_0$ we then define the element
\begin{equation}\label{eq:xidef}
\xi_N^M := \frac{1}{\sqrt{M + 1}} \sum_{r = N}^{N + M} a^r \cd \sqrt{ \inn{r + 1}_q} \in \C O(SU_q(2)),
\end{equation}
and consider the state $\chi_N^M \colon C(SU_q(2)) \to \cc$ given by $ \chi_N^M(x) := h\big( (\xi_N^M)^* x \xi_N^M \big)$. 
That $\chi_N^M$ is indeed a state follows from the formulae in \eqref{eq:haarmatrix} since $u_{00}^n=(a^*)^n$ for all $n \in \nn_0$. 
In order to analyse these states in more detail, it is convenient to first introduce a new circle action. Consider again the left and right circle actions $\sigma_L$ and $\sigma_R$ on $C(SU_q(2))$ defined on generators by
\begin{align}\label{eq:sigma-left-right-formulas}
\sigma_L(z,a):=za \, , \, \, \sigma_L(z,b):= zb \qquad \text{ and } \qquad \sigma_R(z,a):=za \, , \, \, \sigma_R(z,b)=z^{-1}b
\end{align}
A direct computation shows that 
\[
\sigma(z,x):= \sigma_R\big(z, \sigma_L^{-1}(z,x) \big) \q z \in S^1 \, , \, \, x \in C(SU_q(2))
\]
defines a strongly continuous circle action on $C(SU_q(2))$ which preserves $\C O(SU_q(2))$, and we let $\Pi^{\sigma}_m$, $m \in \zz$, denote the spectral projections associated with $\sigma$ (cf.~Section \ref{subsec:spectral-projections}). The circle action $\sigma$ is relevant in connection with the states $\chi_N^M$ and $\epsilon$ since, as we will se below, these only detect its fixed point algebra. We first determine the fixed point algebra in terms of the standard linear basis of $\C O(SU_q(2))$.

\begin{lemma}\label{lem:fixed-pt-basis}
The fixed point algebra of the circle action $\sigma$ on $C(SU_q(2))$ agrees with the norm closure of the linear span
\[
\T{span}_{\cc}\big\{ (b^*b)^m (a^*)^k, a^k (b^*b)^m\mid k,m \in \nn_0 \big\} .
\]
\end{lemma}
\begin{proof}
Since $\sigma$ fixes $a$, $a^*$ and $b^*b$, it is clear that the span in the statement of the lemma is contained in the fixed point algebra. For the opposite inclusion, one may use the standard linear basis \eqref{eq:standard-basis}  together with the spectral projection $\Pi_0^\si \colon C(SU_q(2)) \to C(SU_q(2))$. Indeed, it holds that 
\[
\Pi_0^\si(\xi^{klm}) = \fork{ccc}{ \xi^{klm} & \T{for} & m = l \\ 0 & \T{for} & m \neq l } . \qedhere
\]
\end{proof}

\begin{lemma} \label{lem:states-only-see-fixed-points}
Let $N,M \in \nn_0$. We have that $\epsilon= \epsilon \circ \Pi^{\sigma}_0$ and $\chi_N^M=\chi_N^M\circ \Pi^{\sigma}_0$.
\end{lemma}
\begin{proof}
 This follows immediately since $\epsilon\big( \sigma(z,x) \big) = \epsilon(x)$ and $\chi_N^M\big( \sigma(z,x)\big) = \chi_N^M(x)$ for all $z \in S^1$ and $x \in C(SU_q(2))$. In the case of  $\epsilon$ it suffices to check the relevant identity on the generators $a$ and $b$ and in the case of $\chi_N^M$ the  relevant identity follows since $\si(z, \xi_N^M) = \xi_N^M$ and $h\big( \si(z,x) \big) = h(x)$.
\end{proof}

\begin{lemma}\label{lem:convergence-to-counit}
We have the convergence result $\lim_{N,M \to \infty} \chi^M_N = \epsilon$ with respect to the weak$^*$ topology on $\C S\big(C(SU_q(2))\big)$. 
\end{lemma}
\begin{proof}
By Lemma \ref{lem:fixed-pt-basis} and Lemma \ref{lem:states-only-see-fixed-points}, we only need to treat elements of the form 
\[
(b^*b)^m (a^*)^k \quad \T{ and }  \quad a^k (b^*b)^m, \quad \T{ for }   k,m \in \nn_0.
\]
Since states preserve the involution and 
\[
(b^*b)^m (a^*)^k  = q^{-2km} (a^*)^k(b^*b)^m 
\]
it is enough to check the claim on elements of the form $(a^*)^k(b^*b)^m$. But since
\[
\T{span}_{\cc}\big\{ (b^*b)^m \mid m\in \nn_0\big\}=\text{span}_{\cc}\big\{(a^*)^n a^n \mid n\in \nn_0\big\},
\]
we may, equivalently, verify the convergence on elements of the form $(a^*)^{k+n} a^n$. Let now $k,n \in \nn_0$ be given. We are left with the task of showing that 
\[
\lim_{N,M  \to \infty} \chi_N^M\big( (a^*)^{k+n} a^n\big) = \epsilon\big((a^*)^{k+n} a^n\big)= 1  .
\]
Let us recall the inner product formulae from \eqref{eq:haarmatrix} as well as the fact that $(a^*)^m = u_{00}^m$ for all $m \in \nn_0$. For each $N,M \in \nn_0$ with $M \geq k$ we may thus compute as follows:
\[
\begin{split}
\chi_N^M \big( (a^*)^{k+n} a^n \big) 
& = \frac{1}{M+1}\sum_{i,j = N}^{N+M} \inn{i+1}_q^{1/2} \inn{j+1}_q^{1/2} h\big( (a^*)^{i+k+n} a^{n + j}\big) \\
& = \frac{1}{M+1}\sum_{i = N}^{N+M -k} \inn{i+1}_q^{1/2} \inn{i + k + 1}_q^{1/2} h\big( (a^*)^{i+k+n} a^{i + k + n}\big) \\
& = \frac{1}{M+1}\sum_{i = N}^{N+M - k} \frac{ \inn{i+1}_q^{1/2} \inn{i+k+1}_q^{1/2} }{ \inn{i + k + n + 1}_q} .
\end{split}
\]
Let now $\ep > 0$ be given. Since $\lim_{s \to \infty} \frac{ \inn{s}_q}{\inn{l + s}_q} = 1$ for all $l \in \nn_0$ we may choose $N_0 \in \nn_0$ such that
\[
\Big| \frac{ \inn{i +1}_q^{1/2} \inn{k + i +1}_q^{1/2}}{ \inn{n+k+ i + 1}_q} - 1 \Big| < \ep/2
\]
for all $i \geq N_0$. Furthermore, we may choose $M_0 \geq k$ such that $\frac{k}{M + 1} < \ep/2$ for all $M \geq M_0$. 
For all $M \geq M_0$ and $N \geq N_0$ we then estimate that
\begin{align*}
 \big| \chi_N^M \big( (a^*)^{k+n} a^n \big)  - 1 \big| 
& \leq \frac{1}{M+1}\sum_{i = N}^{N+M - k} \Big| \frac{ \inn{i+1}_q^{1/2} \inn{i+k+1}_q^{1/2} }{ \inn{i + k + n + 1}_q} - 1 \Big| + \Big| \frac{M - k + 1}{M + 1} - 1 \Big|\\
&< \ep/2 + \ep/2 = \ep .
\end{align*}
This proves the proposition.
\end{proof}

We are now ready to introduce the analogue of the Berezin transform in our $q$-deformed setting: 

\begin{dfn} The \emph{quantum Berezin transform} in degree $N,M \in \nn_0$ is the completely positive unital map 
$\be_N^M \colon C(SU_q(2)) \to C(SU_q(2))$ given by $\be_N^M(x) := \left(1 \ot \chi_N^M\right)\De(x)$.
\end{dfn}


\begin{remark}\label{rem:relation-with-podles-berezin}
In \cite{AKK:Podcon}, a quantum Berezin transform was introduced for the standard Podle{\'s} sphere $S_q^2$ in a manner very similar to the one above; see also \cite{IzNeTu:PBD} for an alternative and much more general construction of a Berezin transform on quantum homogeneous spaces. In \cite{AKK:Podcon}, the states defining the Berezin transform were denoted $h_N$, $N \in \nn_0$, and given by $h_N(x):=\inn{N+1}_q h\big((a^*)^Nxa^{N}\big)$ for all $x\in C(S_q^2)\su C(SU_q(2))$. We therefore have that $h_N=\chi_{N}^0\vert_{C(S_q^2)}$. In particular, the restriction of $\be_N^0$ to $C(S_q^2)$ agrees with the Berezin transform $\beta_N$ introduced in \cite{AKK:Podcon}. When $q=1$, we recovered the usual Berezin transform on the classical 2-sphere; see \cite[Section 3.2]{AKK:Podcon} for details on this. Note also that a Berezin transform for quantum homogeneous spaces was introduced in \cite{Sain:Thesis} in the setting of Kac type quantum groups. Since $SU_q(2)$ is only of Kac type when $q=1$ the constructions in \cite{Sain:Thesis} unfortunately do not apply directly in our context. However, as we shall see below, the more ad hoc definition above shares a number of properties with the construction in \cite{Sain:Thesis}.
%
%
\end{remark}

\subsection{The image of the Berezin transform}\label{ss:image-of-berezin}
When investigating  the quantum Gromov-Hausdorff continuity of the family $\big( C(SU_q(2)),L_{t,q}^{\T{max}}\big)_{t,q \in (0,1]}$, a detailed understanding of the image of the Berezin transform $\beta_N^M$ turns out to be imperative. In this section we therefore describe this image explicitly in terms of polynomial expressions in the generators $a,b,a^*,b^*$ for $\C O(SU_q(2))$. 

For each $r,s \in \nn_0$, we introduce the linear functional $\varphi_{r,s} \colon C(SU_q(2)) \to \cc$ given by
\[
\varphi_{r,s}(x) := h\big( (a^*)^s x a^r \big) .
\]
These linear functionals are then related to our states $\chi_N^M$  (see \eqref{eq:xidef}) by the formula
\begin{equation}\label{eq:statphi}
\chi_N^M = \frac{1}{M+1} \sum_{r,s = N}^{N + M} \sqrt{ \inn{r + 1}_q \inn{s + 1}_q } \cd \varphi_{r,s},  \q N,M \in \nn_0 .
\end{equation}
We now wish to determine the image of the Berezin transform $\be_N^M$. To this end we first analyse the linear functionals $\varphi_{r,s}$ in more details.

\begin{lemma}\label{l:imageberI}
Let $n,r,s \in \nn_0$ and $0 \leq i,j \leq n$. It holds that $\varphi_{r,s}(u^n_{ij}) \geq 0$ and that
\[
\varphi_{r,s}(u^n_{ij}) \neq 0 \Longleftrightarrow \big( n - 2j = r - s \T{ and } i = j \T{ and } j \leq s \big) .
\]
\end{lemma}
\begin{proof}
First note that by \eqref{eq:modular-function} we have the identities
\[
\varphi_{r,s}(u_{ij}^n)=h\big((a^*)^s  u_{ij}^n a^r\big)= h\big( \nu(a^r) \cd (a^*)^s u_{ij}^n\big)=q^{-2r}h\big(a^r (a^*)^s u_{ij}^n\big).
\]
Applying the formulae  \eqref{eq:leftmult} we obtain that
\begin{equation}\label{eq:leftpow}
\begin{split}
a^r \cd u^n_{ij} & =  \sum_{k = 0}^r \la_{n,i,j}(k) \cd u_{i+k,j+k}^{n + 2k - r} \quad \T{and} \\
(a^*)^s \cd u^n_{ij} & = \sum_{k = 0}^{ \min\{i,j,s\}} \mu_{n,i,j}(k) \cd u_{i -k,j-k}^{n-2k + s} ,
\end{split}
\end{equation}
where all the coefficients appearing are strictly positive. Now note that $h(u^m_{kl}) = 0$ for all $m > 0$ and $h(u^0_{00}) = 1$.  We see from the formulae in \eqref{eq:leftpow} that if the matrix coefficient $u_{00}^0$ appears in the double sum expressing $a^r (a^*)^s \cd u^n_{ij}$,  then there are terms of the form $u^m_{00}$ in the sum expressing $(a^*)^s \cd u^n_{ij}$. This in turn implies that $s\geq j$  and $i = j$. We thus arrive at the following expressions:
\begin{align*}
h\big( a^r (a^*)^s \cd u^n_{ij} \big) 
& = \cas{
\la_{n-2j + s,0,0}(0) \cd \mu_{n,j,j}(j) \cd h\big( u_{00}^{n - 2j + s - r} \big) & i = j \, , \, \, j \leq s \\
0 & \T{elsewhere} } \\
& = \cas{
\la_{r,0,0}(0) \cd \mu_{n,j,j}(j) &  i = j \, , \, \, j \leq s \, , \, \, n = r -s + 2j \\
0 &  \T{elsewhere} } .
\end{align*}
This proves the lemma.
\end{proof}


 \begin{lemma}\label{l:imageberIA} 
Let $N,M \in \nn_0$ and $m \in \zz$ with $|m| > M$. It holds that $\be_N^M(x) = 0$ for all $x \in A_q^m$.
\end{lemma}
 \begin{proof}
Since $\be_N^M$ preserves the involution and $A_q^m = (A_q^{-m})^*$ we may suppose that $m < -M$. Furthermore, we may assume that $x = u^{2j - m}_{ij}$ for some $j \in \nn_0$ and $i \in \{0,1,\ldots,2j - m\}$ since $\C A_q^m$ is spanned by such matrix coefficients by \eqref{eq:sigma-L-and-R-on-matrix-units}. It then follows from Lemma \ref{l:imageberI} that
\[
\be_N^M(x) = \sum_{k = 0}^{2j - m} u_{ik}^{2j - m} \cd \chi_N^M(u_{kj}^{2j - m}) 
= u_{ij}^{2j-m} \cd \chi_N^M(u_{jj}^{2j - m}) = 0 .
\]
Indeed, for all $r,s \in \{N,\ldots,N+M\}$ we have  $2j - m - 2j = -m > M \geq r - s$.
\end{proof}

\begin{lemma}\label{l:imageberII}
Let $N,M \in \nn_0$ and let $m \in \{0,\ldots,M\}$. Let moreover $j \in \nn_0$ and $i \in \{0,1,\ldots,2j+m\}$. It holds that
\[
\be_N^M(u^{2j+m}_{ij}) \neq 0 \ \Longleftrightarrow \ j \in \{0, \ldots,N+M-m\} .
\]
In this case $\chi_N^M(u^{2j+m}_{jj}) > 0$ and we have the formula
\[
\be_N^M(u^{2j+m}_{ij}) = u^{2j+m}_{ij} \cd \chi_N^M(u^{2j+m}_{jj}) .
\]
Similarly, we have that 
\[
\be_N^M(u^{2j+m}_{i,j+m}) \neq 0 \ \Longleftrightarrow \  j \in \{0,\ldots,N+M-m\} .
\]
In this case $\chi_N^M(u^{2j+m}_{j+m,j+m}) > 0$ and we have the formula
\[
\be_N^M(u^{2j+m}_{i,j+m}) = u^{2j+m}_{i,j+m} \cd \chi_N^M(u^{2j+m}_{j+m,j+m}) .
\]
\end{lemma}
\begin{proof}
Let $k \in \{0,\ldots,2j+m\}$ be given. Using Lemma \ref{l:imageberI} together with \eqref{eq:statphi} we obtain that $\chi_N^M(u^{2j+m}_{kj}) \neq 0$ if and only if $k = j$ and there exist $r,s \in \{N,\ldots,N+M\}$ with $r - s = m$ and $j \leq s$. Since we have assumed that $M\geq m \geq 0$ we then see that $\chi_N^M(u^{2j+m}_{kj}) \neq 0$ if and only if $k = j$ and $j \in \{0,1,\ldots,N+M - m\}$. 
In this case, we moreover have that $\chi_N^M(u^{2j+m}_{kj}) > 0$. The first claim of the lemma (regarding $u^{2j+m}_{ij}$) therefore follows since
\[
\be_N^M(u^{2j+m}_{ij}) = \sum_{k = 0}^{2j+m} u^{2j+m}_{ik} \cd \chi_N^M(u^{2j+m}_{kj}) = u^{2j+m}_{ij} \cd \chi_N^M(u^{2j+m}_{jj}) .
\]
The remaining claim is now a consequence of the positivity of the linear maps $\be_N^M \colon C(SU_q(2)) \to C(SU_q(2))$ and $\chi_N^M \colon C(SU_q(2)) \to \cc$. Indeed, we know from \eqref{eq:adjoint-of-matrix-coefficients} that $(u^{2j+m}_{i,j})^* = (-q)^{j-i} u^{2j+m}_{2j+m - i, j + m}$.
\end{proof} 

\begin{lemma}\label{lem:berezin-image}
Let $N,M \in \nn_0$ and let $m \in \{0,1,\ldots,M\}$. It holds that 
\begin{align*}
\be_N^M(A^{-m}_q) &= \T{span}_{\cc}\big\{ u^{2j + m}_{ij} \mid 0\leq j \leq N+M-m , \ 0\leq i\leq 2j+m    \big\} \q \mbox{and} \\
\be_N^M(A^{m}_q) &= \T{span}_{\cc}\big\{ u^{2j + m}_{i,j+m} \mid 0 \leq j \leq N+M - m, \ 0\leq i \leq 2j + m    \big\} .
\end{align*}
The vector space dimensions are  given by
\[
\T{dim}_{\cc}\big( \be_N^M(A^{-m}_q) \big) = (N+ M + 1)(N+M+1 -m) = \T{dim}_{\cc}\big( \be_N^M(A^m_q) \big) .
\]
In particular, we have that $\be_N^M(A^k_q) \su \C A^k_q$ for all $k \in \{-M,-M+1,\ldots, M\}$.
\end{lemma}
\begin{proof}
 We first remark that the algebraic spectral subspace $\C A_q^{-m}$ is spanned by matrix coefficients of the form $u^{2j+m}_{ij}$ with $j \in \nn_0$ and $i \in \{0,1,\ldots,2j+m\}$ by \eqref{eq:sigma-L-and-R-on-matrix-units}.
  Similarly, since $\C A_q^m= (\C A_q^{-m})^*$ it follows from \eqref{eq:adjoint-of-matrix-coefficients} that $\C A_q^m$ is spanned by matrix coefficients of the form $u^{2j+m}_{i,j+m}$ with $j \in \nn_0$ and $i \in \{0,1,\ldots,2j+m\}$. The first claim regarding the images is then a consequence of Lemma \ref{l:imageberII}. 
The relevant formula for the dimension of the subspaces $\be_N^M(A_q^{-m})$ and $\be_N^M(A^m_q)$ now follows from the computation
\[
\begin{split}
\T{dim}_{\cc}\big(\beta_N^M(A_q^m)\big) & = \T{dim}_{\cc}\big(\beta_N^M(A_q^{-m})\big) = \sum_{j=0}^{N+M-m} (2j+m+1) \\ 
& = (N+M-m+1)(N+M+1). \qedhere
\end{split}
\]
\end{proof}
The images  of the spectral bands under the  Berezin transforms  will serve as our finite dimensional (also known as ``fuzzy'') approximations, analogous to the fuzzy spheres from \cite{Mad:TFS,Rie:MSG} and their $q$-deformed counterparts in \cite{AKK:Podcon}. It will, however, also be convenient to have a description available in terms of the generators of $SU_q(2)$ and we therefore opt to use this as the formal definition. To this end, recall from  \cite[Definition 3.5]{AKK:Podcon} that the quantum fuzzy sphere in degree $N \in \nn_0$ is defined as 
\begin{equation}\label{eq:quantum-fuzzy-sphere-def}
\begin{split}
\T{Fuzz}_N(S_q^2) & := 
\T{span}_{\cc}\big\{(bb^*)^i(ab^*)^j, (bb^*)^i (b a^*)^j \mid i,j\in \nn_0, i+j\leq N \big\}  \su \C O(S_q^2) .
\end{split}
\end{equation}

We now make the following definition:

\begin{dfn}\label{d:fuzzspec}
Let $N,m \in \nn_0$. We define the \emph{fuzzy spectral subspaces} as the finite dimensional vector spaces
\[
\begin{split}
 \T{Fuzz}_N(A_q^m) &:= \sum_{k = 0}^m a^k b^{m - k} \cd \T{Fuzz}_N(S_q^2) \su \C A_q^m  \q \T{and} \\
 \T{Fuzz}_N(A_q^{-m}) &:= \sum_{k = 0}^m (a^*)^k (b^*)^{m - k} \cd \T{Fuzz}_N(S_q^2) \su \C A_q^{-m} .
\end{split}
\]
Moreover, for $K\in \nn_0$ we define the \emph{fuzzy spectral $K$-bands} as
\[
\T{Fuzz}_N(B_q^K):= \sum_{m=-K}^K  \T{Fuzz}_N(A_q^m)\su  \C B_q^K.
\]
\end{dfn}
Note that since $\T{Fuzz}_N(S_q^2)$ increases with $N \in \nn_0$, the same is true for  $\T{Fuzz}_N(A_q^m)$ for all $m\in \zz$.
As mentioned above, the spaces just defined are intimately related to the quantum Berezin transform as the following result shows:

\begin{prop}\label{p:fuzzber}
 Let $N,M,K \in \nn_0$. It holds that $\be_N^M(A_q^m) = \T{Fuzz}_{N + M - |m|}(A_q^m)$  for all $m \in \{-M,\ldots,M\}$. In particular, $\beta_N^M(B_q^K)\subseteq \T{Fuzz}_{N+M}(B_q^K)$ whenever $M\geq K$ and $\T{Fuzz}_N(B_q^K)$ is an operator system (without any constraints on $N,K \in \nn_0$).
\end{prop}
\begin{proof}
Let $m\in \zz$ with  $|m| \leq M$ be given. We focus on the case where $m \in \nn_0$ since the case where $m < 0$ follows from similar arguments. We begin by recalling  from \cite[Lemma 3.4 and Lemma 3.7]{AKK:Podcon} that 
\begin{equation}\label{eq:fuzzpod}
\T{Fuzz}_{N + M - m}(S_q^2)  = \T{span}_{\cc}\big\{ u^{2j}_{ij} \mid 0 \leq j \leq N + M - m \, , \, \, 0 \leq i \leq 2j \big\} .
\end{equation}
 Similarly, we recall from Lemma \ref{lem:berezin-image} that
\[
 \be_N^M(A_q^m)=\T{span}_{\cc}\big\{ u^{2j + m}_{i,j+m} \mid 0 \leq j \leq N+M - m, \ 0\leq i \leq 2j + m    \big\}  .
\]
For $m = 0$,  the identity $\T{Fuzz}_{N + M - m}(A_q^m) = \be_N^M(A_q^m)$  therefore follows immediately. We may thus suppose that $m > 0$. 
Let us start out by proving the inclusion 
\begin{equation}\label{eq:fuzzinc}
\T{Fuzz}_{N + M - m}(A_q^m) \su \be_N^M(A_q^m).
\end{equation}
For each  $l, i,j\in \nn_0$ with $i \leq 2j+l$ it follows from \eqref{eq:leftmult} that
\[
\begin{split}
& a \cd u^{2j+l}_{i,j+l} \in \T{span}_{\cc}\big\{ u^{2j + (l+1)}_{i+1, j + (l+1)} \, , \, \, u^{2(j-1) + (l+1)}_{i,j-1 + (l+1)} \big\} \q \T{and} \\
& b \cd u^{2j + l}_{i,j+l} \in \T{span}_{\cc}\big\{ u^{2j + (l+1)}_{i, j + (l+1)} \, , \, \, u^{2(j-1) + (l+1)}_{i-1,j-1 + (l+1)} \big\} .
\end{split}
\]
Hence, for all $k\in \{0,\dots, m\}$, $j \in \{0,\ldots,N + M - m\}$ and $i \in \{0,\ldots, 2j \}$ it holds that
\[
a^k b^{m-k} u_{ij}^{2j}\in  \T{span}_{\cc}\big\{ u^{2j + m}_{i,j+m} \mid 0 \leq j \leq N+M - m, \ 0\leq i \leq 2j + m    \big\} = \be_N^M(A_q^m) .
\]
By definition of $\T{Fuzz}_{N + M - m}(A_q^m)$, the inclusion in \eqref{eq:fuzzinc} therefore follows.
In order to show that $\T{Fuzz}_{N + M - m}(A_q^m) = \be_N^M(A_q^m)$, it now suffices to establish that
\[
\dim_{\cc}\big( \T{Fuzz}_{N + M - m}(A_q^m) \big) \geq \dim_{\cc}\big( \be_N^M(A_q^m) \big) .
\]
Rewriting the definition of the quantum fuzzy sphere from \eqref{eq:quantum-fuzzy-sphere-def} slightly we obtain
\begin{equation}\label{eq:fuzzy-sphere-2}
\T{Fuzz}_{N + M - m}(S_q^2) = \T{span}_{\cc}\big\{ a^j b^i (b^*)^{i+j} \, , \, \, 
(a^*)^j b^{i + j} (b^*)^i 
 \mid i,j \in \nn_0 \,  \, , i + j \leq N + M - m \big\} .
\end{equation}
From the two extremes, $k=0$ and $k=m$, in Definition \ref{d:fuzzspec}  we obtain  that
\[
\begin{split}
M_1 &:= \big\{ a^{m+j} b^i (b^*)^{i+j} \, , \, \, (a^*)^j b^{m+i+j} (b^*)^i 
\mid  i,j \in \nn_0 \, , \, \, i + j \leq N + M - m \big\} \\
& \subseteq\T{Fuzz}_{N + M - m}(A_q^m) .
\end{split}
\]
Similarly,  fixing $j=0$ in  \eqref{eq:fuzzy-sphere-2} and letting $k$ vary in $\{1,\dots, m-1 \}$ we obtain that
\[
\begin{split}
M_2 & := \big\{ a^k b^{m-k + i} (b^*)^i \mid   1\leq k \leq m-1 \, , \, \, 0 \leq i \leq N+M-m  \big\} \\
& \subseteq \T{Fuzz}_{N + M - m}(A_q^m).
\end{split}
\]
Since $m \geq 1$, we see from \eqref{eq:standard-basis} that the set $M_1\cup M_2$ consists of linearly independent vectors and its cardinality is given by
\[
(m-1) \cd (N + M - m + 1) +  2 \cd\left( \sum_{i = 1}^{N + M - m + 1} i \right)
= (N + M - m + 1)\cd (N + M + 1) ,
\]
which is exactly $\dim_{\cc}(\beta_N^M(A_q^m))$ by Lemma \ref{lem:berezin-image}. This completes the proof of the first part of the lemma.\\

The last two statements of the lemma follow from the first part. Firstly, for $M\geq K$ we have that
\[
\beta_N^M(B_q^K)=\sum_{m=-K}^K \T{Fuzz}_{N+M-|m|}(A_q^m)\subseteq \sum_{m=-K}^K \T{Fuzz}_{N+M}(A_q^m)
=\T{Fuzz}_{ N+M}(B_q^K) .
\]
Secondly, for arbitrary $N,K \in \nn_0$, $\T{Fuzz}_N(B_q^K)$ is an operator system since the Berezin transforms are $*$-preserving:
\[
\T{Fuzz}_N(A_q^m)^*=\beta_N^{|m|}(A_q^m)^*=\beta_N^{|m|}(A_q^{-m})= \T{Fuzz}_N(A_q^{-m}).\qedhere
\]
\end{proof}

\begin{cor}
 Let $m \in \zz$ and $N,K \in \nn_0$. It holds that $\T{dim}_{\cc}  (\T{Fuzz}_N(A_q^m))= (N+|m| +1)(N+1)$. In particular, both $\T{dim}_{\cc}  (\T{Fuzz}_N(A_q^m))$ and $\T{dim}_{\cc} (\T{Fuzz}_N(B_q^K)) $ are independent of  $q \in (0,1]$.
\end{cor}
\begin{proof}
The first identity follows from Lemma \ref{lem:berezin-image} since $\T{Fuzz}_N(A_q^m)=\beta_N^{|m|}(A_q^m)$ by Proposition \ref{p:fuzzber}. Since  $\T{Fuzz}_N(A_q^m) \su \C A_q^m$ one has that 
\[
\T{dim}_{\cc} (\T{Fuzz}_N(B_q^K)) =  \sum_{m=-K}^K  \T{dim}_\cc(\T{Fuzz}_N(A_q^m)),
\]
and $\T{dim}_{\cc} (\T{Fuzz}_N(B_q^K))$ is therefore also independent of  $q \in (0,1]$.
\end{proof}

Inspecting the proof of Proposition \ref{p:fuzzber}, we obtain an explicit linear basis for the fuzzy spectral subspaces:

\begin{cor}\label{c:fuzzbasis}
For each $N\in \nn$ and $m\in \zz$ the fuzzy spectral subspace $\T{Fuzz}_N(A_q^m)$ admits a linear basis consisting of a subset of the standard linear basis \eqref{eq:standard-basis} for $\C O(SU_q(2))$ which is independent of the value of $q$. Concretely the basis can be chosen as follows:
\begin{itemize}
\item For $m>0$ it is given by
\begin{align*}
&\left \{a^{j+m}b^i (b^*)^{i+j}, \  (a^*)^j b^{i+j+m}(b^*)^i  \mid i,j\in \nn_0, \  0\leq i+j\leq N  \right\}\\
&   \cup \ \left \{a^k b^{i+m-k}(b^*)^i \mid   k\in \{1, \dots, m-1\}, \   i\in \{0,\dots, N\}   \right\} .
\end{align*}

\item For $m=0$ it is given by 
\begin{align*}
  & \left\{ a^j b^i (b^*)^{i+j}, \ (a^*)^j b^{i+j} (b^*)^i \mid j\in \{1, \dots, N\}, \   i \in \{0,\dots, N-j\} \right\} \\
  & \cup \ \left \{ b^i (b^*)^i \mid i \in \{0,\ldots,N\} \right \} .
\end{align*}

\item For $m<0$ it is given by
\begin{align*}
&\left \{(a^*)^{j-m}b^{i+j} (b^*)^{i}, \  a^j b^i (b^*)^{i+j-m}  \mid i,j\in \nn_0, \  0\leq i+j\leq N  \right\}\\
&   \cup \ \left \{(a^*)^k b^i (b^*)^{i-m-k}(b^*)^i \mid   k\in \{1, \dots, -m-1\}, \   i\in \{0,\dots, N\}   \right\} .
\end{align*}
\end{itemize}

\end{cor}

The fuzzy approximations of the 2-sphere originated in physics \cite{GrPr:DFS,Mad:TFS,Mad:NCG} and have the feature of carrying an action of $SU(2)$. In the $q$-deformed setting fuzzy approximations of the Podle\'s sphere have also been studied in the mathematical physics literature; see \cite{AlReSc:NWG,GrMaSt:FQFI,GrMaSt:FQFII}. In some sense these ideas can be traced back to the work of Podle\'s \cite{Pod:QS}.

Similarly to the quantum fuzzy spheres, our fuzzy spectral bands also carry a coaction of quantum $SU(2)$:


\begin{prop}\label{prop:fuzzy-coinvariant}
For each $N,K\in \nn_0$, the operator system $\T{Fuzz}_N(B_q^K) \subseteq \C O(SU_q(2))$ is $\C O(SU_q(2))$-coinvariant.  
\end{prop}
\begin{proof}
For every $x\in \T{Fuzz}_N(B_q^K)$ we need to show that $\Delta(x)\in \C O(SU_q(2)) \ot \T{Fuzz}_N(B_q^K)$. Let $m \in \{ -K,-K+1,\dots, K \}$. We shall in fact see that $\De(x) \in \C O(SU_q(2)) \ot \T{Fuzz}_N(A_q^m)$ whenever $x \in \T{Fuzz}_N(A_q^m)$. Indeed, since $\T{Fuzz}_N(A_q^m)=\beta_N^{|m|}(A_q^m)$ by Proposition \ref{p:fuzzber}, the relevant inclusion follows from Lemma \ref{lem:berezin-image} together with the formula for the coproduct on matrix coefficients.
\end{proof} 

Remark that for $q=1$, the comultiplication on $\C O(SU(2))$ is dual to the group multiplication. Letting $\la$ denote the left regular action of $SU(2)$ on $\C O(SU(2))$ and $\T{ev}_g \colon \C O(SU(2)) \to \cc$ denote the evaluation at a point $g \in SU(2)$, we have the formula $\lambda_{g^{-1}}f=(\T{ev}_g \ot 1)\Delta(f)$ for all $f \in \C O(SU(2))$. Thus, in this case the coinvariance in Proposition \ref{prop:fuzzy-coinvariant} does indeed correspond to invariance of $\T{Fuzz}_N(B_1^K)$ under the left regular action of $SU(2)$. \\

In the section to follow, we need to apply the Berezin transform, which is at the moment only defined on $C(SU_q(2))$, to elements in the von Neumann algebra $L^\infty(SU_q(2))$. Since $\Delta$ extends to a normal $*$-homomorphism at the von Neumann algebraic level and each $\chi_N^M$ is normal (being a vector state in the GNS representation), the slice map formula $(1\tens \chi_N^M)\De(x)$ also makes sense at the level of $L^\infty(SU_q(2))$. We could therefore simply extend the Berezin transform $\beta_N^M$ to $L^\infty(SU_q(2))$ using the same formula. However, it will be important to view this extension as a composition of a finite dimensional projection and the original Berezin transform and we therefore take this point of view as our point of departure. \\

For each finite dimensional subspace $F \su \C O(SU_q(2))$ we let $P_F \colon L^2(SU_q(2)) \to L^2(SU_q(2))$ denote the orthogonal projection with image $\La(F) \su L^2(SU_q(2))$. We then define the linear map $\Phi_F \colon \B B(L^2(SU_q(2))) \to \C O(SU_q(2))$ by the formula
\begin{equation}\label{eq:finitephi}
\La( \Phi_F(T) ) := P_F( T \cdot \Lambda(1) ) \q \T{for all } T \in \B B(L^2(SU_q(2))) .
\end{equation}
We record that $\T{Im}( \Phi_F) = F$ and that $\Phi_F$ is WOT-norm continuous, where we recall that WOT refers to the weak operator topology on $\B B\big(L^2(SU_q(2))\big)$. Notice, moreover, that $\Phi_{F_0} \Phi_{F_1} = \Phi_{F_1} \Phi_{F_0} = \Phi_{F_0}$ when $F_0,F_1 \su \C O(SU_q(2))$ are finite dimensional subspaces with $F_0 \su F_1$.
For each $N,M \in \nn_0$, we have that $\beta_N^M(B_q^M) \su \C B_q^M \su \C O(SU_q(2))$ is a finite dimensional subspace and we apply the notation
\[
\Phi_N^M := \Phi_{\beta_N^M(B_q^M)} \colon \B B\big(L^2(SU_q(2))\big) \to \C O(SU_q(2))
\]
for the associated linear map.  Recalling from Proposition \ref{p:fuzzber} that $\be_N^M(B_q^M) \su \T{Fuzz}_{N+M}(B_q^M)$ we now define the \emph{extended Berezin transform} $\wit{\beta}_N^M\colon L^\infty(SU_q(2)) \to \T{Fuzz}_{N+M}(B_q^M)$ by setting 
\begin{align}\label{eq:extended-berezin}
\wit{\beta}_N^M(x):=\beta_N^M(\Phi_N^M(x)) \q \T{for all } x \in L^\infty(SU_q(2)) .
\end{align}

\begin{lemma}\label{lem:extended-berezin-is-ucp}
Let $N,M \in \nn_0$. The extended Berezin transform $\wit{\beta}_N^M$ is ucp (unital completely positive) and satisfies that $\wit{\beta}_N^M(x) = \be_N^M(x)$ for all $x \in C(SU_q(2))$.
\end{lemma}
\begin{proof}
We start by showing that the extended Berezin transform does indeed extend the Berezin transform. By norm-density and linearity, it suffices to verify that $\beta_N^M(\Phi_N^M(u_{ij}^n)) = \beta_N^M(u_{ij}^n)$ for all $n\in \nn_0$ and $i,j\in \{0,\dots, n\}$. If $u_{ij}^n$ is not one of the matrix coefficients spanning $\beta_N^M(B_q^M)=\sum_{m=-M}^M \beta_N^M(A_q^m)$  then we obtain from Lemma \ref{l:imageberIA}, Lemma \ref{l:imageberII} and Lemma \ref{lem:berezin-image} that both sides of the claimed identity are equal to zero (recall here that the different matrix coefficients are orthogonal to one another when embedded in $L^2(SU_q(2))$). Conversely, if $u_{ij}^n \in \beta_N^M(B_q^M)$,   it holds that $\Phi_N^M(u_{ij}^n)=u_{ij}^n$, and the relevant identity therefore holds trivially. \\
We now focus on showing that $\wit{\beta}_N^M$ is completely positive. We first note that this is indeed the case for the Berezin transform $\beta_N^M$, being defined as the composition of the unital $*$-homomorphism $\Delta$ with the slice map induced by the state $\chi_N^M$. Let $x\in L^\infty(SU_q(2))\otimes \mathbb{M}_d(\cc)$ be given. Then there exists a net $(x_\al)_\al$ in $C(SU_q(2))\otimes \mathbb{M}_d(\cc)$ converging in the strong operator topology to $x$. The net $(x_\alpha^*x_\alpha)_\al$ therefore converges in the weak operator topology to $x^*x$ and since $\Phi_N^M$ is WOT-norm continuous we obtain that the net $\left( (\beta_N^M \otimes 1_d)(x_\alpha^* x_\alpha)\right)_\alpha$ converges in norm to $(\wit{\beta}_N^M \otimes 1_d)(x^*x)$. Since each $(\beta_N^M\otimes 1_d)(x_\alpha^* x_\alpha)$ is positive and the positive cone is norm closed we obtain that $(\wit{\beta}_N^M \otimes 1_d)(x^*x)$ is positive. This proves that the extended Berezin transform is completely positive.
\end{proof}

\subsection{Estimates on the Berezin transform}
Our next aim is to analyse the interplay between the Berezin transforms and the twisted derivations defining the Lip-norms $L_{t,q}^{\max}$. At the algebraic level, i.e.~with $L_{t,q}$ instead of $L_{t,q}^{\max}$, this analysis is slightly less complicated (see the remarks preceding Proposition \ref{prop:berezin-and-delta}), but at the analytic level things are more subtle. In the first series of lemmas below, we show how one may, nevertheless,  reduce certain questions to the algebraic setting by means of the projections $\Phi_F$ introduced  in \eqref{eq:finitephi}. \\

Throughout this section, we fix the two parameters $t$ and $q$ in $(0,1]$ unless explicitly stated otherwise.

\begin{lemma}\label{l:projinside}
Let $\xi, \eta \in \C O(SU_q(2))^{\op 2}$. Then there exists a finite dimensional subspace $F_0 \su \C O(SU_q(2))$ such that
\[
\binn{\xi, \pa_q^H(x) \eta} = \binn{\xi, \pa_q^H( \Phi_F(x)) \eta } \q \mbox{and} \q
\binn{\xi, \pa^V_t(x) \eta} = \binn{\xi, \pa^V_t( \Phi_F(x)) \eta } 
\]
whenever $F \su \C O(SU_q(2))$ is a finite dimensional subspace with $F_0 \su F$ and $x \in \T{Lip}_t(SU_q(2))$.
\end{lemma}
\begin{proof}
First consider $y,z \in \C O(SU_q(2))$ and let $F \su \C O(SU_q(2))$ be any finite dimensional subspace containing the vector $y \cd \nu(z)^* \in \C O(SU_q(2))$.  Using that the Haar state is a twisted trace (see \eqref{eq:modular}), we then have that
\begin{equation}\label{eq:innfinproj}
\inn{y, x \cd z} = \inn{y \cd \nu(z)^*, x} = \inn{y \cd \nu(z)^*, \Phi_F(x)} = \inn{y, \Phi_F(x) \cd z}
\end{equation}
for all $x \in C(SU_q(2))$.\\
Let us now focus on the case of the horizontal Dirac operator. The argument is similar for the vertical Dirac operator. Let $x \in \T{Lip}_t(SU_q(2))$ and $\xi, \eta \in \C O(SU_q(2))^{\op 2}$. By definition of $\pa_q^H(x)$ and by Lemma \ref{l:analytic} we have that
\[
\begin{split}
\binn{\xi, \pa_q^H(x) \eta} 
& = \binn{\xi, D^H_q \si_L(q^{1/2},x) \eta} - \binn{\xi, \si_L(q^{-1/2},x) \C D^H_q \eta} \\
& = \inn{\C D^H_q \xi, \Ga_q^{-1} x \Ga_q \eta} - \inn{\xi, \Ga_q x \Ga_q^{-1} \C D^H_q \eta} \\
& = \inn{\Ga_q^{-1}\C D^H_q \xi,  x \Ga_q \eta} - \inn{\Ga_q \xi,  x \Ga_q^{-1} \C D^H_q \eta} .
\end{split}
\]
Since the unbounded operators $\Ga_q$ and $\Ga_q^{-1} \C D^H_q$ both preserve the subspace $\C O(SU_q(2))^{\op 2}$ we obtain the result of the lemma by applying the observation from \eqref{eq:innfinproj} and running the last computation backwards.
\end{proof}

\begin{lemma}\label{lem:fuzzy-three-spheres-exhaust}
Let $n,i,j\in \nn_0$ satisfy that $i,j \leq n$. It holds that $u_{ij}^n\in \be_N^M(B_q^M)$ for all $N, M \in \nn_0$ with $N + M \geq n$ and $M\geq |2j-n|$. In particular, for any finite dimensional subspace $F \su \C O(SU_q(2))$ we may choose a $K_0 \in \nn_0$ such that $F \su \be_0^K(B_q^K)$ for all $K \geq K_0$.
\end{lemma}
\begin{proof}
Let $N,M \in \nn_0$ with $N + M \geq n$ and $M \geq |2j - n|$ be given. Put $m:=|2j-n|$ so that $M \geq m$. 
 Suppose first that $m = n - 2j$. We then have that $u^n_{ij} = u^{2j + m}_{ij}$ and it follows from Lemma \ref{lem:berezin-image} that $u^n_{ij} \in \be_N^M(A^{-m}_q) \su \be_N^M(B_q^M)$, since $j \leq 2j = n - m \leq N + M - m$. Suppose next that $m = 2j - n$. Put $k := j - m$ and notice that $k \geq 0$ since $k = n - j$. We then have that $u^n_{ij} = u^{2j - m}_{ij} = u^{2k + m}_{i,k + m}$ and it again follows from Lemma \ref{lem:berezin-image} that $u^n_{ij} \in \be_N^M(A^m_q)$ since $k = j - m \leq n - m \leq N + M - m$.
\end{proof}

We define the linear map $\de \colon \T{Lip}_q(SU_q(2)) \to \B M_2\big( L^\infty(SU_q(2)) \big)$ by putting
\[
\de(x) := u \cd  \pa_{q,q}(x) \cd u^* \q \T{for all } x \in \T{Lip}_q(SU_q(2)) .
\]
We notice that $\de$ does indeed take values in the von Neumann algebra $\B M_2\big( L^\infty(SU_q(2)) \big)$ since $ \pa_{q,q} = \pa_q^V + \pa_q^H$ takes values here by Corollary \ref{cor:values-in-L-infty}. We moreover remark that $\de$ extends the twisted $*$-derivation
\[
\de = \pma{ \de^3 & -\de^2 \\ - \de^1 & -\de^3 } \colon \C O(SU_q(2)) \to \B M_2\big( \C O(SU_q(2)) \big)
\]
as can be seen by an application of Proposition \ref{p:derV}.

\begin{lemma}\label{lem:adjoint-operator-without-saying-it}
For each $N,M\in \nn_0$ there exists a $K_0 \in \nn_0$ such that 
\begin{enumerate}
\item $\Phi_N^M \de(x) = \Phi_N^M  \de\big(  \Phi_0^K(x) \big)$ for all $x\in \T{Lip}_q(SU_q(2))$ and $K\geq K_0$.
\item $\Phi_N^M \pa_t^V(x) = \Phi_N^M  \pa_t^V\big(  \Phi_0^K(x) \big)$ for all $x \in \T{Lip}_t(SU_q(2))$ and $K \geq K_0$.
\end{enumerate}
\end{lemma}
\begin{proof}
We will only carry out the argumentation for $\de$ since the remaining case follows by a similar but slightly easier argument. \\
Consider the finite dimensional subspace $\be_N^M(B_q^M) \su \C O(SU_q(2))$ and denote its dimension by $d \in \nn$. Let us choose a subset $\{\ze_k \mid k = 1,2,\ldots,d\} \subseteq \be_N^M(B_q^M)$ so that $\{\Lambda(\ze_k) \mid k = 1,2,\ldots,d\}$ constitutes an orthonormal basis for the subspace $\La\big( \be_N^M(B_q^M) \big) \su L^2(SU_q(2))$. The map $\Phi_N^M$ is then given by the expression
\[
\Phi_N^M(T) = \sum_{k = 1}^d \ze_k \binn{\Lambda(\ze_k), T \La(1)} \q T \in \B B\big(L^2(SU_q(2))\big) .
\]
For every vector $\ze \in L^2(SU_q(2))$ we apply the notation
\[
\ze^0 := \pma{\ze \\ 0} \, \, \T{ and } \, \, \, \ze^1 := \pma{0 \\ \ze} \in L^2(SU_q(2))^{\op 2} ,
\]
and let $e_{ij} \in \B M_2(\cc)$ denote the standard matrix units for $i,j \in \{0,1\}$. 
The linear map $\Phi_N^M$ can then be described at the level of $2 \ti 2$-matrices by the expression
\[
\Phi_N^M(T) = \sum_{i,j = 0}^1 \sum_{k = 1}^d e_{ij} \cd \ze_k \binn{\Lambda(\ze^i_k), T \La(1)^j}
\q \T{for all } T \in \B M_2\big( \B B(L^2(SU_q(2))) \big) .
\]
In particular, we have that 
\[
\Phi_N^M( \de(x)) 
= \sum_{i,j = 0}^1 \sum_{k = 1}^d e_{ij} \cd \ze_k \binn{\Lambda(\ze^i_k), \de(x) \La(1)^j}  
= \sum_{i,j = 0}^1 \sum_{k = 1}^d e_{ij} \cd \ze_k \binn{u^* \Lambda(\ze^i_k),  \pa_{q,q}(x) u^* \La(1)^j}
\]
for all $x \in \T{Lip}_q(SU_q(2))$.  It therefore follows from Lemma \ref{l:projinside} 
that we may choose a finite dimensional subspace $F_0  \su \C O(SU_q(2))$ such that
\[
\Phi_N^M( \de(x))  = \Phi_N^M\big( \de(\Phi_F(x)) \big)
\]
for all finite dimensional subspaces $F \su \C O(SU_q(2))$ with $F_0 \su F$ and all $x \in \T{Lip}_q(SU_q(2))$. The result of the present lemma is now a consequence of Lemma \ref{lem:fuzzy-three-spheres-exhaust}.
%
\end{proof}

Let $N, M \in \nn_0$ be given. Recall that the Berezin transform $\beta_N^M \colon C(SU_q(2)) \to C(SU_q(2))$ is defined by slicing the coproduct $\Delta$ on the right tensor-leg with a state, while endomorphisms of the form $\delta_\eta$ with $\eta \in \C U_q (\mathfrak{su}(2))$ are defined by slicing the coproduct on the left tensor-leg. An application of the  coassociativity  of $\Delta$ therefore shows that $\beta_N^M(\delta_\eta(x))=\delta_\eta(\beta_N^M(x))$ for all $x\in \C O(SU_q(2))$ and $\eta \in \C U_q(\G{su}(2))$. In particular, we obtain that
\begin{equation}\label{eq:delber}
\be_N^M(\de(x)) = \de(\be_N^M(x)) \q \T{for all } x \in \C O(SU_q(2)) .
\end{equation}
Furthermore, for each element $x$ belonging to an algebraic spectral subspace $\C A^m_q$ for some $m \in \zz$, we get from Lemma \ref{l:twicommu} , Lemma \ref{l:imageberIA} and Lemma \ref{lem:berezin-image} that
\[
\be_N^M\big( \pa_t^V(x) ) = \pma{ [m/2]_t \be_N^M(x) & 0 \\ 0 & -[m/2]_t \be_N^M(x) }
= \pa_t^V( \be_N^M(x))  \q \T{for all } t \in (0,1] .
\]
We may thus conclude that
\[
\be_N^M(\pa_t^V(x)) = \pa_t^V( \be_N^M(x)) \q \T{for all } x \in \C O(SU_q(2)) .
\]
As a consequence of the analysis carried out above, we shall now see that these identities remain valid also at the level of the Lipschitz algebra. Recall, in this connection, that $\wit{\be}_N^M$ denotes the extension of $\be_N^M$ to $L^\infty(SU_q(2))$ introduced in \eqref{eq:extended-berezin}.

\begin{prop}\label{prop:berezin-and-delta}
For $M, N \in \nn_0$, the following identities are valid:
\begin{enumerate}
\item $\de\big( \be_N^M(x) \big) = \wit{\be}_N^M \de(x)$ for all $x \in \T{Lip}_q(SU_q(2))$;
\item $\pa_t^V\big( \be_N^M(x) \big) = \wit{\be}_N^M \pa_t^V(x)$ for all $x \in \T{Lip}_t(SU_q(2))$.
\end{enumerate}
\end{prop}
\begin{proof}
We focus on proving the identity regarding the  map $\de$. A similar argumentation applies to the twisted $*$-derivation $\pa_t^V$.  Let $x \in \T{Lip}_q(SU_q(2))$ be given. By Lemma \ref{lem:fuzzy-three-spheres-exhaust} and Lemma \ref{lem:adjoint-operator-without-saying-it}, we may choose a $K \in \nn_0$ such that $\be^M_N(B_q^M) \su \be^K_0(B_q^K)$ and such that 
\[
\Phi_N^M\de(x)=\Phi_N^M \de\big( \Phi^K_0(x) \big) .
\]
We now remark that $\Phi^K_0(x)\in \C O(SU_q(2))$ and that $\Phi_N^M  \Phi_0^K=\Phi_N^M$.

Applying these facts together with \eqref{eq:delber} and Lemma \ref{lem:extended-berezin-is-ucp} we obtain the desired result:
\begin{align*}
\wit{\beta}_N^M \de(x) &= \beta_N^M \Phi_N^M \de(x)=  \beta_N^M \Phi_N^M \de\big( \Phi^K_0(x) \big) 
=  \de\big( \beta_N^M  \Phi_N^M \Phi^K_0(x) \big) 
= \de\big( \beta_N^M (x) \big) . \qedhere
\end{align*}
\end{proof}


In the special situation where $t=q$ we have the identity $u\cd  \pa_{q,q} \cd u^*=\de$ and, as we saw above, the map $\de$ commutes with the Berezin transform. As the following result shows, this has the effect that the Berezin transform becomes a contraction for the associated Lip-norm $L_{q,q}^{\max}$. There is no reason to expect this to be the case when $t\neq q$, but Proposition \ref{p:berestigen} below  provides an estimate on how far away the Berezin transform is from being a contraction for the Lip-norm $L_{t,q}^{\max}$. 

\begin{cor}\label{cor:beta-Lip-contractive-when-t=q}
Let $N,M\in \nn_0$. The Berezin transform $\beta_N^M\colon \T{Lip}_q(SU_q(2)) \to \C O(SU_q(2))$ is a Lip-norm contraction for $L_{q,q}^{\max}$; i.e.~we have the inequality $L_{q,q}^{\max}\big(\beta_N^M(x))\leq L_{q,q}^{\max}(x)$ for all $x\in \T{Lip}_q(SU_q(2))$.
\end{cor}
\begin{proof}
 Let $x \in \T{Lip}_q(SU_q(2))$. By Corollary \ref{cor:values-in-L-infty}, we have that $\de(x) = u \cd  \pa_{q,q}(x) \cd u^* \in \B M_2\big(L^\infty(SU_q(2))\big)$ and by Lemma \ref{lem:extended-berezin-is-ucp} the map $\wit{\be}_N^M\colon L^\infty(SU_q(2)) \to C(SU_q(2))$ is ucp, and hence a complete contraction. Using this together with Proposition \ref{prop:berezin-and-delta}, we obtain the relevant inequality:
\[
L_{q,q}^{\max}\big(\beta_N^M(x)\big) = 
\big\| \de(\beta_N^M(x))\big\|=\big\| \wit{\beta}_N^M(\de(x))\big\|\leq \|\de(x)\|=L_{q,q}^{\max}(x). \qedhere
\]
\end{proof}

We now return to the general setting, and will prove that the Berezin transform suitably approximates the identity  operator on the Lip-unit ball.
Most of the results below will be needed in two versions: one version for all of quantum $SU(2)$ and one version which is fine tuned to hold on the spectral bands. 
For $K\in \nn_0$, we will also use $d_{t,q}$ and $d_{t,q}^{\max}$ to denote the metrics on the state space $\C S(B_q^K)$ arising from the restriction of the seminorms $L_{t,q}$ and $L_{t,q}^{\max}$ to the spectral band $B_q^K$ having domains $\C B_q^K$ and $B_q^K \cap \T{Lip}_t(SU_q(2))$, respectively. Hence, for $\mu,\nu \in \C S(C(SU_q(2)))$ we specify that
\begin{align*}
d_{t,q}^{\max}(\mu,\nu)&:= \sup\big\{ |\mu(x)-\nu(x)|\mid x\in C(SU_q(2)) \, , \, \, L_{t,q}^{\max}(x)\leq 1 \big\}\\
d_{t,q}^{\max}(\mu\vert_{B_q^K},\nu\vert_{B_q^K})&:= \sup\big\{ |\mu(x)-\nu(x)|\mid x\in B_q^K \, ,  \,\, L_{t,q}^{\max}(x)\leq 1 \big\},
\end{align*}
and similarly for $d_{t,q}$. Note that by Lemma \ref{lem:paV-on-spectral-subspaces} the domain of the restricted seminorm $L_{t,q}^{\max}|_{B_q^K}$ is independent of $t$, in that we have 
\begin{align}\label{eq:Lip-algebra-intersected-with-band}
\T{Lip}_t(SU_q(2)) \cap B_q^K= \T{Lip}^H (SU_q(2))\cap B_q^K,
\end{align}
where $\T{Lip}^H (SU_q(2))$ is the algebra of horizontally Lipschitz elements introduced in Definition \ref{def:Lipschitz-elements}.

\begin{prop}\label{prop:berezin-approximates-identity}
Let $N,M,K \in \nn_0$. It holds that
\begin{alignat*}{2}
\|\beta_N^M(x)-x \| & \leq d_{t,q}^{\max}(\chi_N^M,\epsilon ) \cd L_{t,q}^{\max}(x) \q && \mbox{for all } x\in C(SU_q(2)) \, \, \, \mbox{and} \\
\|\beta_N^M(x)-x \| & \leq d_{t,q}^{\max}(\chi_N^M\vert_{B_q^K},\epsilon\vert_{B_q^K} ) \cd L_{t,q}^{\max}(x) \q && \mbox{for all } x\in B_q^K .
\end{alignat*}
\end{prop}
\begin{proof}
 When proving the two statements we may focus on the case where $x$ belongs to $\T{Lip}_t(SU_q(2))$ or $\T{Lip}_t(SU_q(2)) \cap B_q^K$ since the seminorms on the right hand side otherwise take the value infinity. Notice first that for every $y \in C(SU_q(2))$ it holds that  
\begin{equation}\label{eq:normvector} 
\|y\|= \sup\big\{| \phi_{\xi,\eta}(y) | \mid \xi,\eta \in L^2(SU_q(2)) \, , \, \, \| \xi \|, \| \eta \| = 1 \big\} ,
\end{equation} 
where we recall that $\phi_{\xi,\eta}$ denotes the linear functional $x\mapsto \inn{\xi, \rho(x) \eta}$.  
Let now $x \in \T{Lip}_t(SU_q(2))$ be given and let $\xi,\eta \in L^2(SU_q(2))$ be unit vectors. Using the identity  \eqref{eq:normvector}, it suffices to show that
\[
\big| \phi_{\xi,\eta} (\beta_N^M(x)-x) \big| \leq d_{t,q}^{\max}(\chi_N^M,\epsilon ) \cd L_{t,q}^{\max}(x) .
\]
This inequality follows from Proposition \ref{prop:slice-and-Lip} and the Fubini theorem for slice maps \cite{Tom:AppFub} via the estimates:
\begin{align*}
\big| \phi_{\xi,\eta} (\beta_N^M(x)-x) \big| 
& = \big| (\chi_N^M - \epsilon) (\phi_{\xi,\eta} \ot 1)\Delta(x) \big| \\
& \leq d_{t,q}^{\max}(\chi_N^M,\epsilon)  \cd L_{t,q}^{\max}\big((\phi_{\xi,\eta} \ot 1)\Delta(x)\big)
\leq d_{t,q}^{\max}(\chi_N^M,\epsilon) \cd L_{t,q}^{\max}(x).
\end{align*}
This proves the first part of the statement. \\
If $x\in B_q^K$ then $\De(x)\in C(SU_q(2))\ot_{\min} B_q^K$ since each of the algebraic spectral subspaces is a left comodule for $\C O(SU_q(2))$. In the last computation in the proof above we therefore have $(\phi_{\xi,\eta} \ot 1)\Delta(x)\in B_q^K$, and hence the rest of the  argument carries over to prove the remaining inequality.
\end{proof}

As indicated above, we now wish to estimate how far the Berezin transform is from being a contraction for the Lip-norm $L_{t,q}^{\max}$. 
In general, there is no hope to commute the Berezin transform directly past the operation $u \cd \pa_{t,q} \cd u^*$ as we could when $t=q$. However, as Proposition \ref{prop:berezin-approximates-identity} shows, the Berezin transform approximates the identity operator well on the Lip-unit ball, and  this makes it possible to obtain strong estimates nevertheless. 
The analytic norm $\|\cdot\|_{t,q}$ introduced in Section \ref{ss:analytic} will be used as a tool in the analysis below, and we first provide an estimate on its values on the entries of the fundamental unitary $u \in \mathbb{M}_2\big(\C O(SU_q(2))\big)$.  We denote these entries by $u_{ij}$, $i,j = 0,1$.

\begin{lemma}\label{l:boumatuni}
For every $i,j \in \{0,1\}$, it holds that
\begin{alignat}{2}
\| u_{ij} \|_{t,q} &= \hspace{0.2cm}\| u_{ij}^* \|_{t,q}  &&\leq q^{-1/2} + t^{-1/2} \q \mbox{and} \notag \\
L_{t,q}^{\T{max}}(u_{ij}) &= L_{t,q}^{\T{max}}(u^*_{ij}) && \leq [1/2]_t + q^{-1/2} .\notag
\end{alignat}
\end{lemma}
\begin{proof}
Let $i,j \in \{0,1\}$ be given. Using that $\sigma_L(s^{1/2}, u_{ij})=s^{j-1/2} u_{ij}$ for all $s \in (0,\infty)$, the result of the lemma follows from the estimate
\[
\| u_{ij}^* \|_{t,q} = \| u_{ij} \|_{t,q} 
\leq \max\big\{ q^{j-1/2} + t^{j-1/2} , q^{-j+1/2} + t^{-j + 1/2} \big\}
\leq q^{-1/2} + t^{-1/2}
\]
together with the estimates
\[
\begin{split}
L_{t,q}^{\T{max}}(u_{ij}^*) & = L_{t,q}^{\T{max}}(u_{ij})
= \| \pa_t^V(u_{ij}) + \pa_q^H(u_{ij}) \| \\
& \leq  \| \pa^3_t(u_{ij}) \|  + \T{max}\big\{ q^{1/2}\| \pa_e(u_{ij}) \| , q^{-1/2} \| \pa_f(u_{ij}) \| \big\}
\leq [1/2]_t + q^{-1/2}, 
\end{split}
\]
where the last inequality follows from \eqref{eq:generators-on-fundamental-unitary}.
\end{proof}


In the following lemma we recall that $\Pi_0^L \colon C(SU_q(2)) \to C(S_q^2)$ denotes the spectral projection onto the Podle\'s sphere; see \eqref{eq:spec-proj-norm-def-on-suq2}. 

\begin{lemma}\label{l:vertradi}
Let $x \in \ker(\Pi_0^L)$. We have the estimate
\[
\| x \| \leq \frac{\pi \cd (t^{1/2} + t^{-1/2})}{\sqrt{3}} \cd L_{t,q}^{\T{max}}(x) .
\]
\end{lemma}
\begin{proof}
Without loss of generality, we may assume that $x \in \ker(\Pi_0^L) \cap \T{Lip}_t(SU_q(2))$ since the right hand side of the desired inequality is equal to infinity otherwise. By Proposition \ref{prop:antiderivative} we then get that $x = \int_t^V \pa^V_t(x)$. It thus follows from Lemma \ref{l:quaintest} and the definition of $\int_t^V$ from \eqref{eq:quaintdef} that
\[
\| x \| \leq  \| \smallint_t^V \| \cd \| \pa^V_t(x) \| \leq \frac{\pi \cd (t^{1/2} + t^{-1/2})}{\sqrt{3}} \cd L_{t,q}^{\T{max}}(x) . \qedhere
\]
\end{proof}

With the above auxiliary results at our disposal, we may now start estimating the error arising when commuting the Berezin transform  past conjugation with the fundamental unitary. This will be relevant when estimating the $L_{t,q}^{\max}$-operator norm of the Berezin transform. An important point of the following lemma is that we are able to control the error term by means of a continuous function in $t$ and $q$. For the statement, we recall that $\ga := \sma{1 & 0 \\ 0 & -1}$.

\begin{lemma}\label{l:estiberuni}
Let $K \in \nn_0$. There exists a continuous, positive function $g_K \colon (0,1] \ti (0,1] \to (0,\infty)$ such that 
\[
\big\| \be^M_N( u \ga x u^* ) - u \be^M_N(\ga x) u^* \big\| \leq 
g_K(t,q) \cd d_{t,q}^{\T{max}}\big( \chi_N^M\vert_{B_q^K},\epsilon\vert_{B_q^K}\big) \cd L_{t,q}^{\T{max}}(x)
\]
for all $N,M \in \nn_0$, all $t,q \in (0,1]$ and all $x \in B_q^K \cap \ker(\Pi_0^L)$ .
\end{lemma}
\begin{proof}
Without loss of generality we may focus on the case where $x \in \Lip_t(SU_q(2)) \cap B_q^K \cap \ker( \Pi_0^L)$, since the right hand side is otherwise equal to infinity. An application of Proposition \ref{prop:berezin-approximates-identity} shows that the following inequalities hold for all $N,M \in \nn_0$ and all $t,q \in (0,1]$; notice in this respect that $u_{ij} x u^*_{kj} \in B_q^K$ for all $i,j,k \in \{0,1\}$:
\[
\begin{split}
& \big \| \be^M_N( u \ga x u^* ) - u \be^M_N(\ga x) u^* \big\|
\leq \sum_{i,j,k = 0}^1 \big\| \be^M_N( u_{ij} x u^*_{kj} ) - u_{ij} \be^M_N(x) u^*_{kj} \big\| \\
& \q \leq \sum_{i,j,k = 0}^1 \big\| \be^M_N(u_{ij} x u^*_{kj}) - u_{ij} x u^*_{kj} \big\|
+ \sum_{i,j,k = 0}^1 \big\| u_{ij} (x - \be^M_N(x) ) u^*_{kj} \big\| \\
& \q \leq d_{t,q}^{\T{max}}\big( \chi_N^M\vert_{B_q^K},\epsilon\vert_{B_q^K}\big)
\cd \sum_{i,j,k = 0}^1 \left( L_{t,q}^{\T{max}}(u_{ij} x u^*_{kj}) + L_{t,q}^{\T{max}}(x) \right) .
\end{split}
\]
Applying Lemma \ref{l:leftrightbound} and Lemma \ref{l:boumatuni} we estimate that
\[
\begin{split}
L_{t,q}^{\T{max}}(u_{ij} x u^*_{kj}) 
& \leq L_{t,q}^{\T{max}}(u_{ij}) \cd \| x \|_{t,q} \cd \| u_{kj} \|_{t,q}
+ \| u_{ij} \|_{t,q} \cd L_{t,q}^{\T{max}}(x) \cd \| u_{kj} \|_{t,q} \\
& \ \ \ + \| u_{ij} \|_{t,q} \cd \| x \|_{t,q} \cd L_{t,q}^{\T{max}}(u_{kj}) \\
& \leq 2 ( [1/2]_t + q^{-1/2} ) (q^{-1/2} + t^{-1/2}) \cd \| x \|_{t,q}
+ (q^{-1/2} + t^{-1/2})^2 \cd L_{t,q}^{\T{max}}(x) .
\end{split}
\]
The result of the lemma is now a consequence of  Lemma \ref{l:anaband} and  Lemma \ref{l:vertradi}: indeed, we have that
\[
\begin{split}
\| x \|_{t,q} 
& \leq \sum_{m = -K}^K (t^{m/2} + q^{m/2}) \cd \| x \| \\
& \leq \sum_{m = -K}^K (t^{m/2} + q^{m/2}) \cd \frac{\pi \cd (t^{1/2} + t^{-1/2})}{\sqrt{3}} \cd L_{t,q}^{\max}(x) . \qedhere
\end{split}
\]
\end{proof}



\begin{prop}\label{p:berestigen}
Let $K \in \nn_0$. There exists a continuous positive function $h_K \colon (0,1] \ti (0,1] \to (0,\infty)$ satisfying that
\begin{enumerate}
\item $h_K(q,q) = 0$ for all $q \in (0,1]$ and;
\item the following estimate holds
\[
L_{t,q}^{\max}\big( \be^M_N(x) \big) \leq
\big( 1 + h_K(t,q) \cd d_{t,q}^{\max}(\chi_N^M\vert_{B_q^K},\epsilon\vert_{B_q^K}) \big) \cd L_{t,q}^{\max}(x) 
\]
for all $N,M \in \nn_0$, all $t,q \in (0,1]$ and all $x \in \T{Lip}_t(SU_q(2)) \cap B_q^K$.
\end{enumerate}
\end{prop}
\begin{proof}
We start out by choosing the continuous positive function $g_K \colon (0,1] \ti (0,1] \to (0,\infty)$ according to Lemma \ref{l:estiberuni}. We then define the continuous positive function $h_K \colon (0,1] \ti (0,1] \to (0,\infty)$ by putting
\[
h_K(t,q) := 2 \cd \sum_{m = 1}^K \big| [m/2]_t - [m/2]_q \big| \cd g_K(t,q) 
\]
and note that $h_K(q,q) = 0$ for all $q \in (0,1]$ as desired. \\

Let now $N,M \in \nn_0$ and $t,q \in (0,1]$ be given. Let moreover $x \in \T{Lip}_t(SU_q(2)) \cap B_q^K$, and remark that by \eqref{eq:Lip-algebra-intersected-with-band}, $x \in \T{Lip}_q(SU_q(2)) \cap B_q^K$ as well. 
We define the element $y = \pa^3_t(x) - \pa^3_q(x)$ and notice that $y \in B_q^K \cap \ker(\Pi_0^L)$ by Lemma \ref{lem:paV-on-spectral-subspaces}. We moreover emphasise the identities 
\[
\pa_{t,q}(x) - \pa_{q,q}(x) = \pa_t^V(x) - \pa_q^V(x) = \ga y, \q \T{where } \ga = \sma{1 & 0 \\ 0 & -1} .
\]

Using Proposition \ref{p:derV} and Proposition \ref{prop:berezin-and-delta} we now compute as follows:
\[
\begin{split}
& u \cd \pa_{t,q}\big( \be^M_N(x)\big) \cd u^*
= u \cd (\pa_t^V - \pa_q^V)\big( \be^M_N(x)\big) \cd u^* + \de\big(\be^M_N(x)\big) \\
& \q = u \cd \be^M_N\big( (\pa_t^V - \pa_q^V)(x) \big) \cd u^* + \wit{\be}^M_N\big( \de(x) \big) \\
& \q = u \cd \be^M_N( \ga y ) \cd u^* - \be^M_N( u \ga y u^* ) 
+ \be^M_N\big( u \cd (\pa_{t,q} - \pa_{q,q})(x) \cd u^* \big) + \wit{\be}^M_N\big( \de(x)\big) \\
& \q = u \cd \be^M_N( \ga y ) \cd u^* - \be^M_N( u \ga y u^* ) + \wit{\be}^M_N\big( u \cd \pa_{t,q}(x) \cd u^*\big) .
\end{split}
\]
Combining the above computation with Lemma \ref{l:estiberuni},  recalling that $\wit{\be}_N^M$ is a complete contraction by Lemma \ref{lem:extended-berezin-is-ucp}, we obtain that
\[
\begin{split}
L_{t,q}^{\max}\big( \be^M_N(x) \big) 
& \leq \big\| u \cd \be^M_N( \ga y ) \cd u^* - \be^M_N( u \ga y u^* ) \big\|
+ L_{t,q}^{\max}(x) \\
& \leq g_K(t,q) \cd  d_{t,q}^{\T{max}}\big( \chi_N^M\vert_{B_q^K},\epsilon\vert_{B_q^K}\big) \cd L_{t,q}^{\max}(y) 
+ L_{t,q}^{\max}(x) .
\end{split}
\]
The result of the present proposition now follows since $y=\sum_{m=-K}^K([m/2]_t-[m/2]_q) \cd \Pi_m^L(x)$ so that 
\[
\begin{split}
L_{t,q}^{\max}(y) 
& \leq \sum_{m = -K}^K \big| [m/2]_t - [m/2]_q \big| \cd L_{t,q}^{\max}( \Pi_m^L(x)) \\
& \leq 2 \sum_{m = 1}^K \big| [m/2]_t - [m/2]_q \big| \cd L_{t,q}^{\max}( x ) ,
\end{split}
\]
where the last inequality follows from  Corollary \ref{c:diracinvariant}. 
\end{proof}


\subsection{Approximation in the quantum Gromov-Hausdorff distance}
As a result of the analysis carried out in this section, we shall see that the quantum Gromov-Hausdorff distance between the two compact quantum metric spaces $\big( C(SU_q(2)), L_{t,q}^{\max}\big)$ and $\big( C(SU_q(2)), L_{t,q}\big)$ is in fact equal to zero; cf.~Corollary \ref{cor:dist-zero} below.  When considering the quantum Gromov-Hausdorff convergence questions in Section \ref{sec:continuity-results}, this result will allow us to work exclusively at the algebraic level, which will simplify matters significantly. We start out with a technical estimate, from which a number of our main results will follow.

\begin{prop}\label{p:fuzztosuq2}
  Let $\de \in (0,1)$. For every $\ep > 0$ there exists a $K_0 \in \nn_0$ and a constant $C \geq 0$ such that  
\[
 \T{dist}_{\T{Q}}\big( (\be_N^M(B_q^{K_0}), L_{t,q}); (C(SU_q(2)), L_{t,q}^{\max}) \big)
\leq d_{t,q}^{\T{max}}\big(\chi_N^M\vert_{B_q^{K_0}},\epsilon\vert_{B_q^{K_0}}\big) \cd C + \ep
  \]
  for all $N,M \in \nn_0$ and all $t,q \in [\de,1]$. Moreover, if $X\subseteq C(SU_q(2))$ is a sub-operator system such that $\T{Dom}(L_{t,q}^{\max})\cap X$ is norm-dense in $X$ and $\beta_N^M(B_q^{K_0})\subseteq X$, then it holds that
  \[
 \T{dist}_{\T{Q}}\big( (X, L_{t,q}^{\T{max}}); (C(SU_q(2)), L_{t,q}^{\max}) \big)
\leq d_{t,q}^{\T{max}}\big(\chi_N^M\vert_{B_q^{K_0}},\epsilon\vert_{B_q^{K_0}}\big) \cd C + \ep
\]
for all $t,q \in [\de,1]$.
\end{prop}
\begin{proof}
  Let $\ep > 0$ be given and choose $K_0 \in \nn_0$ such that $\ep(\de,K_0) \leq \ep$; see \eqref{eq:epnull} for the definition of $\ep(\de,K)$ for $K \in \nn_0$. For every $N,M \in \nn_0$ we remark that the seminorms $L_{t,q}$ and $L_{t,q}^{\max}$ agree on the sub-operator system $\be_N^M(B_q^{K_0}) \su C(SU_q(2))$. This is a consequence of Lemma \ref{l:imageberIA} and Lemma \ref{lem:berezin-image}.
 By Proposition \ref{p:berestigen}, we may choose a constant $C_0 \geq 0$ such that
    \begin{equation}\label{eq:berestidel}
L_{t,q}^{\max}(\be_N^M(x)) \leq \big( 1 + C_0 \cd d_{t,q}^{\max}(\chi_N^M\vert_{B_q^{K_0}},\epsilon\vert_{B_q^{K_0}}) \big) \cd L_{t,q}^{\max}(x)
    \end{equation}
for all $N,M \in \nn_0$, all $t,q \in [\de,1]$ and all $x \in \T{Lip}_t(SU_q(2)) \cap B_q^{K_0}$.
   Combining Proposition \ref{prop:diameter-estimate} with Remark \ref{rem:bounded-diamter} we may choose the constant $C \geq 0$ such that
  \[
C_0 \cd \T{diam}\big( C(SU_q(2)),L_{t,q}^{\max} \big) + 1 \leq  C \q \T{for all } t,q \in [\de,1].
\]
Let now $N,M \in \nn_0$ and $t,q \in [\de,1]$ be given. Define the unital map $\Phi := \be_N^M \ci E_{K_0}^L \colon C(SU_q(2)) \to \be_N^M(B_q^{K_0})$ and note that $\Phi$ is positive since $E_{K_0}^L=\mathbb{M}(\gamma_{K_0})$ is a unital contraction (Lemma \ref{l:compcont}) and $\be_N^M$ is positive by construction. We then obtain from Proposition \ref{p:bandapprox-without-saying-it}, Proposition \ref{prop:berezin-approximates-identity} and Lemma \ref{l:emlip} that
\[
\begin{split}
  \| x - \Phi(x) \| & \leq \| x - E_{K_0}^L(x) \| + \| E_{K_0}^L(x) - \be_N^M(E_{K_0}^L(x)) \| \\
& \leq \ep(\de,K_0) \cd L_{t,q}^{\max}(x) + d_{t,q}^{\max}\big( \chi_N^M\vert_{B_q^{K_0}}, \epsilon\vert_{B_q^{K_0}}\big) \cd L_{t,q}^{\max}(E_{K_0}^L(x)) \\
&  \leq \big( \ep + d_{t,q}^{\max}( \chi_N^M\vert_{B_q^{K_0}}, \epsilon\vert_{B_q^{K_0}}) \big) \cd L_{t,q}^{\max}(x)
  \end{split}
  \]
  for all $x \in \T{Lip}_t(SU_q(2))$.   Another application of Lemma \ref{l:emlip} together with \eqref{eq:berestidel} moreover shows that
  \[
 L_{t,q}^{\max}(\Phi(x)) \leq \big( 1 + C_0 \cd d_{t,q}^{\max}(\chi_N^M\vert_{B_q^{K_0}},\epsilon\vert_{B_q^{K_0}}) \big) \cd L_{t,q}^{\max}(x)
 \]
 for all $x \in \T{Lip}_t(SU_q(2))$.

 Using Corollary \ref{cor:subspacegen} we then see that
 \[
 \begin{split}
& \T{dist}_{\T{Q}}\big( (\be_N^M(B_q^{K_0}), L_{t,q}); (C(SU_q(2)), L_{t,q}^{\max}) \big) \\
& \q \leq d_{t,q}^{\T{max}}(\chi_N^M\vert_{B_q^{K_0}},\epsilon\vert_{B_q^{K_0}}) \cd ( C_0 \cd \T{diam}(C(SU_q(2)), L_{t,q}^{\max}) + 1 ) + \ep \\
& \q \leq d_{t,q}^{\T{max}}(\chi_N^M\vert_{B_q^{K_0}},\epsilon\vert_{B_q^{K_0}}) \cd C + \ep .
\end{split}
\]
This proves the first part of the present proposition. The second part of our proposition now follows from Remark \ref{rem:intermediate}.
\end{proof}

\begin{cor}\label{cor:dist-zero}
  Let $t,q \in (0,1]$. The metrics $d_{t,q}$ and $d_{t,q}^{\max}$ agree on the state space $\C S\big(C(SU_q(2))\big)$. In particular, it holds that
    \[
\T{dist}_{\T{Q}}\big( (C(SU_q(2)), L_{t,q}); (C(SU_q(2)), L_{t,q}^{\max}) \big) = 0 .
    \]
\end{cor}
\begin{proof}
   Let $\mu,\nu \in \C S\big( C(SU_q(2)) \big)$. We trivially have that $d_{t,q}(\mu,\nu)\leq d_{t,q}^{\max}(\mu,\nu)$ so we need to prove the opposite inequality.\\
   For every $K,N,M \in \nn_0$ we recall that the seminorms $L_{t,q}$ and $L_{t,q}^{\max}$ agree on the sub-operator system $\be_N^M(B_q^K) \su C(SU_q(2))$ (see Lemma \ref{l:imageberIA} and Lemma \ref{lem:berezin-image}). We therefore obtain that the two metrics $d_{t,q}$ and $d_{t,q}^{\max}$ agree on the state space $\C S\big( \be_N^M(B_q^K)\big)$.\\
Let $\ep > 0$ be given. Combining the proof of Proposition \ref{p:fuzztosuq2} with Corollary \ref{cor:subgenmet} we may choose a $K_0 \in \nn_0$ and a constant $C \geq 0$ such that 
    \[
    d_{t,q}^{\max}(\mu,\nu)
    \leq d_{t,q}^{\max}\big(\chi_N^M\vert_{B_q^{K_0}}, \epsilon\vert_{B_q^{K_0}}\big) \cd C + \ep/2 + d_{t,q}( \mu, \nu)
    \]
    for all $N,M \in \nn_0$. Next, by Theorem \ref{thm:quantum-su2-as-cqms}, $d_{t,q}^{\max}$ metrises the weak$^*$ topology on $\C S\big(C(SU_q(2))\big)$ and by Lemma \ref{lem:convergence-to-counit} it therefore follows that $\lim_{N,M\to \infty} d_{t,q}^{\max}(\chi_N^M,\epsilon) = 0$. We may thus choose $N,M \in \nn_0$ such that
\[
C \cd d_{t,q}^{\max}\big( \chi_N^M\vert_{B_q^{K_0}}, \epsilon\vert_{B_q^{K_0}} \big) \leq
C \cd d_{t,q}^{\max}( \chi_N^M, \epsilon ) \leq \ep/2 .
\]
Combining these two estimates we obtain that
\[
d_{t,q}^{\max}(\mu,\nu) \leq \ep + d_{t,q}(\mu,\nu) .
\]
Since $\ep > 0$ was arbitrary we have proved that $d_{t,q}^{\max}(\mu,\nu) = d_{t,q}(\mu,\nu)$.\\
The fact that the quantum Gromov-Hausdorff distance between $(C(SU_q(2)),L_{t,q}^{\max})$ and $(C(SU_q(2)),L_{t,q})$ is equal to zero now follows from \cite[Corollary 6.4]{Rie:GHD} (see also the discussion near  Theorem \ref{thm:quantumdist}).
  \end{proof}

We also record a corollary which is an analogue to Corollary \ref{cor:dist-zero} for the spectral bands. Since the proof is similar but easier than the proof of Corollary \ref{cor:dist-zero} we are leaving it out. 

\begin{cor}\label{cor:metrics-agree-on-state-space-of-bands}
  Let $K \in \nn_0$ and let $t,q \in (0,1]$. The metrics $d_{t,q}$ and $d_{t,q}^{\max}$ agree on the state space $\C S(B_q^K)$. In particular, it holds that
    \[
    \T{dist}_{\T Q}\big( (B_q^K,L_{t,q}^{\max}\vert_{B_q^K}); (B_q^K,L_{t,q}\vert_{B_q^K}) \big) = 0 .
    \]
\end{cor}

Lastly, we single out the following consequence of   Proposition \ref{p:fuzztosuq2}, which  shows that our fuzzy approximations do indeed approximate quantum $SU(2)$ in the quantum Gromov-Hausdorff distance. 

\begin{cor}\label{cor:fuzzy-to-quantum-SU2}
Let $t,q\in (0,1]$. It holds that
\[
\lim_{N,K\to \infty }\T{dist}_{\T{Q}}\big( (\T{Fuzz}_{N}(B_q^K), L_{t,q}); (C(SU_q(2)), L_{t,q}^{\max}) \big)=0 .
\]
\end{cor}
\begin{proof}
    Let $\ep > 0$ be given. By Proposition \ref{p:fuzztosuq2}, there exist a $K_0 \in \nn_0$ and a constant $C \geq 0$ such that
  \[
\T{dist}_{\T{Q}}\big( (\be_N^M(B_q^{K_0}), L_{t,q}); (C(SU_q(2)), L_{t,q}^{\max}) \big)
\leq C\cdot d_{t,q}^{\T{max}}\big(\chi_N^M\vert_{B_q^{K_0}},\epsilon\vert_{B_q^{K_0}}\big)  + \ep/2
  \]
  for all $N,M\in \nn_0$.    By Theorem \ref{thm:quantum-su2-as-cqms} and Lemma \ref{lem:convergence-to-counit} we may choose $N_0,M_0 \in \nn_0$ with $M_0 \geq K_0$ such that 
  \[
C \cd d_{t,q}^{\max}(\chi_N^M\vert_{B_q^{K_0}} ,\epsilon\vert_{B_q^{K_0}}) \leq C \cd d_{t,q}^{\max}(\chi_N^M,\epsilon) < \ep/2 \ \ \T{for all }  \  N \geq N_0 {\T{ and }} M\geq M_0 .
\]
For $N\geq N_0+ M_0$ and $K\geq K_0$ we obtain from Proposition \ref{p:fuzzber} that
\[
\beta_{N_0}^{M_0}(B_q^{K_0})\subseteq \T{Fuzz}_{N_0+M_0}(B_q^{K_0})\subseteq \T{Fuzz}_N(B_q^{K_0})\subseteq \T{Fuzz}_N(B_q^K),
\]
and the last part of Proposition \ref{p:fuzztosuq2} therefore shows that
\[
\T{dist}_{\T{Q}}\big( (  \T{Fuzz}_N(B_q^K)   , L_{t,q}); (C(SU_q(2)), L_{t,q}^{\max}) \big)<\ep,
\]
for all $N\geq N_0+M_0$ and all $K\geq K_0$.
\end{proof}

\begin{remark}
The case where $t = q = 1$ is of particular interest since $C(SU_1(2))=C(SU(2))$ and the Lip-norm $L_{1,1}^{\max}$ computes the Lipschitz constant arising from twice the round metric $d_{S^3}$ on $SU(2)\cong S^3\subseteq \rr^4$; see Section \ref{sec:comparison-with-classical} for details. Corollary \ref{cor:fuzzy-to-quantum-SU2} therefore provides a finite dimensional approximation of $C(S^3)$ by subspaces invariant under the $SU(2)$-action (see Proposition \ref{prop:fuzzy-coinvariant}). This yields an $S^3$-analogue of Rieffel's original result \cite[Theorem 3.2]{Rie:MSG} for the 2-sphere.
\end{remark}

\section{Continuity results}\label{sec:continuity-results}

In this section we embark  on our final goal of the paper, which is to prove that the family of compact quantum metric spaces $\big(C(SU_q(2)),L_{t,q}\big)_{t,q\in (0,1]}$ varies continuously in the quantum Gromov-Hausdorff distance; see Theorem \ref{introthm:qgh-continuity}. The result in Corollary \ref{cor:dist-zero} shows that we may choose to work exclusively with the Lip-norm $L_{t,q}$, meaning that the domain equals the coordinate algebra $\C O(SU_q(2))$. Indeed, the corresponding continuity result for the Lip-norm $L_{t,q}^{\max}$ with domain equal to the Lipschitz algebra $\T{Lip}_t(SU_q(2))$ follows automatically. In effect, this allows us to circumvent a lot of analysis and work at a purely (Hopf-)algebraic level. We begin by providing a rough outline of the mains steps in the proof of continuity at a point $(t_0,q_0)\in (0,1] \ti (0,1]$: 
\begin{enumerate}
\item[1.]  We fine tune the result in  Corollary \ref{cor:fuzzy-to-quantum-SU2} by showing that locally around $(t_0,q_0)$ the fuzzy approximations approach quantum $SU(2)$ in a uniform manner. 
\item[2.] Utilising the finite dimensionality of the fuzzy approximation we show that these vary continuously. 
\item[3.] Piecing together these approximation results, we arrive at the main continuity statement in Theorem \ref{thm:continuity-of-quantum-su2} below.
\end{enumerate}
\subsection{Continuity of the fuzzy approximations}
We begin by addressing point 2.~in the above list.

\begin{prop}\label{p:fincont}
Let $K,N \in \nn_0$. The $2$-parameter family of compact quantum metric spaces $\big(\T{Fuzz}_N(B_q^K), L_{t,q} \big)_{t,q\in (0,1]}$ varies continuously in the quantum Gromov-Hausdorff distance. 
\end{prop}
\begin{proof}
Fix a $\de \in (0,1)$. We aim to apply  \cite[Theorem 11.2]{Rie:GHD}, and must therefore provide a fixed  finite dimensional real vector space $V$ with a distinguished vector $e$, a continuous family $\big(\|\cdot\|_{t,q}\big)_{t,q\in [\de, 1]}$ of norms and a continuous family $(M_{t,q})_{t,q\in [\delta, 1]}$ of seminorms such that $(V, e, \|\cdot\|_{t,q}, M_{t,q})$ is an order unit compact quantum metric space isomorphic to 
\[
\big(\T{Fuzz}_N(B_q^K)_{\T{sa}}, 1, \|\cdot\|, L_{t,q}\big) \q \T{for all } t,q\in [\de, 1] .
\]
We are going to apply the unital continuous field of $C^*$-algebras over $[\de,1]$ with total space $C(SU_\bullet(2))$ and with fibre $C(SU_q(2))$ for every $q \in [\de,1]$ which was introduced in Section \ref{subsec:cont-field}. For each $q \in [\de,1]$, we recall that $\T{ev}_q \colon C(SU_\bullet(2)) \to C(SU_q(2))$ denotes the unital $*$-homomorphism which evaluates at the point $q$. 

We moreover recall that  $\C O(SU_{\bullet}(2)) \su C(SU_{\bullet}(2))$ denotes the smallest unital $*$-subalgebra containing $C([\de,1])$ and the generators $a_\bullet$ and $b_\bullet$. Notice also that $\C O(SU_{\bullet}(2))$ is a free $C([\de,1])$-module with basis given by the elements
\[
\xi_\bullet^{klm} := \begin{cases} a_\bullet^k b_\bullet^l (b_\bullet^*)^m & k \geq 0 \\
b_\bullet^l (b_\bullet^*)^m (a_\bullet^*)^{-k} & k < 0 
\end{cases} .
\]
for $k \in \zz$ and $l,m \in \nn_0$. 

For each $q \in [\de,1]$ we then obtain a linear basis for the coordinate algebra $\C O(SU_q(2))$ by applying the evaluation map to the linearly independent subset $\big\{ \xi_\bullet^{klm} \mid (k,l,m) \in \zz \ti \nn_0 \ti \nn_0 \big\} \su C(SU_\bullet(2))$. In particular, we obtain that
\[
\T{ev}_q \colon \T{span}_{\cc}\big\{ \xi_\bullet^{klm} \mid (k,l,m) \in \zz \ti \nn_0 \ti \nn_0 \big\} \longrightarrow \C O(SU_q(2))
\]
is an isomorphism of vector spaces over $\cc$. 
By an application of Corollary \ref{c:fuzzbasis}, we may choose a finite subset $J \su \zz \ti \nn_0 \ti \nn_0$ satisfying that
\[
\T{ev}_q\big( \T{span}_{\cc}\big\{ \xi_\bullet^{klm} \mid (k,l,m) \in J \big\} \big) = \T{Fuzz}_N(B_q^K)
\]
for all $q \in [\de,1]$. We apply the notation
\[
W := \T{span}_{\cc}\big\{ \xi_\bullet^{klm} \mid (k,l,m) \in J \big\} \su C(SU_\bullet(2)) 
\]
and record that $W$ becomes a finite dimensional operator system (indeed, it holds that $\xi^* \in W$ whenever $\xi \in W$ and clearly $1 \in W$ as well). We put $V := W_{\T{sa}}$ and record that the isomorphism
\[
\T{ev}_q \colon W \to \T{Fuzz}_N(B_q^K)
\]
induces an isomorphism of real vector spaces $\T{ev}_q \colon V \to \T{Fuzz}_N(B_q^K)_{\T{sa}}$ for all $q \in [\de,1]$.
For each $t,q \in [\de,1]$ we equip $V$ with the unique order unit space structure such that $\T{ev}_q \colon V \to \T{Fuzz}_N(B_q^K)_{\T{sa}}$ becomes an isomorphism of order unit spaces. We emphasise that this order unit space structure does not depend on the parameter $t \in [\de,1]$. Moreover, we may introduce the seminorm
\[
M_{t,q} \colon V \to [0,\infty) \q M_{t,q}(x_\bullet) := L_{t,q}\big( \T{ev}_q(x_\bullet)\big) .
\]
In this fashion, we get that $(V,M_{t,q})$ becomes an order unit compact quantum metric space which is isometrically isomorphic to the order unit compact quantum metric space $(\T{Fuzz}_N(B_q^K)_{\T{sa}}, L_{t,q})$. 
We remark that the different order unit space structures on $V$ yields a family of norms $\big( \| \cd \|_{t,q}\big)_{t,q \in [\de,1]}$ on $V$. This family becomes continuous since we are dealing with a continuous field of $C^*$-algebras with total space $C(SU_\bullet(2))$. Indeed, for each $t,q \in [\de,1]$ we record that $\| x_\bullet \|_{t,q} = \| \T{ev}_q(x_\bullet) \|$. We therefore only need to show that the family of seminorms $( M_{t,q} )_{t,q \in [\de,1]}$ is continuous as well.\\
It then follows from the discussion in Section \ref{subsec:cont-field} 
that we have two $C([\de,1])$-linear maps
\[
\pa^1_\bullet \T{ and } \pa^2_\bullet \colon \C O(SU_{\bullet}(2)) \to \C O(SU_{\bullet}(2)) 
\]
satisfying that $\T{ev}_q \ci \pa^1_\bullet = \pa^1 \ci \T{ev}_q$ and $\T{ev}_q \ci \pa^2_\bullet = \pa^2 \ci \T{ev}_q$. Moreover, for each $t \in (0,1]$ we may define the $C([\de,1])$-linear map
\[
\pa^3_{t,\bullet} \colon \C O(SU_{\bullet}(2)) \to \C O(SU_{\bullet}(2)) \q
\pa^3_{t,\bullet}(\xi^{klm}_\bullet) := \big[(k + l - m)/2\big]_t \cd \xi^{klm}_\bullet .
\]
By construction we obtain that $\T{ev}_q \ci \pa^3_{t,\bullet} = \pa^3_t \ci \T{ev}_q$. Moreover, for each $x_\bullet \in \C O(SU_\bullet(2))$, we note that the map $(0,1] \to C(SU_\bullet(2))$ defined by $t \mapsto \pa^3_{t,\bullet}(x_\bullet)$ is continuous with respect to the $C^*$-norm on $C(SU_\bullet(2))$. For each $t \in (0,1]$, we may thus consider the $C([\de,1])$-linear map
\[
\pa_{t,\bullet} \colon \C O(SU_\bullet(2)) \to \B M_2\big( C(SU_\bullet(2)) \big) \q
\pa_{t,\bullet} := \pma{ \pa^3_{t,\bullet} & -\pa^2_{\bullet} \\ - \pa^1_{\bullet} & - \pa^3_{t,\bullet}} .
\]
We notice that $\B M_2\big( C(SU_\bullet(2)) \big)$ is again the total space of a continuous field of $C^*$-algebras over $[\de,1]$, this time with fibres $\B M_2\big( C(SU_q(2)) \big)$ for $q \in [\de,1]$. For each $t,q \in [\de,1]$ we moreover have that
\[
M_{t,q}(x_\bullet) = \big\| \pa_{t,q}( \T{ev}_q(x_\bullet)) \big\|
= \big\| \T{ev}_q\big( \pa_{t,\bullet}(x_\bullet) \big) \big\|
\]
for every $x_\bullet \in V$. From these observations we obtain that $(M_{t,q})_{t,q\in [\delta, 1]}$ is a continuous family of seminorms on $V$.\\
The assumptions in \cite[Theorem 11.2]{Rie:GHD} are thereby fulfilled, and since $\de \in (0,1)$ was arbitrary this implies the claimed continuity result.
\end{proof}

In the following subsection we address point 1.~in the road map provided in the beginning of this section. This is the main technical step in the proof of Theorem \ref{introthm:qgh-continuity}. The uniform fuzzy approximation which we are going to establish builds on a combination of the approximation results described in Section \ref{sec:berezin} and the continuity results obtained earlier for the Podle{\'s} sphere in \cite{AKK:Podcon} and \cite{GKK:QI}.

\subsection{Uniformity of the fuzzy approximation}\label{sec:uniform-fuzzy-approximation}

The core result of this section provides a uniform estimate on the Monge-Kantorovi\v{c} distance between the states $\chi_N^M$, $N,M \in \nn_0$, and the counit $\epsilon$. This estimate takes place on a fixed spectral band and the main part of the upper bound is given in terms of the Monge-Kantorovi\v{c} distance between states on the Podle\'s sphere. One of the relevant states is the restriction of the counit  while the remaining states on $S_q^2$ are all of the following form:
\begin{align}\label{eq:def-hj}
h_j \colon C(S_q^2) \to \cc \q h_j(x) := \inn{j+1}_q \cd h\big((a^*)^jxa^j \big) \q j \in \nn_0 .
\end{align}
We emphasise that the state $h_j$ is the restriction of the state $\chi_j^0$ (see \eqref{eq:xidef}) to the Podle\'s sphere $C(S_q^2) \su C(SU_q(2))$. 
We are interested in the algebraic versions of the Monge-Kantorovi\v{c} metrics on quantum $SU(2)$ and the Podle{\'s} sphere defined by
\begin{align*}
d_{t,q}(\mu,\nu)&:=\sup\{|\mu(x)-\nu(x)| \mid x\in \C O(SU_q(2)), L_{t,q}(x)\leq 1\}, \qquad \mu,\nu \in \C S(C(SU_q(2)))\\
d_q^0(\mu,\nu)&:=\sup\{|\mu(x)-\nu(x)| \mid x\in \C O(S_q^2), L_q^0(x)\leq 1\}, \qquad \qquad \hspace{0.15cm} \mu,\nu \in \C S(C(S_q^2)) .
\end{align*}
 We recall from Proposition \ref{p:podisomet} that the seminorm $L_q^0 \colon \C O(S_q^2) \to [0,\infty)$ agrees with the restriction of the seminorm $L_{t,q} \colon \C O(SU_q(2)) \to [0,\infty)$ to $\C O(S_q^2)$ for all values of $t \in (0,1]$.

\begin{lemma}\label{l:leftriga}
Let $m \in \zz$ and $t,q \in (0,1]$. For every $x \in \C A^m_q$ it holds that
\[
\begin{split}
L_q^0\big( (a^*)^m x \big) 
& \leq ( t^{1/2} + t^{-1/2} + 1 ) L_{t,q}(x) \q \mbox{for } m \geq 0  \, \, \mbox{ and} \\
L_q^0\big( x a^{-m} \big)  
& \leq ( t^{1/2} + t^{-1/2} + 1 ) L_{t,q}(x) \q \mbox{for } m \leq 0 .
\end{split}
\]
\end{lemma}
\begin{proof}
We focus on the case where $m \geq 0$ since the remaining case follows by taking adjoints. Suppose thus that $m \geq 0$ and let $x \in \C A^m_q$. We know that $(a^*)^m x \in \C A^0_q$ and an application of Proposition \ref{p:podisomet} shows that $L_q^0\big( (a^*)^m x\big) = L_{t,q}\big( (a^*)^m x\big)$. In particular, we immediately obtain the relevant inequality for $m = 0$. We may thus assume that $m > 0$. Since $(a^*)^m = u_{00}^m$, it follows from \eqref{eq:exppai} that 
\[
\pa_e( (a^*)^m) = 0 \q \T{and} \q \pa_f( (a^*)^m) = u_{01}^m \cd \sqrt{ q^{1-m} \inn{m}_q} .
\]
As a consequence of these identities, we get the estimate 
\[
\| \pa_q^H( (a^*)^m) \| = \| q^{-1/2} \pa_f( (a^*)^m) \| \leq \sqrt{q^{-m} \cd \inn{m}_q} \leq \sqrt{m} \cd q^{-m/2} . 
\]
We moreover notice that Lemma \ref{l:lipdomnorm} and Lemma \ref{l:qnumest} imply the inequalities
\begin{equation}\label{eq:normlipspec}
\| x \| \leq \frac{1}{[m/2]_t} \cd L_{t,q}(x) \leq \frac{ t^{1/2} + t^{-1/2} }{m} \cd L_{t,q}(x) .
\end{equation}
Since $L_{t,q}\big( (a^*)^m x \big)=L_{q}^0\big( (a^*)^m x \big)=\|\pa_q^H\big( (a^*)^m x \big)\|$, the result of the lemma now follows   from Lemma \ref{l:twilip} together with Lemma \ref{l:sigmaonspec} and the estimate in \eqref{eq:praktisk-ulighed}:
\[
\begin{split}
L_{t,q}\big( (a^*)^m x \big)  
& \leq \| \pa_q^H( (a^*)^m) \| \cd q^{m/2} \| x \| + q^{m/2}\| (a^*)^m \| \cd \| \pa_q^H(x) \| \\
& \leq \frac{ t^{1/2} + t^{-1/2}}{ \sqrt{m}} \cd L_{t,q}(x) + L_{t,q}(x) \leq (t^{1/2} + t^{-1/2} + 1) \cd L_{t,q}(x) . \qedhere
\end{split}
\]
\end{proof}

Recall from  Section \ref{ss:image-of-berezin} the linear functionals $\varphi_{r,s} \colon C(SU_q(2)) \to \cc$, $r,s\in \nn_0$, given by
\[
\varphi_{r,s}(x) = h\big( (a^*)^s x a^r \big) .
\]
As noted in \eqref{eq:statphi}, for each $N,M \in \nn_0$, the state $\chi_N^M$ appearing in the definition of the Berezin transform $\be_N^M \colon C(SU_q(2)) \to C(SU_q(2))$ is then given by
\begin{align}\label{eq:chiNM-formula}
\chi_N^M = \frac{1}{M+1} \sum_{s,r = N}^{N + M} \sqrt{ \inn{r+1}_q \inn{s+1}_q} \cd \varphi_{r,s} .
\end{align}

We first describe the  linear functionals $\varphi_{r,s}$ in terms of the  states $h_j$ on the Podle{\'s} sphere introduced in \eqref{eq:def-hj}.

\begin{lemma}\label{l:functionals}
Let $r,s \in \nn_0$. For every $x \in C(SU_q(2))$ it holds that 
\[
\varphi_{r,s}(x) = \begin{cases}
\frac{1}{\inn{r+1}_q} \cd h_r\big( (a^*)^{s-r} \cd \Pi^L_{s-r}(x) \big) & s \geq r \\
\frac{1}{\inn{s+1}_q} \cd h_s\big( \Pi^L_{s-r}(x) \cd a^{r-s} \big) & r \geq s 
\end{cases} .
\]
\end{lemma}
\begin{proof}
By continuity and linearity, we may assume that $x \in A^m_q$ for some $m \in \zz$. Since the  Haar state $h \colon C(SU_q(2)) \to \cc$ vanishes on all but the zeroth spectral subspace and $(a^{*})^sxa^r\in A_q^{m+r-s}$ we then have that $\varphi_{r,s}(x) \neq 0$ if and only if $m + r - s = 0$. Since $\Pi^L_{s-r}(x) $ also vanishes for $m + r - s \neq 0$,
we may assume that $m = s - r$. For $m \geq 0$ we  have that 
\[
\varphi_{r,s}(x) = h\big( (a^*)^s x a^r \big) = h\big( (a^*)^r (a^*)^m x a^r \big) 
= \frac{1}{\inn{r+1}_q} h_r\big( (a^*)^m x \big) .
\]
Likewise, for $m \leq 0$ we get that
\[
\varphi_{r,s}(x) = h\big( (a^*)^s x a^r \big) = h\big( (a^*)^s x a^{-m} a^s \big) 
= \frac{1}{\inn{s+1}_q} h_s\big( x a^{-m} \big) .
\]
This proves the present lemma.
\end{proof}

Inspired by Lemma \ref{l:functionals}, for each $m \in \nn_0$, we now define the bounded operator $P_m \colon C(SU_q(2)) \to C(S_q^2)$ by the formula
\[
P_m(x) := \fork{cc}{ 
(a^*)^m \Pi^L_m(x) + \Pi^L_{-m}(x) a^m & m > 0 \\
\Pi^L_0(x) & m = 0 } .
\]
 Indeed, for every $s,r \in \nn_0$ with $r < s$ we get from Lemma \ref{l:functionals} that
  \begin{equation}\label{eq:sumvarphi}
    \varphi_{s,r}(x) + \varphi_{r,s}(x) = \frac{1}{\inn{r+1}_q} \cd h_r \big( P_{s-r}(x) \big) \, \, \T{ and }  \, \, \,
    \varphi_{r,r}(x) = \frac{1}{\inn{r+1}_q} \cd h_r( P_0(x)) ,
\end{equation}
for all $x \in C(SU_q(2))$. Note also that $P_m(x^*)=P_m(x)^*$ since $\Pi_m(x^*)=\Pi_{-m}(x)^*$ for all $x\in C(SU_q(2))$ and $m\in \nn_0$. For each $N,M \in \nn_0$ we may then express the state $\chi_N^M \colon C(SU_q(2)) \to \cc$ in terms of the bounded operators $P_m$,  $m \in \nn_0$, and the states $h_r \colon C(S_q^2) \to \cc$, $r \in \nn_0$:

\begin{lemma}\label{l:state}
Let $N,M \in \nn_0$. For every $x \in C(SU_q(2))$, it holds that
\[
\chi_N^M(x) = \frac{1}{M+1} \sum_{r = N}^{N + M} \sum_{m = 0}^{N + M - r} \sqrt{ \frac{ \inn{m+r+1}_q}{\inn{r+1}_q}} \cd h_r\big( P_m(x) \big) .
\]
\end{lemma}
\begin{proof}
Using \eqref{eq:chiNM-formula} and \eqref{eq:sumvarphi}, we obtain the desired result from the computation
\[
\begin{split}
 (M + 1) \cd \chi_N^M & =  \sum_{r = N}^{N + M}  \inn{r+1}_q \cd \varphi_{r,r} +  \sum_{r = N}^{N + M} \sum_{s = r + 1}^{N + M} \sqrt{ \inn{r+1}_q \inn{s+1}_q} \cd (\varphi_{s,r} +\varphi_{r,s}) \\
  & = \sum_{r = N}^{N + M} \sum_{s = r}^{N+M}  \sqrt{ \frac{\inn{s+1}_q}{\inn{r+1}_q} } \cd (h_r \ci P_{s-r} ) . \qedhere
  \end{split}
\]
\end{proof}

In order to estimate the distance between the counit $\epsilon$ and the state $\chi_N^M$ for different values of $N,M \in \nn_0$ we introduce the intermediary linear functional $\psi_N^M \colon C(SU_q(2)) \to \cc$ defined by
\[
\psi_N^M(x) = \frac{1}{M+1} \sum_{r = N}^{N + M} \sum_{m = 0}^{N + M - r} \sqrt{ \frac{ \inn{m+r+1}_q}{\inn{r+1}_q}} \cd \epsilon\big( P_m(x) \big) 
\]
for all $x \in C(SU_q(2))$. 
The next two lemmas serve as preparation for Proposition \ref{p:mulspe}, where we provide a uniform upper bound on the Monge-Kantorovi\v{c} distance between the states $\epsilon$ and $\chi_N^M$ on a fixed spectral band.

\begin{lemma}\label{l:epsinter}
Let $n \in \zz$. It holds that
\begin{equation}\label{eq:interepsi}
\big| \psi_N^M(x) - \epsilon(x) \big| \leq 
\big( \tfrac{ 1}{N+1} + \tfrac{ 1}{M + 1} \big) \cd (t^{1/2} + t^{-1/2}) \cd L_{t,q}(x)
\end{equation}
for all $t,q \in (0,1]$, all $x \in \C A_q^n$ and all $N,M \in \nn_0$ with $M \geq |n|$.
\end{lemma}
\begin{proof}
  Since $\psi_N^M$ and $\epsilon$  respect the adjoint operation and $L_{t,q}$ is $*$-invariant, it suffices to treat the case $n\geq 0$. For $n = 0$, the estimate in \eqref{eq:interepsi} clearly holds since the left hand side of the inequality is equal to zero. Suppose therefore that $n > 0$, and let $q\in (0,1]$ be given. For each $r \in \nn_0$ we start out by remarking that 
\begin{equation}\label{eq:quotientest}
\begin{split}
\Big| \tfrac{\sqrt{\inn{n + r + 1}_q}}{\sqrt{\inn{r+1}_q}} - 1  \Big| 
  \leq \tfrac{\inn{n + r + 1}_q}{\inn{r+1}_q} - 1 = q^{2(r+1)} \tfrac{\inn{n}_q}{\inn{r+1}_q} 
  \leq n \cd q^2  \tfrac{1}{\sum_{i = 0}^r q^{-2i}} \leq \tfrac{n}{r+1}  .
\end{split}
\end{equation}
Fix now $N,M \in \nn_0$ with $M \geq n$. From the above inequalities we obtain that
\[
\begin{split}
 \Big| 1 - \frac{1}{M+1} \sum_{r = N}^{N + M - n} \tfrac{\sqrt{\inn{n + r + 1}_q}}{\sqrt{\inn{r+1}_q}}  \Big| & \leq \Big| 1 - \tfrac{M + 1 - n}{M+1} \Big|
+ \frac{1}{M+1} \sum_{r = N}^{N + M - n} \Big| \tfrac{\sqrt{\inn{n + r + 1}_q}}{\sqrt{\inn{r+1}_q}} - 1  \Big|\\
&\leq \tfrac{n}{M+1} + \tfrac{n}{N+1} .
\end{split}
\]
Let furthermore $x \in \C A_q^n$ be given.  Since $\epsilon(a^*) = 1 = \epsilon(a)$ (and since $\epsilon$ is a unital $*$-homomorphism) we know that 
\[
\epsilon( P_m(x)) = \de_{n,m} \cd \epsilon(x) \q \T{for all } m \in \nn_0 .
\]
From this identity, we then get that
\[
\begin{split}
\psi_N^M(x) 
& = \frac{1}{M+1}\sum_{r = N}^{N+M} \sum_{m = 0}^{N + M - r} \tfrac{\sqrt{\inn{m + r + 1}_q}}{\sqrt{\inn{r+1}_q}} \cd \epsilon(P_m(x)) 
 = \frac{1}{M+1}\sum_{r = N}^{N+M -n} \tfrac{\sqrt{\inn{n + r + 1}_q}}{\sqrt{\inn{r+1}_q}} \cd \epsilon(x) .
\end{split}
\]
Combining the above estimates we get
\[
\left| \psi_N^M(x)- \epsilon(x)\right|= \Big| \big(1 - \frac{1}{M+1}\sum_{r = N}^{N+M -n} \tfrac{\sqrt{\inn{n + r + 1}_q}}{\sqrt{\inn{r+1}_q}}\big) \epsilon(x)\Big|\leq ( \tfrac{n}{M+1} + \tfrac{n}{N+1})|\epsilon(x)| .
\] 
Let finally $t \in (0,1]$ be given. The result of the lemma then follows from the above computations together with the estimate
\[
|\epsilon(x)| \leq  \| x \| \leq  \frac{1}{[n/2]_t} L_{t,q}(x) \leq \frac{t^{1/2} + t^{-1/2}}{n}L_{t,q}(x), 
\]
see Lemma \ref{l:lipdomnorm} and Lemma \ref{l:qnumest}.
\end{proof}


\begin{lemma}\label{l:chiinter}
Let $n \in \zz$. The following inequality holds
\[
\begin{split}
\big| \chi_N^M(x) - \psi_N^M(x) \big|
& \leq 
\big( 1 + \tfrac{|n|}{N+1} \big)^{1/2} \cd (t^{1/2} + t^{-1/2} + 1) 
 \cd \sup_{N\leq r  \leq N+M} d_q^0\big(h_r,\epsilon|_{C(S_q^2)}\big) \cd L_{t,q}(x) 
\end{split}
\]
for all $t,q \in (0,1]$, all $x \in \C A^n_q$ and all $N,M \in \nn_0$ with $M \geq |n|$.
\end{lemma}
\begin{proof}
Let $t,q \in (0,1]$, $x \in \C A^n_q$ and $N,M \in \nn_0$ with $M \geq |n|$ be given. Since $\psi_N^M$ and $\chi_N^M$ preserve the adjoint operation and $L_{t,q}$ is $*$-invariant we may, without loss of generality, assume that $n\geq 0$. As in \eqref{eq:quotientest} we have that 
\[
\tfrac{\inn{n + r + 1}_q}{\inn{r+1}_q} \leq 1 + \tfrac{n}{r + 1} \q \T{for all } r \in \nn_0 .
\]
 Remark moreover that $P_m(x) = \de_{n,m} \cdot (a^*)^n x$ for all $m \in \nn_0$. An application of these observations together with Lemma \ref{l:state} yield the following inequalities:
\[
\begin{split}
\big|  \chi_N^M(x) - \psi_N^M(x) \big|
& = \frac{1}{M+1}\Big| 
\sum_{r = N}^{N + M - n} \tfrac{\sqrt{\inn{n + r + 1}_q}}{\sqrt{\inn{r+1}_q}} 
\cd \big( h_r( P_{n}(x)) - \epsilon(P_{n}(x)) \big) \Big| \\
& \leq \frac{1}{M+1} \sum_{r = N}^{N + M - n}\big( 1 + \tfrac{n}{r + 1} \big)^{1/2} 
d_q^0\big( h_r, \epsilon|_{C(S_q^2)} \big) \cd L_{q}^0(P_{n}(x)) \\
& \leq \big( 1 + \tfrac{n}{N + 1} \big)^{1/2} 
\cd \sup_{N\leq r  \leq N+M} d_q^0\big(h_r,\epsilon|_{C(S_q^2)}\big)
\cd L_{q}^0(P_{n}(x)).
\end{split}
\]
The result of the present lemma now follows by noting that Lemma \ref{l:leftriga} entails the inequality $L_{q}^0(P_{n}(x)) \leq (t^{1/2} + t^{-1/2} + 1) \cd L_{t,q}(x)$.
\end{proof}


\begin{prop}\label{p:mulspe}
Let $K \in \nn_0$ and  $\de \in (0,1)$. There exist a constant $C > 0$ and a  positive null sequence $(\ep_{N,M} )_{N,M = 0}^\infty$ such that
\[
d_{t,q}\big(\chi_N^M|_{B_q^K}, \epsilon|_{B_q^K}\big) \leq C \cd\sup_{N\leq r  \leq N+M} d_q^0\big(h_r,\epsilon|_{C(S_q^2)}\big)
+ \ep_{N,M}
\]
for all $(t,q) \in [\de,1] \ti (0,1]$ and all $N,M \in \nn_0$ with $M \geq K$.
\end{prop}
\begin{proof}
We define the constant $C > 0$ by putting
\[
C := (2K+1) \cd ( 1 + K)^{1/2} \cd (\de^{1/2} + \de^{-1/2} + 1),
\]
and the null sequence $(\ep_{N,M} )_{N,M = 0}^\infty$ by putting
\[
\ep_{N,M} := (2K+1) \cd \big(  \tfrac{1}{N+1} + \tfrac{1}{M+1}\big) \cd (\de^{1/2} + \de^{-1/2}) 
\]
for all $N,M \in \nn_0$.\\
Let now $(t,q) \in [\de,1] \ti (0,1]$ be given and let $x \in \C B_q^K$ satisfy that $L_{t,q}(x) \leq 1$. For every $N,M \in \nn_0$ with $M \geq K$, an application of Lemma \ref{l:epsinter} and Lemma \ref{l:chiinter} then shows that
\[
\begin{split}
& \big| \chi_N^M(x) - \epsilon(x) \big|  \leq \big| \chi_N^M(x) - \psi_N^M(x) \big| + \big| \psi_N^M(x) - \epsilon(x) \big| \\
& \q \leq \sum_{n = -K}^K \big| \chi_N^M(\Pi_n^L(x)) - \psi_N^M(\Pi_n^L(x)) \big|
+ \sum_{n = -K}^K \big| \psi_N^M( \Pi_n^L(x)) - \epsilon( \Pi_n^L(x) ) \big|  \\
& \q \leq (2K+1)  \big( 1 + \tfrac{K}{N+1} \big)^{1/2}  (t^{1/2} + t^{-1/2} + 1)  \cd \sup_{N\leq r  \leq N+M} d_q^0\big(h_r,\epsilon|_{C(S_q^2)}\big)\cd L_{t,q}( \Pi_n^L(x)) \\
&  \q \ \ \ +  (2K + 1) \cd \big(  \tfrac{1}{N+1} + \tfrac{1}{M+1} \big) \cd  (t^{1/2} + t^{-1/2}) \cd L_{t,q}( \Pi_n^L(x) ) \\
& \q \leq C \cd \sup_{N\leq r  \leq N+M} d_q^0\big(h_r,\epsilon|_{C(S_q^2)}\big)
+ \ep_{N,M}, 
\end{split}
\]
where the last estimate follows from  Corollary \ref{c:diracinvariant}. This proves the present proposition.
\end{proof}

The next proposition follows by an application of the estimate from Proposition \ref{p:mulspe} together with the core technical result from \cite{AKK:Podcon}.

\begin{prop}\label{p:mulspeII}
 Let $\de \in (0,1)$, $q_0 \in (0,1]$ and $K \in \nn_0$. For every $\ep>0$ there exist an open interval $I$ containing $q_0$ and $N_0,M_0 \in \nn_0$ with $M_0 \geq K$ such that
\[
d_{t,q}\big( \chi_{N_0}^{M_0} |_{B_q^K}, \epsilon|_{B_q^K} \big) < \ep
\] 
for all  $q \in I\cap [\de,1]$ and all $t \in [\de,1]$.
\end{prop}
\begin{proof}
By \cite[Lemma 4.11]{AKK:Podcon}, for every $r \in \nn_0$, we may choose a continuous function  $H_r \colon [\de,1] \to [0,\infty)$ such that 
\[
d_q^0\big(h_r, \epsilon|_{C(S_q^2)} \big) \leq H_r(q) \q \T{for all }  q\in [\de, 1] .
\]
Moreover, by \cite[Lemma 4.12]{AKK:Podcon} we may arrange that $\lim_{r \to \infty} H_r(q_0) = 0$.\\
Let us choose the constant $C > 0$ and the  positive null sequence $(\ep_{N,M})_{N,M = 0}^\infty$ according to Proposition \ref{p:mulspe} with $\de \in (0,1)$ and  $K \in \nn_0$ as given in the statement of the present proposition. Let now $\ep > 0$ be given. Choose $N_0 \geq K$ such that  $\ep_{N,M} < \ep/2$ for all $N,M \geq N_0$ and $H_r(q_0) < \tfrac{\ep}{4 C}$ for all $r \geq N_0$. Since the function $H_r$ is continuous for all $r \in \nn_0$, we may choose our open interval $I$ containing $q_0$ such that
\[
\big| H_r(q_0) - H_r(q) \big| < \tfrac{\ep}{4 C} \q \T{for all } q \in I \cap  [\de,1] \T{ and all } r \in \{N_0,N_0 + 1,\ldots,2N_0\} .
\] 
We now put $M_0 := N_0$ and it then follows from Proposition \ref{p:mulspe} that
\[
d_{t,q}\big( \chi_{N_0}^{M_0} |_{B_q^K}, \epsilon|_{B_q^K} \big) \leq C \cd \sup_{N_0\leq r\leq 2N_0 } H_r(q) + \ep_{N_0,N_0} < \ep
\]
for all $q \in I \cap  [\de,1]$ and all  $t \in [\de,1]$.
\end{proof}

\subsection{Continuity of quantum $SU(2)$}
We are now ready to assemble all the information gathered in the previous sections to obtain a proof of our main continuity result, Theorem \ref{introthm:qgh-continuity} from the introduction.

\begin{theorem}\label{thm:continuity-of-quantum-su2}
The $2$-parameter family of compact quantum metric spaces $\big( C(SU_q(2)), L_{t,q} \big)_{(t,q) \in (0,1]\times (0,1]}$ varies continuously in the quantum Gromov-Hausdorff distance.
\end{theorem}
As noted earlier, since 
\[
\T{dist}_{\T{Q}}\big((C(SU_q(2)), L_{t,q}), (C(SU_q(2)), L_{t,q}^{\max})\big) = 0
\]
by Corollary \ref{cor:dist-zero}, the above theorem also holds true for $L_{t,q}^{\max}$ instead of $L_{t,q}$. 

\begin{proof}
  Let $(t_0,q_0) \in (0,1]\times (0,1]$ and $\ep > 0$ be given and put $\de:=\min\{q_0/2, t_0/2\}$. By Proposition \ref{p:fuzzber},  $\beta_N^M(B_q^K)\subseteq \T{Fuzz}_{N+M}(B_q^K)$  for all $N,M,K \in \nn_0$ with $M \geq K$. Applying Proposition \ref{p:fuzztosuq2}, we may choose $K_0 \in \nn_0$ and a constant $C > 0$ such that 
\[
\T{dist}_{\T{Q}}\big( (\T{Fuzz}_{N+M}(B_q^{K_0}),L_{t,q}) ; (C(SU_q(2)), L_{t,q}^{\max}) \big) \leq d_{t,q}^{\max}(\chi_N^M\vert_{B_q^{K_0}},\epsilon\vert_{B_q^{K_0}}) \cd  C +  \ep/6
\]
for all $N,M\in \nn_0$  with $M \geq K_0$ and all $t,q\in [\de,1]$. By  Corollary \ref{cor:metrics-agree-on-state-space-of-bands} and Proposition \ref{p:mulspeII}, there exist  $N_0,M_0\in \nn_0$ with $M_0 \geq K_0$ and an open interval $I$ with $q_0 \in I$ such that
\[
d_{t,q}^{\max}\big(\chi_{N_0}^{M_0}\vert_{B_q^{K_0}},\epsilon\vert_{B_q^{K_0}}\big)=d_{t,q}\big(\chi_{N_0}^{M_0}\vert_{B_q^{K_0}},\epsilon\vert_{B_q^{K_0}}\big) <  \frac{\ep}{6C}
\]
for all $q \in I \cap [\de,1]$ and all $t\in [\de,1]$.  Hence
\[
\T{dist}_{\T{Q}}\big((\T{Fuzz}_{N_0+M_0}(B_q^{K_0}),L_{t,q}); (C(SU_q(2)), L_{t,q}^{\max})\big) < \frac{\ep}{3}
\]
for all $q \in I \cap [\de,1]$ and all $t\in [\de,1]$.  Note, at this point, that $V := [\de,1] \ti ( I \cap [\de,1] ) \su (0,1] \ti (0,1]$ is a neighbourhood of the point $(t_0,q_0)$. By Proposition \ref{p:fincont}, the compact quantum metric spaces $\big( \T{Fuzz}_{N_0+M_0}(B_q^{K_0}) , L_{t,q} \big)_{t,q\in (0,1]}$ vary continuously in the quantum Gromov-Hausdorff distance, so we may  choose a neighbourhood $U$ of $(t_0,q_0) \in (0,1] \ti (0,1]$ such that
\[
\T{dist}_{\T Q}\big(  (\T{Fuzz}_{N_0+M_0}(B_q^{K_0}),L_{t,q}) ;   ( \T{Fuzz}_{N_0+M_0}(B_{q_0}^{K_0}), L_{t_0,q_0}) \big)<\ep/3
\]
for all $(t,q) \in U$. An application of the triangle inequality for the quantum Gromov-Hausdorff distance \cite[Theorem 4.3]{Rie:GHD}, now yields the estimate
\[
\T{dist}_{\T{Q}}\big( ( C(SU_q(2)),L_{t,q}^{\max}); (C(SU_{q_0}(2)),L_{t_0,q_0}^{\max}) \big) < \ep
\]
for all  $(t,q)\in U \cap V$, thus completing the proof.
\end{proof}

As a last remark we single out the following special case of the above  theorem: As the deformation parameter $q$ tends to $1$, the quantum metric spaces $SU_q(2)$ converge towards $SU(2)$ equipped with its classical round metric rescaled with a factor $2$. To make this statement precise, recall from Section \ref{sec:comparison-with-classical}, that we denote by $d_{S^3}$ the usual round metric on $SU(2)\cong S^3$. We then have the Lip-norm $L_{\T{Lip}}$ which to any continuous function $f \colon SU(2) \to \cc$ assigns the Lipschitz constant with respect to the rescaled metric $2 \cd d_{S^3}$. Comparing with Theorem \ref{t:classical}, the special case of Theorem \ref{thm:continuity-of-quantum-su2} then reads as follows:

\begin{cor}\label{cor:convergence-to-classical-SU2}
As $(t,q) \in (0,1] \ti (0,1]$ tends to $(1,1)$, the quantum metric spaces $(C(SU_q(2)), L_{t,q})$ converge in quantum Gromov-Hausdorff distance to $\big(C(SU(2)), L_{\T{Lip}}\big)$.
\end{cor}

\newcommand{\etalchar}[1]{$^{#1}$}
\providecommand{\bysame}{\leavevmode\hbox to3em{\hrulefill}\thinspace}
\providecommand{\MR}{\relax\ifhmode\unskip\space\fi MR }
\providecommand{\MRhref}[2]{%
  \href{http://www.ams.org/mathscinet-getitem?mr=#1}{#2}
}
\providecommand{\href}[2]{#2}


\begin{thebibliography}{CNNR11}

\bibitem[Agu17]{aguilar:thesis}
Konrad Aguilar, \emph{Quantum {M}etrics on {A}pproximately
  {F}inite-{D}imensional {A}lgebras}, ProQuest LLC, Ann Arbor, MI, 2017, Thesis
  (Ph.D.)--University of Denver. \MR{3705997}

\bibitem[AK18]{AgKa:PSM}
Konrad Aguilar and Jens Kaad, \emph{The {P}odle\'{s} sphere as a spectral
  metric space}, J. Geom. Phys. \textbf{133} (2018), 260--278. \MR{3850270}

\bibitem[AKK21a]{AKK:Podcon}
Konrad Aguilar, Jens Kaad, and David Kyed, \emph{The {P}odle{s} spheres
  converge to the sphere}, Comm. Math. Phys. \textbf{392} (2022), 1029--1061. \MR{4426737}



\bibitem[AKK21b]{AKK:Polyapprox}
\bysame, \emph{Polynomial approximation of quantum {L}ipschitz functions}, Doc.
  Math. \textbf{27} (2022), 765--787. \MR{4432528}
  
  

\bibitem[ARS99]{AlReSc:NWG}
Anton~Yurevich Alekseev, Andreas Recknagel, and Volker Schomerus,
  \emph{Non-commutative world-volume geometries: branes on {$\rm SU(2)$} and
  fuzzy spheres}, J. High Energy Phys. (1999), no.~9, Paper 23, 20.
  \MR{1720690}

\bibitem[BK97]{BiKu:DQQ}
Peter~N. Bibikov and Peter~P. Kulish, \emph{Dirac operators on the quantum
  group {${\rm SU}_q(2)$} and the quantum sphere}, Zap. Nauchn. Sem.
  S.-Peterburg. Otdel. Mat. Inst. Steklov. (POMI) \textbf{245} (1997),
  no.~Vopr. Kvant. Teor. Polya i Stat. Fiz. 14, 49--65, 283. \MR{1627837}

\bibitem[BKM18]{BlKaMe:OCA}
David Blecher, Jens Kaad, and Bram Mesland, \emph{Operator
  {$\ast$}-correspondences in analysis and geometry}, Proc. Lond. Math. Soc.
  (3) \textbf{117} (2018), no.~2, 303--344. \MR{3851325}

\bibitem[Bla96]{Bla:DCH}
\'Etienne Blanchard, \emph{D\'eformations de {$C^*$}-alg\`ebres de {H}opf},
  Bull. Soc. Math. France \textbf{124} (1996), no.~1, 141--215. \MR{1395009}

\bibitem[BMT01]{BMT:comamenability}
Erik {B\'edos}, Gerard~J. {Murphy}, and Lars {Tuset}, \emph{{Co-amenability of
  compact quantum groups}}, {J. Geom. Phys.} \textbf{40} (2001), no.~2,
  130--153 (English).

\bibitem[BS93]{BaSk:UMD}
Saad Baaj and Georges Skandalis, \emph{Unitaires multiplicatifs et dualit\'{e}
  pour les produits crois\'{e}s de {$C^*$}-alg\`ebres}, Ann. Sci. \'{E}cole
  Norm. Sup. (4) \textbf{26} (1993), no.~4, 425--488. \MR{1235438}

\bibitem[BVZ15]{BCZ:CQM}
Jyotishman {Bhowmick}, Christian {Voigt}, and Joachim {Zacharias},
  \emph{{Compact quantum metric spaces from quantum groups of rapid decay}},
  {J. Noncommut. Geom.} \textbf{9} (2015), no.~4, 1175--1200 (English).

\bibitem[CM95]{CoMo:LIF}
Alain Connes and Henri Moscovici, \emph{The local index formula in
  noncommutative geometry}, Geom. Funct. Anal. \textbf{5} (1995), no.~2,
  174--243. \MR{1334867}

\bibitem[CM08]{CoMo:TST}
\bysame, \emph{Type {III} and spectral triples}, Traces in number theory,
  geometry and quantum fields, Aspects Math., E38, Friedr. Vieweg, Wiesbaden,
  2008, pp.~57--71. \MR{2427588}

\bibitem[CNNR11]{CNNR:TEK}
Alan~L. Carey, Sergey Neshveyev, Ryszard Nest, and Adam Rennie, \emph{Twisted
  cyclic theory, equivariant {$KK$}-theory and {KMS} states}, J. Reine Angew.
  Math. \textbf{650} (2011), 161--191. \MR{2770560}

\bibitem[Con85]{Con:NDG}
Alain Connes, \emph{Noncommutative differential geometry}, Inst. Hautes
  \'Etudes Sci. Publ. Math. (1985), no.~62, 257--360. \MR{823176 (87i:58162)}

\bibitem[Con89]{Con:CFH}
\bysame, \emph{Compact metric spaces, {F}redholm modules, and hyperfiniteness},
  Ergodic Theory Dyn. Syst. \textbf{9} (1989), no.~2, 207--220. \MR{1007407}

\bibitem[Con94]{Con:NCG}
\bysame, \emph{Noncommutative geometry}, Academic Press, Inc., San Diego, CA,
  1994. \MR{1303779 (95j:46063)}

\bibitem[Con96]{Con:GFN}
\bysame, \emph{Gravity coupled with matter and the foundation of
  non-commutative geometry}, Comm. Math. Phys. \textbf{182} (1996), no.~1,
  155--176. \MR{1441908}

\bibitem[CP10]{CP:EST}
Partha~Sarathi {Chakraborty} and Arupkumar {Pal}, \emph{{Equivariant spectral
  triples and Poincar\'e duality for \(SU_q(2)\)}}, {Trans. Am. Math. Soc.}
  \textbf{362} (2010), no.~8, 4099--4115 (English).

\bibitem[CvS21]{walter-connes:truncations}
Alain Connes and Walter van Suijlekom, \emph{{Spectral truncations in
  noncommutative geometry and operator systems}}, {Commun. Math. Phys.}
  \textbf{383} (2021), no.~3, 2021--2067 (English).

\bibitem[DLS{\etalchar{+}}05]{DLSSV:DOS}
Ludwik D\c{a}browski, Giovanni Landi, Andrzej Sitarz, Walter van Suijlekom, and
  Joseph~C. V\'{a}rilly, \emph{The {D}irac operator on {${\rm SU}_q(2)$}},
  Comm. Math. Phys. \textbf{259} (2005), no.~3, 729--759. \MR{2174423}

\bibitem[DS03]{DaSi:DSP}
Ludwik D\c{a}browski and Andrzej Sitarz, \emph{Dirac operator on the standard
  {P}odle\'{s} quantum sphere}, Noncommutative geometry and quantum groups
  ({W}arsaw, 2001), Banach Center Publ., vol.~61, Polish Acad. Sci. Inst.
  Math., Warsaw, 2003, pp.~49--58. \MR{2024421}

\bibitem[Edw75]{edwards-GH-paper}
David~A. Edwards, \emph{The structure of superspace}, Studies in topology
  ({P}roc. {C}onf., {U}niv. {N}orth {C}arolina, {C}harlotte, {N}. {C}., 1974;
  dedicated to {M}ath. {S}ect. {P}olish {A}cad. {S}ci.), 1975, pp.~121--133.
  \MR{0401069}

\bibitem[Fri00]{Friedrich:Dirac}
Thomas Friedrich, \emph{Dirac operators in {R}iemannian geometry}, Graduate
  Studies in Mathematics, vol.~25, American Mathematical Society, Providence,
  RI, 2000, Translated from the 1997 German original by Andreas Nestke.
  \MR{1777332}

\bibitem[GKK21]{GKK:QI}
Thomas {Gotfredsen}, Jens {Kaad}, and David {Kyed}, \emph{{Gromov-Hausdorff
  convergence of quantised intervals}}, {J. Math. Anal. Appl.} \textbf{500}
  (2021), no.~2, 13 (English).

\bibitem[GMS01]{GrMaSt:FQFI}
Harald Grosse, John Madore, and Harold Steinacker, \emph{Field theory on the
  {$q$}-deformed fuzzy sphere. {I}}, J. Geom. Phys. \textbf{38} (2001),
  no.~3-4, 308--342. \MR{1829046}

\bibitem[GMS02]{GrMaSt:FQFII}
\bysame, \emph{Field theory on the {$q$}-deformed fuzzy sphere. {II}.
  {Q}uantization}, J. Geom. Phys. \textbf{43} (2002), no.~2-3, 205--240.
  \MR{1919211}

\bibitem[GP95]{GrPr:DFS}
Harald Grosse and Peter Pre\v{s}najder, \emph{The {D}irac operator on the fuzzy
  sphere}, Lett. Math. Phys. \textbf{33} (1995), no.~2, 171--181. \MR{1316346}

\bibitem[Gro81]{gromov-groups-of-polynomial-growth-and-expanding-maps}
Mikhael Gromov, \emph{Groups of polynomial growth and expanding maps}, Inst.
  Hautes \'{E}tudes Sci. Publ. Math. (1981), no.~53, 53--73. \MR{623534}

\bibitem[HM99]{HaMa:PMM}
Piotr~M. Hajac and Shahn Majid, \emph{Projective module description of the
  {$q$}-monopole}, Comm. Math. Phys. \textbf{206} (1999), no.~2, 247--264.
  \MR{1722149}

\bibitem[INT06]{IzNeTu:PBD}
Masaki Izumi, Sergey Neshveyev, and Lars Tuset, \emph{Poisson boundary of the
  dual of {${\rm SU}_q(n)$}}, Comm. Math. Phys. \textbf{262} (2006), no.~2,
  505--531. \MR{2200270}

\bibitem[Kaa17]{Kaa:DAH}
Jens Kaad, \emph{Differentiable absorption of {H}ilbert {$C^*$}-modules,
  connections, and lifts of unbounded operators}, J. Noncommut. Geom.
  \textbf{11} (2017), no.~3, 1037--1068. \MR{3713012}

\bibitem[Kaa21]{Kaa:UKM}
\bysame, \emph{The unbounded {K}asparov product by a differentiable module}, J.
  Noncommut. Geom. \textbf{15} (2021), no.~2, 423--487. \MR{4325713}

\bibitem[Kaa24]{Kaa24}
Jens Kaad, \emph{Noncommutative metric geometry of quantum circle bundles}, Preprint (2024), 1--47,
\href{https://arxiv.org/pdf/2410.03475}{arXiv:2410.03475}.

\bibitem[Ker03]{Ker:MQG}
David Kerr, \emph{Matricial quantum {G}romov-{H}ausdorff distance}, J. Funct.
  Anal. \textbf{205} (2003), no.~1, 132--167. \MR{2020211}

\bibitem[KK21]{KK:DCQ}
Jens {Kaad} and David {Kyed}, \emph{{Dynamics of compact quantum metric
  spaces}}, {Ergodic Theory Dyn. Syst.} \textbf{41} (2021), no.~7, 2069--2109
  (English).

\bibitem[KM24]{KM24}
Jens {Kaad} and Max Holst {Mikkelsen}, \emph{{Spectral metrics on quantum projective spaces}}, {J. Funct. Anal.} \textbf{247} (2024), no.~2, Paper No. 110466, 38







\bibitem[KL12]{KaLe:LGR}
Jens Kaad and Matthias Lesch, \emph{A local global principle for regular
  operators in {H}ilbert {$C\sp *$}-modules}, J. Funct. Anal. \textbf{262}
  (2012), no.~10, 4540--4569. \MR{2900477}

\bibitem[KL13]{KaLe:SFU}
\bysame, \emph{Spectral flow and the unbounded {K}asparov product}, Adv. Math.
  \textbf{248} (2013), 495--530. \MR{3107519}

\bibitem[KRS12]{KRS:RFH}
Ulrich Kr\"{a}hmer, Adam Rennie, and Roger Senior, \emph{A residue formula for
  the fundamental {H}ochschild 3-cocycle for {$SU_q(2)$}}, J. Lie Theory
  \textbf{22} (2012), no.~2, 557--585. \MR{2976934}

\bibitem[KS97]{KlSc:QGR}
Anatoli Klimyk and Konrad Schm\"{u}dgen, \emph{Quantum groups and their
  representations}, Texts and Monographs in Physics, Springer-Verlag, Berlin,
  1997. \MR{1492989}

\bibitem[KS12]{KaSe:TST}
Jens Kaad and Roger Senior, \emph{A twisted spectral triple for quantum
  {$SU(2)$}}, J. Geom. Phys. \textbf{62} (2012), no.~4, 731--739. \MR{2888978}

\bibitem[KV00]{kustermans-vaes-C*-lc}
Johan Kustermans and Stefaan Vaes, \emph{Locally compact quantum groups}, Ann.
  Sci. \'Ecole Norm. Sup. (4) \textbf{33} (2000), no.~6, 837--934.
  \MR{MR1832993 (2002f:46108)}

\bibitem[KW20]{KW:TDO}
Ulrich {Kr\"ahmer} and Elmar {Wagner}, \emph{{Twisted Dirac operator on quantum
  SU(2) in disc coordinates}}, Operator algebras, Toeplitz operators and
  related topics. Selected papers based on the presentations at the
  international workshop, Boca del Rio, Veracruz, Mexico, November 13--19,
  2018. In honor of Nikolai Vasilevskin on the occasion of his 70th birthday,
  Cham: Springer, 2020, pp.~233--253 (English).

\bibitem[Lat05]{Lat:AQQ}
Fr\'{e}d\'{e}ric Latr\'{e}moli\`ere, \emph{Approximation of quantum tori by
  finite quantum tori for the quantum {G}romov-{H}ausdorff distance}, J. Funct.
  Anal. \textbf{223} (2005), no.~2, 365--395. \MR{2142343}

\bibitem[Lat07]{Lat:BLD}
\bysame, \emph{Bounded-{L}ipschitz distances on the state space of a
  {$C^*$}-algebra}, Taiwanese J. Math. \textbf{11} (2007), no.~2, 447--469.
  \MR{2333358}

\bibitem[Lat15]{Lat:DGH}
\bysame, \emph{The dual {G}romov-{H}ausdorff propinquity}, J. Math. Pures Appl.
  (9) \textbf{103} (2015), no.~2, 303--351. \MR{3298361}

\bibitem[Lat16]{Lat:QGH}
\bysame, \emph{The quantum {G}romov-{H}ausdorff propinquity}, Trans. Amer.
  Math. Soc. \textbf{368} (2016), no.~1, 365--411. \MR{3413867}

\bibitem[Lat19]{Lat:MGP}
\bysame, \emph{The modular {G}romov-{H}ausdorff propinquity}, Dissertationes
  Math. \textbf{544} (2019), 70. \MR{4036723}

\bibitem[Lat22]{Lat:GPS}
\bysame, \emph{The {G}romov-{H}ausdorff propinquity for metric spectral
  triples}, Adv. Math. \textbf{404} (2022), Paper No. 108393. \MR{4411527}

\bibitem[Li03]{Li:CQG}
Hanfeng Li, \emph{{$C^*$}-algebraic quantum {G}romov--{H}ausdorff distance},
  Preprint (2003), 1--27, 
\href{https://arxiv.org/abs/math/0312003}{arXiv:math/0312003}.



\bibitem[Li06]{Li:GH-dist}
\bysame, \emph{Order-unit quantum {G}romov-{H}ausdorff distance}, J. Funct.
  Anal. \textbf{231} (2006), no.~2, 312--360. \MR{2195335}

\bibitem[Li09]{Li:ECQ}
\bysame, \emph{Compact quantum metric spaces and ergodic actions of compact
  quantum groups}, J. Funct. Anal. \textbf{256} (2009), no.~10, 3368--3408.
  \MR{2504529}

\bibitem[LM19]{MeLe}
Matthias Lesch and Bram Mesland, \emph{Sums of regular self-adjoint operators
  in {H}ilbert-{$C^*$}-modules}, J. Math. Anal. Appl. \textbf{472} (2019),
  no.~1, 947--980. \MR{3906406}

\bibitem[LP18]{LatPack:Solenoids}
Fr\'{e}d\'{e}ric Latr\'{e}moli\`ere and Judith Packer, \emph{Noncommutative
  solenoids}, New York J. Math. \textbf{24A} (2018), 155--191. \MR{3904875}

\bibitem[Mad92]{Mad:TFS}
John Madore, \emph{The fuzzy sphere}, Classical Quantum Gravity \textbf{9}
  (1992), no.~1, 69--87. \MR{1146044}

\bibitem[Mad99]{Mad:NCG}
\bysame, \emph{An introduction to noncommutative differential geometry and its
  physical applications}, second ed., London Mathematical Society Lecture Note
  Series, vol. 257, Cambridge University Press, Cambridge, 1999. \MR{1707285}

\bibitem[Maj00]{Maj:QNG}
Shahn Majid, \emph{Quantum groups and noncommutative geometry}, J. Math. Phys.
  \textbf{41} (2000), no.~6, 3892--3942. \MR{1768643}

\bibitem[Mes14]{Mes:UCN}
Bram Mesland, \emph{Unbounded bivariant {$K$}-theory and correspondences in
  noncommutative geometry}, J. Reine Angew. Math. \textbf{691} (2014),
  101--172. \MR{3213549}

\bibitem[MR16]{MeRe:NMU}
Bram Mesland and Adam Rennie, \emph{Nonunital spectral triples and metric
  completeness in unbounded {$KK$}-theory}, J. Funct. Anal. \textbf{271}
  (2016), no.~9, 2460--2538. \MR{3545223}

\bibitem[NT05]{NeTu:LFQ}
Sergey Neshveyev and Lars Tuset, \emph{A local index formula for the quantum
  sphere}, Comm. Math. Phys. \textbf{254} (2005), no.~2, 323--341. \MR{2117628}

\bibitem[NT10]{NeTu:DCQ}
\bysame, \emph{The {D}irac operator on compact quantum groups}, J. Reine Angew.
  Math. \textbf{641} (2010), 1--20. \MR{2643923}

\bibitem[OR05]{OzRi:Hyp}
Narutaka Ozawa and Marc~A. Rieffel, \emph{Hyperbolic group {$C^*$}-algebras and
  free-product {$C^*$}-algebras as compact quantum metric spaces}, Canad. J.
  Math. \textbf{57} (2005), no.~5, 1056--1079. \MR{2164594}

\bibitem[Pau02]{Pau:CBM}
Vern Paulsen, \emph{Completely bounded maps and operator algebras}, Cambridge
  Studies in Advanced Mathematics, vol.~78, Cambridge University Press,
  Cambridge, 2002. \MR{1976867}

\bibitem[{Pis}01]{Pisier:Similarity}
Gilles {Pisier}, \emph{{Similarity problems and completely bounded maps.
  Includes the solution to ``The Halmos problem''.}}, vol. 1618, Berlin:
  Springer, 2001 (English).

\bibitem[Pod87]{Pod:QS}
Piotr Podle\'{s}, \emph{Quantum spheres}, Lett. Math. Phys. \textbf{14} (1987),
  no.~3, 193--202. \MR{919322}

\bibitem[Rie98]{Rie:MSA}
Marc~A. Rieffel, \emph{Metrics on states from actions of compact groups}, Doc.
  Math. \textbf{3} (1998), 215--229. \MR{1647515}

\bibitem[Rie99]{Rie:MSS}
\bysame, \emph{Metrics on state spaces}, Doc. Math. \textbf{4} (1999),
  559--600. \MR{1727499}

\bibitem[Rie02]{Rieffel:group-C-star-algebras-as-cqms}
\bysame, \emph{Group {$C^*$}-algebras as compact quantum metric spaces}, Doc.
  Math. \textbf{7} (2002), 605--651. \MR{2015055}

\bibitem[Rie04a]{Rie:GHD}
\bysame, \emph{Gromov-{H}ausdorff distance for quantum metric spaces}, Mem.
  Amer. Math. Soc. \textbf{168} (2004), no.~796, 1--65, Appendix 1 by Hanfeng
  Li. \MR{2055927}

\bibitem[Rie04b]{Rie:MSG}
\bysame, \emph{Matrix algebras converge to the sphere for quantum
  {G}romov-{H}ausdorff distance}, Mem. Amer. Math. Soc. \textbf{168} (2004),
  no.~796, 67--91. \MR{2055928}

\bibitem[Sai09]{Sain:Thesis}
Jeremy Sain, \emph{Berezin quantization from ergodic actions of compact quantum
  groups, and quantum gromov-hausdorff distance}, Preprint (2009), 1--117,
  \href{https://arxiv.org/abs/0906.1829v1}{arXiv:0906.1829v1}.


\bibitem[Sch01]{Sch:BTQ}
Martin Schlichenmaier, \emph{Berezin-{T}oeplitz quantization and {B}erezin
  transform}, Long time behaviour of classical and quantum systems ({B}ologna,
  1999), Ser. Concr. Appl. Math., vol.~1, World Sci. Publ., River Edge, NJ,
  2001, pp.~271--287. \MR{1852228}

\bibitem[Tak03]{Takesaki-vol-II}
Masamichi Takesaki, \emph{Theory of operator algebras. {II}}, Encyclopaedia of
  Mathematical Sciences, vol. 125, Springer-Verlag, Berlin, 2003, Operator
  Algebras and Non-commutative Geometry, 6. \MR{1943006}

\bibitem[Tim08]{Timmermann-book}
Thomas Timmermann, \emph{An invitation to quantum groups and duality}, EMS
  Textbooks in Mathematics, European Mathematical Society (EMS), Z\"{u}rich,
  2008, From Hopf algebras to multiplicative unitaries and beyond. \MR{2397671}

\bibitem[Tom67]{Tom:AppFub}
Jun Tomiyama, \emph{Applications of {F}ubini type theorem to the tensor
  products of {$C^{\ast} $}-algebras}, Tohoku Math. J. (2) \textbf{19} (1967),
  213--226. \MR{218906}

\bibitem[vS21]{walter:GH-convergence}
Walter van Suijlekom, \emph{{Gromov-Hausdorff convergence of state spaces for
  spectral truncations}}, {J. Geom. Phys.} \textbf{162} (2021), 11 (English),
  Id/No 104075.

\bibitem[Wor87]{Wor:UAC}
Stanis\l{}aw~L. Woronowicz, \emph{Twisted {${\rm SU}(2)$} group. {A}n example
  of a noncommutative differential calculus}, Publ. Res. Inst. Math. Sci.
  \textbf{23} (1987), no.~1, 117--181. \MR{890482}

\bibitem[Wor98]{wor:cpqgrps}
\bysame, \emph{Compact quantum groups}, Sym\'{e}tries quantiques ({L}es
  {H}ouches, 1995), North-Holland, Amsterdam, 1998, pp.~845--884. \MR{1616348}

\end{thebibliography}

\end{document}